\documentclass[reqno, a4paper, 12pt]{amsart} 

 \usepackage{amsopn,amstext,amsbsy,amsmath,amscd,amsthm,amsfonts,amssymb}
 \usepackage{geometry} 
 \usepackage[latin9]{inputenc} 
 \usepackage[all]{xy} 
 \usepackage{enumitem}

 \usepackage{scalerel}
 
 \usepackage[toc,page]{appendix}
 
 \usepackage{graphicx}
 
 \usepackage{hyperref}
 
 \usepackage{epstopdf}
 
 \usepackage{xy}

 \theoremstyle{plain}
 \newtheorem{theorem}                 {\bf Theorem}   [section]   
 \newtheorem{proposition}  [theorem]  {\bf Proposition}
 \newtheorem{corollary}    [theorem]  {\bf Corollary}
 \newtheorem{lemma}        [theorem]  {\bf Lemma}
 
 \theoremstyle{definition}
 \newtheorem{example}      [theorem]  {\bf Example}
 \newtheorem{definition}   [theorem]  {\bf Definition}
 \newtheorem{remark}			[theorem] {\bf Remark}
 
 \allowdisplaybreaks

 \newcommand{\trace}{\operatorname{trace}}

 \def \zn{\mathbb Z}
 
 \def \rn{\mathbb R}

 \def \S{\mathcal S}
 
 \def \D{\mathcal D}
 
 \def \C{\mathcal C}
 
 \def \F{\mathcal F}
 \def \H{\mathcal H}
 
 \def \J{\mathcal J}

 \def \M{\mathcal M}
 \def \N{\mathcal N}
 \def \Ol{\mathcal O}
 \def \T{\mathcal T}

 \def \ol{
 	\mathchoice
 	{{\scriptstyle\mathcal{O}}}
 	{{\scriptstyle\mathcal{O}}}
 	{{\scriptscriptstyle\mathcal{O}}}
 	{\scalebox{.7}{$\scriptscriptstyle\mathcal{O}$}}
 }

 \def \v{\mathfrak{v}}

 \def \Hess{\text{Hess}}
 \def \diver{\text{div}}

 \def\one{\mbox{1\hspace{-4.25pt}\fontsize{12}{14.4}\selectfont\textrm{1}}}
 
 \def \nab#1#2{\hbox{$\nabla$\kern -.3em\lower 1.0 ex
 		\hbox{$#1$}\kern -.1 em {$#2$}}}

\begin{document} 

\author{David Lundberg}
\address{Uppsala University \\ Department of Mathematics \\
Box 480, 751 06 UPPSALA}
\email{david.lundberg@math.uu.se}

\title{On Jang's equation and the positive mass theorem for asymptotically hyperbolic initial data sets with dimensions above three and below eight  }


\begin{abstract}
We solve the Jang equation with respect to asymptotically hyperbolic ``hyperboloidal''  initial data in dimensions $n=4,5,6,7$. This gives a non-spinor proof of the positive mass theorem in the asymptotically hyperbolic setting in these dimensions. Our work extends an earlier result of \cite{SakovichPMTah} obtained in dimension $n=3$.
\end{abstract}

\date{\today}

\maketitle

\numberwithin{equation}{section}

\tableofcontents

\section{Introduction}\label{SectionIntroduction}

General Relativity has had a beautiful and fruitful history of interplay with Matematics and, in particular, both Geometry and Partial Differential Equations. A very important example of this is the classical Positive Mass (or Energy) Theorem. In physical terms, the theorem asserts that the mass of an isolated gravitational system with non-negative energy density is non-negative. At first glance, this may seem to be purely physical statement, but it is very geometrical and has far reaching mathematical consequences not only within General Relativity. 
\\ \indent Central to the Positive Energy Theorem is the notions of initial data sets $(M^n, g,k)$. Roughly speaking, this is a Riemannian manifold $(M^n,g)$ and a symmetric $(0,2)$-tensor $k$ that is thought of a ''constant time slice'' in some spacetime with $k$ as its second fundamental form. An initial data set $(M^n, g, k)$ is said to be asymptotically Euclidean if the metric tends to the Euclidean metric in a chart at infinity, that is $g\rightarrow \delta$ at a certain rate. Such initial data sets are used in General Relativity to model isolated systems. Defined by Arnowitt, Deser and Misner \cite{ADM} in 1959, the ADM energy is a flux-integral computed at infinity of $(M^n, g,k)$, which turns out to be a coordinate invariant. The question whether this quantity is non-negative under physically reasonable assumptions is known as the positive energy conjecture. Results proving this conjecture were obtained by Schoen and Yau in \cite{PMTI}, \cite{PMTII} and Witten \cite{Witten81}. The result of \cite{PMTI} was obtained for dimension $n=3$ and $k=0$ (the so-called Riemannian or time-symmetric case, where the dominant energy condition implies $R_g\geq 0$) using minimal surface techniques and the result of \cite{PMTII} was proven for general $k$ (the so-called spacetime case) by reduction to the $k=0$ case using the Jang equation. The result of \cite{Witten81} holds in all dimensions but requires the assumption that $(M^n ,g)$ be spin, which imposes additional restrictions on the topology in dimensions above $3$. The results of \cite{PMTI} and \cite{PMTII} have been extended to dimensions $3\leq n \leq 7$ in \cite{PMTV} and \cite{EichmairPMT}, respectively. 
\\ \indent Since the original work of Schoen and Yau in \cite{PMTI} and \cite{PMTII} the positivity of mass for asymptotically Euclidean manifolds and asyptotically Euclidean initial data sets have remained an active area of research. Very recently, many new methods have been introduced to the field, for example the level set methods (see for instance \cite{agostiniani2023greens}, \cite{BKKS19} and \cite{HKK20}) and $\mu$-bubbles of Gromov (see \cite{LUY21}). Furthermore, the analogue of the minimal surface tecnhique of \cite{PMTI} has been developed for initial data sets in \cite{EichmairHuangLeeSchoen}. We would also like to highlight the recent proofs of optimal rigidity results characterizing the case when the mass is zero, see \cite{HuangLeeRigidityI} and \cite{HuangLeeRigidityII}.
\\ \indent Another important class of initial data is the so-called asymptotically hyperbolic initial data which is characterized by $g\rightarrow b$, where $b$ is the hyperbolic metric. Two main models that are used to define asymptotically hyperbolic initial data are upper unit hyperboloid of Minkowski space, which is an umbilic hypersurface (that is $k=g$), and the $\{t=0\}$-slice of the anti-de Sitter spacetime, which is a totally geodesic hypersurface (that is $k=0$). Mass for asymptotically hyperbolic manifolds was first defined by Wang in \cite{WangPMT} and a positive mass theorem was proven under the assumption that $M^n$ be spin. The result in \cite{WangPMT} was subsequently extended by Chrusciel and Herzlich in \cite{ChruscielHerzlich} to the case of more general asymptotics. These Riemannian results can be interpreted as positive mass theorems for asymptotically anti-de Sitter initial data sets with $k=0$ alternatively asymptotically hyperboloidal initial data sets with $k=g$, where in both cases the dominant energy condition is equivalent to $R_g\geq -n(n-1)$. Positivity of mass for asymptotically hyperbolic manifolds $(M^n,g)$ was also proven in \cite{ACG07} in dimensions $3\leq n \leq7$ under aditional assumptions on the geometry at infinity. These assumptions have recently been removed in \cite{Chrusciel2018OnTM}. 
\\ \indent As already mentioned, in the spacetime hyperbolic setting we can have either $k\rightarrow 0$ or $k\rightarrow g$. Results have been obtained under the spin assumption in both cases. See, for instance, \cite{Zhang99}, \cite{CJL04}, \cite{CM06}, \cite{M06}, \cite{CMT06}, \cite{Zhang04}, \cite{XieZhang08}, \cite{WangXuEM}.
\\ \indent In this work, we establish a positive mass theorem for asymptotically hyperboloidal initial data. For this, we use a technique introduced in \cite{PMTII} known as \emph{Jang equation reduction}. In this procedure one considers the Riemannian product $(M^n\times \rn, g + dt^2)$ and solves a certain prescribed mean curvature equation $\J(f)=0$. One then performs some additional deformations on the graph of the solution $f:M^n\rightarrow \rn$, after which the minimal surface proof used in \cite{PMTI} can be applied, to conclude that the energy is non-negative. We prove the following result, which extends that of \cite{SakovichPMTah} obtained in dimension $n=3$. 

\begin{theorem}\label{TheoremPositiveMassRigidity}
	Let $(M^n,g,k)$ be initial data of type $(\ell, \alpha, \tau, \tau_0)$, where $4\leq n \leq 7$, $\ell\geq 6$, $0\leq \alpha <1$, $\frac{n}{2}<\tau <n$ and $\tau_0>0$. If the dominant energy condition $\mu\geq |J|_g$ holds, then the mass vector is future pointing causal, $E\geq |\vec{P}|$.
	\\ \indent If, in addition, $(M^n,g,k)$ has Wang's asymptotics and $E=0$ then $(M^n, g,k)$ can be isometrically embedded into Minkowski space $\M^{n+1}$ as a spacelike hypersurface with second fundamental form $k$.
\end{theorem}

\noindent Again, the cornerstone of the proof of Theorem \ref{TheoremPositiveMassRigidity} is the method of Jang equation reduction originating from \cite{PMTII}. However, the setting of \cite{PMTII} is asymptotically Euclidean and the dimension is $n=3$, whereas the current work deals with the asymptotically hyperbolic setting and dimensions $4\leq n \leq7$. This requires, in many cases, more recent approaches to the Jang equation reduction. In particular, our construction of the geometric solution relies on geometric measure theory methods developed by Eichmair in \cite{EichmairPMT}, and we also mostly follow \cite{EichmairPMT} when dealing with the blow ups and blow downs of Jang's equation. Regarding the construction of barriers for Jang's equation, we use the ''asymptotic ODE'' method of \cite{SakovichPMTah}, developed specifically to deal with the more complicated asymptotics encountered in the asymptotically hyperbolic setting. We also use the ''graph over barrier'' approach from \cite{SakovichPMTah} to show that the Jang graph has asymptotically Euclidean asymptotics. 
\\ \indent We would like to point out other applications of Jang's equation besides the proof of positive mass theorem. As discussed in Section \ref{SectionJangSolution} the blow ups of the geometric solution occur on the so-called marginally outer (inner) trapped surfaces (MOTS or MITS). This feature has been used to prove existence of MOTS/MITS in certain initial data sets, see \cite{EichmairPlateau} and \cite{AEM11}. Furthermore, \cite{BKS} have applied the Jang's equation to obtain results on stability of the spacetime positive mass theorem in the spherically symmetric setting. There is also a plethora of reduction arguments for various geometric inequalities, starting from \cite{BK11}. 
\\ \indent The paper is organized as follows. In Section \ref{SectionPreliminaries} we clarify the notations and discuss preliminaries required for this work. In Section \ref{SectionBarriers} we construct barriers for Jang's equation assuming Wang's asymptotics. In Section \ref{SectionDirichletProblem} we solve the regularized Jang equation $\J(f_\tau)=\tau f_\tau$ on coordinate balls of radius $R$. In Section \ref{SectionJangSolution} we use techniques from geometric measure theory to obtain a geometric solutions of Jang's equation, that is to say a limit hypersurface for $\tau\rightarrow 0$ and $R\rightarrow \infty$ and discuss its possible blow up sets. In Section \ref{SectionJangAE} we prove that the obtained geometric solution is asymptotically Euclidean. In Section \ref{SectionConformalChanges} we conformally change the metric to achieve zero scalar curvature and we also improve the asymptotics near infinity. Furthermore, we discuss how to deal with conical singularities that arise when compactifying the cylindrical ends by a conformal change. Finally, the positive mass theorem is proven in Section \ref{SectionPositivity}. Some supplementary results are collected in the appendices.

\subsection*{Acknowledgments} 

This project was suggested and supervised by Anna Sakovich, who contributed many important suggestions and constructive comments over its span. The project was partly supported by the Swedish Research Council's grant dnr. 2016-04511. 
\\ \indent The current version of this article is part of the author's thesis at Uppsala University.


\section{Preliminaries}\label{SectionPreliminaries}

\subsection{Initial data sets and their masses}\label{SubSectionIDsetsMasses}

In this work we adopt the following definition of initial data sets:

\begin{definition}\label{DefinitionInitialData}
	
Let $n\geq 3$ and $(M^n,g)$ be an orientable, $n$-dimensional $C^{2}$-regular Riemannian manifold without boundary. Let $k\in C^{1}(\text{Sym}^2 (T^\ast M^n ))$ be a symmetric $(0,2)$-tensor. Then the triple $(M^n,g,k)$ is called an \emph{initial data set}. The  equations
\begin{equation}
	\begin{split}
		R_g - |k|^2_g + \trace^g(k)^2 &= 2 \mu, \\
		\diver^g \big(  k -   \trace^g(k) \cdot g \big) &= J,
	\end{split}
\end{equation}
are called the \emph{constraint equations}, where $\mu$ is the \emph{local mass density} and $J$ is the \emph{local current density}.
The condition
\begin{equation}
	\mu \geq |J|_g
\end{equation}
is called the \emph{dominant energy condition}. The so-called \emph{Riemannian} (or \emph{time-symmetric}) setting is characterized by $k\equiv 0$.
	
\end{definition}

We recall the ''hyperboloidal'' model for $n$-dimensional hyperbolic space $\mathbb{H}^n= (\rn^n, b)$ where $\mathbb{H}^n$ arises as the graph
\begin{equation}\label{EquationUpperHyperboloid}
	\{ (t,r,\theta)\:|\: t = \sqrt{1+r^2} \}	
\end{equation}
in Minkowski space $\M^{n+1}=(\rn \times\rn^{n}, \eta= - dt^2 + dr^2 + r^2\Omega)$, where $r$ and $\theta$ are spherical coordinates on $\rn^n$ and $\Omega$ is the Euclidean metric $\delta$ induced on $\mathbb{S}^{n-1}$. In this case, the induced metric $g$ on the graph and the second fundamental form $k$ of the graph satisfy $g=k=b$, where
\begin{equation}
	b = \frac{dr^2}{1+r^2} + r^2\Omega.
\end{equation}
\\ \indent In this work we use the same definition of asymptotically hyperbolic ''hyperboloidal'' initial data sets as in \cite{DahlSakovichDensityThm}. The reader is referred to \cite{DahlSakovichDensityThm} for the definition of the weighted H\"older spaces $C^{\ell, \alpha}_\tau$ used below.

\begin{definition}\label{DefinitionAHinitialData}
	
	Let $(M^n,g,k)$ be initial an data set. We say that $(M^n,g,k)$ is \emph{asymptotically hyperbolic of type} $(\ell, \alpha, \tau, \tau_0)$ for $\ell \geq 2$, $0\leq \alpha< 1$, $\tau>\frac{n}{2}$ and $\tau_0 >0$, if $g\in C^{\ell, \alpha}_\tau (M^n)$, $k\in C^{\ell-1, \alpha}_\tau(M^n)$ and there is a compact set $K\subset M^n$ and a diffeomorphism $\Psi:M^n\setminus K \rightarrow \rn^n\setminus \overline{B}_1(0)$ such that
	\begin{enumerate}
		\item $e=\Psi_\ast(g)- b \in C^{\ell, \alpha}_\tau(\mathbb{H}^n\setminus \bar{B}_1(0);\text{Sym}^2(T^\ast\mathbb{H}^n))$, \\
		\item $\eta = \Psi_\ast(k-g)\in C^{\ell-1, \alpha}_\tau (\mathbb{H}^n\setminus \bar{B}_1(0);\text{Sym}^2(T^\ast\mathbb{H}^n))$, 
		\item $\Psi_\ast (\mu),\Psi_\ast (J)  \in C^{\ell -2, \alpha}_{\tau_0 + n}(\mathbb{H}^n\setminus \bar{B}_1(0))$. \\
	\end{enumerate}
\end{definition}

In this work it will be sufficient to work with simpler asymptotics similar to those used in \cite{WangPMT}.

\begin{definition}\label{DefinitionWangAsymptotics}
	Let $(M^n,g,k)$ be an asymptotically hyperbolic initial data of type $(\ell, \alpha, \tau=n, \tau_0)$ as in Definition \ref{DefinitionAHinitialData}. Then $(M^n,g,k)$ is said to have \emph{Wang's asymptotics} if  
	\begin{enumerate}
		\item $\Psi_\ast(g) - b=    \textbf{m} r^{-(n-2)}  + \Ol^{\ell, \alpha}(r^{-(n-1)})$,
		\item $\big(\Psi_\ast(k)-b\big)|_{T\mathbb{S}^{n-1}\times T\mathbb{S}^{n-1}}= \textbf{p} r^{-(n-2)}  + \Ol^{\ell-1, \alpha}(r^{-(n-1)})$,
	\end{enumerate}	
	where $\textbf{m},\textbf{p}\in C^{\ell,\alpha}(\mathbb{S}^{n-1};\text{Sym}^2 (T^\ast\mathbb{S}^{n-1}))$ are symmetric $(0,2)$-tensors on $\mathbb{S}^{n-1}$ and $\Omega$ the standard Euclidean metric on $\mathbb{S}^{n-1}$. The expressions $\Ol^{\ell, \alpha}(r^{-\tau})$ are symmetric tensors in $C^{\ell,\alpha}(\mathbb{S}^{n-1};\text{Sym}^2 (T^\ast\mathbb{S}^{n-1}))$ with norms in $C^{\ell, \alpha}_\tau(\mathbb{H}^n)$.
\end{definition}
\noindent Throughout this work we will suppress the dependence on the chart and write, for instance, $\Psi^\ast (g) = g$ as long as there is no risk for confusion.
\\ \indent We now discuss the notion of mass for asymptotically hyperbolic initial datas. Let 
\begin{equation} 
	\N=\{ V\in C^\infty(\mathbb{H}^n) \: |\: \Hess^b V = V b\}.
\end{equation}
Then 
\begin{equation}
	\N = \text{span}_\rn \{ V_0=\sqrt{1+r^2}, V_1= \hat{x}^1r, \ldots,V_n= \hat{x}^nr   \},
\end{equation}
where $\hat{x}^i=\frac{x^i}{r}$ is the $i^{th}$ coordinate of $\rn^n$ restricted to the unit sphere. These functions may be interpreted as the coordinate functions of Minkowski space restricted to the upper unit hyperboloid, as defined in \eqref{EquationUpperHyperboloid}. 
 
\begin{definition}\label{DefinitionMassFunctional}
	Let $(M^n,g,k)$ be an asymptotically hyperbolic initial data set as in Definition \ref{DefinitionAHinitialData} with respect to a chart $\Psi$ at infinity. The map $\M_\Psi: \N \rightarrow \rn$ defined as the integral at infinity
	\begin{equation}\label{EquationMassFunctional}
		\begin{split}
			\M_\Psi(V) &= \lim_{R\rightarrow \infty} \int_{\{r=R\}}\bigg( V\big(\diver^{b}(e)-d \trace^{b}(e)\big) \\
			&\qquad+\trace^{b}(e) d V - (e+2\eta)(\nabla^bV,\cdot )   \bigg)(\vec{n}_r^b)d\mu^{b}
		\end{split}
	\end{equation}
	is called the \emph{mass functional}. Here $\vec{n}^b_r = \sqrt{1+r^2}\partial_r$ is the outward pointing unit normal with respect to the hyperbolic metric. The vector $(E, \vec{P})$, where
	\begin{equation}
		E = \frac{ \M_\Psi(V_0)}{2(n-1)\omega_{n-1}} \qquad \text{and} \qquad P^i = \frac{ \M_\Psi(V_i)}{2(n-1)\omega_{n-1}},
	\end{equation}
	is called the \emph{mass vector}. Its Minkowskian length $m=\sqrt{-|(E,\vec{P})|^2_\eta}=\sqrt{E^2-|\vec{P}|^2}$ is called the \emph{mass}.
\end{definition}
The mass vector $(E, \vec{P})$ is clearly a coordinate dependent. Moreover, for an isometry $A$ of hyperbolic space $\mathbb{H}^n$ it can be shown that if $\Psi$ is a chart at infinity as in Definition \ref{DefinitionAHinitialData} the composition $A\circ \Psi$ is also such a chart. It follows immediately from the definition that the mass functional transforms equivariantly under such a composition, that is $ \M_{A\circ \Psi}(V)= \M_{\Psi}(V\circ A)$. In particular, this shows that the mass is a coordinate invariant. For further details on this, we refer to \cite{ChruscielNagy}, \cite{ChruscielHerzlich} and \cite{MichelMassFormalism}.
\\ \indent If the chart $\Psi$ in Definition \ref{DefinitionAHinitialData} is such that the mass vector takes the form $(E, \vec{0})$ we use terminology coined in \cite{CCS16} and say that $\Psi$ is \emph{balanced}. If the mass vector is causal it is possible to find such a chart. 
\\ \indent The following Theorem \ref{TheoremDensity} is a density theorem, proven in \cite{DahlSakovichDensityThm}, important for our work.

\begin{theorem}\label{TheoremDensity}  
	
	Let $(M^n,g,k)$ be initial data as in Definition \ref{DefinitionAHinitialData} of type $(\ell, \alpha, \tau, \tau_0)$, where $\ell \geq 3$, $0< \alpha <1$, $\frac{n}{2} <\tau <n$ and $0<\tau_0$, and the dominant energy condition $\mu\geq |J|_g$ holds. Then, for any $\epsilon >0$ there exists an initial data set $(M^n,\hat{g}, \hat{k})$ of type $(\ell -1, \alpha, n, \hat{\tau}_0)$, where $\hat{\tau}_0>0$, with Wang's asymptotics (possibly with respect to a different chart $\hat{\Psi}$) such that
	the strict dominant energy condition holds:
	\begin{equation}
		\hat{\mu}>|\hat{J}|_g,
	\end{equation}
	and
	\begin{equation}
		|E-\hat{E}|<\epsilon.
	\end{equation}
 
\end{theorem}

For future reference we include the following definition.

\begin{definition}\label{DefinitionAFinitialData}
	Let $(M^n,g)$ be an $n$-dimensional Riemannian manifold. We say that $(M^n,g)$ is \emph{asymptotically flat} if there is a compact set $K\subset M^n$ and a diffeomorphism $\Psi: M^n\setminus K \rightarrow \rn^n\setminus \overline{B}_1(0)$ such that in the Cartesian coordinates induced by $\Psi$, we have   
	\begin{equation}
		|g_{ij}-\delta_{ij}| + r |g_{ij,k}| + r^2 |g_{ij,k\ell}| = \Ol(r^{-(n-2)}), \qquad \text{as} \qquad r\rightarrow \infty,
	\end{equation}
	which in the coordinate free form reads $|g-\delta|_\delta = \Ol_2(r^{-(n-2)})$. If the scalar curvature $R_g$ is integrable we define the ADM energy:
	\begin{equation}
		E_{ADM} = \lim_{R\rightarrow \infty} \frac{1}{2(n-1)\omega_{n-1}}\int_{\{ r=R\}} \big(\diver^\delta(g) - d (\trace^\delta g) \big)(\vec{n}^\delta_r) d\mu^\delta.
	\end{equation}
	If, furthermore, $m\in \rn$ and $g$ has the asymptotics  
	\begin{equation}
		g  = \bigg( 1+ \frac{m}{2r^{n-2}}\bigg)^{\frac{4}{n-2}}\delta + \Ol_2(r^{-(n-1)}), \qquad    \text{as} \qquad r\rightarrow \infty,
	\end{equation}
	we say that $(M^n,g)$ is \emph{asymptotically Schwarzschildean}. In this case we have $E_{ADM}=m$.
	
\end{definition}

Note that the asymptotics used in Definition \ref{DefinitionAFinitialData} are not the most general ones, but they will be sufficient for this work.

\subsection{Jang's equation}

Let $(M^n,g,k)$ be an initial data. For local coordinates $(x^1, \ldots, x^n)$ we let the metric be $g=g_{ij}dx^i\otimes dx^j$ and $k=k_{ij}dx^i\otimes dx^j$. We use the Einstein summation throughout, so that $g^{i\ell}g_{j\ell}=\delta^i_j$. Further, let $f_{,i}=\partial_i f$ denote the $i^{th}$ coordinate derivative. $f^{,i}=(\nabla^g f)^i=g^{ij}f_{,j}$ is the $i^{th}$ component of the gradient of $f$. The covariant Hessian of $f$ is given by $\Hess_{ij}^g(f)=f_{,ij}- \Gamma^k_{ij}f_{,k}$, where $\Gamma$ are the Christoffel symbols associated to $g$.
\\ \indent For $f\in C^2_{loc}(U)$, where $U\subset M^n$, we consider the equation

\begin{equation}\label{EquationJangsEquation}
	\bigg( g^{ij} - \frac{f^{,i}f^{,j}}{1+|df|^2_g}\bigg) \bigg( k_{ij} - \frac{\Hess_{ij}^g(f)}{\sqrt{1+|df|^2_g}}\bigg)=0,
\end{equation}

\noindent known as \emph{Jang's equation} introduced in \cite{JangPaper}. Throughout this text we will refer to this equation as
\begin{equation}
	\J(f)=0.
\end{equation}
Jang's equation may be geometrically interpreted as follows. We consider the Riemannian product $(M^n\times \rn, g + dt^2)$ and the graph $(\hat{M}^n, \hat{g})$ of a function\footnote{As we will see in Section \ref{SectionJangSolution}, we will not in general obtain a global graph due to blowups/blow-downs.} $f:M^n\rightarrow \rn$. Then the induced metric $\hat{g}= g + df\otimes df$ on the graph has components $\hat{g}_{ij} = g_{ij} + f_{,i}f_{,j}$, and its inverse is
\begin{equation}
	\hat{g}^{ij} = \bigg( g^{ij} - \frac{f^{,i}f^{,j}}{1+|df|^2_g}\bigg).
\end{equation}
Furthermore, the (downward pointing) unit normal is
\begin{equation}
	\vec{n} = \frac{- \partial_t + \nabla^gf}{\sqrt{1+|df|_g^2}}
\end{equation}
and the second fundamental form, given by $\hat{A}(X,Y)=g( \nabla_X Y, \vec{n} )$, has components  
\begin{equation}\label{EquationHessianOfGraph}
	\hat{A}_{ij} = \frac{\Hess^g_{ij}(f)}{\sqrt{1+|df|_g^2}},
\end{equation}
Consequently, the mean curvature $H_{\hat{M}^n} = \trace^{\hat{g}} \hat{A}$ of $(\hat{M}^n, \hat{g})$ is 
\begin{equation}\label{EquationGraphMeanCurvature} 
	H_{\hat{M}^n}= \bigg( g^{ij} - \frac{f^{,i}f^{,j}}{1+|df|^2_g}\bigg)\frac{\Hess_{ij}^g(f)}{\sqrt{1+|df|^2_g}},
\end{equation}
which also equals the divergence of the downward pointing unit normal. Extending $k$ trivially to a symmetric $(0,2)$-tensor on $M^n\times \rn$ by $k(\cdot, \partial_t)=0$, we obtain 
\begin{equation}
	\trace_{\hat{g}}(k)=\bigg( g^{ij} - \frac{f^{,i}f^{,j}}{1+|df|^2_g}\bigg)k_{ij}.
\end{equation}
Hence, Jang's equation $\J(f)=0$ may be viewed as a prescribed mean curvature equation 
\begin{equation}
	H_{\hat{M}^n}=\trace^{\hat{g}}(k).
\end{equation}
It is a quasilinear second order PDE and it follows from the positivity of the metric $\hat{g}$ that it is elliptic. 
\\ \indent The reader is referred to \cite{AEM11} for an extensive summary on Jang's equation and its applications.

\section{Barrier construction}\label{SectionBarriers}

In this section we construct barriers for Jang's equation \eqref{EquationJangsEquation} in the case when the initial data has Wang's asymptotics as in Definition \ref{DefinitionWangAsymptotics}. For this we perform in Section \ref{SubsectionHeuristicAnalysis} a heuristic analysis of Jang's equation in order to better understand the asymptotic behaviour of the solutions. Subsequently, in Section \ref{SubsectionBarriers} we use these results to obtain barriers with the desired asymptotics.
\\ \indent Throughout this work we divide the coordinate indices $i,j$ of $M^n$ into radial and tangential, and use greek letters for the latter.

\subsection{Heuristic analysis}\label{SubsectionHeuristicAnalysis}

We start with the following elementary Example \ref{ExampleJangSolution}.

\begin{example}\label{ExampleJangSolution}
	We consider the standard $n$-dimensional hyperbolic space $(\rn^n,b,k=b)$ and show that $f(r)=\sqrt{1+r^2}$ is a solution to the Jang equation. It is not difficult to see\footnote{Compare to the calculations done in Section \ref{SubsectionBarriers} below.} that
	\begin{equation}
	\frac{\Hess_{ij}^b(f)}{\sqrt{1+|d f|_b^2}} = b_{ij}
	\end{equation}
	for all $i,j$. Thus $f(r)=\sqrt{1+r^2}$ solves Jang's equation.
\end{example}

For $n\geq 4$, $0<\epsilon<1$ and $\alpha,\Psi$ and $q$ smooth functions we make the following ansatz:
 
\begin{equation}\label{EquationJangAnsatz}
	f(r, \theta) = \sqrt{1+r^2} +   \psi(\theta)+ \frac{\alpha(\theta)}{r^{n-3}} + q(r, \theta),
\end{equation}
where $q(r,\theta)=\Ol(r^{-(n-2-\epsilon)})$ with derivatives that decay one order faster per derivative in the $r$-direction; that is $q_{,\mu}(r,\theta),q_{,\mu \nu}(r,\theta)=\Ol(r^{-(n-2-\epsilon)})$, $q_{,r}(r,\theta),q_{,r\mu}(r,\theta)=\Ol(r^{-(n-1-\epsilon)})$ and $q_{,rr}(r,\theta)= \Ol(r^{-(n-\epsilon)})$ and higher order derivatives decay as indicated. Below in Lemma \ref{LemmaJangAnsatz} we see the implications of the requirement $\J(f)=\Ol(r^{-(n+1-\epsilon)})$.
 
\begin{lemma}\label{LemmaJangAnsatz}
	If the function 
	\begin{equation}
		f(r,\theta)=\sqrt{1+r^2}+ \frac{\alpha(\theta)}{r^{n-3}} +\psi(\theta) + q(r,\theta),
	\end{equation}
	satsfies $\J(f)= \Ol(r^{-(n+1-\epsilon)})$, then $\psi(\theta)$ is a constant and $\alpha(\theta)$ is the (unique) solution of
	\begin{equation}\label{EquationAlpha}
		\Delta^{\Omega} \alpha-(n-3)\alpha  = \bigg(\frac{n-2}{2}\bigg)\trace^{\Omega}(\textbf{m})+\trace^{\Omega}(\textbf{p}),
	\end{equation}
	where $\Omega$ is the standard round metric on the sphere $\mathbb{S}^{n-1}$.
\end{lemma}
 
\begin{proof}
	We omit the details of the computation for brevity\footnote{The calculations are very similar in nature to the ones performed in Section \ref{SubsectionBarriers}.} and merely show the result obtained when inserting $f(r,\theta)$ as in \eqref{EquationJangAnsatz} into $\J(f)$:
	\begin{equation}
		\begin{split}
			\J(f)&=  \bigg(\Delta^{\Omega}(\alpha ) -(n-3)\alpha - \bigg(\frac{n-2}{2} \bigg)\trace^{\Omega}(\textbf{m}) -\trace^{\Omega}(\textbf{p})   \bigg)r^{-n} \\ 
			& \qquad + \frac{\Delta^{\Omega}(\psi )}{r^2} \frac{1}{\sqrt{1+r^2+\frac{|d \psi |_{\Omega}^2}{r^2}}} \\
			&\qquad -  \frac{1}{(1+r^2+\frac{|d \psi |_{\Omega}^2}{r^2})^{3/2}}  g^{\mu\lambda} g^{\nu\rho}\psi_{ ,\rho}\psi_{ ,\lambda}\Hess_{\mu \nu}^{\Omega}(\psi ) \\
			&\qquad +2\frac{\sqrt{1+r^2}}{(1+r^2+\frac{|d \psi |_{\Omega}^2}{r^2})^{3/2}}\frac{|d \psi |^2_\Omega}{r^2} \\
			&\qquad +\frac{1}{1+r^2+\frac{|d\psi |_{\Omega}^2}{r^2}} \bigg( \frac{\sqrt{1+r^2}}{\sqrt{1+r^2+\frac{|d \psi |_{\Omega}^2}{r^2}}}  -1 \bigg) \\
			&\qquad +(n-1)\bigg( \sqrt{\frac{1+r^2}{1+r^2+\frac{|d \psi |^2_{\Omega}}{r^2}}}-1\bigg) + \Ol(r^{-(n+1-\epsilon)}) \\
		\end{split}
	\end{equation}
	Requiring that the $\Ol(r^{-3})$-term vanishes implies $\Delta^\Omega \psi = 0$. It is well-known that the harmonic functions on the sphere $(\mathbb{S}^{n-1}, \Omega)$ are precisely the constants. Requiring that the $\Ol(r^{-n})$-term vanishes implies that $\alpha$ must solve \eqref{EquationAlpha}. Multiplying the left hand side of \eqref{EquationAlpha} by $\alpha$, integrating over $(\mathbb{S}^{n-1}, \Omega)$ and integrating by parts, we see that the homogeneous problem has only the trivial solution, and existence of a unique solution $\alpha$ of \eqref{EquationAlpha} follows from Fredholm alternative (see e.g. \cite{Besse}, Appendix I).
	
\end{proof}

\noindent Properties of the associated graph of $f$ in $M^n\times \rn$ are stated in Appendix \ref{SectionJangGraph}.

\begin{remark}\label{Remark3DimBarriers}
We recall for comparison that the corresponding result in \cite{SakovichPMTah} (Proposition 2.6) is
\begin{equation}
	f(r, \theta, \varphi) = \sqrt{1+r^2} + \alpha(\theta, \varphi)\ln(r) + \psi(\theta, \varphi) + q(r,\theta,\varphi),
\end{equation}
where $\alpha$ is the constant
\begin{equation}
	\alpha = \frac{1}{8\pi} \int_{\mathbb{S}^2}\big( \trace^\Omega(\textbf{m})+ 2\trace^\Omega(\textbf{p}) \big)d\mu^{\Omega}
\end{equation}
and 
\begin{equation}
	\Delta^{\Omega}(\psi)= \frac{1}{2}\trace^{\Omega}(\textbf{m})+  \trace^{\Omega}(\textbf{p}) - \alpha.
\end{equation}
\end{remark}



\subsection{Barrier construction}\label{SubsectionBarriers}

In this subsection we construct the barriers for Jang's equation (see Definition \ref{DefinitionBarrier} below), assuming that the initial data has Wang's asymptotics as in Definition \ref{DefinitionWangAsymptotics}. The significance of the barriers is that they ''squeeze'' the solution to Jang's equation near infinity, providing the asymptotic control. In the asymptotically hyperbolic setting the construction of barriers is much more involved compared to the explicit functions used in the asymptotically Euclidean setting of \cite{PMTII} and \cite{EichmairPMT}. 

\begin{definition}\label{DefinitionBarrier}
	Let $(M^n,g,k)$ be given initial data. A function $f_+ \in C^2_{\text{loc}}(M^n_{r_0})$ (respectively $f_-\in C^2_{\text{loc}}(M^n_{r_0})$), where $M^n_{r_0}=\{r\geq r_0\} \subset M^n$, is said to be a \emph{upper barrier} (respectively \emph{lower barrier}) if it satisfies
	\begin{equation}
		f_{+,r}(r_0)=+\infty  \qquad (\text{respectively} \: f_{-,r}(r_0)=-\infty)
	\end{equation}
	and is \emph{supersolution} (respectively \emph{subsolutions}), that is 
	\begin{equation}
		\J(f_+)< 0  \qquad (\text{respectively} \: \J(f_-)>0)
	\end{equation}
	for $r>r_0$.
\end{definition} 

We refer the reader to \cite{SakovichPMTah} for the construction of barriers when $n=3$ and perform a related construction in dimensions $n\geq 4$. For this, we will transform the Jang equation to an asymptotic ODE in the radial variable and construct upper and lower barriers $f_+$ and $f_-$ via a change of variables considered in \cite{MalecMurchada} in the spherically symmetric setting. As in the previous section, $\theta$ denotes a coordinate system on $\mathbb{S}^{n-1}$. Based on the results in Lemma \ref{LemmaJangAnsatz} we choose to make the ansatz \footnote{This is similar to \cite{SakovichPMTah}, where the anzats is $f(r,\varphi, \theta)=\varphi(r) + \psi(\varphi, \theta)$.}
\begin{equation}\label{EquationBarrierAnsatz}
	f(r,\theta)=\frac{\alpha(\theta)}{r^{n-3}}+\varphi(r),
\end{equation}
for the barriers, where $\varphi_{,r}(r)\rightarrow 1$ as $r\rightarrow \infty$. As in \cite{MalecMurchada}, we define 
\begin{equation}
	\begin{split}
		k(r)&= \frac{\sqrt{1+r^2}\varphi_{,r}}{\sqrt{1+(1+r^2)\varphi_{,r}^2}} .
	\end{split}
\end{equation}
For reasons that will become clear below we define
\begin{equation}
	\begin{split}
		\Pi &=\frac{1 +(1+r^2)\varphi_{,r}^2}{ 1+|d f|_g^2} \\
	\end{split}
\end{equation}
and note that a straightforward calculation shows
\begin{equation}\label{EquationGammaInverse}
	\begin{split}
		\Pi &=\bigg( 1 - 2  \sqrt{1+r^2}  \alpha(n-3)r^{-( n-2)}   k\sqrt{1-k^2} \\
		&\qquad +    (1+r^2) (n-3)^2\alpha^2r^{-2(n-2)} (1-k^2)   +   r^{-2(n-3)}|d \alpha|_g^2 (1-k^2)  \bigg)^{-1} .
	\end{split}
\end{equation}
\\ \indent We now rewrite Jang's equation asymmptotically in terms of $k$. Lemmas \ref{LemmaBarrierTraceTerm}- \ref{LemmaBarrierTangentialHessian} below contain some preliminary computations. 
\begin{lemma}\label{LemmaBarrierTraceTerm}
	With the ansatz in \eqref{EquationBarrierAnsatz}, the trace term in Jang's equation is 
	\begin{equation}\label{EquationTraceTerm}
		\begin{split}
			\trace_{\hat{g}}(k)  
			&=\Pi (1+r^{-2(n-3)}|d \alpha|_g^2)(1-k^2)  \\
			&\qquad +   (n-1)+ \frac{\trace^\Omega(\textbf{p}) -\trace^\Omega(\textbf{m})}{r^n} +  \Ol (  r^{-(n+ 1)} )  ,
		\end{split}
	\end{equation}
	where the implicit constant in the $\Ol$-term does not depend on $\varphi$.
\end{lemma}

\begin{proof}
	
	The trace term is explicitly
	\begin{equation}
		\trace_{\hat{g}}(k)=\hat{g}^{rr}k_{rr} + 2\hat{g}^{r\mu} k_{r\mu} + \hat{g}^{\mu\nu}k_{\mu\nu}
	\end{equation}
	and we expand the terms. For the radial term, using $g_{rr}=b_{rr}$, it is not difficult to see that the radial metric component is
	\begin{equation}
		\begin{split}
			\hat{g}^{rr}&= (1+r^2)\bigg(1-\frac{(1+r^2)f_{,r}^2}{1+|df|_g^2}\bigg) \\
			&=(1+r^2)\bigg(\frac{1+r^{-2(n-3)}|d \alpha|_g^2}{1+|d f|_g^2}\bigg).
		\end{split}
	\end{equation}
	With the definition of $\Pi$ it is not difficult to see that
	\begin{equation}
		\begin{split}
			\bigg(\frac{1+r^{-2(n-3)}|d \alpha|_g^2}{1+|d f|_g^2}\bigg) 
			&= \Pi (1+r^{-2(n-3)}|d \alpha|_g^2)(1-k^2) 
		\end{split}
	\end{equation}
	and so, since $k_{rr}=\frac{1}{1+r^2}+ \Ol(r^{-(n+1)})$ from Definition \ref{DefinitionWangAsymptotics}, we find
	\begin{equation}
		\begin{split}
			\hat{g}^{rr}k_{rr}
			&=\Pi (1+r^{-2(n-3)}|d \alpha|_g^2)(1-k^2) + \Ol(r^{-(n+1)}),
		\end{split}
	\end{equation}
	where the implicit constant in the $\Ol$-term does not depend on $\varphi$.
	\\ \indent As for the mixed term $\hat{g}^{r\mu}k_{r\mu}$ we observe that since both $g^{r\mu} =0$, $f^{,r}=\Ol(r^2)$, $f^{,\mu}=\Ol(r^{-(n-1)})$ and $k_{r\mu}=\Ol(r^{-n})$ we immedately obtain $\hat{g}^{r\mu}k_{r\mu}= \Ol(r^{-(n+1)})$.
	\\ \indent We compute the asymptotics of the tangential term $\hat{g}^{\mu\nu}k_{\mu\nu}$:
	\begin{equation}
		\hat{g}^{\mu \nu}k_{\mu \nu} =\bigg( g^{\mu \nu}-\frac{f^{,\mu}f^{,\nu}}{1+|d f|_g^2}\bigg) \bigg(  b_{\mu\nu} + 	\frac{\textbf{p}_{\mu\nu}}{r^{n-2}}+ \Ol ( r^{-(n-1)} ) \bigg). \\
	\end{equation}
	Since $f^{,\mu} = \Ol(r^{-(n-1)})$ it follows that 
	\begin{equation}	
			\hat{g}^{\mu\nu} = b^{\mu\nu} - \frac{\textbf{m}^{\mu\nu}}{r^{n-2}} + \Ol(r^{-(n+3)}),
	\end{equation}
  	where indices on $\textbf{m}$ are raised with $g$. From this and the expression
  	\begin{equation}
  		k_{\mu\nu} = b_{\mu\nu} + \frac{\textbf{p}_{\mu\nu}}{r^{n-2}} +\Ol(r^{-(n-1)})
	\end{equation}
	it is immediate that
	\begin{equation}
		\begin{split}
			\hat{g}^{\mu \nu}k_{\mu \nu}&=(n-1)+ \frac{\trace^{\Omega}(\textbf{p}) -\trace^\Omega(\textbf{m})}{r^n} +  \Ol (  r^{-(n+ 1)} ) 
		\end{split}
	\end{equation}
	and so the assertion follows.

\end{proof}

\begin{lemma}\label{LemmaBarrierRadialHessian}
	With the ansatz in \eqref{EquationBarrierAnsatz}, the radial Hessian term in Jang's equation is 
	\begin{equation}\label{EquationRadialHessian}
		\begin{split}
			\hat{g}^{rr}\frac{\Hess_{rr}^g(f)}{\sqrt{1+|d f|_g^2}} &=
			\sqrt{1+r^2} (1+r^{-2(n-3)}|d \alpha|_g^2)\Pi^{3/2} \\
			& \times  \bigg( k' +      (1-k^2)^{3/2} \sqrt{1+r^2}  (n-3)^2 \alpha r^{-(n-1)}+ \Ol( r^{-(n+1)} ) \bigg),
		\end{split}
	\end{equation}
	where the implicit constant in the $\Ol$-term does not depend on $\varphi$.
\end{lemma}

\begin{proof}
	
	From the definition of $k$ it follows that
	\begin{equation}
		k'(r)=\frac{\sqrt{1+r^2}}{(1+(1+r^2)\varphi_{,r}^2)^{3/2}}\bigg(\varphi_{,rr}+\frac{r}{1+r^2}\varphi_{,r}\bigg).
	\end{equation}
	From the proof of Lemma \ref{LemmaBarrierTraceTerm} the radial metric component $\hat{g}^{rr}$ is  
	\begin{equation}
		\begin{split}
			\hat{g}^{rr} &=(1+r^2)\bigg(\frac{1+r^{-2(n-3)}|d \alpha|_g^2}{1+|d f|_g^2}\bigg).
		\end{split}
	\end{equation}
	Using expressions for the Christoffel symbols of $g$ obtained in Lemma \ref{LemmaWangGeometry} we find
	\begin{equation}
		\begin{split}
			\Hess_{rr}^g(f)&= \varphi_{,rr}+\frac{r}{1+r^2}\varphi_{,r}+ \Hess_{rr}^g\bigg(\frac{\alpha}{r^{n-3}} \bigg)
		\end{split}
	\end{equation}
	where, in turn,
	\begin{equation}
		\begin{split}
			\Hess_{rr}^g\bigg(\frac{\alpha}{r^{n-3}} \bigg) &=  \bigg(\frac{\alpha}{r^{n-3}} \bigg)_{,rr}+ 	\frac{r}{1+r^2}\bigg(\frac{\alpha}{r^{n-3}} \bigg)_{,r} \\
			&= (n-3)^2\frac{\alpha}{r^{n-1}}+ \Ol ( r^{-(n+1)} ).
		\end{split}
	\end{equation}
	Hence, with $\hat{g}^{rr}$ from the proof of Lemma \ref{LemmaBarrierTraceTerm},
	\begin{equation}
		\begin{split}
			\hat{g}^{rr}\frac{\Hess_{rr}^g(f)}{\sqrt{1+|d f|_g^2}}&=(1+r^2) \frac{1+r^{-2(n-3)}|d \alpha|_g^2} {(1+|d 	f|_g^2)^{3/2}}\Hess_{rr}^g(f) \\
			&=\frac{1+r^2}{(1+(1+r^2)\varphi_{,r}^2)^{3/2}} \Pi^{ 3/2} (1+r^{-2(n-3)}|d \alpha|_g^2)  \Hess_{rr}^g(f) \\
			&=\frac{1+r^2}{(1+(1+r^2)\varphi_{,r}^2)^{3/2}} \Pi^{ 3/2}(1+r^{-2(n-3)}|d \alpha|_g^2)  	\bigg(\varphi_{,rr}+\frac{r}{1+r^2}\varphi_{,r}\bigg) \\
			&\qquad + (1+r^2)(1-k^2)^{3/2}\Pi^{ 3/2}(1+r^{-2(n-3)}|d \alpha|_g^2) \\
		  	&\qquad \times   \bigg((n-3)^2\frac{\alpha}{r^{n-1}}+ \Ol ( r^{-(n+1)} )\bigg) \\
			&=\sqrt{1+r^2} (1+r^{-2(n-3)}|d \alpha|_g^2)\Pi^{3/2} \\
			&\qquad \times  \bigg( k' +   (1-k^2)^{3/2}  \sqrt{1+r^2} (n-3)^2\frac{\alpha}{r^{n-1}}+ \Ol ( r^{-(n+1)} ) \bigg),
		\end{split}
	\end{equation}
	as asserted.

\end{proof}

\noindent Similar to the proof of Lemmas \ref{LemmaBarrierTraceTerm} and \ref{LemmaBarrierRadialHessian} calculations yield Lemmas \ref{LemmaBarrierMixedHessian} and \ref{LemmaBarrierTangentialHessian} below.

\begin{lemma}\label{LemmaBarrierMixedHessian}
	With $f$ as in \eqref{EquationBarrierAnsatz}, the mixed Hessian term in Jang's equation is 
	\begin{equation}
		\hat{g}^{\mu r}\frac{\Hess_{\mu r}^g(f)}{\sqrt{1+|d f|_g^2}} =\Ol ( r^{-(n+1)} ),
	\end{equation}
	where the implicit constant in the $\Ol$-term does not depend on $\varphi$.
\end{lemma}

\begin{lemma}\label{LemmaBarrierTangentialHessian}
	With $f$ as in \eqref{EquationBarrierAnsatz}, the tangential Hessian term in Jang's equation is 
	\begin{equation}\label{EquationTangentialHessian} 
		\begin{split}
			\hat{g}^{\mu \nu}\frac{\Hess_{\mu \nu}^g(f)}{\sqrt{1+|d f|_g^2}}&= 
			\bigg( \Delta^\Omega(\alpha )  -(n-3)(n-1)(1+r^2) \alpha       \bigg)\sqrt{\Pi} \sqrt{1-k^2}r^{-(n-1)} \\
			&\qquad + \bigg( \frac{\sqrt{1+r^2}}{r}(n-1) - \frac{\trace^\Omega (\textbf{m})}{r^{n}}\frac{n}{2}    \bigg)  \sqrt{\Pi}k+  \Ol ( 	r^{-(n +1)} ) ,
		\end{split}
	\end{equation}
	where the implicit constant in the $\Ol$-term does not depend on $\varphi$.
\end{lemma}

Combining Lemmas \ref{LemmaBarrierTraceTerm} - \ref{LemmaBarrierTangentialHessian} we obtain the following result:

\begin{lemma}\label{LemmaBarrierAsymptotics}
	With $f$ as in \eqref{EquationBarrierAnsatz}, we have
	\begin{equation}
		\begin{split}
			\frac{\J(f)}{\Pi^{3/2}}&=  \sqrt{1+r^2} (1+r^{-2(n-3)}|d \alpha|_g^2) k' \\
			&\qquad + \sqrt{1+r^2}\bigg(\frac{n-1}{r}\bigg)\bigg(  k   -  \frac{r}{\sqrt{1+r^2}}  \bigg) -    (1+r^{-2(n-3)}|d 	\alpha|_g^2)(1-k^2)   \\
			&\qquad + (n-3)\frac{\alpha}{r^{n-2}}\sqrt{1-k^2}\bigg(          (1-k^2)\frac{1+r^2}{r}   (n-3)    +  \frac{1}{r }   \\
			&\qquad \qquad    - (n-1)\frac{1+r^2}{r}   -  2  (n-1) \frac{1+r^2}{r}      k^2   \\
			&\qquad  \qquad    +\sqrt{1+r^2}k(1-k^2)     +  3 (n-1)\sqrt{1+r^2}k     \bigg) \\
			&\qquad + \bigg(   \bigg(\frac{n-2}{2} \bigg)\trace^\Omega(\textbf{m})+ \trace^\Omega(\textbf{p}) \bigg) \bigg(  	\frac{\sqrt{1-k^2}}{r^{n-1}}   - \frac{1}{r^n} \bigg)   \\
			&\qquad + \frac{\trace^\Omega(\textbf{m})}{r^{n}}\frac{n}{2}  (1-k )+ \Lambda + \Ol(r^{-(n+1)}),
		\end{split}
	\end{equation}
	where 
	\begin{equation}
		\begin{split}
			\Lambda &=\frac{\sqrt{1-k^2}}{r^{n-1}}\bigg(\Delta^\Omega(\alpha) - (n-3)(n-1) \alpha    \bigg)\bigg(1- \frac{1}{\Pi}\bigg) \\
			&\qquad-\frac{|d \alpha|_\Omega^{2}}{r^{2(n-2)}}\frac{(1-k^2)}{2} \bigg(-2(n-1)k + (1-k^2) 	+ 3(n-1)     \bigg) \\
			&\qquad + \frac{\alpha^2}{r^{2(n-2)}}(1+r^2)(1-k^2)(n-3)^2 \bigg(   -2 (n-1) k \frac{r^2}{1+r^2}+ \frac{\sqrt{1+r^2}}{r}(n-1) k\\
			&\qquad  \qquad   -(1-k^2)\frac{1}{2}\bigg(1   + 3    k^2      \bigg) - (n-1)\frac{3}{2}(1+k^2) \bigg) \\
			&\qquad + \frac{\alpha^3}{r^{3(n-2)}}(1+r^2)^{3/2}(1-k^2)^{3/2} (n-3)^3\\
			&\qquad \qquad \times \bigg(    (n-1)    \frac{\sqrt{1+r^2}}{r}  -  \frac{k}{2} (  3   + 5 k^2   ) -\frac{k}{2}( 3      +   k^2)  	\bigg) \\
			&\qquad -   \frac{\alpha^4}{r^{4(n-2)}}(n-3)^4(1-k^2)^2 (1+r^2)^2\\
			&\qquad \qquad \times     \bigg(  (1-k^2)\frac{3}{8}\bigg( -1 +2k^2 + k^4   \bigg)   +(n-1)\frac{3}{8}\bigg(  - 1  + 2    k^2    	+  k^4   \bigg) \bigg) \\
			&\qquad +    \Ol(r^{-(n+1)}).
		\end{split}
	\end{equation}
	
\end{lemma}

\begin{proof}

We divide the terms from Lemmas \ref{LemmaBarrierTraceTerm}, \ref{LemmaBarrierRadialHessian}, \ref{LemmaBarrierMixedHessian} and \ref{LemmaBarrierTangentialHessian} by $\Pi^{3/2}$ and expand. Since $\Pi^{-1}=1+\Ol(r^{-(n-3)})$,  the contribution of the mixed Hessian term is of order $\Ol(r^{-(n+1)})$. To estimate the tangential Hessian term and trace term, which contain powers $\Pi^{-1}, \Pi^{-1/2}, \Pi^{-3/2}$, we rewrite these in terms of functions $\gamma_1, \ldots, \gamma_4$ in a manner explained below. For the tangential Hessian term, we recall \eqref{EquationGammaInverse} and define two functions $\gamma_1$ and $\gamma_2$ as follows:
\begin{equation}\label{EquationGamma1}
	\gamma_1 = (1+r^2) (n-3)^2\alpha^2r^{-2(n-2)} (1-k^2)   +   r^{-2(n-3)}|d \alpha|_g^2 (1-k^2) 
\end{equation}
and it follows that $\gamma_1=\Ol(r^{-2(n-3)})$. We furthermore rewrite 
\begin{equation}
	\begin{split}
		\frac{1}{\Pi}&=\bigg( 1-   2  \sqrt{1+r^2}  \alpha(n-3)r^{-(n-2)}   k\sqrt{1-k^2}+ \gamma_1   \bigg) \\
	&=1- \gamma_2.
	\end{split}
\end{equation}
and $\gamma_2=\Ol(r^{-(n-3)})$. The first term in the right hand side of \eqref{EquationTangentialHessian} is
\begin{equation}
	\begin{split}
		\bigg( \Delta^\Omega(\alpha ) & -(n-3)(n-1)(1+r^2) \alpha       \bigg) \frac{\sqrt{1-k^2}}{\Pi}r^{-(n-1)} = \Ol(r^{-(n-3)}).  \\
	\end{split}
\end{equation}
Furthermore, we have that the second and third terms of the right hand side of \eqref{EquationTangentialHessian} are
\begin{equation}
	\begin{split}
		\frac{\sqrt{1+r^2}}{r}(n-1)\frac{k}{\Pi} 
		&=\frac{\sqrt{1+r^2}}{r}(n-1)k-  2  (n-1)(n-3)    (1+ r^2)    \frac{\alpha}{r^{n-1}}   k^2\sqrt{1-k^2} \\
		&\qquad + \frac{\sqrt{1+r^2}}{r}(n-1)k\gamma_1  
	\end{split}
\end{equation}
and
\begin{equation}
	\frac{\trace^\Omega(\textbf{m})}{r^{n}}\frac{n}{2} \frac{k}{\Pi}=\frac{\trace^\Omega(\textbf{m})}{r^{n}}\frac{n}{2}  k + \Ol(r^{-(n+1)}).
\end{equation}
Combining these estimates we find that the tangential Hessian term in \eqref{EquationTangentialHessian} divided by $\Pi^{3/2}$ is
\begin{equation}
	\begin{split}
		\hat{g}^{\mu\nu}\frac{\Hess^g_{\mu\nu}(f)}{\Pi^{3/2}\sqrt{1+|df|^2_g}} 
		&=\bigg( \Delta^\Omega(\alpha )  -(n-3)(n-1)(1+r^2) \alpha       \bigg) \frac{\sqrt{1-k^2}}{ r^{n-1}} \\
		&\qquad + \bigg( \frac{\sqrt{1+r^2}}{r}(n-1) - \frac{\trace^\Omega(\textbf{m})}{r^{n}}\frac{n}{2}    \bigg) k \\
		&\qquad +  \bigg( \Delta^\Omega(\alpha )  -(n-3)(n-1)(1+r^2) \alpha      \bigg) \frac{\sqrt{1-k^2}}{ r^{n-1}} \gamma_2 	\\
		&\qquad  -  2  (n-1)(n-3)    (1+ r^2)     \alpha   k^2\frac{\sqrt{1-k^2}}{ r^{n-1}}  \\
		&\qquad + \frac{\sqrt{1+r^2}}{r}(n-1)k\gamma_1 +\Ol(r^{-(n+1)}).   \\
	\end{split}
\end{equation}
We similarly expand the trace term divided by $\Pi^{-3/2}$
\begin{equation}
	\begin{split}
		\frac{\trace_{\hat{g}}(k)}{\Pi^{3/2}}  
		&=  \frac{(1+r^{-2(n-3)}|d \alpha|_g^2)(1-k^2) }{\sqrt{\Pi}}  \\
		&\qquad +    \frac{(n-1) }{\Pi^{3/2}}+ \frac{\trace^\Omega(\textbf{p}) -\trace^\Omega(\textbf{m})}{\Pi^{3/2} r^n} +  \Ol (  r^{-(n+ 1)} ) 
	\end{split} 
\end{equation}
by rewriting $\Pi^{-1/2}$ and $\Pi^{-3/2}$ in terms of functions $\gamma_3$ and $\gamma_4$:
\begin{equation}
	\begin{split}
		\frac{1}{\sqrt{\Pi}} &=\bigg(1- \frac{\alpha}{r^{n-2}}(n-3)\sqrt{1+r^2}k\sqrt{1-k^2}+ \gamma_3\bigg) \\
		\frac{1}{\Pi^{3/2}}&=\bigg( 1 -3 \frac{\alpha}{r^{n-2}}\sqrt{1+r^2}(n-3)k\sqrt{1-k^2} + \gamma_4\bigg),
	\end{split}
\end{equation}
where $\gamma_3 = \Ol(r^{-2(n-3)})$ and $\gamma_4=\Ol(r^{-2(n-3)})$. With this, we rewrite the terms in the trace term:
\begin{equation}
	\begin{split}
		\frac{1-k^2}{\sqrt{\Pi}}
		&=(1 - k^2) -   \frac{\alpha}{r^{n-2}}(n-3)\sqrt{1+r^2}k(1-k^2)^{3/2}   +(1 - k^2)\gamma_3
	\end{split}
\end{equation}
and similarly
\begin{equation}
	\begin{split}
		\frac{n-1}{\Pi^{3/2}}&=(n-1)  -3\frac{\alpha}{r^{n-2}}(n-3)(n-1)\sqrt{1+r^2}k\sqrt{1-k^2} + (n-1)\gamma_4.
	\end{split}
\end{equation}
Adding terms and simplifying we compute that the contribution of the trace term to $\frac{\J(f)}{\Pi^{3/2}}$ is  
\begin{equation}
	\begin{split}
		\frac{\trace_{\hat{g}}(k)}{\Pi^{3/2}} 
		& =(1+r^{-2(n-3)}|d \alpha|_g^2)(1 - k^2)-\frac{\alpha}{r^{n-2}}(n-3)\sqrt{1+r^2}k(1-k^2)^{3/2}  \\
		&\qquad  +(1 - k^2)\gamma_3+ (n-1)  -3\frac{\alpha}{r^{n-2}}(n-3)(n-1)\sqrt{1+r^2}k\sqrt{1-k^2} \\
		&\qquad + (n-1)\gamma_4 + \frac{\trace^\Omega(\textbf{p}) -\trace^\Omega(\textbf{m})}{r^n} + \Ol(r^{-(n+1)}).
	\end{split}
\end{equation}
We define $\Lambda$:
\begin{equation}\label{EquationLambda}
	\begin{split}
		\Lambda &=  \bigg( \Delta^\Omega(\alpha )  -(n-3)(n-1)(1+r^2) \alpha      \bigg) \sqrt{1-k^2}\gamma_2r^{-(n-1)} \\
		&\qquad   + \frac{\sqrt{1+r^2}}{r}(n-1)k\gamma_1 -(1 - k^2)\gamma_3 - (n-1)\gamma_4 .\\
	\end{split}
\end{equation}
and collect terms:
\begin{equation}
	\begin{split}
		\frac{\J(f)}{\Pi^{3/2}}
		&=  \sqrt{1+r^2} (1+r^{-2(n-3)}|d \alpha|_g^2) k' \\
		&\qquad + \sqrt{1+r^2}\bigg(\frac{n-1}{r}\bigg)\bigg(  k   -  \frac{r}{\sqrt{1+r^2}}  \bigg) -    (1+r^{-2(n-3)}|d 	\alpha|_g^2)(1-k^2)   \\
		&\qquad + (n-3)\frac{\alpha}{r^{n-2}}\sqrt{1-k^2}\bigg(          (1-k^2) \frac{1+r^2}{r}   (n-3)    +  \frac{1}{r }   \\
		&\qquad     - (n-1)\frac{1+r^2}{r}    -  2  (n-1) \frac{1+r^2}{r}      k^2    \\
		&\qquad      +\sqrt{1+r^2}k(1-k^2)      +  3 (n-1)\sqrt{1+r^2}k    \bigg) \\
		&\qquad + \bigg(   \bigg(\frac{n-2}{2} \bigg)\trace^\Omega(\textbf{m})+ \trace^\Omega(\textbf{p}) \bigg) \bigg(  	\frac{\sqrt{1-k^2}}{r^{n-1}}   - \frac{1}{r^n} \bigg)   \\
		&\qquad + \frac{\trace^\Omega(\textbf{m})}{r^{n}}\frac{n}{2} \big(1-k\big)+ \Lambda + \Ol ( r^{-(n+1)} ), 
	\end{split}
\end{equation}
where we used \eqref{EquationAlpha} to expand the $\Delta^\Omega(\alpha)$-term stemming from the tangential Hessian. Obviously we at least have $\Lambda=\Ol(r^{-1})$, but it will be useful to know the first terms explicitly. We have from \eqref{EquationGammaInverse} and the definition of $\gamma_2$:
\begin{equation}
	\begin{split}
		\gamma_2&= 2\sqrt{1+r^2}\frac{\alpha}{r^{n-2}}(n-3)k \sqrt{1-k^2}- (1+r^2)(n-3)^2\frac{\alpha^2}{r^{2(n-2)}}(1-k^2) \\
		&\qquad - \frac{|d \alpha|_g^2}{r^{2(n-3)}}(1-k^2) 
	\end{split}
\end{equation}
and we compute
\begin{equation}
	\begin{split}
		\gamma_2^2
		&=   4(1+r^2)\frac{\alpha^2}{r^{2(n-2)}}(n-3)^2k^2 (1-k^2)  -4      (1+r^2)^{3/2}(n-3)^3\frac{\alpha^3}{r^{3(n-2)}}k(1-k^2)^{3/2}   \\
		&\qquad    - (1+r^2)^2(n-3)^4\frac{\alpha^4}{r^{4(n-2)}}(1-k^2)^2   + \Ol(r^{-(n+1)}).
	\end{split}
\end{equation}
Similarly, we have
\begin{equation}
	\begin{split}
		\gamma_2^3 
		&=8(1+r^2)^{3/2}\frac{\alpha^3}{r^{3(n-2)}}(n-3)^3k^3 (1-k^2)^{3/2} \\
		&\qquad -12  (1+r^2)^2(n-3)^4\frac{\alpha^4}{r^{4(n-2)}}k^2(1-k^2)^2 +\Ol(r^{-(n+1)}) \\
	\end{split}
\end{equation}
and
\begin{equation}
	\gamma_2^4=16(1+r^2)^2\frac{\alpha^4}{r^{4(n-2)}}(n-3)^4k^4 (1-k^2)^2 + \Ol(r^{-(n+1)}).
\end{equation}
Hence, we may expand:
\begin{equation}
	\begin{split}
		\frac{1}{\sqrt{\Pi}}&= (1-\gamma_2)^{1/2} \\
		&=1-\frac{\gamma_2}{2}+\frac{3}{8}\gamma_2^2 - \frac{5}{16}\gamma_2^3 + \frac{35}{128}\gamma_2^4+ \Ol(r^{-(n+1)}) \\
		&=1  -\sqrt{1+r^2}\frac{\alpha}{r^{n-2}}(n-3)k \sqrt{1-k^2} +\frac{1}{2}\frac{|d \alpha|_g^2}{r^{2(n-3)}}(1-k^2) \\
		&\qquad +\frac{\alpha^2}{r^{2(n-2)}}(1+r^2)(1-k^2)(n-3)^2 \frac{1}{2} \bigg( 1  + 3  k^2     \bigg) \\
		&\qquad + \frac{\alpha^3}{r^{3(n-2)}}(1+r^2)^{3/2}(1-k^2)^{3/2} (n-3)^3 \frac{1}{2} \bigg(  -3  k    -  5 k^3  \bigg) \\
		&\qquad + \frac{\alpha^4}{r^{4(n-2)}}(n-3)^4(1-k^2)^2 (1+r^2)^2 \frac{1}{8} \bigg( -3    + 30   k^2    + 35   k^4   \bigg) \\
		&\qquad +  \Ol(r^{-(n+1)}), \\
	\end{split}
\end{equation}
so that we can read off 
\begin{equation}
	\begin{split}
		\gamma_3&= \frac{1}{2}\frac{|d \alpha|_g^2}{r^{2(n-3)}}(1-k^2) \\
		&\qquad +\frac{\alpha^2}{r^{2(n-2)}}(1+r^2)(1-k^2)(n-3)^2 \frac{1}{2} \bigg( 1  + 3  k^2     \bigg) \\
		&\qquad + \frac{\alpha^3}{r^{3(n-2)}}(1+r^2)^{3/2}(1-k^2)^{3/2} (n-3)^3 \frac{1}{2} \bigg(  -3  k    -  5 k^3  \bigg) \\
		&\qquad + \frac{\alpha^4}{r^{4(n-2)}}(n-3)^4(1-k^2)^2 (1+r^2)^2 \frac{1}{8} \bigg( -3    + 30   k^2    + 35   k^4   \bigg) \\
		&\qquad+  \Ol(r^{-(n+1)})
	\end{split}
\end{equation}
Similarly, we have
\begin{equation}
	\begin{split}
		\frac{1}{\Pi^{3/2}}&=(1-\gamma_2)^{3/2} \\
		&=1- \frac{3}{2}\gamma_2+ \frac{3}{8}\gamma_2^2- \frac{1}{16}\gamma_2^3 + \frac{3}{128}\gamma_2^4 + \Ol(r^{-(n+1)})  \\
		&=1  -3\sqrt{1+r^2}\frac{\alpha}{r^{n-2}}(n-3)k \sqrt{1-k^2} +\frac{3}{2}\frac{|d \alpha|_g^2}{r^{2(n-3)}}(1-k^2) \\
		&\qquad +\frac{\alpha^2}{r^{2(n-2)}}(1+r^2)(1-k^2)(n-3)^2 \frac{3}{2}\bigg( 1  +     k^2     \bigg) \\
		&\qquad + \frac{\alpha^3}{r^{3(n-2)}}(1+r^2)^{3/2}(1-k^2)^{3/2} (n-3)^3\frac{1}{2} \bigg(  -3  k    -    k^3  \bigg) \\
		&\qquad + \frac{\alpha^4}{r^{4(n-2)}}(n-3)^4(1-k^2)^2 (1+r^2)^2 \frac{3}{4} \bigg( -\frac{1}{2}    +      k^2    + \frac{1}{2}   k^4   	\bigg)  \\
		&\qquad   + \Ol(r^{-(n+1)}), \\
	\end{split}
\end{equation}
so that we can read off
\begin{equation}
	\begin{split}
		\gamma_4&= \frac{3}{2}\frac{|d \alpha|_g^2}{r^{2(n-3)}}(1-k^2) \\
		&\qquad +\frac{\alpha^2}{r^{2(n-2)}}(1+r^2)(1-k^2)(n-3)^2 \frac{3}{2}\bigg( 1  +     k^2     \bigg) \\
		&\qquad + \frac{\alpha^3}{r^{3(n-2)}}(1+r^2)^{3/2}(1-k^2)^{3/2} (n-3)^3\frac{1}{2} \bigg(  -3  k    -    k^3  \bigg) \\
		&\qquad + \frac{\alpha^4}{r^{4(n-2)}}(n-3)^4(1-k^2)^2 (1+r^2)^2 \frac{3}{8} \bigg(  -1    +     2 k^2    +     k^4   \bigg)+  	\Ol(r^{-(n+1)}). \\
	\end{split}
\end{equation}
With this at hand we can make the asymptotics of $\Lambda$ in \eqref{EquationLambda} more precise. We find that
\begin{equation}
	\begin{split}
		\bigg(\Delta^\Omega(\alpha)&- (n-3)(n-1)(1+r^2)\alpha \bigg) \frac{\sqrt{1-k^2}}{r^{n-1}} \gamma_2 \\
		&=- (n-3)(n-1)\alpha\frac{\sqrt{1-k^2}}{r^{n-3}} \gamma_2+\frac{\sqrt{1-k^2}}{r^{n-1}}\lambda \\
		&=(n-3)^2(n-1)(1-k^2)\bigg(  -2\frac{\alpha^2}{r^{2(n-3)}}k +   \sqrt{1-k^2}(1+r^2) \frac{\alpha^3}{r^{3(n-2)-1}}   \bigg) \\
		&\qquad +\frac{\sqrt{1-k^2}}{r^{n-1}}\lambda+ \Ol(r^{-(n+1)}),
	\end{split}
\end{equation}
where 
\begin{equation}
	\lambda = \bigg(\Delta^\Omega(\alpha) - (n-3)(n-1) \alpha    \bigg)\gamma_2.
\end{equation}
Clearly, $\lambda=\Ol(r^{-(n-3)})$.	The second term in \eqref{EquationLambda} is 

	\begin{equation}
	\begin{split}
	\frac{\sqrt{1+r^2}}{r}(n-1)k\gamma_1    
	&=\frac{\sqrt{1+r^2}}{r}(1+r^2)(n-3)^2(n-1)k \frac{\alpha^2}{r^{2(n-2)}} (1-k^2) \\
	&\qquad + (n-1)k\frac{|d \alpha|_\Omega^{2}}{r^{2(n-2)}}(1-k^2) +\Ol(r^{-(n+1)}),
	\end{split}
	\end{equation}
where we used \eqref{EquationGamma1}. In summary, we obtain
\begin{equation}
	\begin{split}
		\Lambda
		&=\frac{\sqrt{1-k^2}}{r^{n-1}}\lambda -\frac{|d \alpha|_\Omega^{2}}{r^{2(n-2)}}\frac{(1-k^2)}{2} \bigg(-2(n-1)k + (1-k^2) + 3(n-1)     	\bigg) \\
		&\qquad + \frac{\alpha^2}{r^{2(n-2)}}(1+r^2)(1-k^2)(n-3)^2 \bigg(   -2 (n-1) k \frac{r^2}{1+r^2}+ \frac{\sqrt{1+r^2}}{r}(n-1) k\\
		&\qquad  \qquad   -(1-k^2)\frac{1}{2}\bigg(1   + 3    k^2      \bigg) - (n-1)\frac{3}{2}(1+k^2) \bigg) \\
		&\qquad + \frac{\alpha^3}{r^{3(n-2)}}(1+r^2)^{3/2}(1-k^2)^{3/2} (n-3)^3\\
		&\qquad \qquad \times \bigg(    (n-1)    \frac{\sqrt{1+r^2}}{r}  -  \frac{k}{2} (  3   + 5 k^2   ) -\frac{k}{2}( 3      +   k^2)  	\bigg) \\
		&\qquad -   \frac{\alpha^4}{r^{4(n-2)}}(n-3)^4(1-k^2)^2 (1+r^2)^2\\
		&\qquad \qquad \times     \bigg(  (1-k^2)\frac{3}{8}\bigg( -1 +2k^2 + k^4   \bigg)   +(n-1)\frac{3}{8}\bigg(  - 1  + 2    k^2    +  	k^4   \bigg) \bigg) \\
		&\qquad +   \Ol(r^{-(n+1)}),
	\end{split}
\end{equation}
which completes our assertion.
	
\end{proof}

With this result at hand we can estimate the left hand side of Jang's equation from above and from below, thereby obtaining the eququations for sub- and supersolutions. 

\begin{lemma}\label{LemmaBarrierConstants}
	There exists positive constants $C_1,C_2,C_3,C_4$ such that 
	\begin{equation}\label{EquationSuperBarrier}
	\begin{split}
	\frac{\J(f)}{\sqrt{1+r^2}} &\frac{1}{\Pi^{3/2}(1+r^{-2(n-3)}|d \alpha|_g^2)}  \leq \J_+(k) \\
	&=k'+  \bigg(\frac{n-1}{r}\bigg) \bigg( k -\frac{r}{\sqrt{1+r^2}} \bigg) -  \frac{1-k^2}{\sqrt{1+r^2}} \\
	&\qquad + \frac{\sqrt{1-k^2}}{\sqrt{1+r^2}}\frac{C_1}{r^{n-2}}\bigg|      (1-k^2)\frac{1+r^2}{r}   (n-3)    +  \frac{1}{r }   \\
	&\qquad \qquad    - (n-1)\frac{1+r^2}{r}   -  2  (n-1) \frac{1+r^2}{r}      k^2    \\
	&\qquad \qquad   +\sqrt{1+r^2}k(1-k^2)     +  3 (n-1)\sqrt{1+r^2}k \bigg| \\
	&\qquad  +  \frac{C_1^2}{r^{2(n-2)}}\sqrt{1+r^2}(1-k^2)   \bigg|   -2 (n-1) k \frac{r^2}{1+r^2}+ \frac{\sqrt{1+r^2}}{r}(n-1) k\\
	&\qquad  \qquad   -(1-k^2)\frac{1}{2}\bigg(1   + 3    k^2      \bigg) - (n-1)\frac{3}{2}(1+k^2) \bigg| \\
	&\qquad +  \frac{C_1^3}{r^{3(n-2)}}(1+r^2) (1-k^2)^{3/2}  \\
	&\qquad \qquad \times \bigg|    (n-1)    \frac{\sqrt{1+r^2}}{r}  -  \frac{k}{2} (  3   + 5 k^2   ) -\frac{k}{2}( 3      +   k^2)   \bigg| \\
	&\qquad +  \frac{C_1^4}{r^{4(n-2)}} (1-k^2)^2 (1+r^2)^{3/2}\\
	&\qquad \qquad \times  \bigg| \bigg(  (1-k^2)\frac{3}{8}\bigg( -1 +2k^2 + k^4   \bigg)   +(n-1)\frac{3}{8}\bigg(  - 1  + 2    k^2    +  k^4   \bigg) \bigg)   \bigg| \\
	&\qquad + C_2 \bigg| \frac{\sqrt{1-k^2}}{r^{n }}   - \frac{1}{r^{n+1}} \bigg|   \\
	&\qquad + \frac{C_3}{r^{n+1}}  |1-k |  +  \frac{C_4}{r^{n+2}} 
	\end{split}
	\end{equation}
	and
	\newpage
	\begin{equation}\label{EquationSubBarrier}
	\begin{split}
	\frac{\J(f)}{\sqrt{1+r^2}} &\frac{1}{\Pi^{3/2}(1+r^{-2(n-3)}|d \alpha|_g^2)}  \geq \J_-(k) \\
	&=k'+  \bigg(\frac{n-1}{r}\bigg) \bigg( k -\frac{r}{\sqrt{1+r^2}} \bigg) -  \frac{1-k^2}{\sqrt{1+r^2}} \\
	&\qquad - \frac{\sqrt{1-k^2}}{\sqrt{1+r^2}}\frac{C_1}{r^{n-2}}\bigg|      (1-k^2)\frac{1+r^2}{r}   (n-3)    +  \frac{1}{r }   \\
	&\qquad \qquad    - (n-1)\frac{1+r^2}{r}   -  2  (n-1) \frac{1+r^2}{r}      k^2    \\
	&\qquad \qquad   +\sqrt{1+r^2}k(1-k^2)     +  3 (n-1)\sqrt{1+r^2}k \bigg| \\
	&\qquad  - \frac{C_1^2}{r^{2(n-2)}}\sqrt{1+r^2}(1-k^2)   \bigg|   -2 (n-1) k \frac{r^2}{1+r^2}+ \frac{\sqrt{1+r^2}}{r}(n-1) k\\
	&\qquad  \qquad   -(1-k^2)\frac{1}{2}\bigg(1   + 3    k^2      \bigg) - (n-1)\frac{3}{2}(1+k^2) \bigg| \\
	&\qquad -  \frac{C_1^3}{r^{3(n-2)}}(1+r^2) (1-k^2)^{3/2}  \\
	&\qquad \qquad \times \bigg|    (n-1)    \frac{\sqrt{1+r^2}}{r}  -  \frac{k}{2} (  3   + 5 k^2   ) -\frac{k}{2}( 3      +   k^2)   \bigg| \\
	&\qquad -  \frac{C_1^4}{r^{4(n-2)}} (1-k^2)^2 (1+r^2)^{3/2}\\
	&\qquad \qquad \times  \bigg| \bigg(  (1-k^2)\frac{3}{8}\bigg( -1 +2k^2 + k^4   \bigg)   +(n-1)\frac{3}{8}\bigg(  - 1  + 2    k^2    +  k^4   \bigg) \bigg)   \bigg| \\
	&\qquad - C_2 \bigg| \frac{\sqrt{1-k^2}}{r^{n }}   - \frac{1}{r^{n+1}} \bigg|   \\
	&\qquad - \frac{C_3}{r^{n+1}}  |1-k |  -  \frac{C_4}{r^{n+2}}
	\end{split}
	\end{equation}

\end{lemma}

\begin{proof}
	
The assertion follows directly from Lemma \ref{LemmaBarrierAsymptotics}.

\end{proof} 

The following Proposition will be useful.

\begin{proposition}\label{PropositionSubSuperODE}

Suppose $y\in C^1_{loc}([r_0, \infty))$ solves the a first order ODE
\begin{equation}
	y'+ F(r,y) =0 , \qquad y(r_0)=y_0,
\end{equation}
on $[r_0, \infty)$ and suppose furthermore that there exist sub- and supersolutions $y_-$ and $y_+$ on $[r_0, \infty)$ with $y_-(r_0)\leq y_0$ and $y_0 \leq y_+(r_0)$. Then 
\begin{equation}
	y_- \leq y \leq y_+ 
\end{equation}
on $[r_0, \infty)$.

\end{proposition}

\begin{proof}

We follow the proof of \cite[Lemma 3.3]{SakovichPMTah}. We have
\begin{equation}\label{EquationAuxAss}
	y'(r)+ F(r,y (r))< y_+'(r) + F(r,y_+(r))
\end{equation}
and to see that $y_-(r)\leq y_+(r)$ for $r\geq r_0$ we assume that for some $r_1>r_0$ we have $y_-(r_1)>y_+(r_1)$. Let $r_2=\inf\{r>r_0 \: | \: y_-(r)>y_+(r )\}$. Then $r_0\leq r_2<r_1$ and $y(r_2)=y_+(r_2)$. On the one hand, it follows from \eqref{EquationAuxAss} that $y'(r_2)< y_+'(r_2)$ so that $y(r_2+\epsilon)< y_+(r_2+\epsilon)$ for sufficiently small $\epsilon>0$. On the other hand, this contradicts the definition of $r_2$.
%
\\ \indent Arguing similarly, we obtain $y_-\leq y$.

\end{proof}

We can now prove the existence of sub- and supersolutions.

\begin{lemma}\label{LemmaBarrierSubSuperSols}
	Let $C_{1}, C_2, C_3, C_4$ be as in Lemma \ref{LemmaBarrierConstants}. For $r_0$ large enough, there exists a solution $k_+:[r_0,\infty) \rightarrow \rn$ to $\J_+(k_+)=0$ such that $k_+(r_0)=-1$ and $|k_+|<1$ for $r>r_0$. Similarly, there exists a solution $k_-:  [r_0,\infty)\rightarrow \rn$ to $\J_-(k_-)=0$ such that $k_-(r_0)=+1$ and $|k_-|<1$ for $r>r_0$.
\end{lemma}

\begin{proof}
	
We start with $k_+$. We observe that $k_+^+=+1$ and $k_+^-=-1$ yield by Equations \eqref{EquationSubBarrier} and \eqref{EquationSuperBarrier} $\J_+(k_+^+)>0$ and $\J_+(k_+^-)<0$ for large enough $r_0$. Furthermore $\J_+(k_+)=0$ is on the form $k'(r)+F(r,k)=0$, where $F(r,k)$ is continuous in both variables for large enough $r_0$, and so by Peano's Existence Theorem \cite[Theorem 2.1]{HartmanODE} we have a solution. By Proposition \ref{PropositionSubSuperODE} we have $|k_+|\leq 1$. To verify that $|k_+|<1$ we observe that if there is a point $r_1>r_0$ such that $k(r_1)=+1$ we would have $k'(r_1)<0$. By the continuity of $k$ there must be  $k(r_1-\epsilon)>1$ for any sufficiently small $\epsilon>0$, but this contradicts $|k|< 1$. Similarly, there can be no points $r_1$ with $k_+(r_1)=-1$. This together with \cite[Corollary 3.1]{HartmanODE}, shows that $k_+$ is defined for all $r\geq r_0$.
\\ \indent Existence of $k_-$ with asserted properties is proved in a similar way.

\end{proof}

In Lemma \ref{LemmaBarrierKasymptotics} we establish the asymptotics of $k_+$ and $k_-$.

\begin{lemma}\label{LemmaBarrierKasymptotics}
	For $\epsilon>0$ small enough, there exists $r_0>0$ such that $k_{\pm}$ from Lemma \ref{LemmaBarrierSubSuperSols} satisfy
	\begin{equation}
	k_{\pm}(r)= \frac{r}{\sqrt{1+r^2}} + \Ol ( r^{-(n+1-\epsilon)} ).  
	\end{equation}
	
\end{lemma}

\begin{proof}
	
	We will do finite induction to show the asymptotics. We start with the equation $\J_+(k_+)=0$ and define 
	\begin{equation}
	h(r)= 1 - 2\bigg(\frac{r_0}{r}\bigg)^{2- \epsilon }, 
	\end{equation}
	for $\epsilon>0$ small. Clearly $h(r_0)=-1$ and further we have $h'=2(2-\epsilon)\frac{r_0^{2-\epsilon}}{r^{3-\epsilon}}$ and
	\begin{equation}
		\begin{split}
			\bigg(\frac{n-1}{r}\bigg)  \bigg( h -\frac{r}{\sqrt{1+r^2}} \bigg) &= \bigg(\frac{n-1}{r}\bigg) \bigg(  - 2\bigg(\frac{r_0}{r}\bigg)^{2- \epsilon } + \frac{1}{\sqrt{1+r^2}(r+\sqrt{1+r^2})}   \bigg) \\
			&\leq- (n-1)\frac{r_0^{2-\epsilon}}{r^{3-\epsilon}}+   \bigg(\frac{n-1}{2}\bigg)r^{-3},
		\end{split}
	\end{equation}
	as well as 
	\begin{equation}
	 	1-h^2  =4\bigg(\frac{r_0}{r}\bigg)^{2-\epsilon} \bigg(1-\bigg(\frac{r_0}{r}\bigg)^{2-\epsilon}\bigg).
	\end{equation}
	In particular, $\sqrt{1-h^2}\leq 2(\frac{r_0}{r})^{1-\epsilon/2}$. The $C_1$-term of \eqref{EquationSuperBarrier} may be estimated as follows:
	\begin{equation}\label{EquationEstimate}
		\begin{split}
			\frac{C_1}{r^{n-2}}&\frac{ \sqrt{1-h^2}  }{\sqrt{1+r^2}} \bigg|        (n-3) (1-h^2) \frac{1+r^2}{r}       +  \frac{1}{r }      - 	(n-1)\frac{1+r^2}{r} -  2  (n-1) \frac{1+r^2}{r}   h^2  \\
			&\qquad          +\sqrt{1+r^2}h(1-h^2)      + 3(n-1)\sqrt{1+r^2}h    \bigg| \\
			&\leq\frac{C_1}{r^{n-1}}2\bigg(\frac{r_0}{r}\bigg)^{1-\epsilon/2}\bigg|   \frac{1}{r} -3(n-1)\frac{1+r^2}{r}+3(n-1)\sqrt{1+r^2}  	\bigg| \\
			&\qquad + \frac{C_1}{r^{n-2}}2\bigg(\frac{r_0}{r}\bigg)^{1-\epsilon/2}(1-h^2)\bigg|   (3n-5)\frac{\sqrt{1+r^2}}{r} +h \bigg| \\
			&\qquad + \frac{C_1}{r^{n-1}}12\bigg(\frac{r_0}{r}\bigg)^{2-3\epsilon/2 }   \\
			&\leq \frac{C_8}{r^{n-2}}.
		\end{split}
	\end{equation}
	The $C_1^2$-, $C_1^3$- and $C_1^4$-terms in \eqref{EquationSuperBarrier} are all easily estimated to decay as $\Ol(r^{-(n-2)})$. The $C_2$-, $C_3$- and $C_4$-terms are similarly estimated to $\Ol(r^{-{n+2} })$. Inserting these estimates into $\J_+(h)$ yields
	\begin{equation}\label{EquationSuperBarrierEval1}
		\begin{split}
			\J_+(h) &\leq 2(2-\epsilon)\frac{r_0^{2-\epsilon}}{r^{3-\epsilon}} + \bigg(\frac{n-1}{2}\bigg)\frac{1}{r^3} - 	2(n-1)\frac{r_0^{2-\epsilon}}{r^{3-\epsilon}}+ \frac{C_6}{r^5} \\
			&\qquad -4\bigg(\frac{r_0}{r}\bigg)^{2-\epsilon}\bigg(1-\bigg(\frac{r_0}{r}\bigg)^{2-\epsilon}\bigg)\frac{1}{\sqrt{1+r^2}}  	+\frac{C_7}{r^{n-2}}\\
			&=\bigg(\frac{r_0}{r}\bigg)^{2-\epsilon}\frac{2}{r} \bigg(  \epsilon + 2\bigg(\frac{r_0}{r}\bigg)^{2-\epsilon} - (n-1) \bigg)   	+\Ol(r^{-3}) + \Ol(r^{-(n-2)}). 
		\end{split}
	\end{equation}
	We see that the sum in the paranthesis of the first term is negative for small $\epsilon$ and so we have a subsolution for large enough $r_0$. We have already seen that $k_+^+=+1$ is a supersolution. By Proposition \ref{PropositionSubSuperODE} we have $h \leq k_+ \leq 1$.  
	\\ \indent Now we write $k_+= 1 + g$, where $g= \Ol_1(r^{-(2-\epsilon)})$ from the previous step, where again $\epsilon>0$ is small. Clearly $k_+'=g'$ and straightforward calculations yield 
	\begin{equation}
		\begin{split}
			\bigg(\frac{n-1}{r}\bigg)\bigg( k_+ - \frac{r}{\sqrt{1+r^2}}\bigg) &=\bigg(\frac{n-1}{2}\bigg)\frac{1}{r^3} + 	\bigg(\frac{n-1}{r}\bigg)g + \Ol ( r^{-5} ),   \\
			- \frac{1-k_+^2}{\sqrt{1+r^2}} &=  \frac{2}{r}g  + \Ol ( r^{-(5-2\epsilon)} ).
		\end{split}
	\end{equation}
	We estimate the $C_1$-term of \eqref{EquationSuperBarrier}:
	\begin{equation}
		\begin{split}
			\frac{\sqrt{1-k_+^2}}{\sqrt{1+r^2}}\frac{C_1}{r^{n-2}}&\bigg|      (1-k_+^2)\frac{1+r^2}{r}   (n-3)    +  \frac{1}{r } - 	(n-1)\frac{1+r^2}{r}  \\
			&\qquad \qquad     -  2  (n-1) \frac{1+r^2}{r}      k_+^2    \\
			&\qquad \qquad   +\sqrt{1+r^2}k_+(1-k_+^2)     +  3 (n-1)\sqrt{1+r^2}k_+ \bigg| \\
			&=\Ol ( r^{-(n+1-3\epsilon/2)} )
		\end{split}
	\end{equation}
	Similarly, we estimate the $C_1^2$-, $C_1^3$-, $C_1^4$-terms to be $\Ol(r^{-(n+2)})$, the $C_2$-term to be $\Ol(r^{-(n+1-\epsilon/2)})$ and the $C_3$-term to be $\Ol(r^{-(n+2)})$. Inserting into $\J_+(1+g)=0$ yields 
	\begin{equation}\label{EquationSuperBarrierEval2}
		\begin{split}
			0&= g' + \bigg(\frac{n-1}{2}\bigg)\frac{1}{r^3} + \bigg(\frac{n-1}{r}\bigg)g + \Ol ( r^{-5} ) + 2\frac{g}{r}+ \Ol ( r^{-(5- 	2\epsilon)} ) \\
			&= g' +\bigg(\frac{n+1}{r}\bigg)g + \bigg(\frac{n-1}{2}\bigg) \frac{1}{r^3}    + \Ol ( r^{-(5-2\epsilon)} ) .  
		\end{split}
	\end{equation}
	for any $n\geq 4$. We multiply with $r^{n+1}$ and integrate from $r_0$ to $r$:
	\begin{equation}
		\begin{split}
			0 &=\int_{r_0}^r\bigg( (s^{n+1}g)'+ \bigg(\frac{n-1}{2}\bigg) s^{n-2}  +\Ol(r^{n -4+2\epsilon})\bigg) ds \\
			&=r^{n+1}g(r)- r_0^{n+1}g(r_0) + \frac{r^{n-1}}{2}- \frac{r^{n-1}_0}{2} + \Ol(r^{n-3+2\epsilon}) 
		\end{split}
	\end{equation}
	so that in turn we find $g(r)$:
	\begin{equation}
		g(r)=-\frac{1}{2r^2} + \Ol ( r^{-2(2- \epsilon)} ),
	\end{equation}
	for any $n\geq 4$.
	\\ \indent We repeat the argument inductively, but now we assume that
	\begin{equation}
		k_+= \frac{r}{\sqrt{1+r^2}} + g,
	\end{equation}
	where $g=\Ol(r^{-(2m - \epsilon)})$, for integer $m\geq 2$. $2m=n+1$ is our assertion, so we assume that $2m\leq n$. Straightforward calculations yield
	\begin{equation}
		k_+'=\frac{1}{(1+r^2)^{3/2}} + g', \\
	\end{equation}
	and
	\begin{equation}
		- \frac{1-k_+^2}{\sqrt{1+r^2}}= - \frac{1}{(1+r^2)^{3/2}} + \frac{2}{r} g + \Ol(r^{-(2m+3-\epsilon)}) .
	\end{equation}
	To estimate the $C_1$-term in \eqref{EquationSuperBarrier} we first observe that $\sqrt{1-k_+^2}=\Ol(r^{-1})$. With this at hand we get
	\begin{equation}
		\begin{split}
			\frac{\sqrt{1-k^2_+}}{\sqrt{1+r^2}}&\frac{C_1}{r^{n-2}} \bigg|         (1-k_+^2) \frac{1+r^2}{r}   (n-3)    +  \frac{1}{r }      - 	(n-1)\frac{1+r^2}{r}    -  2  (n-1) \frac{1+r^2}{r}      k_+^2    \\
			&\qquad     +\sqrt{1+r^2}k(1-k_+^2)    +  3 (n-1)\sqrt{1+r^2}k_+     \bigg| \\
			&= \Ol ( r^{-n} ) \times \Ol ( r^{-(2m-1-|\epsilon|)} ) \\
			&=\Ol ( r^{-(n+2)} ).
		\end{split}
	\end{equation}
	We estimate the $C_1^2$-, $C_1^3$-, $C_1^4$-terms to be $\Ol(r^{-(n+2)})$. We estimate the $C_2$-term to be $\Ol(r^{-(n+1-\epsilon/2)})$ and the $C_3$-term to be $\Ol(r^{-(n+2)})$. Clearly both the $C_2$- and $C_3$-terms are $\Ol(r^{-(n+2)})$. Inserting into $\J_+(k_+)=0$ as above yields
	\begin{equation}\label{EquationSuperBarrierEval3}
		0=  g' + \frac{g}{r}  (n+1)+ \Ol ( r^{-(2m+3 -  \epsilon)} )  + \Ol ( r^{-(n+2)} ).  
	\end{equation}
	First, let us consider the case when $2m+3\leq n+2$ so that the first $\Ol$-term dominates. We multiply by $r^{n+1}$ and integrate from $r_0$ to $r$ as above to find:  
	\begin{equation}
		0=   r^{n+1}g(r)-r_0^{n+1}g(r_0)+ \Ol( r^{-(2m+2-(n+1) -  \epsilon)}  )
	\end{equation}
	from which it follows that 
	\begin{equation}
		g(r)=\Ol ( r^{-2(m+1-\epsilon) } ),
	\end{equation}
	or equivalently that
	\begin{equation}
		k_+=\frac{r}{\sqrt{1+r^2}}+ \Ol ( r^{-2(m+1-\epsilon) } ).
	\end{equation}
	Finally, in the case of $2m+3 >n+2$, we recall that $2m\leq n$, which implies $2m=n$. We multiply by $s^{n+1}$ and integrate from $r_0$ to $r$:
	\begin{equation}
		0=r^{n+1}g(r)-r_0^{n+1}g(r_0)+ \Ol(1)
	\end{equation}
	so that 
	\begin{equation}
		g(r)=\Ol ( r^{-(n+1)} ),
	\end{equation}
	which is even stronger than asserted.
	\\ \indent It is not difficult to see that the same procedure works for $k_-$. The only difference is that the $C_i^k$-terms are negative in this case, but the evaluations corresponding to Equations \eqref{EquationSuperBarrierEval1}, \eqref{EquationSuperBarrierEval2} and \eqref{EquationSuperBarrierEval3} for the asymptotics in \eqref{EquationSubBarrier} remain the same and so the argument still works.

\end{proof}

We are now ready to construct the barriers as in Definition \ref{DefinitionBarrier}.

\begin{proposition}\label{PropositionBarrierExistence}
	For $r_0$ sufficiently large, there exists $f_+,f_-:\mathbb{S}^{n-1}\times [r_0,\infty)\rightarrow \rn$ such that
	\begin{enumerate}
		\item $f_+$ (respectively $f_-$) is an upper (respectively lower) barrier in the sense of Definition \ref{DefinitionBarrier};
		\item the asymptotics of $f_+$ and $f_-$ is
		\begin{equation}\label{EquationBarrierAsymptotics}
			f_{\pm}= \sqrt{1+r^2} +\frac{\alpha}{r^{n-3}}+ \Ol ( r^{-(n-2-\epsilon)} ), 
		\end{equation}
		where
		\begin{equation}
		\Delta^\Omega(\alpha)- (n-3)\alpha= \bigg(\frac{n-2}{2} \bigg)\trace^\Omega(\textbf{m})+\trace^\Omega(\textbf{p});
		\end{equation}
		\item $f_-\leq f_+$.
		
	\end{enumerate}
\end{proposition}

\begin{proof}
	
	We let $\epsilon >0$ be given and take $r_0$, $k_{\pm}$ as in Lemmas \ref{LemmaBarrierAsymptotics} and \ref{LemmaBarrierKasymptotics}. Clearly $k'_{\pm}(r_0)\not =0$ and hence, by continuity, $k_{\pm}'\not =0$ on $[r_0,r_0+ \delta]$ for some $\delta>0$. Hence $1\pm k_{\pm}(r)\geq C(r-r_0)$, for some $C>0$, when $r\in  [r_0,r_0+\delta]$. It follows that 
	\begin{equation}
		\varphi_{\pm}'(r)=\frac{k_{\pm}(r)}{\sqrt{(1-k_{\pm}(r)^2)(1+r^2)}}
	\end{equation}
	defines (modulo constants) the continuous functions $\varphi_{\pm}(r)$ on $[r_0,\infty)$, both of which are $C^2((r_0,\infty))$. From the asymptotics in Lemma \ref{LemmaBarrierAsymptotics} it follows that 
	\begin{equation}
		\varphi'_{\pm}(r)=\frac{r}{\sqrt{1+r^2}}    + \Ol ( r^{-(n-1-\epsilon)} )
	\end{equation}
	and since the Jang equation is invariant under vertical translations we assume 
	\begin{equation}
		\varphi_{\pm}(r)= \sqrt{1+r^2} +  \Ol ( r^{-(n-2 -\epsilon)} ).
	\end{equation}
	We define
	\begin{equation}
		f_{\pm}= \varphi_{\pm} + \frac{\alpha}{r^{n-3}}  .
	\end{equation}
	By Lemmas \ref{LemmaBarrierConstants} and \ref{LemmaBarrierSubSuperSols} $f_{\pm}$ will satisfy
	\begin{equation}
		\J(f_+)< 0, \qquad \J(f_-)> 0.
	\end{equation}
	Furthermore, we have
	\begin{equation}
		f_{+,r}(r_0)=-\infty \qquad \text{and} \qquad f_{-,r}(r_0)=+\infty,
	\end{equation}
	since $k_+(r_0)=-1$ and $k_-(r_0)=+1$.
	\\ \indent It remains only to show that $f_-\leq f_+$. We use a version of the Bernstein trick as in \cite{PMTII}, Proposition 3. Clearly, the difference $f_+-f_-$ does not depend on $\theta$ and there must exist a constant $L_0\geq 0$ such that $f_+-f_-> - L_0$ when $r\geq r_0$. We let $L_0$ be the infimum of such constants and show that $L_0=0$. Then 
	\begin{equation}
		(f_+-f_-)(r)\geq -L_0
	\end{equation}
	for all $r\in[r_0,\infty)$ and either the equality is attained at some fixed $r_1\in [r_0, \infty)$ or 
	\begin{equation}
		\lim_{r\rightarrow \infty} (f_+-f_-)(r) = -L_0.
	\end{equation}
	In the latter case, we must obviously have $L_0=0$. In the former case, we first suppose that $r_1=r_0$. But then $(f_+-f_-)(r_0)\geq 0 $ which directly contradicts the properties of the barriers. Now assume that $r_1>r_0$. We let $p\in M^n$ be a point with $r(\Psi(p))=r_1$ out in the chart at infinity. Since $(f_+-f_-)_{,r}(r_0)= 0$ and $f_+-f_-$ is radially symmetric we must have $f_{+,i}=f_{-,i}$ at $p$ and hence also
	\begin{equation}
		g^{ij}-\frac{f_+^{,i}f_+^{,j}}{1+|d f_+|_g^2} = g^{ij}-\frac{f_-^{,i} f_-^{,j}}{1+|d f_-|_g^2} = \hat{g}_{\pm}^{ij}
	\end{equation}
	at $p$. Since $\hat{g}_{\pm}^{ij}$ is the inverse matrix of $\hat{g}_{ij}$, which is positive definite, it is itself positive definite. Clearly, the same arguments also implies that $|d f_+|_g^2=|d f_-|_g^2$. Furthermore, since $p$ is a local minimum, it must follow that the matrix $(f_+-f_-)_{,ij}(p)$ is non-negative definite. Finally, since $f_+$ (and $f_-$) is a supersolution (respectively subsolution) we must have
	\begin{equation}
		\begin{split}
			0&> \J(f_+)-\J(f_-) \\
			&=\bigg(g^{ij} -\frac{f_+^{,i} f_+^{,j} }{1+|d f_+|_g^2}\bigg) \bigg(\frac{\Hess_{ij}^g(f_+)}{\sqrt{1+|d 	f_+|_g^2}}  - k_{ij}   \bigg) \\
			&\qquad - \bigg(g^{ij} -\frac{f_-^{,i} f_-^{,j} }{1+|d f_-|_g^2}\bigg) \bigg(\frac{\Hess_{ij}^g(f_-)}{\sqrt{1+|d 	f_-|_g^2}} - k_{ij}   \bigg) \\
			&=\hat{g}^{ij}_{\pm} \frac{\Hess_{ij}^g(f_+-f_-)}{\sqrt{1+|d f_{\pm}|_g^2}} \\
			&=\hat{g}^{ij}_{\pm} \frac{(f_+-f_-)_{,ij}}{\sqrt{1+|d f_{\pm}|_g^2}},
		\end{split}
	\end{equation} 
	at $p$. We see that on the one hand the last term must be negative but on the other hand it is the trace of a positive definite matrix $\hat{g}_{\pm}^{ij}$ with a non-negative definite matrix $(f_+-f_-)_{,ij}(p)$ and we have a contradiction to $f_+-f_-$ having a local minimum. We have shown that $L_0=0$ and hence $f_+\geq f_-$.  
	
\end{proof}

\section{The regularized Jang equation as a Dirichlet problem}\label{SectionDirichletProblem}

In this section we perform the first step in solving Jang's equation \eqref{EquationJangsEquation} $\J(f)=0$ which is to solve the {\it regularized equation} $\J(f)=\tau f$, where $\tau>0$ is small, on bounded sets. This circumvents the lack of zeroth order derivatives in $\J(f)$ and yields a priori $\tau$-dependent supremum estimates $\sup_{M^n}|f|$. Consequently, we can solve the regularized equation on compact sets $\overline{\Omega}\subset M^n$, provided that $\partial \Omega$ satisfies a certain geometric condition. The procedure is well-known, (see for instance \cite{AEM11}), but we include it for completeness. 
\begin{definition}\label{DefinitionTrappingAssumption}
	Let $\Omega\subset M^n$ be a bounded subset with boundary $\partial \Omega$ and let $H_{\partial \Omega}$ be the mean curvature of $\partial \Omega$ computed as the divergence of the outward pointing unit normal. If
	\begin{equation}\label{EquationTrappingCondition}
		H_{\partial \Omega} > |\trace_{\partial \Omega}(k)|,
	\end{equation}
	then $\Omega$ is said to fulfill the \emph{trapping condition}.
\end{definition}

Let $\tau>0$ be small. We want to solve the Dirichlet problem 
\begin{equation} \label{EquationCapillaryJang}
	\begin{split}
		\J(f_{\tau}) &=\tau f_{\tau} \qquad    \text{on}  \: \Omega, \\
		f_{\tau}&=\varphi \:\:   \qquad \:  \text{on}  \: \partial  \Omega,
	\end{split}
\end{equation}
for $\Omega$ as in Definition \ref{DefinitionTrappingAssumption}. For the remainder of this section, we will supress the index $\tau$ on $f_\tau$ and refer to the solution of \eqref{EquationCapillaryJang} as $f$ for brevity.
\\ \indent The solution to Problem \ref{EquationCapillaryJang} is obtained by the continuity method, where we define the parametrized Jang operator $\J_s(f)=H(f ) - s\trace (k)(f )$, with $s\in [0,1]$, and consider the following parametrized problem:
\begin{equation}\label{EquationContCapJang}
	\begin{split}
		\J_s(f_{s})&=  \tau f_{s} \qquad   \text{in} \: \Omega, \\
		f_{s} &=s \varphi \qquad  \: \text{on} \: \partial  \Omega,
	\end{split}
\end{equation}
where $\varphi \in C^{2,\alpha}(\partial \Omega)$. Here we take $\alpha$ fixed; we will at the end of the proof of Lemma \ref{LemmaSisClosed} find a $0<\beta \leq 1$ and throughout this section we fix $0<\alpha <\beta$. Let $\S\subset [0,1]$ denote the subset of parameters $s$ such that Equations \eqref{EquationContCapJang} has a solution in $C^{2,\alpha}(\bar{\Omega})$. In Lemmas \ref{LemmaSisClosed} and \ref{LemmaSisOpen} below, we will show that $\S$ is both relatively open and closed.

\begin{lemma}\label{LemmaSisClosed}
	Let $\Omega$ satisfy the trapping condition in \eqref{EquationTrappingCondition}. Then $\S$ is closed.
\end{lemma}

\begin{proof}
	
We start by establishing a uniform $C^1(\Omega)$-bound in $s$, which will be subsquently upgraded to a $C^{2,\alpha}(\Omega)$-bound via standard theory for elliptic equations. 
\\ \indent First, we apply a maximum principle argument to show the uniform estimate $\tau |f_s |\leq C$ on $\bar{\Omega}$, where $C$ is a constant depending only on the initial data $(M^n, g,k)$ and not on $s$. If $f_s$ achieves its maximum at some interior point $p$, then $df_s=0$ so that $\hat{g}=g$ and $|df_s|^2_g=0$ at $p$. Further, the Hessian at $p$ reduces to $\Hess_{ij}^g(f_s)= (f_{s})_{,ij}$ and is non-positive definite there. Hence, at a point $p$ we have 
\begin{equation}
	\begin{split}
		\tau f_s  &= \J_s (f_s) \\
		&=g^{ij}\Hess_{ij}^g(f_s) - s\trace^{\hat{g}_s}(k) \\
		&\leq-s\trace^g(k) \\
		&\leq     |\trace^g(k)|\\
	\end{split}
\end{equation}
as $g^{ij}$ is positive definite. Similarly, if $f_s$ has a minimum at $p$, we get $\tau f_s  \geq - |\trace^g(k)|$. Conclusively, we have shown that $\tau |f_s|\leq  |\trace^g(k)|$ at $p$ and hence it follows that
\begin{equation}
	\tau|f_s| \leq \max \bigg( \sup_{\bar{\Omega}}  |\trace^g(k)|,\: \sup_{\partial \Omega}\tau |\varphi | \bigg) ,
\end{equation}
which only depends on the initial data $(M^n,g,k)$ and the boundary data $\varphi$.
\\ \indent We now establish a bound, uniform in $s$, for the gradient $df_s$. We start with an interior estimate $|df_s|_g$ in $\Omega$. Suppose $|df_s|_g^2$ achieves its maximum at some point $p\in \Omega$. We take the covariant derivative of both sides of $\J_s(f_s) = \tau f_s$ and contract with $\nabla^g f_s$ to recover an expression for $|df_s|^2_g$. Straightforward calculations show that 
\begin{equation}\label{EquationGradientlengthDerivative}
	\nabla_k|df_s|_g^2 = 2f^{,i}\Hess_{ik}^g(f_s) 
\end{equation} 
and 
\begin{equation}
	\begin{split}
		(\nabla_k\hat{g}_s)^{ij} 
		&=   - \frac{ f_s^{,j} g^{im} \Hess_{mk}^g(f_s)}{1+|df_s|_g^2}  
		-  \frac{f_s^{,i} g^{jm}\Hess_{mk}^g(f_s)}{1+|df_s|_g^2}  \\
		&\qquad  + 2\frac{f_s^{,i}f_s^{,j}}{(1+|df_s|_g^2)^2}f^{,\ell}_s\Hess_{\ell k}^g(f_s).   \\
	\end{split}
\end{equation}
For our calculations below, it will be convenient to recast the Hessian term in divergence form. We straightforwardly get
\begin{equation}
	\begin{split}
		\bigg( \nabla_m\bigg( \frac{\nabla^g f_s}{\sqrt{1+ |df_s|^2_g}}\bigg) \bigg)^k 
		&=\bigg( g^{k \ell} - \frac{f_s^{,k}f_s^{,\ell}}{1+|df_s|_g^2}\bigg) \frac{\Hess_{\ell m}^g(f_s)}{\sqrt{1+|df_s|_g^2}}  .
	\end{split}
\end{equation}
Contracting over $k$ and $m$ and recalling \eqref{EquationGraphMeanCurvature} yields the familiar divergence form for the mean curvature:
\begin{equation}\label{EquationDivergenceForm}
	\begin{split}
		\diver^g\bigg( \frac{\nabla^gf_s }{\sqrt{1+ |df_s|^2_g}}\bigg)  
		&=H_{\hat{M}^n_s},
	\end{split}
\end{equation}
where $H_{\hat{M}^n_s}$ denotes the mean curvature of $\hat{M}^n_s$, the graph of $f_s$ over $\Omega$. Differentiating we find 
\begin{equation}
	\begin{split}
		H_{,k}^{\hat{M}^n_s} &= \bigg( \nabla_k  \nabla \bigg( \frac{\nabla^g f_s}{\sqrt{1+ |df_s|^2_g}}  \bigg)\bigg)^i_i. \\
	\end{split}
\end{equation}
Using the definition of the Riemann tensor as the commutation of second order derivatives it follows that 
\begin{equation}
	\begin{split} 
		H_{,k}^{\hat{M}^n_s}  &=  \bigg( \nabla_i \nabla \bigg( \frac{\nabla^g f_s}{\sqrt{1+ |df_s|^2_g}}  \bigg)\bigg)^i_k - \text{Ric}^g_{k\ell}  \frac{f^{,\ell}_s}{\sqrt{1+ |df_s|^2_g}} \\
		&=\nabla_i\bigg(  \bigg( g^{mi} -   \frac{f_s^{,i}f_s^{,m} }{1+ |df_s|^2_g } \bigg)\frac{\Hess_{mk}^g(f_s)}{\sqrt{1+ |df_s|^2_g}}      \bigg)  - \text{Ric}^g_{k\ell}  \frac{f_s^{,\ell}}{\sqrt{1+ |df_s|^2_g}}. 
	\end{split}
\end{equation}
Differentiating the trace term we obtain 
\begin{equation}
	\trace_{\hat{g}}(k)_{,k}=(\nabla_k\hat{g})^{ij}k_{ij}+ \hat{g}^{ij}(\nabla_kk)_{ij},
\end{equation}
where
\begin{equation}
	\begin{split}
		(\nabla_k\hat{g})^{ij} k_{ij} 
		&=-2\bigg(g^{ij} - \frac{f_s^{,i}f_s^{,j}}{1+|df_s|_g^2}\bigg) \frac{\Hess_{jk}^g(f_s)f_s^{,\ell}k_{i\ell}}{1+|df_s|_g^2}.
	\end{split}
\end{equation}
In summary, differentiating the regularized Jang equation $\J_s(f_s)= \tau f_s$ yields
\begin{equation}
	\begin{split}
		\tau f^s_{,k}&= \nabla_i\bigg(  \bigg( g^{mi} -   \frac{f_s^{,i}f_s^{,m} }{1+ |df_s|^2_g } \bigg)\frac{\Hess_{mk}^g(f_s)}{\sqrt{1+ 	|df_s|^2_g}}      \bigg)   - \text{Ric}^g_{k\ell}  \frac{f_s^{,\ell}}{\sqrt{1+ |df_s|^2_g}} \\
		&\qquad -s\bigg(-2\bigg(g^{ij} - \frac{f_s^{,i} f_s^{,j}}{1+|df_s|_g^2}\bigg) 	\frac{\Hess_{jk}^g(f_s)f_s^{,\ell}k_{i\ell}}{1+|df_s|_g^2}+ \bigg(g^{ij} - \frac{f_s^{,i}f_s^{,j}}{1+|df_s|_g^2}\bigg) (\nabla_k k)_{ij} \bigg).  
	\end{split}
\end{equation}
We multiply this equation by $f_s^{,k}$ and sum over $k$. To estimate the first term in the right hand side of the resulting equation, we observe that
\begin{equation}
	\begin{split}
		\nabla_i\bigg(\hat{g}^{mi} \frac{\Hess_{mk}^g(f_s)}{\sqrt{1+ |df_s|^2_g}} f^{,k}\bigg) &= \nabla_i \bigg(\hat{g}^{mi} 	\frac{\Hess_{mk}^g(f_s)}{\sqrt{1+ |df_s|^2_g}} \bigg) f_s^{,k}+  \hat{g}^{mi} \frac{\Hess_{mk}^g(f_s)}{\sqrt{1+ |df_s|^2_g}} g^{k\ell}\Hess_{\ell i}^g(f_s) \\
		&= \nabla_i \bigg(\hat{g}^{mi} 	\frac{\Hess_{mk}^g(f_s)}{\sqrt{1+ |df_s|^2_g}} \bigg) f_s^{,k}+  \frac{\trace(\hat{g}^{-1}\Hess^g(f)\: g^{-1}\Hess^g(f))}{\sqrt{1+ |df_s|^2_g} } \\
		&\geq  \nabla_i\bigg(\hat{g}^{mi} \frac{\Hess_{mk}^g(f_s)}{\sqrt{1+ |df_s|^2_g}} \bigg) f_s^{,k} ,
	\end{split}
\end{equation}
From the second term in the right hand side of the resulting equation we obtain
\begin{equation}
	\begin{split}
		\text{Ric}^g_{k\ell}  \frac{f_s^{,\ell}}{\sqrt{1+ |df_s|^2_g}}f_s^{,k} 
		&\leq \frac{| \text{Ric}^g|_g|df_s\otimes df_s|_g}{\sqrt{1+ |df_s|^2_g}} \\
		&\leq C|df_s|_g,
	\end{split}
\end{equation}
where the constant $C$ depends only on the initial data $(M^n,g,k)$. As for the third term, we note that
\begin{equation}
	\begin{split}
		2\hat{g}^{ij}\frac{\Hess_{jk}^g(f_s)f^{,\ell}k_{i\ell}}{1+|df_s|_g^2}f^{,k} &= 	\hat{g}^{ij}\frac{(|df_s|_g^2)_{,j}g^{m\ell}f^s_{,m}k_{i\ell}}{1+|df_s|_g^2}  \\
		&\leq B^k(|df_s|_g^2)_{,k} ,
	\end{split}
\end{equation}
where $B^k$ is bounded. Finally, an estimation of $p_k=\hat{g}^{ij}(\nabla_k k)_{ij }$ yields
\begin{equation}
	\begin{split}
		p_kf_s^{,k}&=\langle p, df_s \rangle_g \\
		&\leq |p|_g|df_s|_g \\
	\end{split}
\end{equation}
Defining $u= |df_s|^2_g$ and adding all terms we arrive at the inequality  
\begin{equation}\label{EquationGradientInequality}
	\tau u\leq \nabla_i(A^{ij}u_{,j})  + B^ku_{,k}+ C\sqrt{u},
\end{equation}
where $A^{ij}$ is positive definite, $B^k$ is bounded, $C$ is a constant and for all $A^{ij}$, $B^k$ and $C$ depend only on the initial data $(M^n,g,k)$. 
\\ \indent At an interior maximum point $p\in \Omega$ of $u$ we must have $u_{,k}=0$. Thus, by \eqref{EquationGradientInequality}, we obtain $\tau |df_s|_g \leq C$ at $p$, where $C$ only depends on the initial data $(M^n, g,k)$.
\\ \indent We now proceed to obtain the boundary gradient estimate. Since $\varphi\in C^{2,\alpha}(\partial \Omega)$ we trivially have a bound on the gradient in the tangential direction. To estimate the gradient in the normal direction we employ the barrier method, where suitable barrier functions $w^-$ and $w^+$ are used to control the normal derivative. Explicitly we require that $\J_s(w^+)<\tau w^+$, $\J_s(w^-)>\tau w^-$ near $\partial\Omega$ and $w^\pm = s\varphi$ on $\partial \Omega$. From the comparison principle of \cite{GilbargTrudinger}, Chapter 10 (but see also Appendix B of \cite{SakovichPMTah}) it follows that in this case barriers satisfy $w^-\leq f_s \leq w^+$, which gives
\begin{equation}
	\frac{w^-(p)-w^-(p_0)}{d_g(p,p_0)} \leq \frac{f_s(p)-f_s(p_0)}{d_g(p,p_0)} \leq \frac{w^+(p)-w^+(p_0)}{d_g(p,p_0)},
\end{equation}
where $p_0\in \partial \Omega$ and $p\in \Omega$. It follows that 
\begin{equation}
	\frac{\partial w^-}{\partial \vec{n}}(p_0) \leq \frac{\partial f}{\partial \vec{n}}(p_0)\leq \frac{\partial w^+}{\partial \vec{n}}(p_0),
\end{equation}
where $\vec{n}$ is the inward pointing unit normal to $\partial \Omega$, and so the full boundary gradient estimate would follow.
\\ \indent In order to construct the barriers we invoke \emph{Fermi coordinates}\footnote{We will use Fermi coordinates again in Section \ref{SectionJangAE}, where a more detailed description is found.} (or \emph{normal geodesic coordinates}). Namely, we let $\rho = \text{dist}(\cdot, \partial \Omega)$ and denote by $N_\rho$ the hypersurfaces (or the leaves) of constant $\rho$, for $\rho$ sufficiently small. Using coordinates $x^\mu$, where $\mu=1, \ldots, n-1$ on $\partial \Omega$, we have a coordinate system $(\rho, x^1, \ldots, x^{n-1})$ defined in a neighbourhood $\{ \rho < \rho_0 \}$ of $\partial \Omega$ where we can write $g=g_\rho + d\rho^2$, where $g_\rho$ is the induced metric on $N_\rho$. By \eqref{EquationTrappingCondition} and the continuity of $\varphi$, there exists a small number $\rho_0$, such that in the neighbourhood $\{ \rho <\rho_0 \}$ we have $H_{N^\rho}- |\trace^{N_\rho}(k)| >\tau |\varphi|$.
\\ \indent We note that in this coordinate system we have $(A^\rho)_{\mu\nu}=\Gamma^\rho_{\mu\nu}$, where $A^\rho$ denotes the second fundamental form of $N_\rho$. In particular, we have $H_{\partial \Omega} = g^{\mu\nu}A^0_{\mu\nu}$.
\\ \indent We let $\varphi$ be extended trivially along the $\rho$-coordinate and claim that $w^\pm  = s\varphi \pm \rho B$ are barriers, where $B$ is a sufficiently large positive constant. We define $Q(f)=\J_s(f)-\tau f$ and need to show that $\pm \J_s(w^\pm) \mp \tau w^\pm <0$ holds for large enough $B$. We have $1+|dw^\pm|^2_g=1+B^2+s^2|\varphi|^2_{g^\rho}=\Ol(B^2)$. For the mean curvature term we have  
\begin{equation}
	\begin{split}
		\bigg(g^{ij}- \frac{w^{\pm,i}w^{\pm,j}}{1+|dw^\pm|_g^2}\bigg) \bigg(  \frac{w^\pm_{,ij} }{\sqrt{1+|dw^\pm|_g^2}}   \bigg) &= 	\Ol(B^{-1}),
	\end{split}
\end{equation}
since all derivatives except for tangential vanish. We also have
\begin{equation}
	\begin{split}
		- \bigg(g^{ij}- \frac{w^{\pm,i}w^{\pm,j}}{1+|dw^\pm|_g^2}\bigg)    \frac{   \Gamma^{k}_{ij}w^\pm_{,k}}{\sqrt{1+|dw^\pm|_g^2}}    &=  	-g^{\mu\nu}  \frac{\Gamma^{\rho}_{\mu\nu}(\pm B)}{B + \Ol(B^{-1})} + \Ol(B^{-1}) \\
		&= -H_{N^\rho } + \Ol(B^{-1}).
	\end{split}
\end{equation}
A straightforward calculation estimates the trace term:
\begin{equation}
	\begin{split}
		\bigg(g^{ij}- \frac{w^{\pm,i}w^{\pm,j}}{1+|dw^\pm|_g^2}\bigg) k_{ij} 
		&= \trace_{N^\rho}(k)  +\Ol(B^{-1}).
	\end{split}
\end{equation}
Taken together, this yields
\begin{equation}
	\begin{split}
		\J_s(w^+) - \tau w^+ &=   -H_{N^\rho}   - s\trace^{N_\rho}(k) - \tau s\varphi - \tau \rho B   + \Ol(B^{-1}) \\
		&\leq - H_{N^\rho} + |\trace^{N_\rho}(k)| + \tau |\varphi| + \Ol(B^{-1}) \\
		&<  0  \\
	\end{split}
\end{equation}
for $B$ sufficiently large. A similar estimate shows $\J_s(w^-)- \tau w^- >0$ so that $w^\pm$ are barriers. The barriers have normal derivatives $\partial_{\vec{n}}w^\pm = \pm B$ and since $B$ only depends on the initial data it follows that we have a uniform in $s$ $C^1$-estimate of $f_s$, which we denote by $||f_s||_{C^1(\Omega)}\leq K_\tau$. 
\\ \indent It is clear that there is a uniform lower bound of the eigenvalues of $\hat{g}_s^{ij}:$ $\lambda_{K_\tau} \leq\lambda(x,f_s, df_s)$ and rewriting Jang's equation as $Qf_s=a^{ij}(x,f_s,df_s) f_{s,ij} + b(x,f_s,df_s)=0$ where 
\begin{equation}
	\begin{split}
		a^{ij}(x,f_s,df_s)&= \bigg( g^{ij} - \frac{f_s^{,i}f_s^{,j}}{1+|df_s|^2_g}\bigg) , \\
	 b(x,f_s, df_s )&=  \bigg( g^{ij} - \frac{f_s^{,i}f_s^{,j}}{1+|df_s|^2_g}\bigg)\bigg( -\Gamma^k_{ij}f_{s,k} -  k_{ij} \bigg),
	\end{split}
\end{equation} 
it is also straightforward to see that there exists a constant $\mu_{K_\tau}$ (uniform in $s$) such that
\begin{equation}
	\begin{split}
		|a^{ij}(x,z, \vec{p})| + |a^{ij}(x,z,\vec{p})_{,p^k}| &+ |a^{ij}(s,z,\vec{p})_{,z}| \\
		&\qquad + |a^{ij}(x,z,\vec{p})_{,x^k}| + |b(x,z,\vec{p})|\leq \mu_{K_\tau}. 
	\end{split}
\end{equation}
From the global H\"older estimate of Ladyzhenskaya and Ural'tseva (cf. Chapter 13 in \cite{GilbargTrudinger}) we then get a uniform bound of the H\"older coefficient $[df_s]\leq C(n,\Omega, K, \mu_{K_\tau}/\lambda_{K_\tau})$ over $\Omega$ and in turn a global bound in $C^{1,\beta}(\overline{\Omega})$, for some $0 <\beta <1$. Applying global Schauder estimates (cf. \cite{GilbargTrudinger} Chapter 6) we get a uniform over $s$ estimate in $C^{2,\beta}(\overline{\Omega})$.  
\\ \indent We now let $\{s_n\}\subset\S$ be any sequence converging to $s\in [0,1]$. It is well-known that the Arzela-Ascoli theorem implies compactness of the embedding $C^{2,\beta}(\bar{\Omega})\xrightarrow{ } C^{2,\alpha}(\bar{\Omega})$ for $0<\alpha<\beta$. In turn, the uniform $C^{2,\beta}(\overline{\Omega})$-estimate hence gives subconvergence in $ C^{2,\alpha}(\bar{\Omega})$ to some $f_s \in C^{2,\alpha}(\bar{\Omega})$. The smoothness of the convergence implies that $\J_s(f_s)=\tau f_s$ so that $s\in \S$. Hence $\S$ is closed.


\end{proof}

With Lemma \ref{LemmaSisOpen} we show that $\S$ is open.

\begin{lemma}\label{LemmaSisOpen}
	Let $\S \subset [0,1]$ be the set of $s$ such that \eqref{EquationContCapJang} has a solution $f_s\in C^{2,\alpha}(\overline{\Omega})$. Then $\S$ is open.
	
\end{lemma}

\begin{proof}
	We aim to show that $\S$ is open with the Implicit Function Theorem. We consider the operator $T:C^{2,\alpha}(\overline{\Omega})\times  \rn \rightarrow C^{0,\alpha}(\overline{\Omega}) \times C^{0,\alpha}(\partial \Omega)\times  \rn$ given
	\begin{equation}
		T(f,s)=(H(f)- s\trace(k)(f)- \tau f  , f|_{\partial \Omega}- s \varphi , s),
	\end{equation}
	Suppose that $f_0$ is a solution of Equations \eqref{EquationContCapJang} for some $s_0$, that is to say $T(f_0,s_0)=(0,0,s_0)$. The linearization of $T$ at $(f_0,s_0)$ is   
	\begin{equation}
		L_{(f_0,s_0)}(h,t) =\bigg( A^{ij}\Hess_{ij}^g(h) + B^k h_{ ,k} - \tau h    - s_0  \trace_g(k)(f_0), h|_{\partial \Omega} - 	s_0\varphi   ,t \bigg),
	\end{equation}
	where
	\begin{equation}
		\begin{split}
			A^{ij} &= \frac{1}{\sqrt{1+|df_0|_g^2}} \bigg(g^{ij}- \frac{f_0^if_0^j}{1+ |df_0|_g^2}\bigg)    \\
			B^k &= (\diver^g A)^k + 2s_0\frac{1}{\sqrt{1+|df_0|_g^2}} A^{ik}f_0^jk_{ij}.
		\end{split}
	\end{equation}
	In order to apply the Implicit Function Theorem we need to show that $L_{(f_0,s_0)}$ is an isomorphism. But if $F\in C^{0,\alpha}(\Omega)$ and $G\in C^{2,\alpha}(\partial \Omega)$, then it is known from the theory of linear elliptic partial differential equations (cf. Chapter 6 of \cite{GilbargTrudinger}) that the problem
	\begin{equation}
		\begin{split}
			A^{ij}\Hess_{ij}^g(h) + B^k h_{ ,k} - \tau h       &= s_0  \trace_g(k)(f_0)+F, \qquad \text{in}\qquad \Omega, \\
			h&= s_0\varphi+G, \qquad \quad \qquad \qquad \text{on}\qquad \partial\Omega
		\end{split}
	\end{equation}
	has a unique solution $h\in C^{2,\alpha}(\overline{\Omega})$. From the Implicit Function Theorem there is an $\epsilon_0>0$ and a $C^1$-map $s\rightarrow f_s$ defined on $|s-s_0|<\epsilon_0$, so that $f_s$ solves Equations \eqref{EquationContCapJang}. Hence $\S$ is open.
\end{proof}

We may now prove the main result of this section.

\begin{proposition}\label{PropositionCapillarizationHasSol}
	Let $\varphi\in C^{2,\alpha}(\partial \Omega)$ and suppose $\Omega$ satisfies the trapping condition in \eqref{EquationTrappingCondition}. Then, for $\tau>0$ so small so as to ensure that $H_{\partial \Omega}- |\trace_{\partial \Omega}(k)|>\tau \varphi$, the Dirichlet problem 
	\begin{equation} \label{EquationCapillaryJang2}
		\begin{split}
			\J(f_{\tau})&= \tau f_{\tau} \qquad     \text{on} \: \Omega , \\
			f_{\tau} &=  \varphi      \qquad  \:\:\: \:  \text{on} \: \partial  \Omega 
		\end{split}
	\end{equation}
	has a solution in $C^{2,\alpha}(\overline{ \Omega}) $. Moreover, if $f_-\leq \varphi\leq f_+$ on $\partial \Omega$ then $f_-\leq f_\tau \leq f_+$ on $\{ r_0 \leq r \} \cap \bar{\Omega}$, where $f_\pm$ are the barriers obtained in Proposition \ref{PropositionBarrierExistence}.
\end{proposition}
 
\begin{proof}
	The proof is immediate from Lemmas \ref{LemmaSisClosed} and \ref{LemmaSisOpen}, as for $s=0$ the trivial function solves Equations \eqref{EquationContCapJang} so that $\S$ is non-empty and hence $\S=[0,1]$. In particular, a solution exists for $s=1$, which solves the Dirichlet Problem in \eqref{EquationCapillaryJang}. 
	\\ \indent To show the assertion about $f_-\leq f_\tau \leq f_+$ we first note that if $\tau>0$ is small enough, then
	\begin{equation}
		\begin{split}
			\J(f_+)- \tau f_+ &>0 \\
			\J(f_-) -\tau f_{\tau} &<0.
		\end{split}
	\end{equation}
	A similar argument as in the proof of Proposition \ref{PropositionBarrierExistence} now applies to show that $f_-\leq f_\tau \leq f_+$ on $\{ r_0 \leq r \} \cap \bar{\Omega}$.
	
\end{proof}

We abuse notation slightly and write $S_R = \Psi^{-1}(S_R)\subset M^n$, where $S_R\subset \rn^n$ is the standard coordinate sphere and $\Psi$ is the diffeomorphism of the initial data. With the following Lemma we show that the coordinate spheres $S_R$ satisfy the trapping condition of Definition \ref{DefinitionTrappingAssumption}. 

\begin{lemma}\label{LemmaSphereTrappingCondition}
	Let $S_R$ be a coordinate sphere with $R>r_0$. Then $S_R$ satisfies the trapping condition of Definition \ref{DefinitionTrappingAssumption} for sufficiently large $R$.
\end{lemma}

\begin{proof}
	
Straightforward compuations show that both 
\begin{equation}
	H_{S_R}=(n-1)+\bigg(\frac{n-1}{2}\bigg)R^{-2} + \Ol(R^{-4})
\end{equation}
and
\begin{equation}
	\trace_{S_R}{k}=(n-1)+ \Ol(R^{-n})
\end{equation}
which proves the assertion.

\end{proof}

\section{A geometric solution to Jang's equation}\label{SectionJangSolution}

In this section we obtain a geometric solution to Jang's equation \eqref{EquationJangsEquation} by subconverging the graphs obtained in Section \ref{SectionDirichletProblem}. More specifically, this is done using Geometric Measure Theory as summarized in Appendix \ref{SectionGMT}. The arguments follow Section 2 of \cite{EichmairPMT} very closely, but we include them for completeness.

\subsection{Limit and regularity}\label{SubSectionLimitRegularity}

We let $R$ be large so that the trapping condition in Definition \ref{DefinitionTrappingAssumption} is satisfied as per Lemma \ref{LemmaSphereTrappingCondition}. For small enough $\tau>0$ and $\varphi = \frac{1}{2}(f_++ f_-)$, where $f_\pm$ are the barriers obtained in Proposition \ref{PropositionBarrierExistence}, we get a solution $f_\tau$ satisfying $\J(f_\tau) = \tau f_\tau$ on $\bar{B}_R=\{ r\leq R\}\subset M^n$ by Proposition \ref{PropositionCapillarizationHasSol}. For any sequences $\{R_k\}_{k=1}^{\infty}$ and $\{\tau_k\}_{k=1}^\infty$ such that $R_k\rightarrow \infty$ and $\tau_k\rightarrow 0$ as $k\rightarrow \infty$ we let $\bar{B}_k = \bar{B}_{R_k}$ and denote by $\{f_k\}_{k=1}^\infty$ the solutions obtained from Proposition \ref{PropositionCapillarizationHasSol}. Further, denote the graphs of $\{f_k\}_{k=1}^\infty$ over $\bar{B}_k$ by $\{ \hat{M}^n_k \}_{k=1}^\infty$. For our choice of boundary data $\varphi$, Proposition \ref{PropositionBarrierExistence} implies that we have $\varphi_k\leq 2R_k$ near infinity, and so it is possible to chose $\{R_k\}_{k=1}^\infty$ and $\{ \tau_k\}_{k=1}^\infty$ such that $\tau_k R_k\leq A$, uniformly over $k$. In turn, the estimate of $\tau_k \sup_{\bar{B}_k}|f_k|$ in the proof of Lemma \ref{LemmaSisClosed} then implies a uniform over $k$ estimate $\tau_k \sup_{\bar{B}_k} |f_k|\leq C$, where $C$ depends only on the initial data and the uniform constant $A$. Finally, by the Cauchy-Schwarz inequality and the estimate $|\hat{g}|_g\leq \sqrt{n}$, we have
\begin{equation}\label{EquationMeanCurvIsBounded}
	\bigg( g^{ij} - \frac{f^{,i} f^{,j}}{1+|df|^2_g}\bigg) k_{ij} \leq  \sqrt{n}  |k|_{g}.
\end{equation}
Thus, it follows that the graphs $\{\hat{M}^n_k\}$ have uniformly bounded mean curvature by some $\lambda$ depending on the initial data and $A$.
\\ \indent We recall the following Harnack Principle, which appears as Lemma 2.3 in \cite{EichmairPlateau}:

\begin{proposition}(''Harnack principle'')\label{PropositionHarnackPrinciple}
	Let $f_k:\Omega \rightarrow \rn$ be $C^3$-functions with open and connected domain $\Omega$, such that for some $\beta>0$,
	\begin{equation}
		\Delta^{k}(\gamma_k^{-1}) \leq  \beta \gamma_k^{-1} + \langle X, d\big( \gamma_k^{-1}\big) \rangle_{k},
	\end{equation}
	where $\gamma_k=\sqrt{1+|df_k|^2}$, $X$ is a locally bounded vector field and $\Delta_k$ is the Laplace-Beltrami operator of the graphs $G_k=\text{graph}(f_k)$. Suppose the graphs $G_k$ converge in $C^3$ to a submanifold $G\subset \Omega\times \rn$. Then, on each component of $G$, $\gamma$ is either everywhere positive or everywhere vanishing.
\end{proposition}

The following Proposition is similar to Proposition 4 in \cite{PMTII} and Proposition 7 in \cite{EichmairPMT}.

\begin{proposition}\label{PropositionJangLimit}
	Let $(M^n, g, k)$ be asymptotically hyperbolic initial data with Wang's asymptotics of type $(\ell, \alpha, \tau=n, \tau_0>0)$ as in Definition \ref{DefinitionWangAsymptotics}, where $4\leq n \leq 7$. There exists an embedded $C^{3,\alpha}_{loc}$-hypersurface $(\hat{M}^n, \hat{g})\subset (M^n\times \rn, g + dt^2)$, with the following properties:
	\begin{enumerate}
		\item $\hat{M}^n$ is the boundary of an open set $\Omega$. We have $H_{\hat{g}} - \trace_{\hat{g}}(k)=0$, where the mean curvature $H_{\hat{g}}$ is computed as the tangential divergence the downward pointing unit normal. Moreover, $\hat{M}^n=\partial \Omega$ is a $\lambda$-minimizing boundary.
		\item $\hat{M}^n$ has finitely many connected components. Each component of $\hat{M}^n$ is either cylindrical of the form $C_\ell\times \rn$, where $C_\ell$ is a closed and properly embedded $C^{3,\alpha}$-hypersurface in $M^n$, or the graph of a function $f$, which solves the Jang equation $\J(f)=0$, on an open subset $U_f\subset M^n$.
		\item The boundary $\partial U_f$ is a closed properly embedded $C^{3,\alpha}$-hypersurface in $M^n$. More specifically, $\partial U_f$ is the disjoint union of components $C^+_\ell$ and $C_\ell^-$, where $f(p)\rightarrow \pm\infty$ uniformly as $p\rightarrow C_\ell^{\pm}$ from $U_f$. These hypersurfaces $C_\ell^\pm$ satisfy $H_{C_\ell} \mp \trace_{C_\ell}(k)=0$, where the mean curvature is computed as the tangential divergence of the outward from $U_f$ pointing unit normal. There exists a $T\geq 1$ such that each component of $\hat{M}^n\cap \{ |t|\geq T\}$ is a graphs over $C_\ell\times [ T, \infty)$. Finally, the graphs $\text{graph}(f- A)$ converge in $C^{3,\alpha}_{loc}$ to $C_\ell^\pm\times \rn$ as $A\rightarrow \pm\infty$.
		\item $\hat{M}^n$ contains a graphical component with domain  
		\begin{equation}
			\{ p\in M^n\: |\: r > r_0  \} \subset U_f,
		\end{equation}
		where $r_0$ is as in Proposition \ref{PropositionBarrierExistence}. Furthermore, $f$ has the asymptotics\footnote{At this stage, we do not show the asymptotic flatness of the graph. This is done in Section \ref{SectionJangAE}.} as in \eqref{EquationBarrierAsymptotics} on this set. 
	\end{enumerate}
\end{proposition}

\begin{proof}
	
We use the results from the Geometric Measure Theory summmarized in Appendix \ref{SectionGMT} to show the existence of the limit and its regularity. We let $\{R_k\}_{k=1}^\infty$ and $\{\tau_k\}_{k=1}^\infty$ be as explained previously in this subsection and denote by $\{f_k\}_{k=1}^\infty$ be the functions obtained from Proposition \ref{PropositionCapillarizationHasSol}. As explained, the graphs $\{\hat{M}^n_k\}_{k=1}^\infty$ of $\{f_k\}_{k=1}^\infty$ over $\{ \bar{B}_k\}_{k=1}^\infty$ have uniformly bounded mean curvature by some $\lambda$ depending only on the initial data $(M^n,g,k)$.
\\ \indent By the Nash-embedding Theorem there exists an isometric embedding $F:M^n\times \rn \rightarrow  \rn^{n+\ell}$ for some $\ell>0$. We denote $F(M^n\times \rn)$ by $N^{n+1}$ to agree with the notation of Appendix \ref{SectionGMT}. The graphs $\{\hat{M}^n_k\}_{k=1}^\infty$ are viewed as currents $\{T_k\}_{k=1}^\infty$ which have multiplicity one and are boundaries $T_k=\partial [[E_k]]$ that are $\lambda$-minimizing. Hence $T_k \in \F_\lambda$ and by Theorem \ref{TheoremGMTclosure} there is a subsequence (denoted by the same index for notational convenience) such that $\{T_k\}_{k=1}^\infty$ converges as currents to some $T\in \F_\lambda$. By Theorem \ref{TheoremGMTregularity} the limit graph $\hat{M}^n$ is regular in the sense of Definition \ref{DefinitionRegularSingularSets} for $4\leq n\leq 6$ and we refer the reader to Remark 4.1 in \cite{EichmairPlateau} for the explanation why it is also regular for $n=7$. From Lemma \ref{LemmaSmoothConvergence} it follows that the convergence is in $C^{1,\alpha}_{loc}$.
\\ \indent The limit satisfies Jang's equation distributionally, as for each $\hat{M}^n_k$ the mean curvature term in divergence form integrates to
\begin{equation}
	\int_{M^n}\diver_g\bigg(  \frac{\nabla^g(f_k)}{\sqrt{1+ |df_k|^2_g}}   \bigg) \varphi d\mu_g = - \int_{M^n} \bigg\langle \frac{\nabla^g (f_k)}{\sqrt{1+ |df_k|^2_g}}, d\varphi     \bigg\rangle d\mu_g,
\end{equation}
for $\varphi\in C^\infty_c(M^n)$, and it follows from the $C^1_{loc}$-convergence of $\{f_k\}$ that $\{f_k\}$ convergens to $f$ distributionally. Hence, $f$ satisfies the (non-regularized) Jang equation weakly. Standard elliptic regularity theory gives regularity up to order $C^{3,\alpha}_{loc}$. 
\\ \indent We now use the Schauder estimates to get $C^{3,\alpha}_{loc}$-convergence. We have that $f_k\rightarrow f$ in $C^{1,\alpha}_{loc}$ and for $u\in C^{2,\alpha}(M^n)$ we define
\begin{equation}
	\begin{split}
		L(u)&= \hat{g}^{ij}u_{,ij}- \hat{g}^{ij}\Gamma^\ell_{ij}u_{,\ell} , \qquad \text{where}\\
		\hat{g}^{ij}&=\bigg(  g^{ij} - \frac{f^{,i}f^{,j}}{1+|df|^2_g}    \bigg)  \qquad \text{and} \\
		L_k(u)&=\hat{g}^{ij}_ku_{,ij}- \hat{g}^{ij}_{k}\Gamma^\ell_{ij}u_{,\ell} - \tau_ku, \qquad \text{where} \\
		\hat{g}^{ij}_k &= \bigg(  g^{ij} - \frac{f^{,i}_kf^{,j}_k}{1+|df_k|^2_g}    \bigg),
	\end{split}
\end{equation}
where $f$ solves the Jang equation $\J(f)=0$ and $f_k$ solves the regularized Jang equation $\J(f_k)=\tau_kf_k$ on $\bar{B}_k$. The $C^{1,\alpha}_{loc}$-convergence of $f_k$ implies uniform bounds of the coefficients of the differential operator $L_k$. We note that $f-f_k$ satisfies $L(f-f_k)=F_k$, where 
\begin{equation}
	F_k=\trace_{\hat{g}_k}(k)- \trace_{\hat{g}}(k) + L(f_k)- L_k(f_k).
\end{equation}
If we show that $||F_k||_{C^{0,\alpha}_{loc}}\rightarrow 0$, as $k\rightarrow \infty$, then $f_k\rightarrow f$ in $C^{2,\alpha}_{loc}$ will follow by the interior Schauder estimate. The $C^{1,\alpha}_{loc}$-convergence established above yields $C^{0,\alpha}_{loc}$-converence to zero of the trace terms. Furthermore, from the equations that $f_k$ satisfy we have a uniform $C^{2,\alpha}_{loc}$-bound on $f_k$ and from this it follows that $L(f_k)- L_k(f_k)\rightarrow 0$ in $C^{0,\alpha}_{loc}$. Iterating this argument we obtain the convergence in $C^{3,\alpha}_{loc}$.
\\ \indent It remains to show that the limit contains at least one graphical component. We write $\gamma_k = \sqrt{1 + |df_k|_g^2}$ and recall that the Jacobi equation \cite[Equation (2.18)]{PMTII} holds for $\gamma_ k$ on $(\hat{M}^n_k, \hat{g}_k)$:
\begin{equation}\label{EquationJacobiEquation}
	\Delta^{\hat{g}_k}\big(\gamma^{-1}_k\big) +  \big(|\hat{A}_k|^2_{\hat{g}} + \text{Ric}^{M^n\times \rn}(\vec{n}_k, \vec{n}_k) + \vec{n}_k( H_{\hat{M}^n_k} ) \big)\gamma_k^{-1} =0,
\end{equation}
where $\hat{A}_k$ is the second fundamental form of $(\hat{M}^n, \hat{g}_k)$, we think of $H_{\hat{M}^n_k}$ as a function trivially extended from the graph $(\hat{M}^n_k, \hat{g}_k)$ to all of $M^n\times \rn$ and $\vec{n}_k$ is the downward pointing unit normal of $\hat{M}^n_k$ in $M^n\times \rn$. We have
\begin{equation}
	 \vec{n}_k( H_{\hat{M}^n_k} ) =  \vec{n}_k( \trace^{\hat{g}_k}k ) + \frac{\tau_k |df_k|_g^2}{\sqrt{1+ |df_k|_g^2}},
\end{equation}
and we estimate the first term on the right hand side following the proof of \cite[Lemma A.1]{EichmairMetzgerJenkinsSerrinType}. Firstly, we have 
\begin{equation}\label{EquationAux000}
	\begin{split}
		\vec{n}_k( \trace^{\hat{g}_k}k )  &= \vec{n}_k^\ell (\langle \hat{g}^k , k \rangle_g )_{,\ell}   \\
		&=\vec{n}_k^\ell  \langle \nabla_\ell \hat{g}^k , k \rangle_g +   \vec{n}_k^\ell  \langle  \hat{g}^k , \nabla_\ell k \rangle_g  \\ 
		&=  \langle \nabla \hat{g}^k , \vec{n}^k \otimes k \rangle_g  +   \langle \theta^k \otimes \hat{g}^k , \nabla  k \rangle_g 
	\end{split}
\end{equation}
where the first line holds since $\trace^{\hat{g}_k}k $ has no $t$-dependence, $\theta^k=df_k /\sqrt{1+|df_k|^2_g}$ is the $1$-form $g$-dual to the part of $\vec{n}_k$ tangential to $M^n$ and $\hat{g}^k$ is the $(0,2)$-tensor obtained by lowering the indices of $\hat{g}^{-1}$ with $g$, so that $\hat{g}^k= g - \frac{df_k \otimes df_k}{1+ |df_k|_g^2}$. We estimate the second term in \eqref{EquationAux000} as follows:
\begin{equation}
	\begin{split}
		\langle \theta^k \otimes \hat{g}^k , \nabla k \rangle_g  &\leq |\theta^k \otimes \hat{g}^k|_g |\nabla k|_g \\
		&= |\theta^k|_g |\hat{g}^k|_g |\nabla k |_g \\
		&\leq \sqrt{ n } |\nabla k|_g,
	\end{split}
\end{equation}
where the tensor $\theta^k\otimes \hat{g}^k$ has components $(\theta^k\otimes \hat{g}^k)_{\ell ij} = \theta^k_\ell \hat{g}^k_{ij}$. To estimate first term in \eqref{EquationAux000}, we first note (recalling that $\hat{g}^{ij}_k = g^{ij} - \vec{n}^i_k\vec{n}^j_k$) that
\begin{equation}
	(\nabla_\ell \hat{g})_{ij}  = - (\nabla_\ell \theta^k)_i \theta^k_j - (\nabla_\ell \theta^k)_j \theta^k_i, 
\end{equation}
It follows that 
\begin{equation}
	\begin{split}
		 \langle \nabla \hat{g}^k , \theta^k \otimes k \rangle_g &= -2 \langle  \nabla \theta^k \otimes \theta^k, \theta^k \otimes k \rangle_g \\
		 &= - 2 \vec{n}_k^i \vec{n}_k^d \hat{g}_k^{mb} \frac{\Hess^g_{mi} f_k}{\sqrt{1+|df_k|_g^2}} k_{bd} \\
		 &= 2 k_{bd}\vec{n}_k^d \hat{g}_k^{mb} (\gamma_k^{-1})_{,m} \\
		 &= \langle X, d\big( \gamma_k^{-1}\big) \rangle_{\hat{g}_k},
	\end{split}
\end{equation}
where we used \eqref{EquationGradientlengthDerivative},\eqref{EquationDivergenceForm} and defined $X_\ell = -2k_{\ell i}\vec{n}_k^i$. We note that 
\begin{equation}\label{EquationXaux}
	\begin{split}
		|X|_{\hat{g}} &=4  \hat{g}_k^{ab}  \vec{n}^i_k \vec{n}^j_k k_{ai}k_{bj} \\
		&= 4|X|_g^2 -4k(\vec{n}_k,\vec{n}_k) k(\vec{n}_k,\vec{n}_k)    \\
		&\leq  4|X|_g^2 \\
		&=16g^{ab}k_{a i}\vec{n}_k^i k_{bj}\vec{n}_k^j \\
		&=16 \langle k\otimes k, g \otimes \theta_k \otimes \theta_k \rangle_g \\
		&\leq 16  |k\otimes k|_g |g \otimes \theta_k \otimes \theta_k|_g \\
		&= 16 |k|_g^2 \sqrt{n} |\theta_k|_g^2 \\
		&\leq C |k|_g^2,
	\end{split}
\end{equation}
so that $|X|_{\hat{g}}$ is bounded. Straightforward calculations show that $\text{Ric}^{M^n\times \rn}_{tt}=0$, $\text{Ric}^{M^n\times \rn}_{ti}=0$ and $\text{Ric}^{M^n\times \rn}_{ij}=\text{Ric}_{ij}^{M^n}$ and so from Lemma \ref{LemmaWangGeometry} we obtain $|\text{Ric}^{M^n\times \rn}|_{ds^2} = |\text{Ric}^{M^n}|_g \leq C(M^n,g,k)$ (where we write $ds^2= g+dt^2$). Hence, the Cauchy-Schwarz inequality gives the estimate $\text{Ric}^{M^n\times \rn}(\vec{n}_k, \vec{n}_k)\leq C(M^n,g,k)$. Consequently, there exists some $\beta\geq 0$ such that
\begin{equation}\label{EquationDifferentialInequality}
	\Delta^{\hat{g}_k}\big(\gamma_k^{-1}\big) \leq \beta \gamma_k^{-1} + \langle X, d\big( \gamma_k^{-1}\big) \rangle_{\hat{g}_k}.
\end{equation}	
From Proposition \ref{PropositionHarnackPrinciple} it now follows that the limit $\hat{M}^n$ consists of graphical components and cylindrical components. In particular there exists an outermost graphical component over some open subset $U_f\subset M^n$. This shows (2.).
\\ \indent To prove assertion $(3)$, we note that depending on whether $f$ blows up or down near some $C_k$, the upward pointing unit normal blows outward or inward, respectively. Hence blow-up hypersurfaces $C_k^\pm$ satisfy $H_{C_k^\pm} \mp \trace_{C_k^\pm}(k)=0$, where the mean curvature is computed as the tangential divergence of the unit normal pointing out of $M^n$. To show that $\hat{M}^n$ is graphical over $C_\ell^\pm \times [T, \infty)$ we note that from Allard's theorem, we obtain a uniform radius $R>0$ such that $\hat{M}^n$ is graphical over the geodesic ball in the tangenplane $T_p\hat{M}^n$ at each point $p$. This rules out the possibility of the downward pointing unit normal $\vec{n}$ being parallel with the cylinder sufficiently close to $\partial U_f$, so that $\hat{M}^n$ is graphical over $C_\ell \times \rn$ there. 
\\ \indent Assertion $(4)$ follows from Proposition \ref{PropositionBarrierExistence}.
	
\end{proof}

\subsection{Topology of the Jang graph}\label{SubSectionGraphTopology}

The discussion in this subsection follows that of \cite{EichmairPMT} rather closely. The Jang graph $(\hat{M}^n, \hat{g})$ has the asymptotics calculated in Proposition \ref{PropositionBarrierExistence} and we denote by $\hat{N}^n$ the end of $\hat{M}^n$ over the asymptotically hyperbolic end $N^n$. Furthermore we have $\ell$ cylindrical ends $\hat{C}_1, \ldots, \hat{C}_\ell$. We focus only on the graphical component. The boundary of $U_f$ is the disjoint union of closed hypersurfaces $\partial U_f= C_1\cup \ldots \cup C_\ell$, where each $\hat{C}_i$ is asymptotic to $C_i\times \rn$. We denote by $\sigma_i=g|_{C_i}$ the induced metric on $C_i$. 
\\ \indent We begin by discussing some properties of the closed manifolds $C_i$. In \cite{PMTII} it was shown that when the strict dominant energy condition holds in a neighbourhood of each $C_i$ they are topologically spheres by the Gauss-Bonnet Theorem. In Lemma \ref{LemmaMOTSprops} we show that the analogue of the result holds in dimensions $n\geq 4$.

\begin{lemma}\label{LemmaMOTSprops}
	
Suppose that the strict dominant energy condition $\mu > |J|_g$ holds locally around a component $C_i\subset \partial U_f$. Then the spectrum of the conformal Laplacian of each $(C_i, \sigma_i)$
\begin{equation}
	L  =-\Delta^{C_i}  + c_{n-1} R_{C_i}  , \qquad c_{n} = \frac{n-2}{4(n-1)},
\end{equation}
is positive. In particular, $C_i$ has positive Yamabe type.

\end{lemma}

\begin{proof}
	
We recall the \emph{Schoen-Yau identity}:
\begin{equation}\label{EquationSchoenYauId}
	R_{\hat{g}}=2(\mu -J(\omega))+|\hat{A}-k|_{\hat{g}}^2+2|q|_{\hat{g}}^2 -2\diver^{\hat{g}}(q)
\end{equation}
where $\hat{A}$ is the second fundamental form of the Jang graph $(\hat{M}^n, \hat{g})$, 
\begin{equation}
	\omega = \frac{\nabla^g f}{\sqrt{1+ |df|_g^2}}, \qquad \text{and} \qquad q_i=\frac{f^{,j}}{\sqrt{1+ |df|_g^2}}(\hat{A}_{ij}- k_{ij}).
\end{equation}
The estimate $J(\omega)\leq |J|_g$ follows from the Cauchy-Schwarz inequality, which together with the inequality of arithmetic and geometric means also implies 
\begin{equation}\label{EquationAux3}
	- \langle d(\varphi^2), q\rangle_{\hat{g}}\leq  |d\varphi|_{\hat{g}}^2+\varphi^2|q|_{\hat{g}}^2. 
\end{equation}
For $\varphi\in C_c^1(\hat{M}^n)$, we multiply \eqref{EquationSchoenYauId} by $\varphi^2$ and integrate by parts over $\hat{M}^n$:
\begin{equation}\label{EquationAux5}
	\begin{split}
		\int_{\hat{M}^n} \bigg( 2(\mu - J(\omega)) + |\hat{A}- k |_{\hat{g}}^2\bigg)\varphi^2 d\mu^{\hat{g}} &= \int_{\hat{M}^n} \bigg(    	R_{\hat{g}} - 2|q|_{\hat{g}}^2 +2 \diver^{\hat{g}}(q)  \bigg) \varphi^2 d\mu^{\hat{g}} \\
		&=\int_{\hat{M}^n} \bigg(    R_{\hat{g}}\varphi^2 - 2|q|_{\hat{g}}^2 \varphi^2 - 2\langle d(\varphi^2), q \rangle_{\hat{g}}  \bigg)   	d\mu^{\hat{g}} \\
		&\leq \int_{\hat{M}^n} \bigg(    R_{\hat{g}}   \varphi^2 + 2|d\varphi|_{\hat{g}}^2\bigg)  d\mu^{\hat{g}},
	\end{split}
\end{equation}
where we used \eqref{EquationAux3} in the last line and that $\hat{M}^n$ has no boundary for the partial integration. Hence, the last integral is positive as a consequence of the strict dominant energy condition $2(\mu- J(\omega) )\geq 2(\mu- |J|_g )\geq \delta>0$. Note also that for a compactly supported function $\varphi\in C_c^1(C_i\times \rn)$ we get from the converge in $C^{3,\alpha}_{loc}$ of $\text{graph}(f- A)$ as $A\rightarrow \pm\infty$ to $C_\ell^\pm\times \rn$ in Proposition \ref{PropositionJangLimit} that
\begin{equation}
	\begin{split}
		\delta \int_{C_i\times \rn} \varphi^2 d\mu^{\sigma_i+dt^2} &\leq \int_{C_i\times \rn}\bigg( R_{C_i}\varphi^2 + 	2|d\varphi|_{\sigma_i}^2 \bigg) d\mu^{\sigma_i+dt^2} \\
		&\leq \int_{C_i\times \rn}\bigg( R_{C_i}\varphi^2 +  c_{n-1}^{-1}|d\varphi|_{\sigma_i+dt^2}^2 \bigg) d\mu^{\sigma_i+dt^2},
	\end{split} 
\end{equation}
where we used $2c_{n-1}\leq 1$ in the last line. 
\\ \indent We now let $\varphi(p,t)=\xi(p)\chi(t)$, for $p\in C_i$ and with $\chi$ a smooth cutoff function such that $\chi(t)=1$ for $|t|\leq T$, $\chi(t)=0$ for $|t|\geq T+2$ and $|\chi_{,t}|\leq 2$. We have
\begin{equation}
	\begin{split}
		0 &<\delta \int_{C_i\times \rn}c_{n-1} \varphi^2 d\mu^{\sigma_i+dt^2} \\
		&\leq  \int_{C_i\times \rn}  c_{n-1}   R_{C_i} \varphi^2  d\mu^{\sigma_i+dt^2} +   \int_{C_i \times \rn} |d\varphi|^2_{\sigma_i+dt^2}  	d\mu^{\sigma_i+dt^2} \\
		&=  \bigg( \int_{C_i} c_{n-1} R_{C_i} \xi^2 d\mu^{\sigma_i} \bigg) \bigg(\int_\rn \chi^2dt \bigg) +  \int_{C_i\times \rn}  \bigg( 	\chi^2|d\xi|_{\sigma_i}^2+   \xi^2 \chi_{,t}^2\bigg)d\mu^{\sigma_i+dt^2}   \\
		&\leq  \bigg( \int_{C_i} R_{C_i}\xi^2 d\mu^{\sigma_i} \bigg) \bigg(\int_\rn \chi^2 dt \bigg)  +    \bigg(  \int_{C_i} |d\xi|_{\sigma_i}^2 	d\mu^{\sigma_i} \bigg) \bigg(\int_\rn \chi^2 dt\bigg) \\
		&\qquad + \bigg(\int_{C_i}\xi^2  d\mu^{\sigma_i} \bigg)\bigg(\int_\rn \chi_{,t}^2 dt\bigg). 
	\end{split}
\end{equation}
We divide both sides by the integral $\int^\infty_{-\infty}\chi^2dt$ and after letting $T\rightarrow \infty$ we get 
\begin{equation}
	0< \int_{C_i} c_{n-1} R_{C_i} \xi^2 d\mu^{\sigma_i} +  \bigg(  \int_{C_i} |d\xi|_{\sigma_i}^2 d\mu^{\sigma_i} \bigg).
\end{equation}
By choosing the function $\xi$ to vanish on all but components of $\partial U_f$ except $C_i$, we get that the integral of the scalar curvature $R_{C_i}$ is positive. It is standard that the first eigenfunction of the conformal Laplacian $-\Delta^{C_i}  + c_{n-1} R_{C_i} $ is positive and so it follows that the conformal Laplacian has positive spectrum on each $C_i$. In particular, its first eigenfunction $\varphi_i>0$ can be used to conformally change the metric on $C_i$ to $\varphi_i^{\frac{4}{n-2}} \sigma_i$, which will have positive scalar curvature. It is standard theory that the existence of a positive scalar curvature metric is equivalent to the Yamabe type being positive (we refer the reader to \cite{SchoenNotes} for a comprehensive discussion on this).
	
\end{proof}

Next we deform the cylindrical ends $\hat{C}_1, \ldots, \hat{C}_\ell$ to \emph{exact cylinders}. Clearly, there exists a $T\geq 1$ such that $\hat{M}^n\cap \{ |t| \geq T\}$ can be perturbed to agree \emph{exactly} with the disjoint union $C_1 \times \rn \cup \ldots \cup C_\ell\times \rn$. We denote the induced, complete metric by $\tilde{g}$. In the case when $U_f=M^n$, we let $\tilde{g}=\hat{g}$. In both cases $\tilde{g}=\hat{g}$ on $\hat{N}^n$. We refer to the exact cylindrical ends as $\tilde{C}_1, \ldots, \tilde{C}_\ell$ and the deformed graph as $\tilde{M}^n$. There is a compact set $\tilde{K} \subset \tilde{M}^n$ such that $\tilde{M}^n = \tilde{K}\cup \hat{N}^n \cup \tilde{C}_1\cup \ldots \cup \tilde{C}_\ell$. We note that $( \tilde{M}^n, \tilde{g})$ will not satisfy the mean curvature equations but will satisfy the necessary inequalities derived in Lemma \ref{LemmaPsiProps} below. Since $\hat{C}_i$ and $C_i\times (T, \infty)$ are diffeomorphic we may equivalently pull up the metric $\tilde{g}$ to $\hat{C}_i$ and consider the manifold $(\hat{M}^n, \tilde{g})$ (with slight abuse of notation). There is a compact set $\hat{K}$ so that $\hat{M}^n = \hat{K}\cup \hat{N}^n\cup \hat{C}_1 \cup \ldots \hat{C}_\ell$. Clearly, $\tilde{g}$ and $\hat{g}$ are uniformly equivalent\footnote{We recall that two metrics $g_1$ and $g_2$ are \textbf{uniformly equivalent} if there exists a constant $C\geq 1$ such that $C^{-1}g_1\leq g_2\leq Cg_1$ as quadratic forms.}. We will alternate between the conventions after convenience.
\\ \indent For each exact cylindrical end $(\tilde{C}_i, \tilde{g})$, we define a function $\Psi_i=\exp(-\sqrt{\lambda_i}t_i)\varphi_i$, where $\lambda_i>0$ is the principal eigenvalue and $\varphi_i$ is the respective principal (positive) eigenfunction, of the operator $-\Delta^{C_i}  + c_{n} R_{C_i}$ in Lemma \ref{LemmaMOTSprops}. We arrange so that the coordinate $t_i$ carries sign so that $\Psi_i\rightarrow 0$ in the cylindrical infinities and the cylinders are exact for $t_i\in (T, \infty)$. 
\\ \indent We define a positive function $\Psi>0$ via
\begin{equation}\label{EquationPsiDef}
	\Psi(p)=
	\begin{cases}
		1, \qquad &\text{if}\qquad p\in \hat{K}\cup \hat{N}^n, \\
		\Psi_i,   &\text{if}\qquad p\in \tilde{C}_i, \: t_i\geq 1.
	\end{cases}
\end{equation}
In Lemma \ref{LemmaPsiProps} we establish some of the properties of $\tilde{g}_{\Psi} =\Psi^{\frac{4}{n-2}}\tilde{g}$.  

\begin{lemma}\label{LemmaPsiProps}
	
The scalar curvature of $\tilde{g}_\Psi$ vanishes on the exact cylinders $\tilde{C}_i$, that is to say $R_{\tilde{g}_\Psi}|_{\tilde{C}_i}=0$. Moreover, $(\tilde{C}_i, \tilde{g}_\Psi)$ is isometric to 
\begin{equation}\label{EquationAux4}
	\bigg(C_i\times \bigg(0, \frac{n-2}{2\sqrt{\lambda_i}}\bigg), \varphi_i^{\frac{4}{n-2}}\bigg(\frac{4 \lambda_i s_i^2}{(n-2)^2}\sigma_i+ ds_i^2\bigg)\bigg) 
\end{equation}
and is uniformly equivalent  to the metric cone 
\begin{equation}
	\bigg(C_i\times \bigg(0, \frac{n-2}{2\sqrt{\lambda_i}}\bigg), s_i^2\sigma_i+ ds_i^2 \bigg).
\end{equation}
Furthermore, let $\varphi \in C^1(\hat{M}^n)$ be such that $\text{supp}(\varphi)\cap (\tilde{C}_1\cup \ldots \cup \tilde{C}_\ell)$ is compact. Then 
\begin{equation}
	\frac{1}{2}\int_{\hat{M}^n} |d\varphi|_{\tilde{g}}^2d\mu^{\tilde{g}} + c_n\int_{\hat{N}^n}|\hat{A}-k|_{\hat{g}}^2d\mu^{\hat{g}}\leq \int_{\hat{M}^n} \bigg( |d\varphi|_{\tilde{g}}^2 + c_nR_{\tilde{g}}\varphi^2  \bigg)   d\mu^{\tilde{g}}.
\end{equation}
Finally, 
\begin{equation}\label{EquationAux6}
	\int_{\hat{M}^n} \frac{1}{2}\Psi^{-2}|d(\Psi \varphi)|^2_{\tilde{g}_\Psi}d\mu^{\tilde{g}_\Psi} + c_n\int_{\hat{N}^n}|\hat{A}-k|^2_{\hat{g}}\varphi^2 d\mu^{\hat{g}} \leq \int_{\hat{M}^n} \bigg( |d\varphi|^2_{\tilde{g}\Psi} + c_nR_{\tilde{g}_\Psi} \varphi^2\bigg) d\mu^{\tilde{g}_\Psi},
\end{equation}

\end{lemma}

\begin{proof}
	
We recall the well-known formula for the conformal transformation of scalar curvature:
\begin{equation}
	c_nR_{\tilde{g}_{\Psi}}= \Psi^{-\frac{n+2}{n-2}}\big( - \Delta^{\tilde{g}}\Psi  +  c_n\Psi R_{\tilde{g}}   \big).
\end{equation}
Using the equation that $\lambda_i$ and $\varphi_i$ satisfy it is straightforward to check that $\Delta^{\tilde{g}} \Psi  = \Psi c_n R_{\tilde{g}}$ on $C_i\times \rn$. Hence the scalar curvature $R_{\tilde{g}_\Psi}=0$ in $(C_i\times \rn, \Psi^{\frac{4}{n-2}}(\sigma_i + dt^2))$. It is furthermore easily seen that $R_{C_i\times \rn} = R_{C_i}$ with the product metric and so the scalar curvature $R_{\tilde{g}_\Psi}=0$ in $(C_i ,\Psi^{\frac{4}{n-2}} \sigma_i)$.
\\ \indent For the assertion about the isometry, we consider the map $\Phi:C_i\times \rn \rightarrow C_i\times \rn$ explicitly defined by
\begin{equation}
	\Phi(\theta,t)=\bigg(\theta, s_i=\frac{n-2}{2\sqrt{\lambda_i}}\exp\bigg(-\frac{2\sqrt{\lambda_i}t}{n-2}\bigg)\bigg).
\end{equation}
It is straightforward to verify that $\Phi$ is an isometry from $(\tilde{C}_i, \Psi^{\frac{4}{n-2}} \tilde{g})$ to the manifold in \eqref{EquationAux4}, which also implies the uniform equivalence to the metric cone.
\\ \indent We note that far enough into the cylindrical ends we have from the Schoen-Yau identity \eqref{EquationSchoenYauId} and the strict dominant energy condition that 
\begin{equation}
	R_{\tilde{g}} - |\hat{A}-k|_{\tilde{g}}^2 - 2|q|_{\tilde{g}}^2 +  2\diver^{\tilde{g}}(q)   \geq  (\mu -|J|_g). 
\end{equation}
The same treatment of this equation as in \eqref{EquationAux5} in the proof of Lemma \ref{LemmaMOTSprops} yields
\begin{equation}
	\int_{\hat{M}^n} c_n\bigg((\mu- |J|^2_{\tilde{g} }) + |\hat{A}-k|^2_{\tilde{g} } \bigg) \varphi^2 d\mu^{\tilde{g} } \leq \int_{\hat{M}^n} c_n\bigg( R_{\tilde{g}}\varphi^2 + 2|d\varphi|^2_{\tilde{g}} \bigg) d\mu^{\tilde{g}}.
\end{equation}
By the dominant energy condition and the inclusion $\hat{N}^n\subset \hat{M}^n$ it follows that
\begin{equation}
	\int_{\hat{N}^n} c_n  |\hat{A}-k|^2_{\tilde{g} }  \varphi^2 d\mu^{\tilde{g} } \leq \int_{\hat{M}^n} c_n\bigg( R_{\tilde{g}}\varphi^2 + 2|d\varphi|^2_{\tilde{g}} \bigg) d\mu^{\tilde{g}}.
\end{equation}
Adding $\int_{\hat{M}^n}\frac{1}{2}|d\varphi|^2_{\tilde{g}}d\mu^{\tilde{g}}$ to both sides we get
\begin{equation}\label{EquationAux7}
	\begin{split}
		\int_{\hat{M}^n}\frac{1}{2} |d\varphi|^2_{\tilde{g}} d\mu^{\tilde{g}} + \int_{\hat{N}^n} c_n  |\hat{A}-k|^2_{\tilde{g} }  \varphi^2 d\mu_{\tilde{g} } &\leq \int_{\hat{M}^n} \bigg( c_nR_{\tilde{g}}\varphi^2 + \bigg( \frac{1}{2} + 2c_n     \bigg)|d\varphi|^2_{\tilde{g}} \bigg) d\mu^{\tilde{g}} \\
		&\leq \int_{\hat{M}^n} \bigg( c_nR_{\tilde{g}}\varphi^2 + |d\varphi|^2_{\tilde{g}} \bigg) d\mu^{\tilde{g}}
	\end{split}
\end{equation}
since $\frac{1}{2}+ 2c_n\leq 1$.
\\ \indent To show \eqref{EquationAux6}, we replace $\varphi$ by $\Psi \varphi$ in \eqref{EquationAux7}. Clearly, we have $d\mu^{\tilde{g}_\Psi}=\Psi^{\frac{2n}{n-2}} d\mu^{\tilde{g}}$. Hence, it follows straightforwardly that the first term in the left hand side of \eqref{EquationAux7} is
\begin{equation}
	\int_{\hat{M}^n}\frac{1}{2}|d(\Psi\varphi)|^2_{\tilde{g}} d\mu^{\tilde{g}} =\int_{\hat{M}^n}\frac{1}{2}\Psi^{-2}|d(\Psi \varphi)|^2_{\tilde{g}\Psi} d\mu^{\tilde{g}_\Psi}.
\end{equation}
The second term remains the same as $\Psi=1$ on $\hat{N}^n$. Similarly, we obtain 
\begin{equation}
	\begin{split}
		\int_{\hat{M}^n}  |d(\Psi\varphi)|_{\tilde{g}}^2    d\mu^{\tilde{g}}  
		&=\int_{\hat{M}^n}  \bigg( \Psi^2|d\varphi|^2_{\tilde{g}} + 2\Psi \varphi \langle d\varphi, d\Psi\rangle_{\tilde{g}} + \varphi^2|d\Psi|^2_{\tilde{g}}  \bigg)    d\mu^{\tilde{g} } \\
		&=\int_{\hat{M}^n}   |d\varphi|^2_{\tilde{g}_\Psi}    d\mu_{\tilde{g}_\Psi} + \int_{\hat{M}} \bigg( 2\Psi \varphi \langle d\varphi, d\Psi\rangle^{\tilde{g} } + \varphi^2|d\Psi|^2_{\tilde{g} }      \bigg)    d\mu^{\tilde{g} } ,\\ 
	\end{split}
\end{equation}
and recalling the equation $\Psi$ satisfies we get
\begin{equation}
	\begin{split}
		\int_{\hat{M}^n} c_nR_{\tilde{g}}(\Psi \varphi)^2 d\mu^{\tilde{g}} &=\int_{\hat{M}^n} \bigg(   \Delta^{\tilde{g}} \Psi+  c_n R_{\tilde{g}_\Psi} \Psi^{\frac{n+2}{n-2}}\bigg) \Psi \varphi^2  d\mu^{\tilde{g} } \\
		&=\int_{\hat{M}^n}   \Psi  \varphi^2\Delta^{\tilde{g}} \Psi     d\mu^{\tilde{g}}+ \int_{\hat{M}^n}           c_n R_{\tilde{g}_\Psi}    \varphi^2  d\mu^{\tilde{g}_\Psi } \\
	\end{split}
\end{equation}
Finally, since $\Psi=1$ on $\tilde{K}^n\cup \hat{N}^n$ and $f\Delta^{\tilde{g}}h= \diver_{\tilde{g}} (hdf) + \langle dh, df \rangle_{\tilde{g}}$, we have
\begin{equation}
	\begin{split}
		\int_{\hat{M}^n}   \Psi  \varphi^2\Delta^{\tilde{g}} \Psi     d\mu^{\tilde{g}} 
		&=\int_{\hat{M}^n}  \bigg( \diver_{\tilde{g}} (\Psi  \varphi^2d\Psi)  -  \langle d(\Psi\varphi^2), d\Psi\rangle_{\tilde{g}}  \bigg) d\mu^{\tilde{g}}  \\
		&=\int_{\tilde{C}_1\cup \ldots \cup \tilde{C}_\ell}  \diver^{\tilde{g}} (\Psi  \varphi^2d\Psi) d\mu_{\tilde{g}}  - \int_{\hat{M}^n} \bigg( \varphi^2 |d\Psi|^2_{\tilde{g}} + 2\Psi\varphi \langle   d\varphi  , d\Psi\rangle_{\tilde{g}}   \bigg) d\mu_{\tilde{g}}   \\
	\end{split}
\end{equation}
The first term vanishes since $\varphi$ is compactly supported in $\tilde{C}_1\cup \ldots \cup \tilde{C}_\ell$. This shows \eqref{EquationAux6}.

\end{proof}

At this point it is convenient to introduce a new ''distance'' function $s\in C^{3,\alpha}_{loc}(\tilde{M}^n)$ on $\tilde{M}^n$. We let 
\begin{equation}\label{EquationDistanceFunction}
	s(p)=
	\begin{cases}
		r(p), \qquad &\text{when} \qquad p\in \hat{N}^n, \\
		\frac{n-2}{2\sqrt{\lambda_i}}\exp(-\frac{2\sqrt{\lambda_i}t_i}{n-2}),  &\text{when} \qquad p\in C_i.
	\end{cases}
\end{equation}
When we have no blow-ups or blow-downs, we simply let $s(p)=r(p)$ globally. In the case of blow-ups or blow-downs, this distance function will measure the distances to the infinities at the exact cylinders. 
\\ \indent By adding a point at the infinities we obtain a complete metric $(\tilde{M}^n, \tilde{g}_\Psi)$, where $s$ tends to zero at the points.

\begin{remark}\label{RemarkHarmonicCapacity}

At this stage it is convenient to define the following function:
\begin{equation}\label{EquationChiFunction}
	\chi_\epsilon (p) = 
		\begin{cases}
			0 \qquad & 0 \leq s(p) \leq \epsilon, \\
			-1 + s(p)/\epsilon     &\epsilon \leq s(p) \leq 2\epsilon, \\
			1     & 2\epsilon \leq s(p) 
		\end{cases}
\end{equation}
for $\epsilon >0$. It is not difficult to see that $|d\chi_\epsilon|_{\tilde{g}_\Psi}= \Ol(\epsilon^{-1})$ so that
\begin{equation}
	\int_{\hat{M}^n} |d\chi_\epsilon|^2_{\tilde{g}_\Psi} d\mu^{\tilde{g}_\Psi} = \Ol(\epsilon^{n-2}).
\end{equation}
It follows that the conical singularities have vanishing $\tilde{g}_\Psi$-harmonic capacity.
 
\end{remark}

\section{Asymptotic flatness of the Jang graph}\label{SectionJangAE}

In this section we show that the induced metric on the Jang graph is asymptotically flat, that is $|\hat{g}-\delta|_\delta = \Ol_2(r^{-(n-2)})$ (see Definition \ref{DefinitionAFinitialData}). To achieve this, we will show that 
\begin{equation}
	f= \sqrt{1+r^2} +\frac{\alpha}{r^{n-3}}+ \Ol^3 (r^{-(n-2-\epsilon)}).
\end{equation}
As explained in \cite[Section 6]{SakovichPMTah} we cannot directly apply the rescaling technique as in \cite{PMTII} for this purpose. Instead, we follow the argument in \cite{SakovichPMTah}. We will write $f$ as a height function $h$ over an asymptotically Euclidean manifold (roughly the graph of the lower barrier constructed in Proposition \ref{PropositionBarrierExistence}) sufficiently near the hyperbolic infinity of $(M^n,g)$. The rescaling technique can then be applied and in turn this will be translated into the desired fall-off properties of $f$. 
\\ \indent As opposed to \cite{SakovichPMTah}, we will firstly estimate the norm of the second fundamental form $\hat{A}$ of the graph of $f$ near the hyperbolic infinity $N^n$ of $M^n$.

\subsection{Estimates for the second fundamental form near infinity}\label{SubsectionAsymptotic2ndFF}

In this section we want to establish an asymptotic estimate on the second fundamental form $\hat{A}$ of the Jang graph obtained in Proposition \ref{PropositionJangLimit}. We state for convenience the expected asymptotic behaviour of $\hat{A}$ from Lemma \ref{LemmaJangGraph2ndFF}:

\begin{equation}\label{EquationLowerBarrier2ndFFnorm}
	|\hat{A}|^2_{\hat{g}}=  \frac{1}{(1+r^2)^2} +   (n-1) + \Ol(  r^{-(n+1-\epsilon)} ).
\end{equation}
Having shown that the graph $\hat{M}^n$ is asymptotically Euclidean as in Definition \ref{DefinitionAFinitialData} the asymptotic decay of $\hat{A}$ in \eqref{EquationLowerBarrier2ndFFnorm} will be immediate. However at this stage we have only the asymptotic control of $\hat{M}^n$ obtained with the barriers. 
\\ \indent We will first derive an interior gradient estimate following \cite{EichmairPlateau} in the asymptotically flat setting. Then we prove an estimate for $\hat{A}$ in Proposition \ref{Proposition2ndffEst}.

\begin{lemma}\label{LemmaInteriorGradientEstimate}
	Let $(\hat{M}^n, \hat{g})$ be the Jang graph obtained in Proposition \ref{PropositionJangLimit}. Fix $p_0\in U_f$ and $\rho\in (0, \text{inj}_{p_0}(M^n,g) )$ such that $B_\rho(p_0)\subset U_f$, where $B_\rho(p_0)$ is the geodesic ball in $(M^n,g)$ centered at $p_0$ with radius $\rho$. Suppose there exists a real number $T>0$ such that either $f(p)\leq f(p_0)+ T$ or $f(p_0)-T\leq f(p)$ for every $p\in B_\rho(p_0)$. Then $|df|_g(p_0)\leq C$, where $C$ depends only on $T$, $\rho$ and the restrictions of $g$ and $k$ to $B_\rho(p_0)$.
\end{lemma}

\begin{proof}

The following argument was used in the proofs of \cite[Lemma 2.1]{EichmairPlateau} and \cite[Lemma 2.1 in Appendix A]{EichmairMetzgerJenkinsSerrinType} (see also \cite[Theorem 1.1]{Spruck}). 
\\ \indent We will only consider the case when $f(p)\leq f(p_0)+ T$ as the case when $f(p)\geq f(p_0)- T$ is similar. For a point $p\in B_\rho(p_0)$, let
\begin{equation}
	\varphi(p)= f(p) - f(p_0) + \rho - \frac{T+\rho}{\rho^2} \text{dist}^2(p_0, p), 
\end{equation}
where $\text{dist}(p_0, p)$ is the geodesic distance between $p_0$ and $p$ in $(M^n,g)$. Further, let 
\begin{equation}
	\Omega = \{ p\in B_\rho(p_0) \: |\: \varphi(p) > 0 \}.
\end{equation}
Then $\varphi=0$ on $\partial \Omega$, $p_0\in \Omega$, $\varphi\leq T+ \rho$ on $B_\rho(p_0)$ and $\varphi\leq 0$ on $\partial B_\rho(p_0)$ so that $\Omega \subset B_\rho(p_0)$. We let $\gamma = \sqrt{1 +|df|_g^2}$, extend $\gamma$ trivially to $U_f\times \rn$ to be constant in the $t$-variable and view $\gamma$ as a function on the graph $\hat{M}^n$. We note that 
\begin{equation}\label{EquationAux333}
	\begin{split} 
		0 &= \Delta^{\hat{g}} 1 \\
		&= \gamma^{-1}\Delta^{\hat{g}} \gamma + 2\langle d\gamma ,d(\gamma^{-1}) \rangle_{\hat{g}} + \gamma \Delta^{\hat{g}} (\gamma^{-1})\\
		&=  \gamma^{-1}\Delta^{\hat{g}} \gamma - \frac{2}{\gamma^2} |d\gamma|_{\hat{g}} + \gamma \Delta^{\hat{g}} (\gamma^{-1}).
	\end{split} 
\end{equation}
Let $K\geq 1$ be a real number to be specified later and consider the cutoff-function $\eta = e^{K\varphi}-1$. Clearly $\eta \geq 0$ in $\Omega$ and $\eta=0$ on $\partial \Omega$. In what follows, we view $\eta$ as a function on the graph $\hat{M}^n$. At a point in $\Omega$ where $\eta \gamma$ attains a maximum, we have both $0=d(\eta \gamma) = \eta d\gamma + \gamma d\eta$ and $d\eta = \eta \gamma d(\gamma^{-1})$. Combining these equations with \eqref{EquationAux333} we obtain   
\begin{equation}\label{EquationMaximumAux}
	\begin{split}
		0 &\geq \gamma^{-1} \Delta^{\hat{g}} (\eta \gamma)\\
		&=\frac{\eta }{\gamma} \Delta^{\hat{g}} \gamma + \frac{2}{\gamma} \langle d\gamma, d\eta\rangle_{\hat{g}} + \Delta^{\hat{g}} \eta   \\
		&=  \frac{2\eta}{\gamma^2} |d\gamma|_{\hat{g}} - \gamma\eta \Delta^{\hat{g}} \gamma^{-1} + \frac{2}{\gamma} \langle d\gamma, d\eta\rangle_{\hat{g}} + \Delta^{\hat{g}} \eta \\
		&= - \gamma \eta \Delta^{\hat{g}} (\gamma^{-1})  + \Delta^{\hat{g}} \eta \\
		&\geq - \beta \eta - \gamma \eta \langle X , d(\gamma^{-1})\rangle_{\hat{g}} +\Delta^{\hat{g}} \eta \\
		&=- \beta \eta -  \langle X , d\eta\rangle_{\hat{g}} +\Delta^{\hat{g}} \eta
	\end{split}
\end{equation}
where the vector field $X$ and the constant $\beta$ are defined before \eqref{EquationXaux} and in \eqref{EquationDifferentialInequality}, respectively, in the proof of Proposition \ref{PropositionJangLimit}. We have
\begin{equation}
	\begin{split} 
		\langle X, d\eta \rangle_{\hat{g}} &\leq |X|_{\hat{g}} |d\eta|_{\hat{g}} \\
		&= |X|_{\hat{g}} e^{K\varphi} K |d\varphi|_{\hat{g}} \\
		&\leq \sqrt[4]{n}Ke^{K\varphi }  |X|_g  |d\varphi|_{\hat{g}} \\
		&\leq D Ke^{K\varphi }   |d\varphi|_{\hat{g}},
	\end{split}  
\end{equation} 
where the second inequality follows in a similar fashion as \eqref{EquationMeanCurvIsBounded} and the last inequality is shown in \eqref{EquationXaux}. A computation shows that
\begin{equation}
	\Delta^{\hat{g}} \eta = K^2e^{K\varphi} |d\varphi|^2_{\hat{g}} + Ke^{K\varphi} \Delta^{\hat{g}} \varphi.
\end{equation} 
and so \eqref{EquationMaximumAux} implies 
\begin{equation}\label{EquationMaximumAuxII}
	\begin{split}
		0 &\geq - \beta \eta -  \langle X , d\eta\rangle_{\hat{g}} +\Delta^{\hat{g}} \eta \\
		&\geq \beta + \big( K\Delta^{\hat{g}} \varphi  + K^2 |d\varphi|^2_{\hat{g}} - KD|d\varphi|_{\hat{g}}-  \beta\big) e^{K\varphi}. \\
	\end{split}
\end{equation}
We compute
\begin{equation}\label{EquationSomething}
	\begin{split}
		|d\varphi|^2_{\hat{g}} &= \bigg( g^{ij} - \frac{f^{,i} f^{,j}}{1 + |df|^2_g} \bigg)\varphi_{,i} \varphi_{,j} \\
		&=|d\varphi|^2_g - \frac{\langle df, d\varphi\rangle_g^2}{1+ |df|^2_g} \\
		&=\bigg| df - \bigg( \frac{T + \rho}{\rho^2}\bigg)d \big(\text{dist}^2(p_0, \cdot) \big) \bigg|_g^2  \\
			&\qquad - \frac{    \big\langle df, df- \big( \frac{T + \rho}{\rho^2}\big)d \big(\text{dist}^2(p_0, \cdot) \big)    \big\rangle_g^2}{1+ |df|^2_g} \\
		&= \frac{ |df|_g^2- 1+ \big( 1- \big( \frac{T+\rho}{\rho^2} \big) \langle df, d \big(\text{dist}^2(p_0, \cdot) \big) \rangle_g  \big)^2 }{1+ |df|_g^2} \\
		&\qquad + \bigg( \frac{T + \rho}{\rho^2}\bigg)^2 \big| d \big(\text{dist}^2(p_0, \cdot) \big) \big|_g^2
	\end{split}
\end{equation}
and note that $1\leq\liminf_{|df|_g\rightarrow \infty} |d\varphi|^2_{\hat{g}}$. To compute $\Delta^{\hat{g}} \varphi$, we note that for any function $u$ defined on $M^n$ and viewed as a function on $\hat{M}^n$, we have
\begin{equation}
	\Hess^{\hat{g}}_{ij} u = \Hess^g_{ij} u - f^{,k}u_{,k} \frac{\Hess^g_{ij} f }{1+ |df|^2_g}
\end{equation} 
and so, from \eqref{EquationHessianOfGraph}, we obtain
\begin{equation}\label{EquationGraphLaplacian}
	\Delta^{\hat{g}} u = \hat{g}^{ij}\Hess^g_{ij} u - du(\nabla^g f) \frac{H_{\hat{M}^n}}{\sqrt{1+|df|^2_g}}.
\end{equation}
In particular, when $u=f$, this yields 
\begin{equation}\label{EquationGraphLaplacianII}
	\Delta^{\hat{g}}f = \frac{H_{\hat{M}^n}}{\sqrt{1+|df|_g^2}},
\end{equation}
so that $f$ is $\hat{g}$-harmonic precisely when the graph $(\hat{M}^n, \hat{g})$ is minimal. Furthermore, we note that
\begin{equation}\label{EquationAuxBlah}
	\begin{split}
		\Delta^{\hat{g}} \big( \text{dist}^2(p_0, \cdot) \big) &= 2\text{dist}(p_0, \cdot) \Delta^{\hat{g}} \big( \text{dist}(p_0, \cdot) \big) + 2 |d \big( \text{dist}(p_0, \cdot) \big)|_{\hat{g}}^2 \\
		&\leq 2\rho  \Delta^{\hat{g}} \big( \text{dist}(p_0, \cdot) \big) +2 \sqrt{n}   |d \big(\text{dist}(p_0, \cdot)\big)|_{g}^2.
	\end{split}
\end{equation}
Inserting $\text{dist}(p_0, \cdot)$ into \eqref{EquationGraphLaplacian} we obtain
\begin{equation}\label{EquationLaplacianDistanceEstimate}
	\begin{split}
		\Delta^{\hat{g}} \text{dist}(p_0, \cdot) &= \hat{g}^{ij}\Hess^g_{ij} \text{dist}(p_0, \cdot) - d\big(\text{dist}(p_0, \cdot)\big)(\nabla^g f) \frac{H_{\hat{M}^n}}{\sqrt{1+|df|^2_g}} \\
		&\leq \sqrt{n} | \Hess^g \text{dist}(p_0, \cdot) |_g + |d \big(\text{dist}(p_0, \cdot) \big) |_g |df|_g \frac{H_{\hat{M}^n}}{\sqrt{1+|df|^2_g}} \\
		&\leq \sqrt{n}C + |H_{\hat{M}^n}| \\
		&\leq \sqrt{n}(C + |k|_g),
	\end{split}
\end{equation} 
where we used the Hessian comparison theorem to estimate $| \Hess^g \big( \text{dist}(p_0, \cdot) \big) |_g\leq C$ and $|d \big(\text{dist}(p_0, \cdot) \big)|_g^2=1$ in the second inequality and \eqref{EquationMeanCurvIsBounded} was used in the last inequality. From the estimates \eqref{EquationGraphLaplacianII} - \eqref{EquationLaplacianDistanceEstimate}, we conclude that $|\Delta^{\hat{g}} \varphi|\leq \tilde{C}$, where $\tilde{C}$ is a constant depending only on the initial data and $\rho$. Inserting this estimate into \eqref{EquationMaximumAuxII}, we obtain 
\begin{equation}
	0 \geq K |d\varphi|^2_{\hat{g}} - D |d\varphi|_{\hat{g}} - (\beta +\tilde{C}) .
\end{equation}
Choosing $K$ large enough so that $|d\varphi|_{\hat{g}}<1$ and recalling \eqref{EquationSomething}, we conclude that $|df|_g$ is bounded in terms of $T$, $\rho$ and the restriction of $g$ and $k$ to $B_\rho(p_0)$ at a point $q$ in $\Omega$ where $\eta \gamma$ attains a local maximum. Denote this constant by $A$, so that $|df|_g(q)\leq A$. Then the estimate $\eta \gamma(p_0)\leq \eta \gamma(q)$ together with the bound $\varphi\leq T + \rho$ implies $(e^{K\rho}-1)\sqrt{1 + |df|_g^2(p_0)} \leq (e^{K(\rho + T)}-1)\sqrt{1+ A^2}$, which proves our claim.

\end{proof}

\begin{proposition}\label{Proposition2ndffEst}
	Let $(\hat{M}^n, \hat{g})$ the Jang graph obtained in Proposition \ref{PropositionJangLimit} and $\hat{A}$ its second fundamental form. Then $|\hat{A}|_{\hat{g}} \leq C(M^n, g,k)$ sufficiently far out in $\hat{N}^n$.

\end{proposition}

\begin{proof}
	
We get the estimate of $|\hat{A}|_{\hat{g}}$ as follows; the estimate in Lemma \ref{LemmaInteriorGradientEstimate} in $B_\rho(p_0)$ improves to a $C^{1,\alpha}$-bound as in Section \ref{SectionDirichletProblem}, except that since $B_\rho(p_0)$ is open we restrict to the closed ball $\overline{B}_{\rho/2}(p)$ which is a compact manifold with boundary (and the coefficient estimates necessary follow from the $C^1$-estimate). The arguments in Lemma \ref{LemmaSisClosed} yields a $C^{2,\alpha}$-bound on $f$: $||f||_{C^{2,\alpha}(\bar{B}_{\rho/2})} \leq C(M^n, g,k)$. In particular, we obtain an estimate on $|\Hess^g (f) |_g$ and so the bound on $|\hat{A}|_g$ follows by considering the coordinate expression for $\hat{A}$ in terms of $f$ as given explicitly in \eqref{EquationJangsEquation}. 
\\ \indent It remains to verify the bound on the intrinsic norm $|\hat{A}|_{\hat{g}}$. It is easily seen that 
\begin{equation}
	\begin{split}
		|\hat{A}|^2_{\hat{g}} &= 
		\langle \hat{g}\otimes \hat{g}, \hat{A} \otimes \hat{A} \rangle \\
	&	\leq |\hat{g}\otimes \hat{g}|_g |\hat{A} \otimes \hat{A} |_g \\
		&= |\hat{g}|^2_g |\hat{A}|^2_g,
	\end{split}
\end{equation}
so that the bound follows since $|\hat{g}|^2_g \leq \sqrt{n}$.
\\ \indent Finally, from the fall-off properties of the barriers and the properties of the initial data the estimate is uniform sufficiently far out in $N^n$.

\end{proof}

\subsection{Setup and Fermi coordinates}\label{SectionFermiSetup}

We now describe how to rewrite the Jang graph in terms of a height function $h$ over an asymptotically Euclidean manifold. Firstly, from Section \ref{SectionBarriers} we know that 
\begin{equation}
	f= \sqrt{1+r^2} +\frac{\alpha}{r^{n-3}}+ \Ol  (r^{-(n-2-\epsilon)}) ,
\end{equation} 
and that there are functions $f_\pm$ such that $f_- \leq f \leq f_+$ with the same asymptotics defined on $M^n_{r_0}= \{ r\geq r_0\}$. It is clear from the properties of $f_\pm$ that there exists functions with the same asymptotics but with derivatives of better fall-off properties, defined on $M^n_{r_1}$, for $r_1> r_0$. We let these functions be denoted by $f_\pm$ (as we will not use the barriers of Proposition \ref{PropositionBarrierExistence} again). The functions are such that $f_-\leq f\leq f_+$ and 
\begin{equation}
	f_\pm= \sqrt{1+r^2} +\frac{\alpha}{r^{n-3}}+ \Ol^3  (r^{-(n-2-\epsilon)})
\end{equation} 
and we denote them by \emph{upper-} and \emph{lower barriers}. We denote the graphs of the barriers $f_\pm$ by $\hat{M}^n_\pm$. 
\\ \indent Next, we invoke \emph{Fermi coordinates} (or \emph{normal geodesic coordinates}) adapted to $\hat{M}^n_-$. Let $\Psi=(x^1, \ldots, x^n)$ be an asymptotically Euclidean Cartesian coordinate system on $\hat{M}^n_-$. From Proposition \ref{Proposition2ndffEst} we know that $\hat{M}^n$ has uniformly bounded geometry sufficiently far out in $\hat{N}^n$. By \cite[Chapter 2]{Eldering} $\hat{N}^n$ has an open neighbourhood in $M^n\times \rn$, denoted by $N_\gamma(\hat{M}^n_-)$, and a diffeomorphism $y:N_\gamma(\hat{M}^n_-)\simeq \hat{M}^n_-\times (-\gamma, \gamma)$ for some positive radius $\gamma>0$. We take the coordinates $y$ on $N_\gamma(\hat{M}^n)$ to be $y(\cdot, 0)=\Psi(\cdot)$ and $\frac{\partial y}{\partial \rho}=\vec{n}^{\rho}$, where $\vec{n}^{\rho}$ is the upward pointing unit normal to the $\rho$-level set $\hat{M}^n_\rho= y(\cdot, \rho)$. In these coordinates the metric $g+ dt^2$ takes the form $d\rho^2 + \hat{g}_\rho$, where $\hat{g}_\rho$ is the induced metric on $\hat{M}^n_\rho$. The induced metric $\hat{g}_0$ on $\hat{M}^n_0=\hat{M}^n_-$ will be denoted by $\hat{g}_-$. We will denote by $\langle \cdot, \cdot \rangle$ the metric $ds^2=g+ dt^2$ on $M^n\times \rn$. The second fundamental form of $\hat{M}^n_\rho$ is denoted by $\hat{A}_\rho$, where the convention is $(\hat{A}_\rho)_{ij} = \langle \nabla_{\partial_i}\partial_j, \vec{n}^\rho\rangle$.
\\ \indent We will show that we can write the Jang solution constructed in Section \ref{SectionJangSolution} as the graph of $h$:
\begin{equation}
	\hat{M}^n\cap \{ r\geq r_2 \} =\{(p,h(p))\in \hat{M}^n_-\times (-\gamma, \gamma)\: |\: p\in \hat{M}^n_-  \}\cap \{ r\geq r_2 \},
\end{equation}
where $h:\hat{M}^n_-\rightarrow [0,\gamma)$ and the relevance of $r_2\geq r_1$ will be explained in Corollary \ref{CorollaryHeightFunExistence}.
\\ \indent The following proposition, which is proven in Appendix \ref{SectionRicattiSystem} using the bound on $\hat{A}$ obtained in Proposition \ref{Proposition2ndffEst}, will be useful. 

\begin{proposition}\label{PropositionFermiCoord}
	There exists constants $\rho_0>0$ and $C\geq 1$ such that $|\hat{A}_\rho|_{\hat{g}_\rho}<C$ and $C^{-1}\delta\leq \hat{g}^\rho\leq C\delta$ for any $0\leq \rho \leq \rho_0$. Furthermore, all partial derivatives of $(\hat{g}_\rho)_{ij}$ and $(\hat{A}_\rho)^i_j$ up to order $3$ in the Fermi coordinates are bounded. 
\end{proposition}

\subsection{Existence of the height function and a priori estimates}

In this section we obtain the height function $h:\hat{M}^n_-\rightarrow [0,\gamma)$ as described above. Moreover we derive a priori estimates of $h$.
\\ \indent As a first step, we obtain the following ''tilt-excess'' estimate or the normal.

\begin{lemma}\label{LemmaTiltExcess}
Let $\vec{n}$ and $\vec{n}_-$ be the upward pointing unit normals of $\hat{M}^n$ and $\hat{M}^n_-$, respectively, extended parallelly to $M^n\times \rn$. Then there exists a constant $C>0$ such that for every $p\in M^n\times \rn$, with $r(p)>r_1$, we have
\begin{equation}
	|\vec{n}(p) - \vec{n}_-(p) |^2 \leq  C  r(p)^{-(n-2-\epsilon)}.
\end{equation}

\end{lemma}

\begin{proof}
For any point $p\in M^n\times \rn$ we write $p_{M^n }=\text{proj}_{M^n}$, where $\text{proj}_{M^n}:M^n\times \rn\rightarrow M^n$ is the standard projection operator. Similarly, we write $p_\rn = \text{proj}_\rn$, so that $p = (p_{M^n}, p_{\rn})\in M^n\times \rn$. Clearly $r(p_{M^n})= r(p)$.
\\ \indent Let $p\in \hat{M}^n$ be such that $r(p)> 2r_1$. We shift $\hat{M}^n_-$ vertically so that it coincides with $\hat{M}^n$ at $p$ and denote the resulting submanifold by $\hat{M}^n_{p}$. $\hat{M}^n_{p}$ is then the graph of the function $f_{p} : M^n  \rightarrow \rn$ defined for $r(p)>r_1$, explicitly given by 
\begin{equation}
	f_{p}= f_- + (f(p_M)- f_-(p_M)).
\end{equation}
We define the function $F_-: M^n\times \rn \rightarrow \rn$ by $F_-(x^1, \ldots, x^n, t) = t- f_{p}(x^1, \ldots, x^n)$. We have $F_-=0$ on $\hat{M}^n_p$ and it follows straightforwardly that  
\begin{equation}
	\vec{n}_- = \frac{\nabla^{ds^2}F_-}{|\nabla^{ds^2}F_-|} = \frac{\partial_t - \nabla^g f_p}{1+|df_p|_g^2}
\end{equation}
on $M^n\times \rn$ for $r>r_1$.  
\\ \indent For a point $q\in \hat{M}^n$ we let $\gamma$ be a unit speed geodesic in $\hat{M}^n$ such that $\gamma(0)=p$ and $\gamma(s)=q$. Since $F_-(p)=0$ and $\hat{g}(\nabla^{\hat{g}} F_-, \dot{\gamma}(0))= \langle \nabla^{ds^2}F_-, \dot{\gamma}(0)\rangle$, Taylor's formula for $F_-(\gamma(s))$ at $s=0$ gives  
\begin{equation}\label{EquationFexpansion}
		F_-(q)
		=\langle \nabla^{ds^2}F_-, \dot{\gamma}(0)\rangle s + \Hess^{\hat{M}^n}(F_-)(\dot{\gamma}(\theta s), \dot{\gamma}(\theta s))\frac{s^2}{2},	
\end{equation}
where $\theta \in (0,1)$ and the Hessian is evaluated at $\gamma(s\theta)$. The claim will follow for specific choices of $s$ and $\dot{\gamma}(0)$.
\\ \indent From the properties of the replaced barriers we know that there exists a constant $C_0>0$ such that
\begin{equation}
	0 \leq (f_+-f_-)(r)  \leq  C_0 r^{-(n-2-\epsilon)}  
\end{equation}
for $r\geq r_1$. Let $\delta = (2^{n-1}+1) C_0 r(p)^{-(n-2-\epsilon)}$ and let $q$ be such that $\text{dist}_{\hat{M}^n}(p,q) = \sqrt{\delta}$. We claim that we may assume 
\begin{equation}
	\frac{r(p)}{2}\leq r(q) \leq 2r(p).
\end{equation}
Indeed, if we assume $r(q) > 2 r(p)$ we get
\begin{equation}
	\begin{split}
		\sqrt{(2^{n-1}+1) C_0r(p)^{-(n-2-\epsilon)}} &= \sqrt{\delta} \\
		&=\text{dist}_{\hat{M}^n}(p,q) \\
		&\geq \text{dist}_{M^n}(p_{M^n}, q_{M^n}) \\
		&\geq \int_{r(p)}^{r(q)} \frac{dr}{\sqrt{1+r^2}} \\
		&\geq   \int_{r(p)}^{2r(p)} \frac{dr}{\sqrt{1+r^2}} \\
		&\geq \frac{r(p)}{\sqrt{ 1 + 4r^2(p)}} \\
		&\geq \frac{1}{\sqrt{r^{-2}(p)+4}},
	\end{split}
\end{equation}
which cannot be true for sufficiently large $r_1$. A similar calculation yields a contradiction with the assumption $r(q)< \frac{r(p)}{2}$.
\\ \indent Since now $\frac{r(p)}{2}\leq  r(q) \leq 2r(p)$, we have $r(q)\geq r_1$, so that $f_-(q_{M^n})$ and $f_+(q_{M^n})$ are well-defined. Let $\tilde{q}\in \hat{M}^n_p$ be the point such that $\text{proj}_{M^n}(\tilde{q})= \text{proj}_{M^n}(q)$. Then 
\begin{equation}
	\begin{split}
		\text{dist}_{M^n\times \rn} (q,\tilde{q}) &= |f_{p} (q_{M^n})- f(q_{M^n})| \\
		&\leq | f_-(q_{M^n}) - f(q_{M^n})| + |f(p_{M^n}) - f_-(p_{M^n})| \\
		&\leq | f_-(q_{M^n}) - f_+(q_{M^n})| + |f_+(p_{M^n}) - f_-(p_{M^n})| \\
		&\leq  C_0 r(q_{M^n})^{-(n-2-\epsilon)}  +  C_0 r(p_{M^n})^{-(n-2-\epsilon)}  \\
		&\leq   C_0( 2^{(n-2-\epsilon)}+1 ) r(p_{M^n})^{-(n-2-\epsilon)} \\
		&\leq (2^{n-1}+1) C_0  r(p)^{-(n-2-\epsilon)} \\
		&=\delta.
	\end{split}
\end{equation}
This estimate, together with $F_-(\tilde{q})=0$ and $\nabla^{ds^2}F_-$ being constant along the $\rn$-factor then implies 
\begin{equation}
	F_-(q) \leq |F_-(q)- F_-(\tilde{q})| \leq \delta|\nabla^{ds^2}F_-|(q).
\end{equation}
so that we get an estimate of the left hand side of \eqref{EquationFexpansion}.
\\ \indent The right hand side of \eqref{EquationFexpansion} can be estimated as follows. We recall the definition of the Hessian of a $C^2(M^n)$-function $f$ on $(M^n,g)$ a $C^2$-smooth Riemannian manifold with Levi-Civita connection $\nabla$:  
\begin{equation}
	\Hess^g(f)(X,Y)=X(Y(f)) - df(\nabla_X Y).
\end{equation}
It follows, for $X,Y$ tangential vector fields on $\hat{M}^n$, that
\begin{equation}
	\Hess^{M^n\times \rn}(F_-)(X,Y) = \Hess^{\hat{M}^n}(F_-)(X,Y) - \hat{A}(X,Y)dF_-(\vec{n}) 
\end{equation}
and in turn, since $dF_-(\vec{n}) \leq |\nabla^{ds^2}F_-|$ from the Cauchy-Schwarz inequality, that
\begin{equation}
	| \Hess^{M^n\times \rn}(F_-)  - \Hess^{\hat{M}^n}(F_-)  |_{\hat{g}}  \leq |\nabla^{ds^2}F_-|  |\hat{A}|_{\hat{g}}.
\end{equation}
With the tensor equality $\Hess^{M^n\times \rn}(F_-) = |\nabla^{ds^2} F_-| \hat{A}_-$ (which follows from the fact that $M^n\times \{p\}$ is totally geodesic in $M^n \times \rn$ and similar equations for the Hessian as in the proof of Lemma \ref{LemmaJangEqHeightFun}) and the boundedness of $\hat{A}$ from Proposition \ref{Proposition2ndffEst} we obtain  
\begin{equation}
	\begin{split}
		|\Hess^{\hat{M}^n}(F_-) |_{\hat{g}} &\leq |\Hess^{M^n\times \rn}(F_-)|_{\hat{g}} + |\nabla^{ds^2}F_-| |\hat{A}|_{\hat{g}} 	\\
		&\leq C|\nabla^{ds^2}F_-| 
	\end{split}
\end{equation}
follows, where the last inequality follows from the expression of the second fundamental form together with the estimates calculated in \eqref{EquationLowerBarrier2ndFFnorm} and the estimates in Proposition \ref{Proposition2ndffEst}. 
\\ \indent With these estimates at hand, we choose $s=\sqrt{\delta}$ in \eqref{EquationFexpansion} and obtain 
\begin{equation}\label{EquationGradFestimate}
\delta |\nabla^{ds^2}F_-|(q) \geq \sqrt{\delta} \langle \nabla^{ds^2}F_-, \dot{\gamma}(0)\rangle - C\delta \sup_{0\leq \theta\leq 1}|\nabla^{ds^2}F_-|(\gamma(\theta s)).
\end{equation}
From Lemma \ref{LemmaJangGraphMetric} it follows that 
\begin{equation}
	|\nabla^{ds^2} F_-|^2 = 1 + r^2 + \Ol ( r^{-(n-2-\epsilon)} )
\end{equation}
and so $|\nabla^{ds^2}F_-|=   r+ \Ol(r^{-1})$ and since $\gamma$ is a geodesic such that $\gamma(0)=p$ and $\gamma(s)=q$ we can also estimate 
\begin{equation}
	\begin{split}
		\sup_{0\leq \theta\leq 1 } |\nabla^{ds^2}F_-|(\gamma(\theta s)) &\leq 2 \max\{ r(p), r(q)  \} \\
		&< 4r(q) \\
		&< 8 |\nabla^{ds^2}F_-|(q).  
	\end{split}
\end{equation}
This estimate combined with \eqref{EquationGradFestimate} gives
\begin{equation}
	\langle \vec{n}_-(p), \dot{\gamma}(0)\rangle \leq C \sqrt{\delta}, 
\end{equation}
where $C$ may be a larger constant. In turn, for the particular choice
\begin{equation}
	\dot{\gamma}(0)= \frac{\nabla^{\hat{g}}F_-(p)}{|\nabla^{\hat{g}}F_-(p)|},
\end{equation}
where $\nabla^{\hat{g}}F_-$ is the gradient of $F_-$ on $\hat{M}^n$,at the point $p\in \hat{M}^n$ we get
\begin{equation}
	\begin{split}
		C\sqrt{\delta} &\geq \bigg\langle \vec{n}_-, \frac{\nabla^{\hat{g}}F_- }{|\nabla^{\hat{g}}F_- |} \bigg\rangle \\
		&=  \frac{\langle \vec{n}_-, \nabla^{ds^2}F_- - \langle \nabla^{ds^2}F_-, \vec{n}\rangle \vec{n}\rangle}{|\nabla^{\hat{g}}F_-|} \\
		&=\frac{\langle \vec{n}_-,\nabla^{ds^2}F_-\rangle - \langle \nabla^{ds^2}F_-, \vec{n}\rangle  \langle \vec{n}, 	\vec{n}_-\rangle}{|\nabla^{\hat{g}}F_-|}  \\
		&=  \frac{|\nabla^{ds^2}F_-|\langle \vec{n}_-, \vec{n}_- - \langle \vec{n}, \vec{n}_-\rangle \vec{n}\rangle}{|\nabla^{\hat{g}}F_- 	|}     \\
		&=  \frac{(1 - \langle \vec{n}, \vec{n}_-\rangle^2)|\nabla^{\hat{g}}F_-|}{  |\nabla^{ds^2}F_- - \langle \nabla^{ds^2}F_-, 	\vec{n}\rangle \vec{n} |} \\
		&=\frac{1- \langle \vec{n},\vec{n}_-\rangle^2}{|\vec{n}_- -\langle \vec{n}, \vec{n}_-\rangle \vec{n}|} \\
		&= \sqrt{ 1 - \langle \vec{n}, \vec{n}_-\rangle^2}.
	\end{split}
\end{equation}
It follows that $\langle \vec{n}, \vec{n}_-\rangle^2(p)= 1+ \Ol(\delta)$, which proves the assertion.

\end{proof}

With Lemma \ref{LemmaTiltExcess} at hand we can prove existence of the height function. 

\begin{corollary}\label{CorollaryHeightFunExistence}
	There exists a non-negative $C^{3,\alpha}_{loc}$-function $h:\hat{M}^n_-\rightarrow \rn$ and $r_2>0$ such that $\hat{M}^n\cap (M^n_{r_2} \times \rn)=\text{graph}(h)$ in the Fermi coordinates as described in section \ref{SectionFermiSetup}. 
\end{corollary}

\begin{proof}
	
We use the same notation as in Lemma \ref{LemmaTiltExcess}. Let $F:M^n_{r_1}\times \rn \rightarrow \rn$ be given by $F(p,t)= t- f(p)$. Then 
\begin{equation}
	\frac{\nabla^{ds^2}F}{|\nabla^{ds^2}F|}   = \vec{n} \qquad \text{in}\qquad M^n\times \rn, \\
\end{equation}
where $F=0$ precisely on $\hat{M}^n$. We want to show that $\partial_\rho F$ is bounded away from $0$ on $\hat{M}^n\cap (M^n\times \rn)$, for $r>r_2$, for large enough $r_2$. The claim will then follow from the Implicit Function Theorem.
\\ \indent We let $q\in \hat{M}^n\cap (M^n\times \rn)$, where $r(q)>3r_1$. Let $q_-$ denote the orthogonal projetion of $q$ to $\hat{M}^n_-$, and let $q_{M^n}=\text{proj}_{M^n}(q)$. We may assume as in Lemma \ref{LemmaTiltExcess} that $\frac{1}{2}r(q)\leq r(q_-)\leq 2r(q)$. Since $\partial_\rho =\vec{n}_-$ on $\hat{M}^n_-$, we have 
\begin{equation}
	\frac{\partial_\rho F(q_-)}{|\nabla^{ds^2}F(q_-)|} = \langle \vec{n}(q_-), \vec{n}_-(q_-)\rangle = 1+ \Ol\big(  r(q_-)^{-\frac{n-2-\epsilon}{2}}  \big), 
\end{equation}
from Lemma \ref{LemmaTiltExcess}, where $\langle \cdot, \cdot \rangle$ is the inner product on $M^n\times \rn$. It is not difficult to see that the norm $|\nabla^{ds^2} \vec{n}|$ is bounded by a constant depending only on the bounds on $|\Hess^g(f)|_g$ and $|df|_g$, 
which is in turn bounded by the calculations in the proof of Proposition \ref{Proposition2ndffEst}.	We obtain 
\begin{equation}
	\begin{split}
		\bigg| \frac{\partial_\rho F(q)}{|\nabla^{ds^2}F(q)|} - \frac{\partial_\rho F(q_-)}{|\nabla^{ds^2}F(q_-)|}\bigg| &\leq \bigg| 	\frac{\nabla^{ds^2} F(q)}{|\nabla^{ds^2}F(q)|} - \frac{\nabla^{ds^2} F(q_-)}{|\nabla^{ds^2} F(q_-)|}\bigg| \\
		&=|\vec{n}(q)- \vec{n}(q_-)| \\
		&\leq  |\nabla^{ds^2}\vec{n}| \cdot  \text{dist}_{M^n\times \rn}(q, q_-) \\
		&\leq C(f(q_{M^n})- f_-(q_{M^n})) \\
		&= \Ol ( r(q)^{-(n-2-\epsilon)}   ).
	\end{split}
\end{equation}
It follows that
\begin{equation}
	\frac{\partial_\rho F(q)}{|\nabla^{ds^2}F(q)|}= 1+ \Ol (   r(q)^{-(n-2-\epsilon)}     )
\end{equation}
from the triangle inequality.
\\ \indent Finally, since $|\nabla^{ds^2}F|=\sqrt{1+|df|_g^2}\geq 1$ we get
\begin{equation}
	\partial_\rho F (q) \geq \frac{1}{2}|\nabla^{ds^2} F(q)|\geq \frac{1}{2},
\end{equation}
for $r(q)\geq r_2$, where $r_2$ is sufficiently large. 

\end{proof}

Since the Jang graph is located between the barriers $f_+$ and $f_-$, the height function $h$ must fall off as the difference $f_+-f_-= \Ol (  r^{-(n-2-\epsilon)}  )$ in view of Lemma \ref{LemmaTiltExcess}. In Lemma \ref{LemmaHeightFunDecays} below we refine estimate further, and also establish a priori estimates for the coordinate derivatives $h_{,k}$ and $h_{,k\ell}$.  

\begin{lemma}\label{LemmaHeightFunDecays}
	
Let $h:\hat{M}^n_-\rightarrow \rn$ be the non-negative height function in Corollary \ref{CorollaryHeightFunExistence}. Then 
\begin{equation}
	|h|   = \Ol(   r^{-(n-1-\epsilon)} ), \qquad |dh|_\delta    =  \Ol (  r^{-(n-1-\epsilon)/2}  ), \qquad |\Hess^\delta (h)|_\delta    = \Ol(1).  
\end{equation}

\end{lemma}

\begin{proof}
	
We first prove the assertion about the fall-off of $h$. For $r$ large enough both $f_+$ and $f_-$ are both strictly increasing. Let $p\in \hat{M}^n_+$ and $q\in \hat{M}^n_-$ be such that $p$ is projected orthogonally to $q$. We define $z\in \hat{M}^n_-$ so that $p$ projects radially along the $M^n$-factor to $z$. In other words, $p$ and $z$ have the same coordinates with exception to the radial coordinates. In particular, $z=(z_{M^n}, f_-(z_{M^n}))=(z_{M^n}, f_+(p_{M^n}))$. Furthermore, $h(q)\leq \text{dist}_{M^n\times \rn}(p,q)\leq \text{dist}_{M^n\times \rn}(p,z)$ and so we only need to estimate the geodesic distance between $p$ and $z$ in $M^n\times \rn$.  
\\ \indent Denote $r_p=r(p_{M^n})$, $r_z=r(z_{M^n})$ and $r_q=r(q_{M^n})$ for brevity. Now, from the properties of $f_+$ and $f_-$ we have
\begin{equation}
	\begin{split}
		f_-(r_z, \theta)-f_-(r_p, \theta) &=f_+(r_p, \theta)-f_-(r_p, \theta) \\ 
		&= \Ol( r_p^{-(n-2-\epsilon)} )  .
	\end{split} 
\end{equation}
Consequently, by the Mean Value Theorem there exists $\beta \in [0,1]$ from such that 
\begin{equation}
	\begin{split}
		f_{-,r}(\beta r_z +(1-\beta)r_p, \theta)(r_z-r_p)  	&= \Ol (  r_p^{-(n-2-\epsilon)} ).
	\end{split}
\end{equation}
Since
\begin{equation}
	f_{-,r}(r,\theta)= 1 + \Ol (r^{-2}) 
\end{equation}
it follows that
\begin{equation}
	r_z - r_p =   \Ol \big( r_p^{-(n-2-\epsilon)} \big).  
\end{equation}
Thus
\begin{equation}\label{EquationAuxIII}
	\begin{split}
		\text{dist}_{M^n\times \rn}(p,z) &=\int_{r_p}^{r_z} \frac{dr}{\sqrt{1+r^2}} \\
		&\leq \frac{r_z-r_p}{\sqrt{1+r_p^2}} \\
		&= \Ol (  r_p^{-(n-1-\epsilon)}   ).
	\end{split}
\end{equation}
It only remains to show that we may replace $r_p$ with $r_q$ in \eqref{EquationAuxIII}. We have
\begin{equation}
	\begin{split}
		\frac{C }{r_p^{n-1-\epsilon}} & \geq \text{dist}_{M^n\times \rn}(p,z) \\
		&\geq \text{dist}_{M^n\times \rn}(p,q) \\
		&\geq \text{dist}_{M^n}(p_{M^n},q_{M^n}) \\
		&\geq\bigg| \int_{r_p}^{r_q} \frac{dr}{\sqrt{1+r^2}} \bigg| \\
		&\geq \frac{|r_q-r_p|}{\sqrt{1 + (r_p+r_q)^2}} \\
		&=\frac{|r_qr_p^{-1}-1|}{\sqrt{r_p^{-2} + (r_qr^{-1}_p+1)^2}} \\
		&\geq \frac{|r_qr_p^{-1} -1 |}{\sqrt{2}(r_qr_p^{-1}+1)},
	\end{split}
\end{equation}
which implies that $r_qr^{-1}_p\rightarrow 1$ as $r_p\rightarrow \infty$. Hence 
\begin{equation}
	h(q)=  \Ol(r_q^{-(n-1-\epsilon)})
\end{equation}
as asserted.
\\ \indent We now establish the bound on $h_{,k}$. Let $p\in \hat{M}^n$ and let $p_-$ be the orthogonal projection on $\hat{M}^n_-$. For $\rho_0=h(p_-)$ we consider the function $\Phi=\Phi(\rho)=\rho_0-\rho$. As in the proof of Lemma \ref{LemmaTiltExcess} we conclude that $|\Hess^{\hat{g}} \Phi|\leq C$. If $\gamma$ is a unit speed geodesic in $\hat{M}^n$ such that $\gamma(0)=p$ and $\gamma(s)=q$, we have
\begin{equation}\label{EquationExpansionAgain}
	\Phi(q)\geq d\Phi(\dot{\gamma}(0))\text{dist}_{\hat{M}^n}(p,q) - C \text{dist}_{\hat{M}^n}^2(p,q).
\end{equation}
We have already proven the first assertion about the fall-off of $h$, and so we set $s=\sqrt{\delta}$, where $\delta=(2^{n-1}+1) C_0r(p)^{-(n-2+\epsilon)}$, where $C_0$ is such that $(f_+-f_-)(r) \leq C_0r^{-(n-2-\epsilon)}$.
\\ \indent Let $q_-$ denote the orthogonal projection of $q$ on $\hat{M}^n_-$. We have 
\begin{equation}
	\frac{r(p_-)}{2}\leq r(q_-) \leq 2r(p_-)
\end{equation}
in this case as well, given that $r(p)$ is large enough. The left hand side of \eqref{EquationExpansionAgain} may be estimated as follows: 
\begin{equation}
	\begin{split}
		|\Phi(q)| &\leq |h(p_-)- h(q_-)| \\
		&\leq \frac{C_0}{r(p_-)^{n-2+\epsilon}} + \frac{C_0}{r(q_-)^{n-1-\epsilon}} \\
		&\leq(2^{n-1}+1)\frac{C_0}{r(p_-)^{n-1-\epsilon}}  \\
		&= \delta.
	\end{split}
\end{equation}
As a consequence it follows from \eqref{EquationExpansionAgain} that $d\Phi(\dot{\gamma}(0))\leq C\sqrt{\delta}$ for some $C>0$. We choose 
\begin{equation}
	\dot{\gamma}(0)= \frac{\nabla^{\hat{g}}\Phi}{|\nabla^{\hat{g}}\Phi|}. 
\end{equation}
At a point $(p_-, \rho_0)=(p_-,h(p_-))$ we get that
\begin{equation}
	\begin{split}
		C\sqrt{\delta} &\geq \frac{\langle \partial_\rho, \nabla^{\hat{g}}\Phi \rangle}{| \nabla^{\hat{g}}\Phi  |} \\
		&=\frac{\langle \partial_\rho,  \partial_\rho - \langle\vec{n}, \partial_\rho\rangle \vec{n} \rangle}{ \sqrt{1- \langle \vec{n}, \partial_\rho \rangle^2}} \\
		&= \sqrt{1- \langle \vec{n}, \partial_\rho \rangle^2},
	\end{split}
\end{equation}
where $\nabla^{\hat{g}}\Phi$ is the vector $\hat{g}$-dual to $d\Phi$ and 
\begin{equation}
	\vec{n} = \frac{\partial_\rho - \nabla^{g_\rho}h}{\sqrt{1+|dh|^2_{g_\rho}}}
\end{equation}
is the upward pointing unit normal to $\hat{M}^n$. From this it follows that
\begin{equation}
	1 - \frac{1}{1+ |dh|^2_{\hat{g}_\rho}} \leq C^2\delta,
\end{equation}
so that $|dh|^2_{\hat{g}_\rho}=\Ol(\delta)$. It now follows from the uniform equivalence of $\hat{g}^\rho$ and $\delta$ in Proposition \ref{PropositionFermiCoord} that $|dh|^2_{\delta} = \Ol(r^{-(n-1-\epsilon)})$.
\\ \indent Finally we show the asserted decay of $|\Hess^\delta (h)|_\delta$. The following argument constitutes the proof of Lemma 5.2 in \cite{SakovichPMTah}. Let $p\in \hat{M}^n$ be such that $r(p)$ is close to the infinity and $\Theta$ be the biggest eigenvalue of $\hat{g}$. Let $\vec{X}=X^ie_i$ be an eigenvector to $\hat{g}$ with eigenvalue $\Theta$, where $e_i=\partial_i + h_{,i}\partial_\rho$ and $(X^1)^2 + \ldots + (X^n)^2=1$. For $\vec{Y}=X^i\partial_i \in T\hat{M}_\rho^n$ we then have $|\vec{Y}|_\delta=1$. We have
\begin{equation}
	\begin{split}
		\Theta &= \hat{g}(\vec{X}, \vec{X}) \\
		&=\langle \vec{X}, \vec{X} \rangle \\
		&=\hat{g}^\rho_{ij}X^iX^j + h_{,i}h_{,j}X^iX^j \\
		&=|\vec{Y}|^2_{\hat{g}^\rho} + (dh(\vec{Y}))^2 \\
		&\leq (1+ |dh|_{\hat{g}^\rho}^2) | \vec{Y}|_{\hat{g}^\rho}^2 \\
		&\leq C(1+ |dh|_\delta^2),
	\end{split}
\end{equation}   
where we in the last line used the uniform equivalence of $\delta$ and $\hat{g}^\rho$ on $\hat{M}^n_\rho$ (see Proposition \ref{PropositionFermiCoord}.) Furthermore, this also yields a lower bound for the lowest eigenvalue $\Theta^{-1}$ of $\hat{g}^{-1}$. 
\\ \indent We now think of the bilinear forms in terms of their matrix representation in the basis $\{e_1, \ldots, e_n\}$. We let $O$ be an orthogonal matrix such that $O \hat{g}^{-1} O^T = D$, where $D$ is a diagonal, matrix and let $\tilde{A} = O \hat{A} O^T$. Then, in terms of matrices, we have
\begin{equation}
	\begin{split}
		|\hat{A}|_{\hat{g}}^2 &= \trace(\hat{g}^{-1} \hat{A} \hat{g}^{-1} \hat{A} ) \\
		&=\trace(  D \tilde{A} D \tilde{A}) \\
		&\geq \Theta^{-2}\trace(\tilde{A}^2) \\
		&= \Theta^{-2}\trace(\hat{A}^2) \\
		&\geq \frac{\trace(\hat{A}^2)}{C(1+ |dh|_\delta^2)^2} \\
		&= \frac{\sum_{ij=1}^n \hat{A}^2_{ij}}{C(1+ |dh|_\delta^2)^2}.
	\end{split}
\end{equation}
In turn, the coordinate expression of the components of $\hat{A}$ in terms of $h$ (see the proof of Lemma \ref{LemmaJangEqHeightFun} for the details):
\begin{equation}
	\hat{A}_{ij}^2 = \frac{ \big(h_{,ij} - (\hat{\Gamma}^\rho)^k_{ij}h_{,k} + \hat{A}_{ij}^\rho + 2(\hat{A}^\rho)^k_ih_{,j}h_{,k}    \big)^2} {1+|dh|^2_{\hat{g}_\rho}},
\end{equation}
and from the inequality $(a+b)^2\geq a^2/2 - b^2$, the fall-off $|d h|_\delta$ and Proposition \ref{PropositionFermiCoord} we obtain the estimate
\begin{equation}
	|h_{,ij}|^2 \leq C(1+|dh|_\delta^2)^3 |\hat{A}|_{\hat{g}}^2
\end{equation}
which implies the final assertion of our claim, in view of Proposition \ref{Proposition2ndffEst}.

\end{proof}

\subsection{The Jang equation in terms of the height function}

We now rewrite the Jang equation in terms of the height function in the Fermi coordinates. For this, we consider the Jang graph $\hat{M}^n$ as the level set $\{F=0\}$, for the function $F(x^1, \ldots, x^n, \rho)=h(x^1, \ldots, x^n) - \rho$. We recall that in the Fermi coordinates the metric $ds^2=dt^2+ g$ on $M^n\times \rn$ takes the form $ds^2=d\rho^2 + \hat{g}_\rho$, where $\hat{g}_\rho$ is $ds^2$ induced on $\hat{M}^n_\rho$, the hypersurface lying at geodesic distance $\rho$ from $\hat{M}^n_0=\hat{M}_-^n$. The Christoffel symbols and the second fundamental form of the $\rho$-level sets $\hat{M}^n_\rho$ will be denoted by $\hat{\Gamma}^\rho$ and $\hat{A}^\rho$ and the $\rho$-index may be suppressed when convenient. We write $i,j,k,\ell$ for the base coordinates in the Fermi coordinate system.
\\ \indent We start by rewriting the Jang equation in terms of $h$.

\begin{lemma}\label{LemmaJangEqHeightFun}
In the Fermi coordinates as described in Section \ref{SectionFermiSetup}, the Jang equation is the following equation for the height function $h$:
\begin{equation}\label{EquationHeightJangEq}
a^{ij}h_{,ij} + b^kh_{,k} = c,
\end{equation}
where
\begin{equation}
	\begin{split}
		a^{ij}&=\frac{\hat{g}^{ij}}{\sqrt{1+|dh|^2_{\hat{g}_\rho}}}, \\
		b^k &= - \frac{\hat{g}^{ij} (\hat{\Gamma}^\rho)^k_{ij}}{\sqrt{1+|dh|^2_{\hat{g}_\rho}}} - 2\hat{g}^{ik}k_{i\rho}, \\
		c&=\hat{g}^{ij}\bigg(  - \frac{(\hat{A}^\rho)_{ij} + 	2\hat{g}_\rho^{k\ell}(\hat{A}^\rho)_{i\ell}h_{,j}h_{,k}}{\sqrt{1+|dh|^2_{\hat{g}_\rho}}} + k_{ij} + h_{,i}h_{,j} k_{\rho \rho}          \bigg).
	\end{split}
\end{equation}
\end{lemma}

\begin{proof}

The Christoffel symbols $\Gamma$ of the metric $ds^2$ are straightforwardly found to be 
\begin{equation}
	\begin{split}
		\Gamma^\rho_{\rho\rho}&=0, \qquad  \Gamma^\rho_{i\rho}=0, \qquad \Gamma^{\rho}_{ij}= - 	\frac{1}{2}\hat{g}^\rho_{ij,\rho} \\
		\Gamma^k_{\rho\rho}&=0 , \qquad  \Gamma^k_{i\rho}=\frac{1}{2}\hat{g}^{k\ell}_\rho\hat{g}_{i\ell, \rho}^\rho \qquad 	\Gamma^k_{ij} = (\hat{\Gamma}^\rho)^{k}_{ij}.
	\end{split}
\end{equation}
With the convention $\langle \partial_\rho, \nabla_{\partial_i}\partial_j\rangle = (\hat{A}^\rho)_{ij}$ on $\hat{M}^n_\rho$ we get $(\hat{A}^\rho)_{ij}=\Gamma^\rho_{ij}$. In particular, this implies 
\begin{equation}
	\begin{split}
		\Gamma^k_{\rho i}
		&= -\hat{g}^{k\ell}_\rho (\hat{A}^\rho)_{i\ell}.
	\end{split}
\end{equation}
In turn, the Hessian components are:
\begin{equation}\label{EquationHessians}
	\begin{split}
		\Hess_{\rho\rho}^{ds^2}(F) 
		&=0, \\
		\Hess_{\rho i}^{ds^2}(F) 
		&=(\hat{A}^\rho)^k_i h_{,k},  \\
		\Hess_{ij}^{ds^2}(F) 
		&=\Hess_{ij}^{\hat{g}_\rho}(h) + (\hat{A}^\rho)_{ij}. 
	\end{split}
\end{equation}
In the Fermi coordinates the vector $-\partial_\rho + \nabla^{\hat{g}_\rho}h$ is normal to $\hat{M}^n$ at the point $(x^1, \ldots, x^n, \rho)$ and the vector $\partial_i + h_{,i}\partial_\rho$ is tangent to $\hat{M}^n$ at the same point. The induced metric on $\hat{M}^n$ has components 
\begin{equation}
	\begin{split}
		\hat{g}_{ij}&= ds^2(\partial_i + h_{,i}\partial_\rho, \partial_i + h_{,i}\partial_\rho    ) \\
		&=\hat{g}^\rho_{ij} + h_{,i}h_{,j},
	\end{split}
\end{equation}
and similarly 
\begin{equation}
	\hat{g}^{ij}= \hat{g}_\rho^{ij} - \frac{h^{,i}h^{,j}}{1+ |dh|^2_{\hat{g}_\rho}},
\end{equation}
where in both cases it is understood that $\rho=h(x^1, \ldots, x^n)$ and the indices are raised and lowered by $\hat{g}^{ij}_\rho$. The components on the second fundamental form can be calculated using the tensor identity $\Hess^{ds^2}(F)= |\nabla^{ds^2} F| \hat{A}$ discussed in the proof of Lemma \ref{LemmaTiltExcess}:
\begin{equation}
	\begin{split}
		|\nabla^{ds^2} F| \hat{A}_{ij}&= \Hess^{ds^2}(F)( \partial_i + h_{,i}\partial_\rho, \partial_j + h_{,j}\partial_\rho     ) \\
		&= \Hess^{ds^2}_{ij}(F)+h_{,i}\Hess^{ds^2}_{j\rho}(F) + h_{,j}\Hess^{ds^2}_{i\rho}(F)+ h_{,i}h_{,j}\Hess^{ds^2}_{\rho \rho}(F) \\
		&=\Hess_{ij}^{\hat{g}_\rho}(h) + (\hat{A}^\rho)_{ij} + h_{,j} (\hat{A}^\rho)^k_i h_{,k} + h_{,i}(\hat{A}^\rho)^k_j h_{,k},
	\end{split}
\end{equation}	
where we used the equalities in \eqref{EquationHessians}. In turn, the mean curvature $H_{\hat{M}^n}$ is
\begin{equation}
	\begin{split}
		H_{\hat{M}^n}&= \hat{g}^{ij}\hat{A}_{ij} \\
		&=\hat{g}^{ij}\frac{  \Hess_{ij}^{\hat{g}_\rho}(h) + (\hat{A}^\rho)_{ij} + 2h_{,j}(\hat{A}^\rho)^k_i h_{,k}  }{ \sqrt{1+ 		|dh|^2_{g_\rho} }}.
	\end{split}
\end{equation}
\\ \indent We calculate the trace-term:
\begin{equation}
	\begin{split}
		\trace_{\hat{g}}(k)&=\hat{g}^{ij}k(  \partial_i + h_{,i}\partial_\rho, \partial_i + h_{,i}\partial_\rho     ) \\
		&=\hat{g}^{ij} (k_{ij} + 2h_{,i}k_{j\rho} + k_{\rho\rho}) , 
	\end{split}
\end{equation}
where we used the symmetry of the inverse metric.
\\ \indent This yields the result.

\end{proof}


In Lemmas \ref{LemmaRhoExpansionMeanCurv} and \ref{LemmaRhoExpansionTrace} we obtain some asymptotic expansions to be used later.

\begin{lemma}\label{LemmaRhoExpansionMeanCurv}
	Let $H_\rho$ be the mean curvature of the hypersurface $\hat{M}^n_\rho$. Then
	\begin{equation}
		\begin{split}
			H_\rho &= H_- + \Ol(r^{-(n+1-\epsilon)} ) ,  \\
		\end{split}
	\end{equation}
\end{lemma}

\begin{proof}
Throughout this proof we abbreviate the Riemann tensor $\text{Riem}^{M^n\times \rn}=\text{Riem}$ for convenience. We Taylor expand $H_\rho$ in the $\rho$-variable and hence need expressions for the first and second derivatives in the $\rho$-variables. The second fundamental form $\hat{A}^\rho$ of $\hat{M}^n_\rho$ satisfies the \emph{Mainardi equation}:
\begin{equation}
-  (\hat{A}^\rho)_{j,\rho}^i + (\hat{A}^\rho)_k^i(\hat{A}^\rho)_j^k = \text{Riem}^i_{\rho \rho j},
\end{equation}
where indices are raised with $\hat{g}_\rho$. Taking the trace yields
\begin{equation}
	H_{\rho,\rho}= \text{Ric}^{M^n\times \rn}_{\rho \rho} + |\hat{A}^\rho|^2_{\hat{g}_\rho}.
\end{equation}
We take the second (coordinate) derivative with respect to $\rho$ and use the ODE that $\hat{A}^\rho$ satisfies:
\begin{equation}
	\begin{split}
		H_{\rho, \rho\rho} &= 2 (\hat{A}^\rho)_{j,\rho}^i (\hat{A}^\rho)_i^j +  \text{Ric}^{M^n\times \rn}_{\rho \rho,\rho} \\
		&=2   (\hat{A}^\rho)_k^i (\hat{A}^\rho)_j^k (\hat{A}^\rho)_i^j -2 \text{Riem}^i_{\rho \rho j}   (\hat{A}^\rho)_i^j +   	\text{Ric}_{\rho \rho,\rho}^{M^n\times \rn} .
	\end{split}
\end{equation}
Since both $\Gamma^k_{\rho \rho}=0$ and $\Gamma^\rho_{\rho \rho}=0$, we have $(\nabla_{ \rho}\text{Ric}^{M^n\times \rn})_{\rho\rho} = \text{Ric}^{M^n\times \rn}_{\rho \rho ,\rho }$. We then get 
\begin{equation}
	| H_{\rho, \rho \rho}| \leq 2|\hat{A}^\rho|_{\hat{g}_\rho}^3 + 2|\hat{A}^\rho|_{\hat{g}_\rho}|R^{M^n\times \rn}|_{\hat{g}_\rho} + |\nabla \text{Ric}^{M^n\times \rn}|_{\hat{g}_\rho},
\end{equation}
where all terms are bounded by Proposition \ref{PropositionFermiCoord} and the assumptions on the initial data.
\\ \indent Note that we have
\begin{equation}\label{EquationRicciAux}
	\begin{split}
		\text{Ric}^{M^n\times \rn}_{tt}	&=0, \\
		\text{Ric}^{M^n\times \rn}_{tj} &= 0, \\
		\text{Ric}_{ij}^{M^n\times \rn} &= \text{Ric}^{M^n}_{ij},
	\end{split}
\end{equation}
where $i,j$ are the coordinates on $M^n$. This gives
\begin{equation}
	\text{Ric}^{M^n\times \rn}(\vec{n}_-, \vec{n}_-) = \text{Ric}^{M^n}_{rr}(\vec{n}_-^r)^2 + 2\text{Ric}^{M^n}_{r\mu}\vec{n}_-^r\vec{n}_-^\mu + \text{Ric}^{M^n}_{\mu\nu}\vec{n}_-^\mu \vec{n}_-^\nu,
\end{equation}
where we had no $\vec{n}^t$-terms by \eqref{EquationRicciAux}. Straightforward calculations, using Lemma \ref{LemmaJangGraphMetric}, yield
\begin{equation}\label{EquationAux11}
	\begin{split}
		\vec{n}_-^r&=\frac{g^{rk}f^-_{,k}}{\sqrt{1+|df_-|^2_g}} \\
		&= r - (n-3)\frac{\alpha}{r^{n-3}} + \Ol(r^{-(n-2-\epsilon)}), \\
	\end{split}
\end{equation}
and similarly
\begin{equation}\label{EquationAux12}
	\begin{split}
		\vec{n}_-^\mu 
		&=\frac{b^{\mu\nu}\alpha_{,\nu}}{r^{n-1}}  - \frac{\textbf{m}^{\mu\nu}\alpha_{,\nu}}{r^{2n-3}} + \Ol( r^{-(2n+2-\epsilon)}).
	\end{split}
\end{equation}
Therefore we get, using Lemma \ref{LemmaWangGeometry},
\begin{equation}
	\begin{split}
		\text{Ric}_{rr}^{M^n} (\vec{n}_-^r)^2 
		&=   -(n-1)\frac{r^2}{1+r^2} + \Ol ( r^{-(n-1-\epsilon)} )   \\
	\end{split}
\end{equation}
together with
\begin{equation}
	\begin{split}
		\text{Ric}_{r\mu}^{M^n}\vec{n}_-^r \vec{n}_-^\mu 
		&= \Ol(r^{-(2n+1)}).
	\end{split}
\end{equation}
and 
\begin{equation}
	\begin{split}
		\text{Ric}_{\mu\nu}^{M^n} \vec{n}_-^\mu \vec{n}_-^\nu 
		&= \Ol( r^{-2n}). 
	\end{split}
\end{equation}
We are now able to assert fall-off rates about the Ricci tensor and the norm of the second fundamental form. We have
\begin{equation}
	\begin{split}
		\text{Ric}^{M^n\times \rn} (\vec{n}_-, \vec{n}_-) + |\hat{A}_-|^2_{\hat{g}_-} &= -(n-1)\frac{r^2}{1+r^2} + 	\Ol(r^{-(n-1-\epsilon)}) \\
		&\qquad + (n-1) + \frac{1}{(1+r^2)^2} + \Ol( r^{-n}) \\
		&=\Ol(r^{-2}),
	\end{split}
\end{equation}
where the decay of the second fundamental form follows from Lemma \ref{LemmaJangGraph2ndFF}. 
\\ \indent With these bounds, we obtain
\begin{equation}
	\begin{split}
		H_\rho &= H_0 + (  H_{\rho,\rho} |_{\rho=0})\rho + \Ol(\rho^2) \\
		&=H_- + \big(\text{Ric}^{M^n\times \rn}(\vec{n}_-, \vec{n}_-) + |\hat{A}_-|^2_{\hat{g}_-} \big)\rho + \Ol( r^{-2(n-1-\epsilon)} 	) \\
		&=H_- +\Ol (  r^{-(n+1-\epsilon)} ) + \Ol( r^{-2(n-1-\epsilon)}) \\
		&=H_- +\Ol (  r^{-(n+1-\epsilon)} ) ,  
	\end{split}
\end{equation}
since on $\hat{M}^n$ we have $\rho= h(x^1, \ldots, x^n)$ and $|h|=\Ol(r^{-(n-1-\epsilon)})$ by Lemma \ref{LemmaHeightFunDecays}. This completes the proof.
\end{proof}

\begin{lemma}\label{LemmaRhoExpansionTrace}
	Let $\trace_{\hat{g}_\rho}(k)$ be the trace of $k$ on the hypersurface $\hat{M}^n_\rho$. Then
	\begin{equation}
		\trace^{\hat{g}_\rho}(k) =\trace^{\hat{g}_-}(k)  + \Ol (  r^{-(n+1-\epsilon)}  )  .  
	\end{equation}
\end{lemma}

\begin{proof}
	We Taylor expand in the $\rho$-variable up to second order as in the proof of Lemma \ref{LemmaRhoExpansionMeanCurv}. We have 
	\begin{equation}
		\trace^{ds^2}(k) = \trace^{\hat{g}_\rho}(k) + k_{\rho\rho}.
	\end{equation}
	With $\nabla_{ \partial \rho}\partial_\rho=0$, we again get $ k_{\rho\rho, \rho} = (\nabla_{\rho}k)_{\rho \rho}$ as with the mean curvature. Hence
	\begin{equation}
		\trace^{\hat{g}_\rho}(k)_{,\rho}= \trace_{ds^2}(k)_{,\rho} -   (\nabla_{\rho}k)_{\rho \rho}
	\end{equation}
	and
	\begin{equation}
		\trace_{\hat{g}_\rho}(k)_{,\rho \rho} = \nabla_{\rho}\nabla_{\rho}\trace_{ds^2}(k) - (\nabla_{\rho}\nabla_{\rho}k)_{\rho \rho}.
	\end{equation}
	It follows that $\trace^{\hat{g}_\rho}(k)_{,\rho \rho}$ is bounded for $\rho \in [0,\rho_0]$ and hence
	\begin{equation}
		\begin{split}
			\trace_{\hat{g}_\rho}(k) &= \trace_{\hat{g}_0}(k) + (\trace_{\hat{g}_\rho}(k)_{,\rho})|_{\rho=0}\rho + \Ol(\rho^2) \\
			&=\trace_{\hat{g}_-}(k) +   (\vec{n}_-(\trace_{ds^2}(k)) - (\nabla_{\vec{n}_-}k)(\vec{n}_-, \vec{n}_-))\rho + \Ol(\rho^2). \\
		\end{split}
	\end{equation}
	To estimate the first term in the $\rho$-coefficient, we observe that the trace term, computed in the product coordinates, is
	\begin{equation}
		\begin{split}
			\trace_{ds^2}(k)&=g^{ij}k_{ij} \\
			&=n + \frac{\trace^{\Omega}(\textbf{p})-\trace^{\Omega}(\textbf{m})}{r^n}  + \Ol(r^{-(n+1)}) \\
		\end{split}
	\end{equation}
	which is an immediate consequence of Definition \ref{DefinitionWangAsymptotics}. It follows that
	\begin{equation}
		\begin{split}
			\vec{n}(\trace_{ds^2}(k))
			&=\vec{n}^t\trace_{ds^2}(k)_{,t} + \vec{n}^r\trace_{ds^2}(k)_{,r} +\vec{n}^\mu\trace_{ds^2}(k)_{,\mu} \\
			&=\Ol(r^{-n}),
		\end{split}
	\end{equation}
	where we used the expansions for $\vec{n}^k$ in \eqref{EquationAux11} and \eqref{EquationAux12}. 
	\\ \indent In order to expand the covariant derivative $(\nabla_{\vec{n}_-}k)(\vec{n}_-, \vec{n}_-)$ we calculate its components in the original coordinates $M^n\times \rn$-coordinates, where we let capital letters $I,J,K,L$ run over the coordinates $t, r$ and $\mu$.  Recall that $k$ has been extended trivially so that $k_{it}=k_{tt}=0$. It follows by inspection that $(\nabla_L k)_{IJ}=0$, if at least one of the indices $I,J,L$ is $t$. We estimate the remaining components using Lemma \ref{LemmaWangGeometry}, omitting details. Differentiating in the $r$-direction we obtain
	\begin{equation}
		\begin{split}
			(\nabla_rk)_{rr} 	&=\Ol ( r^{-(n+3)}  ), \\
			(\nabla_r k)_{r\mu} 
			&= \Ol ( r^{-(n+1)}  ), \\
			(\nabla_r k)_{\mu \nu} 	&=n\frac{  \textbf{m}_{\mu\nu} -   \textbf{p}_{\mu\nu}   }{r^{n-1}} + \Ol ( r^{-(n-1+\epsilon)} ).
		\end{split}
	\end{equation}
	Differentiation in the $\mu$-direction yields
	\begin{equation}
		\begin{split}
			(\nabla_\mu k)_{rr} &=\Ol ( r^{-(n+1)} ), \\
			(\nabla_\mu k)_{r\nu} &=  \frac{\textbf{m}_{\mu\nu} - \textbf{p}_{\mu \nu}}{r^{n-1}} + \Ol ( r^{-(n-\epsilon)} ) , \\
			(\nabla_\mu k)_{\rho \sigma} &=  \Ol (  r^{-(n-3)}   ) .
		\end{split}
	\end{equation}
	Combining these results, we obtain
	\begin{equation}
		\begin{split}
			(\nabla_{\vec{n}_-} k)(\vec{n}_-, \vec{n}_-) &= (\nabla_r k)_{rr}(\vec{n}_-^r)^3+ 2(\nabla_r k)_{r\mu} (\vec{n}_-^r)^2\vec{n}_\mu 	+ (\nabla_r k)_{\mu\nu}(\vec{n}_-^r \vec{n}_-^\mu \vec{n}_-^\nu) \\
			&\qquad + (\nabla_t k)_{rr}(\vec{n}_-^r)^2\vec{n}_-^t + 2(\nabla_t k)_{r\mu} \vec{n}_-^r \vec{n}_-^t \vec{n}_-^\mu + (\nabla_t 	k)_{\mu\nu}(\vec{n}_-^t \vec{n}_-^\mu \vec{n}_-^\nu)  \\ \\
			&=\Ol(r^{-(n-1-\epsilon)}),
		\end{split}
	\end{equation}
	From the estimates of the components of $\vec{n}$ obtained in the proof of Lemma \ref{LemmaRhoExpansionMeanCurv} the it now follows that $(\nabla_{\vec{n}}k)(\vec{n}, \vec{n})=\Ol(r^{-{n+1-\epsilon}})$ and in turn the assertion on $(\nabla_\rho k)_{\rho\rho}$ follows and so also the main assertion.

\end{proof}

With Lemmas \ref{LemmaRhoExpansionMeanCurv} and \ref{LemmaRhoExpansionTrace} and Definition \ref{DefinitionHolderNorm} at hand, we can improve the fall-off properties of $h$ asserted in Lemma \ref{LemmaHeightFunDecays}. For this purpose, we recall the definition of weighted H\"older spaces on asymptotically Euclidean manifolds.

\begin{definition}\label{DefinitionHolderNorm}

Let $\bar{B}$ be a closed ball in $\rn^n$ centered at the origin. For $k\in \zn_{\geq 0}$, $\alpha \in (0,1)$ and $\tau \in \rn$ we define the weighted H\"older space $C^{k,\alpha}_\tau(\rn^n\setminus \bar{B})$ to be the set of functions $f\in C_{loc}^{k,\alpha}(\rn^n\setminus \bar{B})$ with finite weighted H\"older norm:
\begin{equation}
	\begin{split}
		||f||_{C^{k,\alpha}_\tau(\rn^n\setminus \bar{B})} &= \sum_{|I|\leq k} \sup_{x\in \rn^n\setminus \bar{B}} |x|^{|I|+\tau}| f_{,I}(x)| \\
		&\qquad + \sum_{|I|=k} \sup_{x\in \rn^n\setminus \bar{B}}  |x|^{k+\alpha + \tau} \sup_{4|x-y|<|x| } \frac{| f_{,I }(x) - f_{,I }(y)|}{|x-y|^\alpha} <  \infty,
	\end{split}
\end{equation}
where we write $f_{,I}= \partial_I f = \partial_1^{i_1}\ldots \partial_{n}^{i_n}f$ for $I=(i_1, \ldots, i_n)$ and $|I|=i_1+\ldots + i_n$.
\\ \indent This definition generalizes in a standard way to define weighted H\"older spaces on $C^k$-manifolds $(M^n,g)$ that are diffeomorphic to $\rn^n\setminus \bar{B}_R$ outside a compact set $K$ (see \cite{EichmairHuangLeeSchoen} Definitions 1 and 2) and to sections of more general tensor bundles. In what follows, we will write $T=\Ol^{k,\alpha}(r^{-\tau})$ for a tensor $T\in C^{k,\alpha}_\tau(M^n\setminus K)$, where $K$ is a compact set. 

\end{definition}

 We recall that at this stage the base coordinates in the Fermi coordinate system are Cartesian unless otherwise stated. 

\begin{lemma}\label{LemmaHeightFunDecaysII}
	The height function satisfies
	\begin{equation}
		h = \Ol^{2,\alpha}(r^{-(n-1-\epsilon)})
	\end{equation}
	for some $\alpha\in (0,1)$ and $|h_{,ijk} | = \Ol^\alpha(r^{-(n+1-\epsilon)})$.
\end{lemma}

\begin{proof}
	
We apply elliptic theory to the uniformly elliptic equation \eqref{EquationHeightJangEq}. Explicitly, we use Interior Schauder estimates and standard bootstrap procedure and for this we need to estimate the H\"older norms of the coefficients $a^{ij}$ and $b^k$ and $c$. We emphazise that here the coefficients and the inhomogeneous part live on the graph $\hat{M}^n$ so that $\rho=h(x^1, \ldots, x^n)$ and so $a^{ij}, b^k$ and $c$ are functions only of the base coordinates in the sense that $a^{ij}(x^1, \ldots, x^n) = a^{ij}(x^1, \ldots, x^n,h(x^1, \ldots, x^n) )$. Hence, we use the chain rule to distinguish between the pure coordinate derivative $a^{ij}_{,k}$ and the implicit coordinate derivative $a^{ij}(x^1, \ldots, x^n)_{,k}=a^{ij}(x^1, \ldots, x^n,  h(x^1, \ldots, x^n))_{,k} + a^{ij}(x^1, \ldots, x^n, h(x^1, \ldots, x^n))_{,\rho} h_{,k}$ and similarly for $b^k$ and $c$. We will abuse the notation and write $a^{ij}_{,k}=a^{ij}_{,k} + a^{ij}_{,\rho}h_{,k}$ whenever it is clear if $a^{ij}_{,k}$ denotes the total derivative or the partial derivative with respekt to $x^k$. We recall from Lemma \ref{LemmaHeightFunDecays} that at this stage we have $h=\Ol(r^{-(n-1-\epsilon)})$, $h_{,k}=\Ol(r^{-(n-1-\epsilon)/2})$ and $h_{,k\ell} = \Ol(1)$.
\\ \indent It is convenient to estimate the Cartesian coordinate derivatives of $\hat{g}^{ij}_-$, which will be used below. Since the Christoffel symbols of $\delta$ vanish in Cartesian coordinates, we have $|\hat{g}^{ij}_{-,k}|^2 \leq |\nabla^\delta \hat{g}^{-1}_{-}|^2_\delta$. Using Lemma \ref{LemmaJangGraphMetricInverseDerivative} we obtain $|\nabla^\delta \hat{g}_-^{-1}|^2_\delta = \Ol( r^{-2(n-1)} )$, where $\nabla^\delta$ denotes covariant differentiation with respect to $\delta$, so that $\hat{g}^{ij}_{-,k}=\Ol(r^{-(n-1)})$.
\\ \indent We now compute the derivative of $a^{ij}$. To begin with, we note that
\begin{equation}
	\begin{split}
		\hat{g}^{ij} &= \bigg(\hat{g}_\rho^{ij} - \frac{h^{,i}h^{,j}}{1+|dh|^2_{\hat{g}_\rho}} \bigg) \\
		&= \hat{g}_\rho^{ij} + \Ol (   r^{-(n-1-\epsilon)} ),
	\end{split}
\end{equation}
where we used Lemma \ref{LemmaHeightFunDecays} and the uniform equivalence of $\hat{g}_\rho$ with $\delta$ (see Proposition \ref{PropositionFermiCoord}). Consequently
\begin{equation}
	\begin{split}
		\hat{g}^{ij}_{,k} &= \hat{g}^{ij}_{,k} + \hat{g}^{ij}_{,\rho} h_{,k} \\
		&= \hat{g}_{\rho,k}^{ij}  + \Ol(r^{-(n-1-\epsilon)/2}) \\
		&= \hat{g}_{-, k}^{ij} + \Ol(r^{-(n-1-\epsilon)/2})
	\end{split}
\end{equation}
and so it follows that $a^{ij}_{,k}=\Ol(r^{-(n-1-\epsilon)/2})$.
\\ \indent In a similar way, estimating $b^k$ we get the estimate 
\begin{equation}
	||a^{ij}||_{C^{0,\alpha}(B_2(p))} + ||b^k||_{C^{0,\alpha}(B_2(p))} \leq \Lambda,
\end{equation}
We now improve the decay of $|dh|_\delta$. It is convenient to note that 
\begin{equation}
	\begin{split}
		c&= - \frac{H_\rho}{\sqrt{1+|dh|^2_{\hat{g}^\rho}}} + \trace_{\hat{g}^\rho}(k) + 3 \frac{\langle dh \otimes dh , \hat{A}^\rho \rangle_{\hat{g}^\rho}}{ (1+|dh|^2_{\hat{g}^\rho})^{\frac{3}{2}}} \\
		&\qquad - \frac{ \langle dh \otimes dh, k \rangle_{\hat{g}^\rho}}{1+ |dh|^2_{\hat{g}^\rho}} - k_{\rho\rho} \frac{|dh|^2_{\hat{g}^\rho}}{1+|dh|^2_{\hat{g}^\rho}},
	\end{split}
\end{equation}
which, combined with Lemma \ref{LemmaRhoExpansionMeanCurv}, Proposition \ref{PropositionFermiCoord} and Lemma \ref{LemmaHeightFunDecays} implies $c= - H_\rho + \trace_{\hat{g}_\rho}(k) +  \Ol (   r^{-(n-1-\epsilon)} )$. Moreover, from the estimates in Lemmas \ref{LemmaRhoExpansionMeanCurv} and \ref{LemmaRhoExpansionTrace} we have
\begin{equation}\label{EquationCest}
	\begin{split}
		c &= -\big(   H_- + \Ol (   r^{-(n +1-\epsilon)} )    \big) + \big(  \trace_{\hat{g}_-}(k)  + \Ol (   r^{-{(n+1-\epsilon)}} ) \big)+ \Ol (   r^{-(n-1-\epsilon)} ) \\
		&=   - \big( H_- -  \trace_{\hat{g}_-}(k) \big)  + \Ol (   r^{-(n-1-\epsilon)} )     \\
		&=  \Ol (   r^{-(n-1-\epsilon)} ),  
	\end{split}
\end{equation}
where we used Lemmas \ref{LemmaJangGraphMetric} and \ref{LemmaJangGraph2ndFF} in the last line. From Strong $L^p$-regularity and Sobolev inclusions we then have
\begin{equation}
	\begin{split}
		||h||_{C^{1,\alpha} (B_1(p))} &\leq C ||h||_{W^{2,q}(B_2(p))} \\
		& \leq C \big( ||h||_{L^q(B_3(p))}      + ||c||_{L^q(B_3(p))}   \big) \\
		&\leq C \big( ||h||_{C^0(B_3(p))}      + ||c||_{C^0(B_3(p))}   \big)\\
		&= \Ol(r^{-(n-1-\epsilon)}),
	\end{split}
\end{equation}
where $q$ was chosen large enough so that $2>n/q$ for the Sobolev inclusion, the estimate \eqref{EquationCest} was used and the constant $C$ may change line by line but remains independent of $n$. In particular, this gives an improved estimate $|d h|_\delta = \Ol (   r^{-(n-1-\epsilon)} )$ which, in turn, gives $c=\Ol(r^{-(n+1-\epsilon)})$ by reworking the above argument and, similarly, we obtain $a^{ij}_{,k} = \Ol(r^{-(n-1-\epsilon)})$.
\\ \indent We estimate the H\"older norm of $c$ by Taylor expanding $c_{,\ell}$ in the $\rho$-variable. From the proof of Lemma \ref{LemmaRhoExpansionMeanCurv} we know the linear term of the $\rho$-expansion of $H_\rho$ and it is immediate from the Lemmas in Section \ref{SectionJangGraph} that we have 
\begin{equation}
	\begin{split}
		\text{Ric}^{M^n\times \rn}(\vec{n}_-, \vec{n}_-)_{,r} + (|\hat{A}_-|^2_{\hat{g}_-})_{,r} &= -\frac{(n-1)}{r^3} + \Ol(r^{-5}), \\
		\text{Ric}^{M^n\times \rn}(\vec{n}_-, \vec{n}_-)_{,\mu} + (|\hat{A}_-|^2_{\hat{g}_-})_{,\mu} &= \Ol(r^{-4}), \\
	\end{split}
\end{equation}
and so
\begin{equation}
	\begin{split}
		\big|d\big( \text{Ric}^{M^n\times \rn}(\vec{n}_-, \vec{n}_-)  + |\hat{A}_-|^2_{\hat{g}_-}  \big) \big|_\delta &=    \Ol(r^{-3}).
	\end{split}
\end{equation}
Furthermore, from the proof of Lemma \ref{LemmaRhoExpansionTrace} we know the first term of the $\rho$-expansion of $\trace_{\hat{g}_\rho}(k)$ and it is immediate from the Lemmas in Section \ref{SectionJangGraph} that
\begin{equation}
	\begin{split}
		\trace_{\hat{g}_-}(k)_{,r} - (\nabla_{\vec{n}_-}k)(\vec{n}_-,\vec{n}_-)_{, r} &= \Ol(r^{-(n+2-\epsilon)}), \\
		\trace_{\hat{g}_-}(k)_{,\mu} - (\nabla_{\vec{n}_-}k)(\vec{n}_-,\vec{n}_-)_{, \mu} &= \Ol(r^{-(n+1-\epsilon)}),
	\end{split}
\end{equation}
so that 
\begin{equation}
	\begin{split}
		\big|  d\big( \trace_{\hat{g}_-}(k) - (\nabla_{\vec{n}_-}k)(\vec{n}_-,\vec{n}_-)  \big) \big|_\delta &= \Ol(r^{-(n+1-\epsilon)}).
	\end{split}
\end{equation}
Differentiating $c$ with respect to the tangential variables and Taylor expanding gives
\begin{equation}
	\begin{split}
		c_{,\ell} &= - ( H_\rho - \trace_{\hat{g}_\rho}(k))_{,\ell} +  \Ol ( r^{-2(n-1- \epsilon)}) \\
		&= - \big( H_- - \trace_{\hat{g}_-}(k) \big)_{,\ell}+ \big( \text{Ric}^{M^n\times \rn}(\vec{n}_-, \vec{n}_-) + 	|\hat{A}_-|^2_{\hat{g}_-}    \big)_{,\ell}\rho \\
		&\qquad + \big( \nabla_{\vec{n}_-}(\trace_{\hat{g}_-}(k)) - (\nabla_{\vec{n}_-}k)_{\vec{n}_-,\vec{n}_-}    \big)_{,\ell}\rho +   	\Ol ( r^{-2(n-1- \epsilon)}) \\
		&=\Ol ( r^{-(n+2-\epsilon)})
	\end{split}
\end{equation}
where we considered the tangential derivatives of $c$ from above to conclude the first equality and used that $f_-$ is an approximate Jang solution so that $\J(f_-)=\Ol(r^{-(n+1-\epsilon)})$. Differentiation with respect to the $\rho$-coordinate yields
\begin{equation}
	\begin{split}
		c_{,\rho} &= - \big( H_\rho - \trace_{\hat{g}_\rho}(k) \big)_{,\rho} + \Ol ( r^{-2(n-1- \epsilon)}) \\
		&=- \big( \text{Ric}^{M^n\times \rn}(\vec{n}_-, \vec{n}_-) + |\hat{A}_-|^2_{\hat{g}_-}    \big) \\
		&\qquad + \big(   \nabla_{\vec{n}_-}(\trace_{\hat{g}_-}(k)) - (\nabla_{\vec{n}_-}k)(\vec{n}_-,\vec{n}_-)    \big) + \Ol ( 	r^{-2(n-1- \epsilon)}) \\
		&=\Ol(r^{-2}),
	\end{split}
\end{equation}
where we used the expansion in $\rho$ from Lemmas \ref{LemmaRhoExpansionMeanCurv} and \ref{LemmaRhoExpansionTrace}. It follows that $c(x^1, \ldots, x^n, h(x^1, \ldots, x^n))_{,\ell} = \Ol ( r^{-(n+1-\epsilon)})$, which in turn gives the bound $||c||_{C^{0, \alpha}(B_1(p))}= \Ol(r^{-(n+1-\epsilon)})$. In turn, applying Schauder estimates yields 
\begin{equation}
	\begin{split}
		||h||_{C^{2,\alpha}(B_1(p))} &\leq C \big(||h||_{C^{0 }(B_2(p))} + ||c||_{C^{0,\alpha}B_2(p)} \big) \\
		&= \Ol(r^{-(n-1-\epsilon)}). 
	\end{split}
\end{equation}
In particular, we note that $h_{,ij}= \Ol(r^{-(n-1-\epsilon)})$ and so $|h|+|dh|_\delta + |\Hess^\delta (h)|_\delta = \Ol(r^{-(n-1-\epsilon)})$. 
\\ \indent Since $|c_{,k\ell}|=\Ol(1)$ we obtain 
\begin{equation}
	\begin{split}
		\frac{|c_{,\ell}(x) - c_{,\ell}(y)|}{|x-y|^\alpha} &= \bigg(   \frac{|c_{,\ell}(x) - c_{,\ell}(y) | }{|x-y| } \bigg)^\alpha |c_{,\ell}(x) - c_{,\ell}(y) |^{(1-\alpha)} \\
		&= \Ol(r^{-(1-\alpha)(n-1-\epsilon)}).
	\end{split}
\end{equation}
In turn, it follows that $||c||_{C^{1,\alpha} (B_2(p) )}=\Ol(r^{-(1-\alpha)(n+1-\epsilon)  }) $, so that we may apply Schauder estimates to obtain
\begin{equation}\label{EquationFirstEstimate}
	\begin{split}
		||h||_{C^{3,\alpha} ( B_1(p) )} &\leq C\big( ||h||_{C^{1,\alpha} ( B_2(p) )} + ||c||_{C^{1,\alpha} (B_2(p) )}      \big) \\
		&=\Ol(r^{-(1-\alpha)(n+1-\epsilon)  }).
	\end{split}
\end{equation}
In particular we note that $h_{,ijk} = \Ol(r^{-(1-\alpha)(n+1-\epsilon) })$.  
\\ \indent Next we improve the $C^{2,\alpha}$-estimate using rescaling to decay rates where the decay increases one order per derivative. We let $p_0\in \hat{M}^n_-$ be close to the infinity and we write $\Psi(p_0)=x_0\in \rn^n\setminus \bar{B}_R(0)$ where $x_0=(x_0^1, \ldots, x_0^n)$. Let $r_0=r(x_0)$ and 
\begin{equation}
	\tilde{x} = \frac{x - x_0}{\sigma},
\end{equation}
where $\sigma=r_0/2$. From the chain rule we get that the Jang equation in terms of $h$ as in Lemma \ref{LemmaJangEqHeightFun} and these new coordinates is
\begin{equation}\label{EquationRescaled}
	a^{\tilde{i}\tilde{j}}h_{,\tilde{i}\tilde{j}} + \sigma b^{\tilde{k}} h_{,\tilde{k}} = \sigma^2c,
\end{equation}
where we used the sub-indices $\tilde{k}, \tilde{i}, \tilde{j}$ to denote the partial derivatives in the rescaled coordinates. Furthermore, we let
\begin{equation}
	\tilde{U}_r = \{|\tilde{x}|<r\},
\end{equation}		
where $r>0$. In particular, if $\tilde{x} \in \tilde{U}_1$ then $\frac{1}{2}r_0\leq r(\tilde{x}) \leq \frac{3}{2}r_0$.
\\ \indent We want to apply Schauder estimates as above to the rescaled \eqref{EquationRescaled} in order to get the better decay. Hence, we need to verify the structure conditions and estimate the H\"older norm of $c$. The coefficient matrix $a^{ij}$ is again positive and so we need only to estimate the H\"older norms of $a^{ij}$ and $b^k$.
\\ \indent We start with the H\"older estimate on $a^{\tilde{i}\tilde{j}}$ using the chain rule on $a^{\tilde{i}\tilde{j}}=a^{\tilde{i}\tilde{j}}(\tilde{x}, h(\tilde{x}))$ and keeping the stronger decay of $|dh|_\delta$ in mind:
\begin{equation}
	\begin{split}
		a^{\tilde{i}\tilde{j}}_{,\tilde{k}}  &= a^{\tilde{i}\tilde{j}}_{,\ell}x_{,\tilde{k}}^\ell + 	a^{\tilde{i}\tilde{j}}_{,\rho}h_{,\ell}x_{,\tilde{k}}^\ell \\
		&=(a^{\tilde{i}\tilde{j}}_{,k}  + a^{\tilde{i}\tilde{j}}_{,\rho}h_{,k})\sigma \\
		&=\Ol(r^{-(n-2-\epsilon)}),
	\end{split}
\end{equation}
which in turn yields the estimate $\max_{\tilde{U}_1} |a^{\tilde{i} \tilde{j}}_{,\tilde{k}}| = \Ol(r_0^{-(n-2-\epsilon)})$, where we used the equivalence of $r$ and $r_0$ noted above. We find the following estimate for the H\"older coefficient:
\begin{equation}
	\begin{split}
		\frac{ |a^{\tilde{i}\tilde{j}} (\tilde{x}) - a^{\tilde{i}\tilde{j}}(\tilde{y})|}{|\tilde{x}- \tilde{y}|^\alpha} &= 	\frac{ 	|a^{\tilde{i}\tilde{j}} (\tilde{x}) - a^{\tilde{i}\tilde{j}} (\tilde{y})|^\alpha}{|\tilde{x}- \tilde{y}|^\alpha} | a^{\tilde{i}\tilde{j}} (\tilde{x}) - a^{\tilde{i}\tilde{j}} (\tilde{y})|^{1-\alpha} \\
		& =   \Ol( r^{-\alpha(n-2-\epsilon)} ) ,
	\end{split}
\end{equation}
where finally we used that $a^{ij}$ is asymptotically bounded.  
\\ \indent We now consider the term $b^k$ and it is convenient to split into two parts:  
\begin{equation}
	b_1^k = -2 \hat{g}^{kj}k_{\rho j}, \qquad \text{and} \qquad b_2^k = -a^{ij}\hat{\Gamma}^k_{ij}.
\end{equation}
First we consider first $b_1^k$ and so we need to estimate both $k_{\rho j}$ and its derivative $k_{\rho j, \ell}$. The covariant derivative $(\nabla dt)$ of $dt$ vanishes identically; in the $M^n\times \rn$-coordinates we have already discussed that any Christoffel symbol containing at least one index $t$ must vanish, and hence
\begin{equation}
	\begin{split}
		(\nabla_L dt)_N = dt_{,L}  - \Gamma_{LN}^t dt_t - \Gamma_{LN}^k dt_k =0,
	\end{split}
\end{equation}
where $L,N$ denote any index $t,r,\mu$ on $M^n\times \rn$. Returning to the Fermi coordinate system we observe that 
\begin{equation}
	dt_{\rho,\rho} = (\nabla_\rho dt)_\rho + \Gamma_{\rho \rho}^\rho dt_\rho + \Gamma_{\rho \rho}^\ell dt_\ell = 0,
\end{equation}
since $\nabla_{\partial_\rho}\partial_\rho=0$. Hence $dt_\rho$ is constant along the $\rho$-coordinate and so $dt_\rho(x,\rho) = (dt_\rho)(x,0)= dt(\vec{n}_-)(x)$. We have already seen
\begin{equation}
	\begin{split}
		dt(\vec{n}_-)&= \vec{n}^t_- \\
		&=\frac{1}{\sqrt{1+r^2}} + \Ol ( r^{-2n} )
	\end{split}
\end{equation}
and moreover $dt(\partial_j)\leq |dt|_{ds^2}|\partial_j|_{ds^2} =1$ from the Cauchy-Schwarz inequality. It follows straightforwardly from Definition \ref{DefinitionWangAsymptotics} that $|k-g|_{ds^2}^2=|k-g|^2_g = \Ol(r^{-n})$ so that $(k-g)_{\rho j}^2=\Ol(r^{-n})$ follows from the uniform equivalence of $\hat{g}_\rho$ with $\delta$ and so we obtain 
\begin{equation}
	\begin{split}
		k_{\rho j } &= ds^2_{\rho j } - dt_{\rho}dt_j + (k-g)_{\rho j } \\
		&=-dt(\vec{n}_-) dt_j + (k-g)_{\rho j }  \\
		&= \Ol (r^{-1} ). 
	\end{split}
\end{equation}
From this it follows that $b_1^k = \Ol(r_0^{-1})$ from which we conclude that $\max_{\tilde{U}_1} |b^k_1| = \Ol(r^{-1}_0)$.  
\\ \indent Next we estimate the derivative $b^k_{1,\ell}$. We calculate the norm of $\nabla k$ in order to estimate the tangential derivatives. It follows from the proof of Lemma \ref{LemmaRhoExpansionTrace} that $|\nabla k |^2_g=\Ol(r^{-n})$ and from the Christoffel symbols calculated in the beginning of the proof of Lemma \ref{LemmaJangEqHeightFun} we note that $(\hat{\Gamma}^\rho)^\ell_{j\rho} = - (\hat{A}^\rho)_{j}^\ell$ for below convenience.
\\ \indent We estimate the Christoffel symbols associated to $\hat{g}^-$ in the Cartesian coordinates and for convenience we write $\hat{g}_{ij}^-=\delta_{ij}+b_{ij}$ and estimate the decay of $b_{ij,k}$. The estimate is done as for $\hat{g}_-^{-1}$ above; we have $ |\hat{g}_{ij,k}^-|^2 \leq|\nabla g^-|^2_\delta$, where $\nabla$ is the Levi-Civita connection associated to $\delta$. In turn, the components of the $(0,3)$-tensor $\nabla \hat{g}^-$ are calculated in the polar coordinate system in Lemma \ref{LemmaJangGraphMetricDerivative} and so the norm of $\nabla \hat{g}^-$ is estimated as $|\nabla \hat{g}^-|^2_\delta = \Ol (  r^{-2( n-1)} )$. It follows that $\hat{g}^-_{ij,k}= b_{ij,k}=\Ol(r^{-(n-1)})$. Together with the boundedness of the derivatives of $\hat{g}_\rho$ and $\hat{g}^\rho$ from Proposition \ref{PropositionFermiCoord}, this implies $\hat{\Gamma}_- =\Ol(r^{-(n-1)})$ so that
\begin{equation}
	\begin{split}
		\hat{\Gamma}_\rho &= \hat{\Gamma}_0 + (\hat{\Gamma}_\rho)_{,\rho}|_{\rho=0}\rho + \Ol(\rho^2) \\
		&= \hat{\Gamma}_- + \Ol (  r^{-(n-1-\epsilon)} ) \\
		&= \Ol (  r^{-(n-1-\epsilon)} ),
	\end{split}
\end{equation} 
where we omitted the coordinate indices for convenience. In turn, we may estimat the $\rho$-derivative of $k_{\rho j}$:
\begin{equation}
	\begin{split}
		k_{\rho j, \rho } &= (\nabla_\rho k)_{\rho j} + (\hat{\Gamma}^\rho)^L_{\rho \rho}k_{L j} + (\hat{\Gamma}^\rho)^L_{\rho j}k_{L 	\rho} \\
		&=(\nabla_\rho k)_{\rho j} + (\hat{\Gamma}^\rho)_{\rho j}^\ell k_{\rho \ell} \\
		&=  (\hat{\Gamma}_-)_{\rho j}^\ell k_{\rho \ell} + \Ol(r^{-n/2}) \\
		&=- (\hat{A}^-)_{j}^\ell  k_{\rho \ell}  + \Ol(r^{-n/2}) \\
		&=\Ol(r^{-1}),
	\end{split}
\end{equation}
where $L$ runs over indices $k,\ell, i,j$ and $\rho$ and we used the calculations above for $k_{\rho j}$ and Cauchy-Schwarz for the covariant derivative. 
\\ \indent In order to estimate the tangential (to $\hat{M}^n_{\rho}$) coordinate derivative $k_{\rho j, \ell}$ we first need to estimate $dt_{j, \ell}$ and $dt(\vec{n}_-)_{,\ell} $. There is firstly
\begin{equation}
	\begin{split}
		dt_{j, \ell} &= (\nabla_\ell dt)_j - (\hat{\Gamma}^\rho)^k_{\ell j}dt_k \\
		&=-(\hat{\Gamma}_-)^k_{\ell j}dt_k + \Ol ( r^{-(n-1-\epsilon)}  ) \\
		&= \Ol ( r^{-(n-1-\epsilon)}  ). \\
	\end{split}
\end{equation}
To estimate the norm of $dt_{\vec{n}_-,\ell} $ we calculate the $\delta$-norm of the differential in polar coordinates:
\begin{equation}
	\begin{split}
		|d(dt  (\vec{n}_- )) |_\delta^2 &= \delta^{rr}(dt_{\vec{n}_-})_{,r}^2 + \delta^{\mu\nu} 	(dt  (\vec{n}_- ))_{,\mu}(dt  (\vec{n}_- ))_{,\nu} \\
		&=  \bigg( -\frac{r}{(1+r^2)^{3/2}} + \Ol ( r^{-(2n+1)} )     \bigg)^2 +   \Ol ( r^{-2(2n+1)}  ) \\
		&=\frac{r^2}{(1+r^2)^3} + \Ol ( r^{-(2n+3)}  )
	\end{split}
\end{equation}
so that $dt_{\vec{n}_-,\ell} =\Ol(r^{-2})$. Hence,
\begin{equation}
	\begin{split}
		k_{\rho j, \ell}&=   -dt_{\vec{n}_-, \ell} dt_j -dt_{\vec{n}_-} dt_{j, \ell} + (k-g)_{\rho j, \ell}      \\
		&=-dt(\vec{n}_- )_{ ,\ell} dt_j + (\hat{\Gamma}^\rho)^k_{\ell j}dt(\vec{n}_- )dt_k + (\nabla_{ \ell}(k-g))_{\rho j} \\
		&\qquad + (\hat{\Gamma}^\rho)_{\ell \rho}^k(k-g)_{k j} + (\hat{\Gamma}^\rho)_{\ell j}^k(k-g)_{\rho k}  \\
		&=\Ol(r^{-2}),
	\end{split}
\end{equation}
by previous estimates. Applying the chain rule as we did for $a^{ij}$, we obtain
\begin{equation}
	\begin{split}
		k_{\rho j , \tilde{k}} &=( k_{\rho j,k} + k_{\rho j,\rho} h_{,k})\sigma  \\
		&\leq \frac{C}{r^2}r_0 \\
		&\leq \frac{C'}{r_0}.
	\end{split}
\end{equation}
This gives us a H\"older bound on $k_{\rho j}$ and together with the estimate on $a^{ij}$ obtained above we get the an estimate on $b_1^k$:
\begin{equation}
	\sup_{\tilde{U}_1}|b^k_{1,\tilde{k}}|= \Ol( r_0^{-1}).
\end{equation}
Hence
\begin{equation}
	\begin{split}
		\sigma \frac{|b_1^k(\tilde{x}) - b_1^k(\tilde{y})|}{|\tilde{x}-\tilde{y}|^\alpha} &=\sigma \frac{|b_1^k(\tilde{x}) - 	b_1^k(\tilde{y})|^\alpha}{|\tilde{x}-\tilde{y}|^\alpha}|b_1^k(\tilde{x}) - b_1^k(\tilde{y})|^{1-\alpha} \\
		&\leq Cr_0r_0^{-\alpha} r_0^{-(1-\alpha)} \\
		&= C,
	\end{split}
\end{equation}
so that $b_1^k$ is H\"older bounded. In particular, we obtain $||b_1^k||_{C^{0,\alpha}(\tilde{U}_1)}=\Ol(1)$.
\\ \indent To estimate $b_2^k$ we may again use the estimates on $a^{ij}$ and the Christoffel symbols obtained above to assert that 
\begin{equation}
	b_2^k = \Ol ( r_0^{-(n-1)} ).
\end{equation}
To find its derivative in the $\tilde{x}^k$-direction, we have estimate the Christoffel symbols via the estimates on $\hat{g}^-$ and its derivatives. In particular, we need the second order covariant derivative of $\hat{g}^-$. The components of the covariant derivative of a $(0,3)$ tensor $T$ is
\begin{equation}
	(\nabla_\ell  T)_{ijk} = T_{ijk,\ell} - \Gamma^m_{\ell i}T_{mjk} - \Gamma^m_{\ell j}T_{imk} - \Gamma^m_{\ell k}T_{ijm}, 
\end{equation}
so that 
\begin{equation}
	(\nabla_n \nabla_k\hat{g}^-)_{ij} =   (\nabla_k \hat{g}^-)_{ij,n} - \Gamma^m_{nk}(\nabla_m \hat{g}^-)_{ij} - \Gamma^m_{ni}(\nabla_k \hat{g}^-)_{m j} - \Gamma^m_{nj}(\nabla_k \hat{g}^-)_{i m}.
\end{equation}
Here, both $\nabla$ and $\Gamma$ are associated to the Euclidean metric $\delta$. We calculate below the $2\cdot  2\cdot 3=12$ components. Components with two derivations in the $r$-direction are:
\begin{equation}
	\begin{split}
		(\nabla_r \nabla_r\hat{g}^-)_{rr} &=-2(n-1)(n-2)(n-3)\frac{\alpha}{r^{n}} + \Ol ( r^{-(n+1-\epsilon)}), \\
		(\nabla_r \nabla_r\hat{g}^-)_{r\mu} 	&=(n-2)(n-3)\frac{\alpha}{r^{n-1}}   + \Ol ( r^{-(n-\epsilon)}) , \\
		(\nabla_r \nabla_r\hat{g}^-)_{\mu\nu} 	&= \Ol ( r^{-(n-\epsilon)}).
	\end{split}
\end{equation}
For the mixed second derivative with the first derivation in a tangential direction, we have the following components:
\begin{equation}
	\begin{split}
		(\nabla_r \nabla_\rho\hat{g}^-)_{rr} &= 3(n-2)^2\frac{\alpha}{r^{n-1}} + \Ol ( r^{-(n-\epsilon)}), \\ 
		(\nabla_r \nabla_\rho\hat{g}^-)_{r\mu}  	&=  \Ol ( r^{-(n-2)}), \\
		(\nabla_r \nabla_\rho\hat{g}^-)_{\mu\nu} 	&=  \Ol ( r^{-(n-3)}) 
	\end{split}
\end{equation}
For the mixed second derivative with the first derivation in the radial direction, we have the following components:
\begin{equation}
	\begin{split}
		(\nabla_\rho \nabla_r\hat{g}^-)_{rr} 	&=  \Ol ( r^{-(n-1)}), \\
		(\nabla_\rho \nabla_r\hat{g}^-)_{r\mu}	&=  \Ol ( r^{-(n-2)}), \\
		(\nabla_\rho \nabla_r\hat{g}^-)_{\mu\nu}  	&=  \Ol ( r^{-(n-3)}).
	\end{split}
\end{equation}
Finally, for the second derivative with only tangential directions, we have the following components:
\begin{equation}
	\begin{split}
		(\nabla_\rho \nabla_\sigma\hat{g}^-)_{rr}	&=   \Ol ( r^{-(n-2)}) , \\
	 	(\nabla_\rho \nabla_\sigma\hat{g}^-)_{r\mu} 	&=  \Ol ( r^{-(n-3)}), \\
		(\nabla_\rho \nabla_\sigma\hat{g}^-)_{\mu\nu}  	&=  \Ol ( r^{-(n-4)}).
	\end{split}
\end{equation}
With this we may estimate the norm of $\nabla \nabla \hat{g}^-$:
\begin{equation}
	\begin{split}
		|\nabla \nabla \hat{g}^-|^2_\delta &= \delta^{ij}\delta^{k\ell}\delta^{mn}\delta^{op} (\nabla_i \nabla_k\hat{g}^-)_{m o} (\nabla_j 	\nabla_\ell\hat{g}^-)_{n p} \\
		&=\Ol(r^{-2n})  
	\end{split}
\end{equation}
It follows that $\hat{g}^-_{ij,k\ell}=\Ol(r^{-n})$ in Cartesian coordinates.
\\ \indent From this we may estimate the decay of the coordinate derivatives of the Christoffel symbols. We have $ \hat{g}_-^{ij}=\Ol(1)$, $\hat{g}^-_{ij,k}=\Ol(r^{-(n-1)})$, $(\hat{g}_-^{ij})_{,k}=\Ol(r^{- (n-1)})$ and from the above $\hat{g}^-_{ij,k\ell}=\Ol(r^{-n})$. It follows that 
\begin{equation} 
	\begin{split}
		(\hat{\Gamma}^k_{ij-})_{,\ell}&=\Ol ( r^{-(n-1)}),
	\end{split}
\end{equation} 
so that, by similar arguments as for $b_1^k$, we get
\begin{equation}
	b_{1,\ell}^k= \Ol ( r^{-2(n-1)}).
\end{equation}
From this we get the estimate for the H\"older norm: 
\begin{equation}
	\begin{split}
		\sigma \frac{|b_2^k(\tilde{x}) - b_2^k(\tilde{y})|}{|\tilde{x}- \tilde{y}|^\alpha} &= \sigma \frac{|b_2^k(\tilde{x}) - 	b_2^k(\tilde{y})|^\alpha}{|\tilde{x}- \tilde{y}|^\alpha} |b_2^k(\tilde{x}) - b_2^k(\tilde{y})|^{1-\alpha} \\
		&\leq Cr_0r_0^{- (n-1)\alpha } r_0^{ -(n-1)(1-\alpha)}   \\
		&=Cr_0^{-(n-2 ) },
	\end{split}
\end{equation}
for $\tilde{x}, \tilde{y}\in \tilde{U}_1$. With this we get the H\"older bound
\begin{equation}
	||\sigma b^k_2||_{ C^{0,\alpha}(\tilde{U}_1)}=\Ol (  r_0^{ -(n-2)   } ).
\end{equation}
\\ \indent Finally we obtain the H\"older estimate on $c$. We have already estimated the partial derivatives $c_{,\ell}=\Ol(r^{-(n+1-\epsilon)})$ and $c_{,\rho}=\Ol(r^{-2})$ so that in total $c_{,\ell} = \Ol(r^{-(n+1-\epsilon)})$. It follows that
\begin{equation}
	\begin{split}
		c_{,\tilde{k}}&= c(x, h(x))_{,\ell}  x(\tilde{x})^\ell_{,\tilde{k}} \\
		&=c(x, h(x))_{,k} \sigma \\
		&\leq  Cr^{-(n+1-\epsilon)}r_0 \\
		&\leq Cr_0^{-(n-\epsilon)}.
	\end{split}
\end{equation}
From this we may estimate
\begin{equation}
	\begin{split}
		\sigma^2\frac{|c(\tilde{x})- c(\tilde{y})|}{|\tilde{x}-\tilde{y}|^\alpha} &= \sigma^2\frac{|c(\tilde{x})- 	c(\tilde{y})|^\alpha}{|\tilde{x}-\tilde{y}|^\alpha}|c(\tilde{x})- c(\tilde{y})|^{1-\alpha} \\
		&\leq Cr_0^2r^{-(n-\epsilon)\alpha}r_0^{-(n+1-\epsilon)(1-\alpha)} \\
		&=C'r_0^{-(n-2-\alpha+\epsilon)} 
	\end{split}
\end{equation}
so that $||\sigma^2 c||_{C^{0,\alpha}(\tilde{U}_1)} = \Ol ( r_0^{-( n-1- \alpha- 2\epsilon)})$.
\\ \indent We may now finally assert the decay on $h$. From Lemma \ref{LemmaHeightFunDecaysII} and our obtained H\"older estimate on $\sigma^2c$ we get
\begin{equation}
	\begin{split}
		||h||_{C^{2,\alpha}( \tilde{U}_{1/2})} &= C\big( ||h||_{C^{0}(\tilde{U}_1)} + ||\sigma^2c||_{C^{0,\alpha}(\tilde{U}_1)} \big) \\
		&=\Ol  ( r_0^{-(n-1-\alpha -\epsilon)}  ).
	\end{split}
\end{equation}
from a straightforward application of Interior Schauder Estimates. Unscaling the coordinates gives the desired decay, modulo redefinition of $\epsilon$. 
\\ \indent To obtain H\"older estiates of the third derivatives we first observe that, since $c_{, \tilde{k} \tilde{\ell}}=\Ol(r^{-2}_0)$, we have 
\begin{equation}
	\begin{split}
		\sigma^2    \frac{| c_{,\tilde{k}} (\tilde{x}) - c_{,\tilde{k} } (\tilde{y})|}{|\tilde{x}- \tilde{y}|^\alpha}   &= \sigma^2    \frac{| c_{,\tilde{k}} (\tilde{x}) - c_{,\tilde{k} } (\tilde{y})|^\alpha}{|\tilde{x}- \tilde{y}|^\alpha} | c_{,\tilde{k}} (\tilde{x}) - c_{,\tilde{k}}|^{1-\alpha} \\
		&\leq C r^2_0 r^{-(n-\epsilon')} \\
		&= Cr^{-(n-2-\epsilon')},
	\end{split}
\end{equation}
where the constant changes in the last line and $\epsilon'$ is small. Applying Schauder estimates we obtain
\begin{equation}
	\begin{split}
		||h||_{C^{3,\alpha}(\tilde{U}_{1/2})} &\leq C \big(||h||_{C^0(\tilde{U}_{1/2})} + ||\sigma^2 c||_{C^{1,\alpha}(\tilde{U}_{1/2})} \big) \\
		&= \Ol(r^{-(n-2-\epsilon')}).
	\end{split}
\end{equation}
Recalling that $c_{,ijk} = \sigma^{-3}c_{, \tilde{i}\tilde{j}\tilde{k}}$ we obtain the assertion $c_{,ijk} = \Ol(r^{-(n+2-\epsilon')})$, as asserted.

\end{proof}

With Lemma \ref{LemmaHeightFunDecaysII} at hand, we can show that the induced metric $\hat{g}$ on the Jang graph $\hat{M}^n$ is asymptotically Euclidean in the following sense.

\begin{proposition}\label{PropositionJangGraphIsAE}
	The induced metric $\hat{g}$ on the Jang graph $\hat{M}^n$ in $(M^n\times \rn, g + dt^2)$ satisfies
	\begin{equation}
		\hat{g} = \hat{g}_- + \Ol^{2,\beta} (  r^{-(n-1-\epsilon)}  ), 
	\end{equation}
	where $\hat{g}_-$ is the induced metric on  $\hat{M}^n_-$.
\end{proposition}

\begin{proof}
	
We let $e=\hat{g} - \hat{g}_-$. Estimating as in the proof of Lemma \ref{LemmaHeightFunDecaysII} yields
\begin{equation}
	\begin{split}
		e_{ij} &= (\hat{g} - \hat{g}_-)_{ij} \\
		&= \hat{g}^\rho_{ij} + h_{,i}h_{,j} - \hat{g}^-_{ij} \\
		&=(  \hat{g}^{\rho}_{ij,\rho}|_{\rho=0})\rho  + \Ol(\rho^2) +  h_{,i}h_{,j} \\
		&=\Ol ( r^{-( n-1-  \epsilon)}),
	\end{split}
\end{equation}
from Proposition \ref{PropositionFermiCoord} and Lemma \ref{LemmaHeightFunDecaysII}.  
\\ \indent Here it is convenient to estimate the components and coordinate derivatives in Cartesian coordinates of $\hat{A}^-$. Computations using results from Section \ref{SectionJangGraph} yield $|\hat{A}^-|^2_{\delta} = (n-1) + \Ol(r^{-4})$, $|\nabla^\delta \hat{A}^-|^2_\delta = \Ol(r^{-2})$ and $|\nabla^\delta \nabla^\delta \hat{A}^-|^2_\delta = \Ol(r^{-4})$. Arguing as in the proof of Lemma \ref{LemmaHeightFunDecaysII} we obtain $\hat{A}^-_{ij} = \Ol(1)$, $\hat{A}^-_{ij,\ell} = \Ol(r^{-1})$ and $(\hat{A}^-)_{ij,k\ell} = \Ol(r^{-2})$.
\\ \indent With the estimates at hand we may procede to estimate
\begin{equation}
	\begin{split}
		e_{ij,\ell } &=  (\hat{g}^\rho_{ij} - \hat{g}^-_{ij})_{,\ell}   +  (\hat{g}^\rho_{ij} - \hat{g}^-_{ij})_{,\rho} h_{,\ell} + 	(h_{,i} h_{,j})_{,\ell} \\
		&=   \bigg((  \hat{g}^{\rho}_{ij,\ell \rho}|_{\rho=0})\rho  + \Ol(\rho^2) \bigg)  +  \bigg((  	\hat{g}^{\rho}_{ij,\rho }|_{\rho=0})   + \Ol(\rho ) \bigg)  h_{,\ell} +(h_{,i} h_{,j})_{,\ell}     \\
		&=- 2  (\hat{A}^-)_{ij,\ell} \rho  -2 (\hat{A}^-)_{ij}  h_{,\ell}    +\Ol(r^{-2(n-1-\epsilon)})  \\ 
		&=\Ol(r^{-(n-\epsilon)}),
	\end{split}
\end{equation}
where we used $\hat{g}^\rho_{ij,\rho} = -2 \Gamma^\rho_{ij}=-2  (\hat{A}^\rho)_{ij}$ from the proof of Lemma \ref{LemmaJangEqHeightFun}, together with the estimates for the derivatives of $h$ from Lemma \ref{LemmaHeightFunDecaysII}.
\\ \indent Next, we recall the decay of $h_{,ijk}=\Ol(r^{-(n+1-\epsilon)})$ from Lemma \ref{LemmaHeightFunDecays} so that $(h_{,i} h_{,j})_{,k\ell} = \Ol(r^{-2( n-1-\epsilon)})$ and compute
\begin{equation}
	\begin{split}
		(\hat{g}^\rho_{ij} - \hat{g}^-_{ij})_{,\ell k} &=  \bigg( (\hat{g}^{\rho}_{ij,\rho k \ell}|_{\rho=0}) \rho  + \Ol(\rho^2) \bigg)  \\
		&=\bigg( -2(\hat{A}^-)_{ij,k\ell} \rho + \Ol(r^{-2(n-1-\epsilon)}) \bigg)  \\
		&=\Ol(r^{-(n+1-\epsilon)}),
	\end{split}
\end{equation}
where we have used the fact that $\hat{g}^\rho_{ ,\rho  }$ has three bounded derivatives to motivate the Taylor expansion. Similar computations yield $(\hat{g}^\rho_{ij} - \hat{g}^-_{ij})_{,k\rho} h_{,k}  = \Ol(r^{-(n+1-\epsilon)})$ and $(\hat{g}^\rho_{ij} - \hat{g}^-_{ij})_{,\rho\rho} h_{,\ell} h_{,k} = \Ol ( r^{-2(n-\epsilon)} )$ so that 
	\begin{equation}
		\begin{split}
			e_{ij ,k\ell} &=  \bigg( (\hat{g}^\rho_{ij} - \hat{g}^-_{ij})_{,\ell} \bigg)_{,k}  +  \bigg( (\hat{g}^\rho_{ij} - \hat{g}^-_{ij})_{,\rho} h_{,\ell} \bigg)_{,k} +  (h_{,i} h_{,j})_{,k\ell} \\
			&=    (\hat{g}^\rho_{ij} - \hat{g}^-_{ij})_{,\ell k}   +  (\hat{g}^\rho_{ij} - \hat{g}^-_{ij})_{,\ell\rho} h_{,k} \\
			&\qquad   +    (\hat{g}^\rho_{ij} - \hat{g}^-_{ij})_{, \rho k} h_{,\ell}  +   (\hat{g}^\rho_{ij} - \hat{g}^-_{ij})_{,\rho\rho} 	h_{,\ell} h_{,k} +   (\hat{g}^\rho_{ij} - \hat{g}^-_{ij})_{,\rho} h_{,\ell k} \\
			&\qquad + (h_{,i} h_{,j})_{,k\ell}. \\
			&=\Ol (   r^{-(n+1-\epsilon)}). 
		\end{split}
	\end{equation}
 	These estimates imply $|e|_\delta + r|\nabla e|_\delta + r^2|\nabla \nabla e|_\delta = \Ol(r^{-(n-1-\epsilon)})$, or $e \in \Ol_2(r^{-(n-1-\epsilon)})$. 
\\ \indent In order to estimate the H\"older coefficient we write $e_{ij}= (e_{ij} - h_{,i}h_{,j}) + h_{,i}h_{,j}$. Arguing as in the proof of Lemma \ref{LemmaHeightFunDecaysII} we obtain the H\"older estimate $(e_{ij} - h_{,i}h_{,j})_{,k\ell} =\Ol^{0, \beta}(r^{-(n+1-\epsilon)})$. Similarly $(h_{,i}h_{,j})_{,k\ell}=\Ol^{0,\beta}(r^{-(n+1-\epsilon)})$ also follows from Lemma \ref{LemmaHeightFunDecaysII}.

\end{proof}

As a consequence $\hat{g}$ is asymptotically flat as in Definition \ref{DefinitionAFinitialData}. 

\begin{corollary}\label{CorollaryJangGraphAF}
	
	The Jang graph obtained in in Proposition \ref{PropositionJangLimit} is asymptotically flat as in Definition \ref{DefinitionAFinitialData}. 
	
\end{corollary}

\begin{proof}

At this stage we only know from Proposition \ref{PropositionJangLimit} that 
\begin{equation}
	f = \sqrt{1+r^2} + \frac{\alpha}{r^{n-3}} + q(r, \theta),
\end{equation}
where $q=\Ol(r^{-(n-2-\epsilon)})$. So we need to show that $q=\Ol^3(r^{-(n-2-\epsilon)})$. We show this by comparing metric components as follows. On the one hand, we have
\begin{equation}
	\begin{split}
		\hat{g}_{rr}&= g_{rr} + f_{,r}f_{,r} \\
		&=\frac{1}{1+r^2} + \bigg( \frac{r}{\sqrt{1+r^2}}- (n-3)\frac{\alpha}{r^{n-3}} + q_{,r}(r,\theta)   \bigg)^2 \\
		&=1   -2(n-3)\frac{\alpha}{r^{n-2}} + 2q_{,r}(r,\theta) +Q(r, \theta),
	\end{split}
\end{equation} 
where $Q$ is a function with faster fall-off than $q_{,r}$. On the other hand, it follows from Proposition \ref{PropositionJangGraphIsAE} and Lemma \ref{LemmaJangGraphMetric} that
\begin{equation}
	\begin{split}
		\hat{g}_{rr} &= \hat{g}^-_{rr} + \Ol_2(r^{-(n-1-\epsilon)}), \\
		&=1- 2(n-3)\frac{\alpha}{r^{n-2}} + \Ol_2(r^{-(n-1-\epsilon)}).
	\end{split}
\end{equation}
Hence $q_{,r}(r, \theta) =  \Ol(r^{-(n-1-\epsilon)})$. Similarly, on the one hand we have
\begin{equation}
	\begin{split}
		\hat{g}_{r\mu} &= g_{r\mu} + f_{,r}f_{,\mu} \\
		&= \bigg( \frac{r}{\sqrt{1+r^2}}- (n-3)\frac{\alpha}{r^{n-2}} + q_{,r}(r,\theta)   \bigg)\bigg(   \frac{\alpha_{,\mu}}{r^{n-3}} + 	q_{,\mu}(r, \theta)  \bigg) \\
		&=\frac{\alpha_{,\mu}}{r^{n-3}} + q_{,\mu}(r,\theta) + \tilde{Q}(r, \theta),
	\end{split}
\end{equation}
where $\tilde{Q}$ is a functon with faster fall-off than $q_{,\mu}$. On the other hand, we have 
\begin{equation}
	\begin{split}
		\hat{g}_{r\mu} &= \hat{g}^-_{r\mu} + \Ol_2(r^{-(n-1-\epsilon)}) \\
		&= \frac{\alpha_{,\mu}}{r^{n-3}} +\Ol_2(r^{-(n-2-\epsilon)}) 
	\end{split}
\end{equation}
from Proposition \ref{PropositionJangGraphIsAE} and Lemma \ref{LemmaJangGraphMetric}. It follows that $q_{,\mu}(r, \theta) = \Ol(r^{-(n-2-\epsilon)})$ and in turn $|dq|_\delta= \Ol(r^{-(n-1-\epsilon)})$.
\\ \indent Repeating the above argument, we obtain the desired estimates for the second and third derivatives.

\end{proof}

\section{The conformal structure of the Jang graph}\label{SectionConformalChanges}

In this section we apply a series of conformal changes and deformations to the Jang graph obtained in Section \ref{SectionJangSolution} resulting in an asymptotically Euclidean manifold to which the Riemannian positive mass Theorem can be applied. We also need to handle the additional complications related to the possible presence of the conical singularities in Section \ref{SectionJangSolution}. 
\\ \indent We note that at this stage we know both that the Jang graph $(\hat{M}^n, \tilde{g}_\Psi)$ is asymptotically flat in the sense of Definition \ref{DefinitionAFinitialData} (see Corollary \ref{CorollaryJangGraphAF}) and we know that the scalar curvature $R_{\hat{g}}$ is integrable (see Appendix \ref{SectionJangGraph}). In particular the ADM energy is well-defined (see Appendix \ref{SectionJangGraphADMmass} for a computation of the ADM energy of the Jang graph). 
\\ \indent We recall the decompositions $\hat{M}^n = \hat{C}_1\cup \ldots \cup \hat{C}_\ell \cup \hat{K} \cup \hat{N}^n$ (and $\tilde{M}=\tilde{C}_1 \cup \ldots \cup \tilde{C}_\ell \cup \tilde{K}\cup \hat{N}^n$) from the end of Section \ref{SectionJangSolution}, where the $\hat{C}_i$ are the cylindrical ends ($\tilde{C}_i$ are the exact cylindrical ends), $\hat{K}$ (and $\tilde{K}$) is compact and $\hat{N}^n$ is the asymptotically flat end. Furthermore, we have $\tilde{g}_\Psi=\hat{g}$ on $\hat{N}^n$. 

\subsection{Conformal change}\label{SubsectionConformalChange}

We now want to find a conformal factor $u>0$ such that the conformal change 
\begin{equation}
	\tilde{g}_\Psi \rightarrow u^{\frac{4}{n-2}}\tilde{g}_\Psi = (u\Psi)^{\frac{4}{n-2}}\tilde{g},
\end{equation}
where $\Psi$ is defined in \eqref{EquationPsiDef}, yields vanishing scalar curvature. For this, we need to solve the {\it Yamabe equation}:
\begin{equation}\label{EquationConformalChange}
	-\Delta_{\tilde{g}_\Psi}u + c_nR_{\tilde{g}_\Psi}u=0, \qquad c_n=\frac{n-2}{4(n-1)}.
\end{equation}
Our argument here follows \cite[Section 3]{EichmairPMT} very closely, but we include it here for completeness. 
\\ \indent We recall from Section \ref{SubSectionGraphTopology} that the metric $\tilde{g}_\Psi$ is not complete without having added points at the cylindrical infinities. We let $\sigma$ and $\sigma^{-1}$ be regular values of the distance function $s$ in \eqref{EquationDistanceFunction} and define the manifold $S_\sigma=\{\sigma \leq s \leq \sigma^{-1}\}$ with boundary $\partial S_\sigma = \{s=\sigma\}\cup \{s=\sigma^{-1}\}$. The solution of \eqref{EquationConformalChange} is obtained by first solving the sequence of Dirichlet problems 
\begin{equation}\label{EquationDirichletSystemSigma}
	\begin{cases}
		-\Delta^{\tilde{g}\Psi}u_\sigma + c_n R_{\tilde{g}\Psi}&u_\sigma=0 \qquad \text{in} \qquad S_\sigma, \\
		&u_\sigma=1 \qquad \text{on} \qquad \partial S_\sigma
	\end{cases}	
\end{equation}
and then pass to the limit $\sigma \rightarrow 0$.
\\ \indent The following Lemma, which we state without proof, is a Sobolev-type inequality that will be useful below.

\begin{lemma}\label{LemmaSobolevInequality}(\cite[Lemma 18]{EichmairPMT} but cf. also \cite[Lemma 3.1]{PMTI})\label{EquationSobolevInequality}
	Let $(M^n,g)$ be a complete connected Riemannian manifold, possibly with boundary, such that there exists a compact set $K\subset M^n$ and a diffeomorphism $\Psi=(x^1, \ldots, x^n):M^n\setminus K \rightarrow \rn^n \setminus \bar{B}_1(0)$ so that for some constant $C\geq1$ we have that $C^{-1}\delta \leq g \leq C\delta$ as quadratic forms. For every $1\leq p <n$ there exists a constant $C=C(M^n, g, p)$ such that 
	\begin{equation}
		\bigg(  \int_{M^n}|\varphi|^{\frac{np}{n-p}}d\mu^g     \bigg)^{\frac{n-p}{np}} \leq C \bigg(    \int_{M^n} |d\varphi|_g^pd\mu^g  	\bigg)^{\frac{1}{p}}, 
	\end{equation}
	for all compact supported functions $\varphi\in C^{1}_c(M^n)$.
\end{lemma}

\noindent We decompose $u_\sigma = 1+ v_\sigma$. In Lemma \ref{LemmaDirichletSolution} below we establish existence, uniqueness, regularity, positivity and uniform boundedness of the solutions $u_\sigma$ to the Dirichlet problems \ref{EquationDirichletSystemSigma}.

\begin{lemma}\label{LemmaDirichletSolution}
	The Dirichlet problems
	\begin{equation}\label{EquationDirichletSystem}
		\begin{cases}
			-\Delta^{\tilde{g}_\Psi}v_\sigma + c_n R_{\tilde{g}_\Psi}&v_\sigma = - c_n R_{\tilde{g}_\Psi} \qquad  \text{in} \qquad S_\sigma   \\
			&v_\sigma = 0 \qquad \qquad \: \:\: \: \text{on} \qquad \partial S_\sigma .
		\end{cases}
	\end{equation}
	have unique solutions $v_\sigma\in C^{2,\alpha}(S_\sigma)$ for any regular value $\sigma$ of $s$. Furthermore there is a uniform bound
	\begin{equation}
		||v_\sigma||_{C^{2,\alpha}_{loc}(S_\sigma)}<C,
	\end{equation}
	where the constant $C$ is independent of $\sigma$. Finally, the functions $u_\sigma=1+v_\sigma$ are positive on $S_\sigma$.
\end{lemma}

\begin{proof}
	
We use Fredholm alternative to show existence and uniqueness of solutions $v_\sigma$ to the Dirichlet problem \eqref{EquationDirichletSystem}. Multiplying the homogeneous problem $-\Delta_{\tilde{g}_\Psi}v_\sigma + c_n R_{\tilde{g}_\Psi}v_\sigma=0$ by $v_\sigma$ and performing a partial integration over $S_\sigma$ in view of $v_\sigma=0$ on $\partial S_\sigma$ we obtain the associated variational form:
\begin{equation}
		0= \int_{S_\sigma}\bigg(  |dv_\sigma|_{\tilde{g}_\Psi}^2   + c_n R_{\tilde{g}_\Psi}v_\sigma^2\bigg)  d\mu^{\tilde{g}_\Psi}.
\end{equation}
From \eqref{EquationAux6} in Lemma \ref{LemmaPsiProps} (together with a standard approximation argument using that smooth functions are dense in $W^{1,2}$) we see that this implies
\begin{equation}
	\int_{S_\sigma} \Psi^{-2}|d(\Psi v_\sigma)|^2_{\tilde{g}_\Psi} d\mu^{\tilde{g}_\Psi} =0.
\end{equation}
It follows that $\Psi v_\sigma$ is constant and, in turn since $v_\sigma=0$ on $\partial S_\sigma$, that $v_\sigma$ must vanish. Hence Problem \eqref{EquationDirichletSystem} has trivial kernel and we have a unique solutions $v_\sigma \in W^{1,2}(S_\sigma)$ by Fredholm alternative. By standard elliptic regularity theory any weak solution $v_\sigma \in W^{1,2}(S_\sigma)$ is also $C^{2,\alpha}(S_\sigma)$-regular. We extend $v_\sigma$ by zero to a compactly supported Lipschitz function on $\hat{M}^n$.
\\ \indent In what follows we let $0<\sigma_0<1/2$ to be small enough so that when $s<2\sigma_0$, that is to say ''far enough in the exact cylinders'', the scalar curvature $R_{\tilde{g}_\Psi}$ vanishes. Additionally, we may assume that $\sigma_0$ is small enough so that for all $\sigma\in (0,2\sigma_0)$, both $\sigma$ and $\sigma^{-1}$ are regular values of $s$.
\\ \indent In order to prove uniform $C^{2, \alpha}(S_\sigma)$-boundedness of $v_\sigma$ we perform estimates separately on $S_\sigma\cap \{ s \geq 2\sigma_0\}$ and on $S_\sigma\cap \{s\leq 2\sigma_0\}$. For the former, we first establish uniform in $\sigma$ bounds for the norms $||v_\sigma||_{L^q(\{s\geq \sigma_0\} )}$, where $q=\frac{2n}{n-2}$. Multiplying \eqref{EquationDirichletSystem} by $v_\sigma$ and integrating by parts over $S_\sigma$ we obtain:
\begin{equation}\label{EquationAuxillaryIntegral}
	-\int_{S_\sigma} c_nR_{\tilde{g}_\Psi}v_\sigma d\mu^{\tilde{g}_\Psi} = \int_{S_\sigma}\big(  |dv_\sigma|_{\tilde{g}_\Psi}^2   + c_n R_{\tilde{g}_\Psi}v_\sigma^2 \big)  d\mu^{\tilde{g}_\Psi}.
\end{equation}
We now have
\begin{equation}
	\begin{split}
		\bigg(\int_{\{s\geq \sigma_0\} } |v_\sigma|^{\frac{2n}{n-2}} d\mu^{\tilde{g}_\Psi}\bigg)^{\frac{n-2}{n}} &\leq  	C_1\bigg(\int_{\{s\geq \sigma_0\} } |\Psi v_\sigma|^{\frac{2n}{n-2}} d\mu^{\tilde{g}_\Psi}\bigg)^{\frac{n-2}{n}} \\
		&\leq C_2\int_{\{s\geq \sigma_0\} }\Psi^{-2}|d(\Psi v_\sigma)|^2_{\tilde{g}_\Psi} d\mu^{\tilde{g}_\Psi} \\
		&\leq C_2\int_{S_\sigma} \Psi^{-2}|d(\Psi v_\sigma)|^2_{\tilde{g}_\Psi} d\mu^{\tilde{g}_\Psi} \\
		&\leq C_2\int_{S_\sigma}\big(  |dv_\sigma|_{\tilde{g}_\Psi}^2   + c_n R_{\tilde{g}_\Psi}v_\sigma^2\big)  d\mu^{\tilde{g}_\Psi} \\
		&\leq 2C_2\int_{S_\sigma} |R_{\tilde{g}_\Psi}| |v_\sigma| d\mu^{\tilde{g}_\Psi} \\
		&\leq 2C_2  \bigg(  \int_{\{s\geq \sigma_0\}} |R_{\tilde{g}_\Psi}|^{\frac{2n}{n+2}} d\mu^{\tilde{g}_\Psi}   \bigg)^{\frac{n+2}{2n}} 	\\
		&\qquad \times \bigg(     \int_{\{s\geq \sigma_0\}} |v_\sigma|^{\frac{2n}{n-2}} d\mu^{\tilde{g}_\Psi}  \bigg)^{\frac{n-2}{2n}}
	\end{split}
\end{equation}
where $C_1,C_2$ do not depend on $\sigma$. The first inequality follows by the boundedness of $\Psi$ from below, the second inequality follows from Lemma \ref{LemmaSobolevInequality} applied to the manifold with boundary $(\{s\geq \sigma_0 \}, \tilde{g}_\Psi)$ and the fact that $\Psi$ is bounded from above, the third inequality follows since $\text{supp}(v_\sigma) \cap \{s\geq \sigma_0\}\subset S_\sigma$, the fourth inequality from \eqref{EquationAux6} in Lemma \ref{LemmaPsiProps}, the fifth inequality follows from \eqref{EquationAuxillaryIntegral} and the final inequality follows from H\"older's inequality and the fact that $R_{\tilde{g}_\Psi}$ vanishes for $s\leq \sigma_0$. Since $R_{\tilde{g}_\Psi}=R_{\hat{g}}=\Ol(r^{-(n+\min\{ \tau_0, \epsilon\})})$ we have
\begin{equation}
	\int_{\{s\geq \sigma_0\}}|R_{\tilde{g}_\Psi}|^{\frac{2n}{n+2}}d\mu^{\tilde{g}_\Psi}< \infty. 
\end{equation}
Thus, we have established a bound of $||v_\sigma||_{L^q(\{s\geq 2\sigma_0\}  )}$. By successively applying $L^p$-estimates and Sobolev inequalities we get a uniform bound on $||v_\sigma||_{C^{\alpha}(\{s\geq 2\sigma_0\}  )}$. In turn, global Schauder estimates gives a uniform $C^{2,\alpha}_{loc}$ on $\{s\geq2 \sigma_0\}$. 
\\ \indent To prove a uniform bound on $S_\sigma\cap \{s\leq 2\sigma_0\}$, we observe that since $R_{\tilde{g}_\Psi}=0$ on $\{s\leq \sigma_0\}$ we have that $v_\sigma$ is $\tilde{g}_\Psi$-harmonic there. From the maximum principle we get that the maximum and minimum of $v_\sigma$ in $\{ \sigma \leq s \leq 2\sigma_0\}$ are attained on the boundary $\{s=\sigma\}\cup\{s=2\sigma_0\}$. The uniform bound on $\v_\sigma$ on $\{ s\geq 2\sigma_0 \}$ together with the vanishing of $v_0$ on $\{ s=\sigma\}$ implies a uniform bound of $v_\sigma$ on $S_\sigma$. Standard elliptic theory converts this $L^\infty$-bound into a $C^{2,\alpha}_{loc}$-bound on $S_\sigma$. 
\\ \indent To show positivity of $u_\sigma=1+v_\sigma$ we let $\epsilon>0$ be a regular value of $-u_\sigma$, for a fixed $\sigma$. Then $w_\sigma =\min\{u_\sigma + \epsilon, 0\}$ is a Lipschitz function with support in $S_\sigma$. We use $w_\sigma$ as a test function in \eqref{EquationAux6} of Lemma \ref{LemmaPsiProps}:
\begin{equation}
	\begin{split}
		\frac{1}{2}\int_{\hat{M}^n } \Psi^{-2}|d(\Psi w_\sigma)|^2_{\tilde{g}_\Psi} d\mu^{\tilde{g}_\Psi} &\leq \int_{\hat{M}^n} 	|dw_\sigma|^2_{\tilde{g}_\Psi}+ c_nR_{\tilde{g}_\Psi}w_\sigma^2 d\mu^{\tilde{g}_\Psi} \\
		&= \int_{\hat{M}^n} w_\sigma \bigg(-\Delta^{\tilde{g}\Psi}w_\sigma + c_nR_{\tilde{g}\Psi}w_\sigma\bigg) d\mu^{\tilde{g}_\Psi} \\
		&\leq \int_{\{u_\sigma <-\epsilon\}} (u_\sigma + \epsilon)\bigg(-\Delta^{\tilde{g}_\Psi}(u_\sigma + \epsilon) + 	c_nR_{\tilde{g}_\Psi}(u_\sigma + \epsilon)\bigg) d\mu^{\tilde{g}_\Psi} \\
		&=\int_{\{u_\sigma < - \epsilon\}} c_n(u_\sigma + \epsilon)\epsilon R_{\tilde{g}_\Psi} d\mu^{\tilde{g}_\Psi} \\
		&=\int_{S_\sigma} \one_{ \{u_\sigma < - \epsilon\}} (u_\sigma + \epsilon)\epsilon R_{\tilde{g}_\Psi}d\mu^{\tilde{g}_\Psi}
	\end{split}
\end{equation}
It follows from the Dominated Convergence Theorem that we may take the outside limit $\epsilon \rightarrow 0$ to get that $|d(\Psi u_\sigma) |_{\tilde{g}_\Psi}=0$ on $\{u_\sigma <0\}$ so that, in turn, $\Psi u_\sigma$ is constant on $\{u_\sigma <0\}$. From this we get that $\{u_\sigma<0\}$ is empty and so we have shown that $ u\geq 0$ on $S_\sigma$. By the Harnack inequality, \cite[Corollary 8.21]{GilbargTrudinger}, we have $u_\sigma >0$ on $S_\sigma$.

\end{proof}

In what follows we will need:

\begin{theorem}(\cite[Lemma 5]{MeyersPDEexpansion}, "Basic integral estimate for the Poisson equation")\label{TheoremMeyersPDEexpansion}
Consider the Poisson equation $\Delta^\delta u = f$, where $\delta$ is the Euclidean metric on $\rn^n$ and 
\begin{equation}
	f=\Ol\bigg(\frac{(\ln r)^q}{r^{2+p+\gamma}}\bigg)
\end{equation}
where $p$ is an integer, $q$ is a non-negative integer, $0\leq \gamma < 1$ and $f$ is Hölder continuous. For $n>2$ the equation has a particular solution 
\begin{equation}
	u= 
	\begin{cases}
		\Ol\big(\frac{(\ln r)^q}{r^{p+\gamma}}\big) & \text{if} \qquad 0<p<n-2   \qquad \text{or}\qquad \gamma>0, \\
		\Ol\big(\frac{( \ln r)^{q+1}}{r^{p}}\big) & \text{otherwise.}
	\end{cases}
\end{equation}
\end{theorem}

Combining Lemma \ref{LemmaDirichletSolution} and Theorem \ref{TheoremMeyersPDEexpansion}, we can finally prove the existence and establish desired properties of a solution to \eqref{EquationConformalChange}.


\begin{proposition}\label{PropositionConformalChange}
Let $\hat{E}_{ADM}$ be the ADM energy of $(\hat{M}^n,\tilde{g}_\Psi)$. There exists a positive solution $u\in C^{2,\alpha}_{loc}(\hat{M}^n)$  to
\begin{equation}
	- \Delta^{\tilde{g}_\Psi} u + c_n R_{\tilde{g}_\Psi}u=0, \qquad c_n=\frac{n-2}{4(n-1)},
\end{equation}
bounded below and above by positive constants with the asymptotic expansion 
\begin{equation}\label{EquationAsymptotics}
	u=1+\frac{A+2c_n \alpha}{r^{n-2}}+ \Ol^2 ( r^{-(n-1-\epsilon)} ),
\end{equation}
where $A$ satisfies
\begin{equation}\label{EquationA}
	 A < 2\frac{c_n (n-4)}{\omega_{n-1}} \int_{ \mathbb{S}^{n-1}}     \alpha   d\Omega.
\end{equation}
\end{proposition}

\begin{proof}

We construct a converging subsequence of $\sigma$. For this, let $\{\sigma_k\}_{k=1}^\infty$ be a sequence of positive numbers such that $\sigma_k\rightarrow 0$ as $k\rightarrow \infty$ and such that each $\sigma_k$ is a regular value of the distance function $s$. From the uniform bounds in the $C^{2,\alpha}_{loc}$-norm, the compactness of the embedding $C^{2,\alpha}_{loc} \rightarrow C^{2,\beta}_{loc}$ as mentioned before Lemma \ref{LemmaSisOpen} and a standard diagonalization argument we get a convergent subsequence to some $u\in C^{2,\beta}_{loc}(\hat{M}^n)$, where $\beta<\alpha$, and $u$ solves the Yamabe equation \eqref{EquationConformalChange}. 
\\ \indent From the uniform bound $||v_\sigma||_{L^q(\{s\geq \sigma_0\})} <C$, where $q= \frac{2n}{n-2}$ and $C$ does not depend on $\sigma$, we get that
\begin{equation}
	\int_{\hat{N}^n}|v|^\frac{2n}{n-2} d\mu^{\tilde{g}_\Psi} < \infty
\end{equation}
from Fatou's lemma, and hence $v \rightarrow 0$ as $r\rightarrow \infty$ so that $u\rightarrow 1$ as $r\rightarrow \infty$. 
\\ \indent We now turn to the proof of the asymptotics in \eqref{EquationAsymptotics}. Since $\Psi=1$ on $\hat{N}^n$ we use the notation $\tilde{g}_\Psi=\hat{g}$, where $\hat{g}$ is asymptotically Euclidean (see Section \ref{SectionJangAE}). We let $\Psi(p)=(x^1(p), \ldots, x^n(p))$ be the Cartesian coordinates induced by the chart at infinity. We write $u=1+v$ so that $v$ satisfies the equation $-\Delta^{\hat{g}}v + c_n R_{\hat{g}}v = - c_n R_{\hat{g}}$, which has a coordinate expression of the form $L(v)=a^{ij}v_{,ij}+ b^kv_{,k} + cv = f$. Recalling from Lemma \ref{LemmaJangGraphRicci} that $f=-c_nR_{\hat{g}}=-2c_n \frac{\Delta^\Omega \alpha}{r^{n}}  + \Ol^1(r^{-(n+1-\epsilon)})$ and using the rescaling technique used in Section \ref{SectionJangAE} we obtain $v_{,r}=\Ol(r^{-1})$ and $v_{,\mu}=\Ol(1)$ and similarly $v_{,rr}=\Ol(r^{-2})$, $v_{,r\mu}=\Ol(r^{-1})$ and $v_{,\mu\nu}=\Ol(1)$. We now follow the proof of \cite[Proposition 7.7]{SakovichPMTah} and write $v= 2c_n \frac{\alpha}{r^{n-2}} + \tilde{v}$. Inserting this expression into the equation that $v$ satisfies and using the expansion of $R_{\hat{g}}$ yields
\begin{equation}\label{EquationSomeSome}
	-2c_n \Delta^{\hat{g}} \bigg( \frac{\alpha}{r^{n-2}} \bigg) - \Delta^{\hat{g}} \tilde{v} + c_n R_{\hat{g}} \tilde{v} = -2c_n \frac{\Delta^\Omega \alpha}{r^n} + \Ol(r^{-(n+1-\epsilon)}). 
\end{equation}
We expand the first term in terms of the Euclidean metric $\delta$ using Lemma \ref{LemmaJangGraphChristoffelSymbols}:
\begin{equation}
	\begin{split}
		\Hess^{\hat{g}}_{rr} \bigg( \frac{\alpha}{r^{n-2}}\bigg) &= \bigg( \frac{\alpha}{r^{n-2}}\bigg)_{,rr} - \hat{\Gamma}_{rr}^r\bigg( \frac{\alpha}{r^{n-2}}\bigg)_{,r} - \hat{\Gamma}_{rr}^\mu \bigg( \frac{\alpha}{r^{n-2}}\bigg)_{,\mu} \\
		&=(n-1)(n-2)\frac{\alpha}{r^{n}} + \Ol(r^{-(2n-1-\epsilon)}),
	\end{split}
\end{equation}
and similarly we find 
\begin{equation}
	\Hess^{\hat{g}}_{r\mu} \bigg( \frac{\alpha}{r^{n-2}}\bigg) = -(n-1)\frac{\alpha_{,\mu}}{r^{n-1}} + \Ol(r^{-(2n-3)})
\end{equation}
and
\begin{equation}
	\Hess^{\hat{g}}_{ \mu \nu} \bigg( \frac{\alpha}{r^{n-2}}\bigg) = \frac{\Hess_{\mu\nu}^\Omega (\alpha)}{r^{n-2}}  + (n-2)\frac{\alpha}{r^n}\delta_{\mu\nu} + \Ol(r^{-(2n-3-\epsilon)}).
\end{equation}
It follows that 
\begin{equation}
	\Delta^{\hat{g}} \bigg( \frac{\alpha}{r^{n-2}}\bigg) = \frac{\Delta^\Omega (\alpha)}{r^n}  + \Ol(r^{-(2n-1-\epsilon)}) 
\end{equation}
so that \eqref{EquationSomeSome}, combined with the estimate $\tilde{v} R_{\hat{g}}= \Ol(r^{-n})$, reduces to $-\Delta^{\hat{g}} (\tilde{v} )= \Ol(r^{-n})$. We now expand the Laplacian in terms of the Euclidean metric $\delta$:
\begin{equation}
	\begin{split}
		\Hess_{rr}^{\hat{g}}(\tilde{v}) 
		&=\tilde{v}_{,rr}-\hat{\Gamma}_{rr}^{r}\tilde{v}_{,r}-\hat{\Gamma}_{rr}^{\mu}\tilde{v}_{,\mu} \\
		&=\tilde{v}_{,rr}-\bigg((n-2)(n-3)\frac{\alpha}{r^{n-1}} + \Ol ( r^{-(n-\epsilon)} ) \bigg)\tilde{v}_{,r}  - \bigg(     \Ol ( 	r^{-(n+1-\epsilon)})  \bigg)   \tilde{v}_{,\mu} \\
		&=\Hess_{rr}^{\delta}(\tilde{v}) + \Ol(r^{-n}),     
	\end{split}
\end{equation}
where we used the Christoffel symbols from Lemma \ref{LemmaJangGraphChristoffelSymbols} for $\hat{g}$ in the second line and the fact that $\Gamma_{rr}^r=\Gamma_{rr}^{\mu}=0$ for the Euclidean metric in the last line. Similarly, we find $\Hess_{r\mu}^{\hat{g}}(\tilde{v}) = \Hess_{r\mu}^{\delta}(\tilde{v}) + \Ol(r^{-(n-1)})$ and $\Hess_{\mu \nu}^{\hat{g}}(\tilde{v})= \Hess_{\mu \nu}^{\delta}(\tilde{v}) + \Ol(r^{-(n-2)}) $.
Using Lemma \ref{LemmaJangGraphMetric}, we obtain $\Delta^{\hat{g}} \tilde{v} =\Delta^{\delta} \tilde{v}+ \Ol(r^{-n})$
so that
\begin{equation}
	\Delta^{\delta} \tilde{v}  = \Ol(r^{-n}). 
\end{equation}
We decompose $\tilde{v}$ into homogeneous and particular parts: $\tilde{v}=\tilde{v}_h+\tilde{v}_p$, where $\tilde{v}_h=  Ar^{-(n-2)}$ is the solution of $\Delta^\delta \tilde{v}_h = 0$ such that $\tilde{v}_h\rightarrow 0$ as $r\rightarrow \infty$. To estimate the particular solution $\tilde{v}_p$ we apply Theorem \ref{TheoremMeyersPDEexpansion} with $\gamma=q=0$ and $p=n-2$. The particular solution then decays as
\begin{equation}
	\tilde{v}_p = \Ol\bigg(  \frac{\ln r}{r^{n-2}}  \bigg),
\end{equation}
which is still slower than $\tilde{v}_h$. We bootstrap the argument in order to improve this decay rate. The Interior Schauder estimate now yields the bound $||\tilde{v}_p||_{C^{2,\alpha}(B_r(x_0))}=\Ol((\ln r )\cdot r^{-(n-2)})$ and the rescaling technique yields $|d\tilde{v}_p|_\delta=\Ol((\ln r )\cdot r^{-(n-1)})$ and $|\Hess^\delta(\tilde{v}_p)|_\delta=\Ol((\ln r )\cdot r^{-n})$. Hence, $\tilde{v}_{p,r}=\Ol((\ln r )\cdot r^{-(n-1)})$ and $\tilde{v}_{p,\mu}=\Ol((\ln r )\cdot r^{-(n-2)})$ and similarly $\tilde{v}_{p,rr}=\Ol((\ln r )\cdot r^{-n})$, $\tilde{v}_{p,r\mu}=\Ol((\ln r )\cdot r^{-(n-1)})$ and $\tilde{v}_{p,\mu\nu}=\Ol((\ln r )\cdot r^{-(n-2)})$. With the improved estimate $R_{\hat{g}}\tilde{v} = \Ol((\ln r) \cdot r^{-(2(n-1))})=\Ol(r^{-(n+1-\epsilon)})$, \eqref{EquationSomeSome} reduces to $-\Delta^{\hat{g}}( \tilde{v}) = \Ol(r^{-(n+1-\epsilon)})$. Furthermore, we find $\Hess^{\hat{g}}_{rr}(\tilde{v}) = \Hess^{\delta}_{rr}(\tilde{v}) + \Ol\big( (\ln r)\cdot r^{- 2(n-1) }  \big)$ and similarly $\Hess^{\hat{g}}_{r\mu}(\tilde{v}) = \Hess^{\delta}_{r\mu}(\tilde{v})+\Ol\big( (\ln r) \cdot r^{-(n-1)}  \big)$ and $\Hess^{\hat{g}}_{ \mu \nu}(\tilde{v}) = \Hess^{\delta}_{ \mu \nu}(\tilde{v})+\Ol\big( (\ln r) \cdot r^{-(2n-5)}  \big)$ so that $\Delta^{\hat{g}}(\tilde{v}) = \Delta^\delta(\tilde{v}) + \Ol\big( r^{-(n+1-\epsilon)} \big)$. Applying Theorem \ref{TheoremMeyersPDEexpansion} with $p=n-2$, $q=0$ and $\gamma=1-\epsilon>0$ to obtain the decay 
\begin{equation}%
	v_p=\Ol ( r^{-(n-1-\epsilon)} ).
\end{equation}
Using Interior Schauder estimates together with the rescaling technique yet again we obtain asserted decay of $u$.
\\ \indent Finally, it remains only to prove the inequality that $A$ satisfies. We observe that the Schoen-Yau identity \eqref{EquationSchoenYauId} combined with the strict dominant energy condition holding near $\partial U_f$ imply 
\begin{equation}
	R_{\tilde{g}} - |\hat{A}-k|^2_{\tilde{g}} - 2|q|_{\tilde{g}}^2 + 2\diver^{\tilde{g}} q\geq \frac{1}{2}(\mu - |J|_g).
\end{equation}
Moreover, the metric $\tilde{g}_{u\Psi}=(u \Psi)^{\frac{4}{n-2}} \tilde{g}$ is scalar flat, hence $-\Delta^{\tilde{g}} (u\Psi) + c_nR_{\tilde{g}} (u\Psi) = 0$. Further, we note that $\diver^{\tilde{g}}((u\Psi)^2q) = 2(u\Psi) q(\nabla^{\tilde{g}}(u\Psi)) + (u\Psi)^2\diver^{\tilde{g}}q$ and $\diver^{\tilde{g}}((u\Psi) d (u\Psi)) = |d (u\Psi)|_{\tilde{g}}^2+ (u\Psi)\Delta_{\tilde{g}}(u\Psi)$. Furthermore, from the Cauchy-Schwartz and geometric-arithmetic mean inequalities it follows that 
\begin{equation}\label{EquationAuxII}
	-2(u\Psi)q(\nabla^{\tilde{g}}(u\Psi)) \leq  (u\Psi)^2|q|^2_{\tilde{g}}+|d(u\Psi)|^2_{\tilde{g}}.
\end{equation}
Combining all these facts we obtain
\begin{equation}
	\begin{split}
		\frac{1}{2}&c_n(u\Psi)^2(\mu - |J|_g) + c_n(u\Psi)^2|\hat{A}-k|^2_{\tilde{g}}\\
		&\leq (u\Psi)\Delta^{\tilde{g}} (u\Psi)    - 2c_n(u\Psi)^2|q|_{\tilde{g}}^2 + 2c_n(u\Psi)^2\diver^{\tilde{g}} q \\
		&= \big( \diver^{\tilde{g}}((u\Psi) d (u\Psi)) -|d (u\Psi)|_{\tilde{g}}^2  \big)  -2c_n(u\Psi)^2|q|^2_{\tilde{g}} \\
		&\qquad + 2c_n\big(\diver^{\tilde{g}}((u\Psi)^2q)- 2(u\Psi) 	q(\nabla^{\tilde{g}}(u\Psi))\big) \\
		&=\diver^{\tilde{g}} \big(   (u\Psi) d (u\Psi)   + 2c_n(u\Psi)^2q\big) - |d (u\Psi)|_{\tilde{g}}^2  \\
		&\qquad  -2c_n(u\Psi)^2|q|^2_{\tilde{g}} - 4c_n(u\Psi) q(\nabla^{\tilde{g}}(u\Psi))  \\
		&\leq \diver^{\hat{g}} \big(  (u\Psi) d (u\Psi)   + 2c_n(u\Psi)^2q\big) - |d (u\Psi)|_{\tilde{g}}^2 \\
		&\qquad -  2c_n(u\Psi)^2|q|^2_{\tilde{g}} + 2c_n\big(|d(u\Psi)|^2_{\tilde{g}}+ (u\Psi)^2|q|^2_{\tilde{g}} \big) \\
		&=\diver^{\tilde{g}} \big(  (u\Psi) d (u\Psi)   + 2c_n(u\Psi)^2q \big) +  (2c_n-1) |d(u\Psi)|^2_{\tilde{g}} \\
		&\leq \diver^{\tilde{g}} \big(  (u\Psi) d (u\Psi)   + 2c_n(u\Psi)^2q \big),
	\end{split}	
\end{equation}
since $(2c_n-1)=-\frac{n}{2(n-1)}$. We want to integrate this inequality over $\hat{M}^n$ with respect to the measure $\mu^{\tilde{g}}$. In order to do so we need to verify integrability in the asymptotically flat end $\hat{N}^n$ as well as the cylindrical ends. We have 
\begin{equation}
	\begin{split}
		|\diver^{\tilde{g}} \big(  (u\Psi) d (u\Psi)   + 2c_n(u\Psi)^2q \big)| &\leq  | d(u \Psi)|^2_{\tilde{g}} + c_n(u\Psi)^2 |R_{\tilde{g}} |  \\
		&\qquad  + 2c_n(u\Psi)^2 |\diver^{\tilde{g}} q |+ 4c_n|u\Psi q(\nabla^{\tilde{g}} (u\Psi))| \\
		&\leq (1+2c_n)| d(u \Psi)|^2_{\tilde{g}} +   c_n(u\Psi)^2 |R_{\tilde{g}} | \\
		&\qquad  + 2c_n(u\Psi)^2 |\diver^{\tilde{g}} q |+ 2c_n (u\Psi)^2 |q|^2_{\tilde{g}} \\
		&\leq 2(1+2c_n)\Psi^2|du|^2_{\tilde{g}} + 2(1+2c_n)u^2|d\Psi|^2_{\tilde{g}}  +   c_n(u\Psi)^2 |R_{\tilde{g}} | \\
		&\qquad  + 2c_n(u\Psi)^2 |\diver^{\tilde{g}} q |+ 2c_n (u\Psi)^2 |q|^2_{\tilde{g}}, \\
	\end{split}
\end{equation}
where we used the equation $-\Delta^{\tilde{g}} (u\Psi) + c_nR_{\tilde{g}} (u\Psi) = 0$ and the inequality \eqref{EquationAuxII}. Now, since $u$ is bounded and $\Psi\in W^{1,2}( \{ s\leq \sigma_0\})$ by construction we see that all terms, except for possibly the first term, are integrable on $\{ s\leq \sigma_0\}$. The integrability of this term was shown in \cite{EichmairPMT} and for convenience we assist the reader with the argument. Since $R_{\tilde{g}_\Psi}=0$ on $\{ s\leq \sigma_0\}$ it follows that $-\Delta^{\tilde{g}_\Psi}u=0$ there. We get, after multiplying $du$ by a test function $\xi\in C^1_c( \{ s\leq \sigma\})$, using the Divergence Theorem and the equality $\diver^{\tilde{g}_\Psi}(\xi du)=du(\nabla^{\tilde{g}_\Psi} \xi) + \xi \Delta^{\tilde{g}_\Psi} u$ that
\begin{equation}
	\int_{\{s\leq \sigma\}} du (\nabla^{\tilde{g}_\Psi} \xi) d\mu^{\tilde{g}_\Psi} = \int_{ \{ s=\sigma\}} \xi du(\vec{n}_{\tilde{g}_\Psi}) d\mu^{\tilde{g}_\Psi}.
\end{equation}
By choosing $\xi= u \chi_\epsilon^2$, where $\chi_\epsilon$ is as in Remark \ref{RemarkHarmonicCapacity} for a small $\epsilon$, and applying the Cauchy-Schwarz inequality and the inequality of arithmetic and geometric means we obtain 
\begin{equation}
	\begin{split}
		\int_{ \{ s=\sigma\}} u\chi_\epsilon^2 du(\vec{n}_{\tilde{g}_\Psi}) d\mu^{\tilde{g}_\Psi} &= \int_{\{s\leq \sigma\}} \chi_\epsilon^2 |du|^2_{\tilde{g}_\Psi} d\mu^{\tilde{g}_\Psi} + 2\int_{\{s\leq \sigma\}} u\chi_\epsilon \langle du, d\chi_\epsilon\rangle_{\tilde{g}_\Psi} d\mu^{\tilde{g}_\Psi} \\
		&\geq 2\int_{\{s\leq \sigma\}} \chi_\epsilon^2 |du|^2_{\tilde{g}_\Psi} d\mu^{\tilde{g}_\Psi} -  \int_{\{s\leq \sigma\}} \big( \delta\chi_\epsilon^2 |du|^2_{\tilde{g}_\Psi} + \frac{1}{\delta}u^2 |d\chi_\epsilon|^2_{\tilde{g}_\Psi}   \big) d\mu^{\tilde{g}_\Psi} 
	\end{split}
\end{equation}
where $\delta>0$ is small. It follows that
\begin{equation}
	(1-\delta) \int_{\{s\leq \sigma\}} \chi_\epsilon^2 |du|^2_{\tilde{g}_\Psi} d\mu^{\tilde{g}_\Psi} \leq  \frac{1}{\delta} \int_{\{s\leq \sigma\}} u^2 |d\chi_\epsilon|^2_{\tilde{g}_\Psi}    d\mu^{\tilde{g}_\Psi}  +\int_{ \{ s=\sigma\}} u\chi_\epsilon^2 du(\vec{n}_{\tilde{g}_\Psi}) d\mu^{\tilde{g}_\Psi}.
\end{equation}
Letting $\epsilon\rightarrow 0$ and using the Monotone Convergence Theorem together with the easily verified equality 
\begin{equation}
	\int_{\{s\leq \sigma\}} \Psi^2 |du|^2_{\tilde{g}} d\mu^{\tilde{g}} = \int_{\{s\leq \sigma\}} |du|^2_{\tilde{g}_\Psi} d\mu^{\tilde{g}_\Psi}
\end{equation}
shows the claimed integrability. In turn, this shows that the divergence term is integrable over $\hat{M}^n$ with respect to the measure $\mu^{\tilde{g}}$ and by the Dominated Convergence Theorem we obtain 
\begin{equation}\label{EquationAux8}
	\begin{split}
		0 &< \int_{\hat{M}^n} \diver^{\tilde{g}} \big(  (u\Psi) d (u\Psi)   + 2c_n(u\Psi)^2q \big) d\mu^{\tilde{g}} \\
		&= \lim_{\sigma \rightarrow 0} \int_{\{ s=\sigma^{-1}\} }  \big(  u d u   + 2c_nu^2q \big)(\vec{n}_{\hat{g}})  d\mu^{\hat{g}} \\
		&\qquad + \lim_{\sigma\rightarrow 0} \int_{\{ s=\sigma \}}  \big(  (u\Psi) d (u\Psi)   + 2c_n(u\Psi)^2q \big)(\vec{n}_{\tilde{g}})   d\mu^{\tilde{g}} ,\\
	\end{split}
\end{equation}
where we recalled that $\tilde{g}=\hat{g}$ and $\Psi=1$ in the asymptotically flat end $\hat{N}^n$ of $\hat{M}^n$. The second integral is 
\begin{equation}
	\begin{split}
		\lim_{\sigma\rightarrow 0} \int_{\{ s=\sigma \}}  \big(  (u\Psi) d (u\Psi)   + 2c_n(u\Psi)^2q \big)(\vec{n}_{\tilde{g}})   d\mu^{\tilde{g}}&=  \lim_{\sigma\rightarrow 0} \int_{\{ s=\sigma \}}   u\Psi^2 d u(\vec{n}_{\tilde{g}})   d\mu^{\tilde{g}} \\
		&= \lim_{\sigma\rightarrow 0} \int_{\{ s=\sigma \}} u d u(\vec{n}_{\tilde{g}_\Psi})   d\mu^{\tilde{g}_\Psi} \\
		&=\lim_{\sigma\rightarrow 0} \int_{\{ s\leq\sigma \}} |d u |^2_{\tilde{g}_\Psi}   d\mu^{\tilde{g}_\Psi} \\
		&=0,
	\end{split}
\end{equation}
where the first inequality follows, as $\vec{n}_{\tilde{g}}=\partial_t$ close towards the cylindrical ends and $q$ acts trivially on $\partial_t$ and further $d\Psi(\partial_t)\rightarrow 0$ as $\sigma\rightarrow 0$, the second equality follows by standard computations and the final equality follows by letting $\xi= u\chi_\epsilon$ in \eqref{EquationAux8} and arguing as above. The first integral in \eqref{EquationAux8} is 
\begin{equation}\label{EquationAinequality}
	\begin{split}
		0&< \lim_{\sigma \rightarrow 0} \int_{\{s = \sigma^{-1}\}}   \big(   u d u   + 2c_nu^2q\big) (\vec{n}_{\hat{g}}) d\mu^{\hat{g}} \\
		&=\lim_{\sigma \rightarrow 0} \int_{ \{s=\sigma^{-1} \}} \big(  (n-2)\big(- A -2c_n \alpha    +2c_n (n-3) \alpha \big)\sigma^{ n-1 } +   \ol ( \sigma^{  n-1 } 	)   \big) d\mu^{\hat{g}} \\
		&=  -(n-2)\int_{ \mathbb{S}^{n-1}}     \big( A -2c_n (n-4) \alpha \big) d\Omega, \\
		&= -(n-2)\omega_{n-1}\big(  A + 2c_n\bigg(\frac{n-4}{n-3}\bigg)\hat{E}_{ADM}\big),    
	\end{split}
\end{equation}
where $\Omega$ is the standard induced Euclidean measure on $\mathbb{S}^{n-1}$, we used the asymptotics of the components of $q$ from Lemma \ref{LemmaJangGraphRicci}, used arguments from the proof of Proposition \ref{PropositionJangGraphADMmass} in the last line and that $\vec{n}_{\hat{g}}= \partial_r + \vec{V}$, where the components of $\vec{V}$ decay at least as $\Ol(r^{-1})$. The asserted inequality follows. 

\end{proof}
	
Finally, we comment on the relation between the energies of the original asymptotically hyperbolic metric $g$ and the scalar flat metric $\tilde{g}_{u\Psi}$ which exists by Proposition \ref{PropositionConformalChange}. Let the ADM energy of $\tilde{g}_\Psi$ be denoted by $\hat{E}_{ADM}$ (since $\tilde{g}_\Psi=\hat{g}$ in the infinity) and the ADM energy of the metric $\tilde{g}_{u\Psi}$ be denoted by $\hat{E}_{ADM}^u$. We compute the following:
\begin{equation}\label{EquationADMconformalChange}
	\begin{split}
		\hat{E}_{ADM}^u &= \lim_{R\rightarrow \infty}\frac{1}{2\omega_{n-1}(n-1)}  \int_{ \{r=R\}} \big( \diver^\delta( u^{\frac{4}{n-2}}\hat{g}) - d \trace^\delta (u^{\frac{4}{n-2}}\hat{g})  \big)(\partial_r) d\mu^\delta \\
		&= \hat{E}_{ADM} + \lim_{R\rightarrow \infty}\frac{1}{2\omega_{n-1}(n-1)}  \int_{ \{r=R\}} \frac{4u^{\frac{4}{n-2}-1}}{n-2}\big( \hat{g}(\nabla^\delta u, \partial_r) - \trace^\delta \hat{g} du(\partial_r)  \big) d\mu^\delta \\
		&= \hat{E}_{ADM} + \lim_{R\rightarrow \infty}\frac{2 }{ \omega_{n-1}(n-1)(n-2)} \\
		&\qquad \times  \int_{ \{r=R\}} \frac{ u^{\frac{4}{n-2}-1}}{    r^{ n-1 }} \big( A + 2c_n \alpha  \big) \big( -(n-2)  +n(n-2)  + \Ol(r^{-(1-\epsilon)})  \big) d\mu^\delta \\
		&=\hat{E}_{ADM} + 2A + \frac{4c_n}{\omega_{n-1}} \int_{\mathbb{S}^{n-1}} \alpha d\Omega.
	\end{split}
\end{equation}
Hence, using Lemmas \ref{PropositionWangMassVector} and \ref{PropositionJangGraphADMmass} and \eqref{EquationA}, we obtain 
\begin{equation}
	\begin{split}
		\hat{E}_{ADM}^u &= \hat{E}_{ADM}+ 2A + \frac{4c_n}{\omega_{n-1}} \int_{\mathbb{S}^{n-1}} \alpha d\Omega \\
		&< \hat{E}_{ADM} +\frac{4c_n}{\omega_{n-1}}(n-4)  \int_{ \mathbb{S}^{n-1}}         \alpha  d\Omega +  \frac{4c_n}{\omega_{n-1}} \int_{\mathbb{S}^{n-1}} \alpha d\Omega \\
		&=\hat{E}_{ADM} - 4c_n \hat{E}_{ADM} \\
		&= \frac{\hat{E}_{ADM} }{n-1} \\
		&=E. 
	\end{split}
\end{equation}

\subsection{Deformation to Asymptotically Schwarzschildean metric}

In this Section we remedy the fact that the metric $\tilde{g}_{u\Psi}= u^{\frac{4}{n-2}}\tilde{g}_\Psi=(u\Psi)^{\frac{4}{n-2}}\tilde{g}$ is in general not asymtotically Schwarzschildean in the sense of Definition \ref{DefinitionAFinitialData}. More specifically, Theorem \ref{TheoremConformalChangeToAS} below guarantees that we may approximate $\tilde{g}_\Psi$ with a metric $\bar{g}$ having asymptotics as in \cite{PMTI}. This is achieved applying the argument that was used in \cite{SakovichPMTah} which in turn relies on the construction of \cite{PMTIII}. 

\begin{theorem}\label{TheoremConformalChangeToAS}
	Let $\tilde{g}_{u\Psi}$ be the metric constructed on $\hat{M}^n$ in Proposition \ref{PropositionConformalChange} and let $\hat{E}_{ADM}^u$ be its ADM energy. For any $\epsilon >0$ there exists a scalar flat metric $\bar{g}$ on $\hat{M}^n$ with associated ADM energy $|\bar{E}_{ADM} - \hat{E}_{ADM}^u| \leq \epsilon$, which outside of a compact set in $\hat{M}^n$ is conformally flat: 
	\begin{equation}
		\overline{g}= \varphi^{\frac{4}{n-2}}\delta.
	\end{equation}
	In addition, the conformal factor $\varphi$ satisfies
	\begin{equation}
		\begin{split}
			\varphi  = 1 + \frac{\bar{E}_{ADM}}{2r^{n-2}} &+ \Ol ( r^{-(n-1)}  ), \qquad |d \varphi|_\delta  =  \Ol ( r^{-(n-1)} ) \qquad \text{and} \\ \qquad	&|\Hess^{\tilde{g}_{u\Psi}} ( \varphi ) |_\delta  =  \Ol ( r^{-(n-1)} ).
		\end{split}
	\end{equation}
\end{theorem}

\begin{proof}
	
We follow \cite{PMTIII} and write our metric $\tilde{g}_{u\Psi}$ as the sum of the Schwarzschild metric $g_S$ with the same ADM energy and a correction that does not contribute to the ADM energy:
\begin{equation}
	\tilde{g}_{u\Psi} = \bigg( 1+ \frac{\hat{E}_{ADM}^u}{2r^{n-2}}\bigg)^{\frac{4}{n-2}}\delta + h,
\end{equation}
where $h=\Ol_2(r^{-(n-2)})$. For $R>0$ large, we let $\xi_R$ be a $C^{3,\alpha}$-regular cutoff function that satisfies
\begin{equation}
	\xi_R(r) = 
		\begin{cases}
			0 \qquad \text{if} \qquad r<R, \\
			1 \qquad \text{if} \qquad r>2R,
		\end{cases}
\end{equation}
and 
\begin{equation}
	|\xi_R|\leq 1, \qquad |d \xi_R|_\delta = \Ol(R^{-1}), \qquad |\Hess \: \xi_R|_\delta = \Ol(R^{-2}).
\end{equation}
With this function we deform $\tilde{g}_{u\Psi}$ to a new metric:
\begin{equation}
	g_R= \tilde{g}_{u\Psi} - \xi_R(r)h.
\end{equation}
It is not difficult to see that, for $r>2R$,
\begin{equation}
		|g_R - \delta|^2_\delta  = \bigg( \frac{4}{n-2}\frac{\hat{E}_{ADM}^u}{2r^{n-2}} + \Ol(r^{-2( n+2)})     \bigg)^2 n, 
		%
\end{equation}
and
\begin{equation}
	|\nabla (g_R - \delta)|_\delta^2 	=\bigg(  \frac{2\hat{E}_{ADM}^u}{r^{n-1}} + \Ol ( r^{-(2n-3)} )    \bigg)^2 n.
\end{equation}
We have $R_{\tilde{g}_{u\Psi}}=0$ and it is well-known that the Schwarzschild metric is scalar flat. Hence $R_{ g_R}=0$ for both $r<R$ and $r>2R$. 
\\ \indent We now estimate the scalar curvature $R_{g_R}$ of $g_R$ in $\{R \leq r \leq 2R\}$. For this, we expand $R_{g_R}$ around $R_{g_S}$ using the formulas in \cite[Section 4.1]{MichelMassFormalism}. We have $ R_{g_R} 	 = R_{g_S} + DR_{g_S}( (1-\xi_R)h) + Q((1-\xi_R)h)$, where
\begin{equation}
	DR_{g_S}( (1-\xi_R)h) = \diver^{g_S} \bigg( \diver^{g_S}((1-\xi_R)h)- d\trace^{g_S}((1-\xi_R)h) \bigg),
\end{equation}
and $Q$ is a quadratic term that may be estimated as follows:
\begin{equation}
	Q((1-\xi_R)h) \leq C \big( |\nabla (1-\xi_R)h|^2_{g_S} + |(1-\xi_R)h|_{g_S} |\nabla \nabla(1-\xi_R)h|_{g_S}  \big),
\end{equation}
where $\nabla$ is the covariant derivative associated to $g_S$. Estimating these terms yields $R_{g_R} = \Ol(R^{-n})$. In particular it follows that
%
\begin{equation}
	\begin{split}
		\bigg(\int_{\hat{N}^n} | R_{g_R}|^{\frac{n}{2}}  \bigg)^{\frac{2}{n}} 	&=\Ol( R^{-(n-2)}).
	\end{split}
\end{equation}
since $R_{g_R}$ vanishes outside the annulus $\{R < r < 2R\}$ that has volume of order $R^{n}$. For reasons that will be clear below, we let $R$ be large enough so that
\begin{equation}\label{EquationAssumption1}
	C \bigg(\int_{\hat{N}^n} |c_nR_{g_R}|^{\frac{n}{2}}  d\mu^{g_R} \bigg)^{\frac{2}{n}} < 1,
\end{equation}
where $C$ is the constant in the Sobolev inequality \eqref{EquationSobolevInequality} with $p=2$.
\\ \indent We will now construct a solution $\varphi^R>0$ of the equation $\Delta^{g_R}\varphi^R - c_n R_{g_R} \varphi^R =0$. Let $\varphi^R=1+v^R$. Then $v^R$ satisfies the equation 
\begin{equation}
	-\Delta^{g_R}v^R + c_n R_{g_R}v^R = -c_n R_{g_R}.
\end{equation}
We solve this equation by similar methods to the ones in the proof of Proposition \ref{PropositionConformalChange}. Here we consider, for $\rho>0$ large, the mixed Dirichlet/Neumann problem
\begin{equation} 
	\begin{cases}
		-\Delta^{g_R}v^R_\rho + c_n& R_{g_R}v^R_\rho = -c_n R_{g_R} \qquad  \text{on} \qquad S_\rho   \\
		&\vec{n}_\rho(v^R_\rho) = 0 \qquad  \qquad  \: \text{on} \qquad \partial^- S_\rho \\
		& \qquad v^R_\rho = 0 \qquad \qquad   \: \text{on} \qquad \partial^+ S_\rho .
	\end{cases}
\end{equation}
where $S_\rho=\{ \rho^{-1} < s < \rho  \}$, $\partial^+S_\rho=\{s=\rho\}$ and $\partial^-S_\rho=\{ s=\rho^{-1}\}$, $s$ is the distance function defined in \eqref{EquationDistanceFunction} and $\vec{n}_\rho$ is the outward pointing unit normal of $\partial^- S_\rho$. To prove the existence of $v^R$, we consider the homogeneous problem. Multiplication by $v_\rho^R$ and integration by parts yields
\begin{equation}
	\begin{split}
		\int_{S_\rho}|dv_\rho^R|^2_{g_R}d\mu^{g_R} &= - \int_{ \{ R < s < 2R\} } c_nR_{g_R} (v_\rho^R)^2 d\mu^{g_R}   	\\
		&\leq    \bigg(\int_{\{ R < s < 2R\} } |c_nR_{g^R}|^{\frac{n}{2}}  d\mu^{g_R} \bigg)^{\frac{2}{n}} \bigg( \int_{\{ R < s < 2R\} } |v_\rho^R|^{\frac{2n}{n-2}} d\mu^{g_R}  \bigg)^{\frac{n-2}{n}} \\
		&\leq  C\bigg(\int_{\{ R < s < 2R\} } |c_nR_{g^R}|^{\frac{n}{2}}  d\mu^{g_R} \bigg)^{\frac{2}{n}} \int_{S_\rho}|dv_\rho^R|^2_{g_R}d\mu^{g_R}
	\end{split}
\end{equation}
where we used H\"older's inequality in the second line and the Sobolev inequality (together with an approximation argument, using that smooth functions are dense in $W^{1,2}(\hat{M}^n)$) in the last line. It follows that
\begin{equation}
	1 \leq  C\bigg(\int_{S_\rho} |c_nR_{g_R}|^{\frac{n}{2}}  d\mu^{g_R} \bigg)^{\frac{2}{n}},
\end{equation}
where $C$ is the Sobolev constant from Lemma \ref{LemmaSobolevInequality}, which contradicts the assumption in \eqref{EquationAssumption1}. Hence the homogeneous problem admits only the trivial solution and a unique solution exists by the Fredholm alternative. 
\\ \indent To show regularity and uniform boundedness in $C^{2,\alpha}_{loc}$-norm we let $R_0>0$ be large and for $R>2R_0$ observe that 
\begin{equation}
	\begin{split}
		\bigg(\int_{S_\rho\cap \{ s>R_0\}} |v^R_\rho|^{\frac{2n}{n-2}} d\mu^{g_{R}} \bigg)^{\frac{n-2}{ n}} &\leq \int_{S_\rho} |dv^R_\rho|^{2} d\mu^{g_R} \\
		&\leq \int_{S_\rho} \big(   c_n  R_{g_R}(v^R_\rho)^2   -c_n R_{g_R}v^R_\rho \big)d\mu^{g_R} \\
		&\leq C \bigg( \int_{\hat{M}^n} |R_{g_R}|^{ \frac{n}{2}} d\mu^{g_R} \bigg)^{ \frac{2}{n}} \bigg( \int_{S_\rho\cap \{ s>R_0\}} |v^R_\rho|^{\frac{2n}{n-2}}  \bigg)^{\frac{n-2}{n}} \\
		&\qquad + C \bigg( \int_{\hat{M}^n} |R_{g_R}|^{\frac{2n}{n+2}}  \bigg)^{\frac{n+2}{2n}} \bigg( \int_{S_\rho\cap \{ s>R_0\}} |v^R_\rho|^{\frac{2n}{n-2}}  \bigg)^{\frac{n-2}{2n}},
	\end{split}
\end{equation}
where we used the Sobolev inequality \eqref{EquationSobolevInequality} and inclusion in the first inequality, the equation and boundary conditions that $v^R_\rho$ satisfies in the second and H\"older's inequality in the final line. Using the estimate $R_{g_R}=\Ol(R^{-n})$ obtained above we observe that 
\begin{equation}
	\bigg( \int_{\hat{M}^n} |R_{g_R}|^{ \frac{n}{2}} d\mu^{g_R} \bigg)^{ \frac{2}{n}} = \Ol(R^{-(n-2)})	 \qquad \text{and} \qquad \bigg( \int_{\hat{M}^n} |R_{g_R}|^{\frac{2n}{n+2}}  \bigg)^{\frac{n+2}{2n}}= \Ol(R^{-\frac{(n-2)}{2}})
\end{equation}
and so, for $R$ sufficiently large, we may absorbe the first term into the left hand side to obtain the estimate
\begin{equation}\label{EquationAux9}
	\bigg(\int_{S_\rho\cap \{ s>R_0\}} |v^R_\rho|^{\frac{2n}{n-2}} d\mu^{g_{R}} \bigg)^{\frac{n-2}{ 2n}} = \Ol(R^{-\frac{(n-2)}{2}}),
\end{equation}
where the implicit constant in the $\Ol$-term does not depend on $\rho$. The same arguments as in the proof of Proposition \ref{PropositionConformalChange} yields the $C^{2,\alpha}_{loc}( \{ s>R_0\})$ bound $||v^R_\rho||_{C^{2,\alpha}( \{ s>R_0\})}=\Ol(R^{-\frac{(n-2)}{2}})$. Hence, $\varphi^R_\rho=1+v^R_\rho$ is bounded above and below on $\{s>R_0\}$ by positive constants independent of $\rho$. Furthermore, $\varphi^R_\rho$ cannot be maximized or minimized on $\{s = \rho^{-1}\}$, as by the Hopf Maximum principle this would contradict the boundary condition $\vec{n}_\rho(v_\rho^R)=0$ unless $v^R_\rho$ is constant. Hence, $\varphi^R_\rho>0$ on $\{ \rho^{-1} < s < R_0\}$ and in turn on $\{ \rho^{-1} < s < \rho\}$ with uniform in $\rho$ bounds above and below. 
\\ \indent Proceeding with the same diagonalization argument as in the proof of Proposition \ref{PropositionConformalChange} we obtain the desired solution $v^R\in C^{2,\beta}(\{ s>2R_0\})$ on $\hat{M}^n$, where $\beta <\alpha$ and $v^R$ solves $-\Delta^{g_R}v^R + c_nR_{g_R}v^R = - c_nR_{g_R}$.
\\ \indent We verify the fall-off properties of $\varphi^R$. It is well-known
\begin{equation}
	\varphi^R = 1 + \frac{A^R}{r^{n-2}} + \Ol_2(r^{-(n-1)})
\end{equation}
when the metric is identically Schwarzschild at the infinity. 
Consequently, integrating the equation that $\varphi^R$ satisfies we obtain
\begin{equation}
	\begin{split}
		\int_{ \{ R < s < 2R   \} } c_nR_{g_R}\varphi^R d\mu^{g_R} &= \int_{\{ R < s < 2R \}} 	\Delta^{g_R} \varphi^R d\mu^{g_R} \\
		&=\lim_{\sigma\rightarrow 0} \int_{ \{  s < \sigma^{-1} \}  } \diver^{g_R}( d\varphi^R) d\mu^{g_R} \\
		&= \lim_{\sigma\rightarrow 0} \int_{ \{s=\sigma^{-1}\} }   d\varphi^R(\vec{n}^R)  d\mu^{g_R}   \\
		&= \lim_{\sigma\rightarrow 0}\int_{\{s=\sigma^{-1}\} } \bigg(  -(n-2)\frac{A^R}{ r^{n-1}} + \Ol_1 ( r^{-n} )    \bigg)    d\mu^{g_R}  \\
		&=-(n-2) A^R \omega_{n-1} .
	\end{split}
\end{equation}
Thus, 
\begin{equation}
	A^R = - \frac{ c_n}{(n-2)\omega_{n-1}}\int_{ \hat{M}^n }  R_{g_R}\varphi^R d\mu^{g_R},
\end{equation}
in conclusion.
\\ \indent We now show that $A^R\rightarrow 0 $ as $R \rightarrow \infty$. We have 
\begin{equation}
	\bigg( \int_{\{ s> 2R_0\}} 	|v^R|^{\frac{2n}{n-2}} d\mu^{g_R}  \bigg)^{ \frac{n-2}{2 n}} = \Ol(R^{-\frac{(n-2)}{2}})
\end{equation}
from \eqref{EquationAux9} and Fatou's lemma. It follows that
\begin{equation}
	\begin{split}
		\bigg| A^R  +& \frac{c_n}{(n-2)\omega_{n-1}}\int_{\hat{M}^n}R_{g_R}  d\mu^{g_R}   \bigg| \\
		&\leq 	\frac{c_n}{(n-2)\omega_{n-1}} \int_{\hat{M}^n} |R_{g_R}| |v^R| d\mu^{g_R} \\
		&\leq 	\frac{c_n}{(n-2)\omega_{n-1}}\bigg(  \int_{\hat{M}^n} | R_{g_R} |^{\frac{2n}{n+2}} d\mu^{g_R} \bigg)^{\frac{n+2}{2n}} \bigg( \int_{\{s > R_0\}} 	|v^R|^{\frac{2n}{n-2}} d\mu^{g_R}  \bigg)^{ \frac{n-2}{2n}} \\
		&=\Ol(R^{-(n-2)}) \\
	\end{split}
\end{equation}
from the estimates above. From the formalism of \cite[Section 4.1]{MichelMassFormalism} we consider the metric $g_R= \delta + e$ and expand the scalar curvature $R_{g_R} = R_\delta + DR_\delta (e) + Q(e)$, where
\begin{equation}
	\begin{split}
		DR_\delta(e) &= \diver^\delta \bigg( \diver^\delta(e)- d\trace_\delta(e)   \bigg)  - \langle \text{Ric}^\delta, e\rangle_\delta \\
		&=\diver^\delta U(g_R , \delta) ,
	\end{split}
\end{equation}
and in turn $U(g_R, \delta)= \diver^\delta(e)- d\trace^\delta(e) $. $Q(e)$ may be estimated as
\begin{equation}
	\begin{split}
		Q(e) &\leq C\bigg(  |\nabla e|_\delta^2+  |e|_\delta |\nabla \nabla e|_\delta   \bigg) \\
		&=\Ol ( r^{-2( n-1)} ). \\
	\end{split}
\end{equation}
A computation shows that we have
\begin{equation}
	U(g_S, \delta)_r =(n-1)\bigg( 1 + \frac{\hat{E}_{ADM}^u}{2r^{n-2}}  \bigg)^{\frac{4}{n-1}-1}\bigg( \frac{\hat{E}_{ADM}^u}{2r^{n-1}}\bigg).
\end{equation}
It follows that
\begin{equation}
	\begin{split}
		\int_{\hat{M}^n} R_{g_R} d\mu^{g_R} &= \int_{\{ R < s < 2R\}} \diver^\delta U(g_R , \delta) d\mu^{g_R} + \Ol(R^{-(n-2)}) 	\\
		&= \int_{\{ R < s < 2R\}} \diver^\delta U(g_R , \delta) d\mu^{g_R} + \Ol(R^{-(n-2)}) \\
		&=\int_{\{s=2R \} }U(g_S , \delta)(\partial_r ) d\mu^{g_R} - \int_{ \{s= R \}}U(\tilde{g}_{u\Psi} , \delta)(\partial_r) d\mu^{g_R}+ 	\Ol(R^{-(n-2)}) .\\
	\end{split}
\end{equation}
Both integrals tend to $2(n-1)\omega_{n-1}\hat{E}_{ADM}^u$ as $R\rightarrow \infty$. It follows that $A^R\rightarrow 0$ as $R\rightarrow \infty$. In summary, the metric
\begin{equation}
	\bar{g}  = (\varphi^R)^{\frac{4}{n-2}}g_R
\end{equation}
is scalar flat and is also conformally flat in $\{ s > 2R\}$ with conformal factor
\begin{equation}
	\varphi= \varphi^R\bigg( 1+ \frac{\hat{E}_{ADM}^u}{2r^{n-2}}    \bigg).
\end{equation}
The ADM energy of $\bar{g}$ is $\bar{E}_{ADM} = \hat{E}_{ADM}^u + 2A^R$ (see calculation \eqref{EquationADMconformalChange}) so that if $R$ is chosen large enough the assertion about the energies also holds.

\end{proof}
	
\subsection{Undarning of the conical singularitites}\label{SubsectionConicalBarriers}
	
We need to adress the conical singularities, which we denote by $\{P_1, \ldots, P_\ell\}$. The result is due to Eichmair and we include the proof for completeness.
	
\begin{lemma}\label{LemmaCappingCylinders}
	There exists $0<w\in C^{2,\alpha}_{loc}(\hat{M}^n)$ such that $\Delta^{\bar{g} }w \leq 0 $ with strict inequality when $r$ is large, such that 
	\begin{equation}
		w = \frac{B}{r^{n-2}}+ \Ol_2(  r^{-(n-\epsilon)} ),  
	\end{equation}
	as $r\rightarrow \infty$, where $B$ is some constant, and such that for $C\geq 1$ we have 
	\begin{equation}
		\frac{1}{C}\frac{1}{u\varphi s^{n-2}} \leq w \leq \frac{C}{u\varphi s^{n-2}}
	\end{equation}
	as $s\rightarrow 0$.
\end{lemma}
	
\begin{proof}
	
We let $\sigma_0$ be as in Proposition \ref{PropositionConformalChange} and we recall that $\bar{g}=\varphi^{\frac{4}{n-2}}(u\Psi)^{\frac{4}{n-2}}\tilde{g}$ with vanishing scalar curvature $R_{\bar{g}}=0$ in $\{ s \leq 2\sigma_0\}$. A straightforwad computation shows that the metric $(\varphi u s^{n-2})^{-\frac{4}{n-2}}\bar{g}=s^{- 4}\Psi^{\frac{4}{n-2}}\tilde{g}$ has vanishing scalar curvature on $\{s\leq 2\sigma_0\}$ and hence  
\begin{equation}
	\Delta^{\bar{g}}\bigg( \frac{1}{\varphi us^{n-2}}\bigg)=0 
\end{equation}
in $\{ s \leq 2\sigma_0\}$. Fix a non-negative function $w_0 \in C^{2,\alpha}_{loc}(\hat{M}^n)$ that coincides with $(\varphi u s^{n-2})^{-1}$ in $\{ s\leq 2\sigma_0\}$ and such that $\text{supp}(w_0)\cap \{ s > 2\sigma_0\}$ is compact. oow, fix a non-negative function $q \in C^{2,\alpha}_{loc}(\hat{M}^n)$ such that $\text{supp}(q)\cap \{ s > 2\sigma_0\}=\emptyset$ and such that $q=r^{-2n}$ for large $r$. For $\sigma\in (0,\sigma_0)$, we consider the Dirichlet problem 
\begin{equation}\label{EquationWSigma}
	\begin{cases}
		-\Delta^{\bar{g}}(w_0+ w_\sigma) &=  q \qquad    \text{in}  \: S_\sigma, \\
		w_\sigma&=0   \qquad \:  \text{on}  \: \partial  S_\sigma.
	\end{cases}
\end{equation}
By the Maximum principle and the $\bar{g}$-harmonicity of $w_0$ we see that the homogeneous problem has only the trivial solution and hence by Fredholm alternative and elliptic regularity there exists a unique solution $w_\sigma\in C^{2,\alpha}_{loc}(S_\sigma)$. Furthermore $w_0+w_\sigma>0$ by the Maximum Principle. Let us extend $w_\sigma$ by zero to a Lipschitz function globally on $\hat{M}^n$. Similarly to the proof of Proposition \ref{PropositionConformalChange}, we get
\begin{equation}
	\begin{split}
		\frac{1}{C_1}\bigg(  \int_{\{ s\geq \sigma_0\} } |w_\sigma|_{\bar{g}}^{\frac{2n}{n-2}} d\mu^{\bar{g}}     \bigg)^{\frac{n-2}{n}} &\leq 	\bigg(  \int_{\{ s\geq \sigma_0\} } |dw_\sigma|_{\bar{g}}^2  d\mu_{\bar{g}}   \bigg)^{\frac{n-2}{n}} \\
		&\leq \int_{\{ \sigma \leq s \leq \sigma^{-1}\}} |dw_\sigma|_{\bar{g}}^2 d\mu^{\bar{g}} \\
		&=\int_{\{ \sigma \leq s  \leq \sigma^{-1}\}} w_\sigma (q + \Delta^{\bar{g}}w_0) d\mu^{\bar{g}} \\
		&\leq \int_{\{   s \geq \sigma_0 \}} |w_\sigma| |q + \Delta^{\bar{g}}w_0| d\mu^{\bar{g}} \\
		&=\bigg(\int_{\{   s  \geq \sigma_0 \}} |w_\sigma|_{\bar{g}}^{\frac{2n}{n-2}}  d\mu^{\bar{g}} \bigg)^{\frac{n-2}{2n}} \\
		&\qquad \times \bigg( \int_{\{   s  \geq \sigma_0 \}} |q+\Delta^{\bar{g}}w_0 | d\mu^{\bar{g}} \bigg)^{\frac{n+2}{2n}}.
	\end{split}
\end{equation}
The constants $C_1$ and $C_2$ do not depend on $\sigma$. From the decay of $q$ and compact support of $w_\sigma$ we get finiteness of the last integral and hence
\begin{equation}
	\int_{\{s  \geq \sigma_0\}}|w_\sigma|^{\frac{2n}{n-2}} d\mu^{\bar{g}} \leq C, 
\end{equation}
where $C$ does not depend on $\sigma$. It follows, as in the proof of Proposition \ref{PropositionConformalChange}, that we get a uniform $C^{2,\alpha}_{loc}$-bound on $\{s \geq 2\sigma_0\}$ and a standard diagonalization argument shows that we may further take a subsequential limit $w_0+w_{\sigma_k} \rightarrow w= w_0 + \lim_{k\rightarrow \infty} w_{\sigma_k}\in C^{2,\alpha}_{loc}(\hat{M}^n)$ that solves $-\Delta^{\bar{g}} w=q$. Since $w$ is subharmonic and non-negative it follows by the Hopf Maximum principle that $w>0$. Since $\varphi u$ is bounded below and above by positive constants on $\{0 <s\leq 2\sigma_0\}$ we conclude that the same follows for $w$ as $s$ is small.

\end{proof}

\section{The positive mass theorem}\label{SectionPositivity}

In this section we show the positive mass assertion $E\geq \vec{P}$ and the rigidity statement that $E=0$ only if $(M^n, g, k)$ is initial data for Minkowski space.  

\subsection{Positivity; $E\geq |P|$}
 
In this section we prove the following result for asymptotically hyperbolic initial data $(M^n, g,k)$ as in Definition \ref{DefinitionAHinitialData}.
  
\begin{theorem}\label{TheoremPositiveMass}
	Let $(M^n,g,k)$ be initial data of type $(\ell, \alpha, \tau, \tau_0)$, where $4\leq n \leq 7$, $\ell\geq 6$, $0<\alpha <1$, $\frac{n}{2}<\tau <n$ and $\tau_0>0$. If the dominant energy condition $\mu\geq |J|_g$ holds, then the mass vector is future pointing causal; $E\geq |\vec{P}|$.
\end{theorem}
  
\begin{proof}
	
By Theorem \ref{TheoremDensity} we can assume that our initial data has Wang's asymptotics as in Definition \ref{DefinitionWangAsymptotics} and is of type $(\ell- 1, \alpha, \tau=n, \tau_0')$, for some $\tau_0'>0$, with a strict dominant energy condition $\mu> |J|_g$ satisfied. We form the Riemannian product $( M^n \times \rn, g+dt^2)$ and solve Jang's equation $J(f)=0$ for a function $f$ defined on its domain $U_f\subset M^n$ as in Proposition \ref{PropositionJangLimit}. After the deformations in \ref{SubSectionGraphTopology} we obtain the graphical $(\hat{M}^n, \tilde{g}_\Psi)$ which is asymptotically flat by Proposition \ref{PropositionJangGraphIsAE} and has integrable scalar curvature $R_{\hat{g}}$ at the asymptotically flat infinity $\hat{N}^n$. As in Section \ref{SectionConformalChanges} we denote the ADM energy of this end by $\hat{E}_{ADM}$. Performing the conformal deformation of the metric $\tilde{g}_\Psi$ to the metric $\tilde{g}_{u \Psi}$ (with vanishing scalar curvature) we get a new energy $\hat{E}^u_{ADM}$. By the discussion at the end of Subsection \ref{SubsectionConformalChange} we know that $\hat{E}_{ADM}\leq E$, where $E$ is the zeroth component of the mass vector of $(M^n, g)$ computed in Proposition \ref{PropositionWangMassVector}. This metric is, however, not asymptotically Schwarzschildean as in Definition \ref{DefinitionAFinitialData} and so we deform $\tilde{g}_{u \Psi}$ to $\bar{g}$ using Theorem \ref{TheoremConformalChangeToAS} and the new ADM energy $\bar{E}_{ADM}$ is arbitrarily close to $\hat{E}_{ADM}^u$. The new metric is conformally flat sufficiently close to infinity: $\bar{g}=\varphi^{\frac{4}{n-2}} \delta$, where $\varphi$ has asymptotics as stated in the same Theorem, and $\bar{g}$ has vanishing scalar curvature. 
\\ \indent As in the proof of \cite[Proposition 14]{EichmairPMT}, we let $\epsilon>0$ be small and consider the metric $g^\epsilon=(1+\epsilon w)^{\frac{4}{n-2}}\overline{g}$, where $w$ is the solution in Lemma \ref{LemmaCappingCylinders}. Since $\overline{g}$ is scalar flat, it follows that $R_\epsilon=R_{g^\epsilon}\geq 0$ from the subharmonicity of $w$, and $R_\epsilon >0$ for large $r$. Near infinity, the metric then has the asymptotic form 
\begin{equation}
	\begin{split}
		g^\epsilon_{ij}&=(1+\epsilon w)^{\frac{4}{n-2}}\varphi^{\frac{4}{n-2}}\delta_{ij} \\
		&=\bigg( 1+\frac{ \bar{E}_{ADM} +\epsilon 2 B  }{2r^{(n-2)}}    \bigg)^{\frac{4}{n-2}}\delta_{ij} + \Ol_2(r^{-(n-1)}).
	\end{split}
\end{equation}
On each component $\tilde{C}_i$ the metric $g^\epsilon$ is uniformly equivalent to the metric $\sigma_i^2\gamma_i+ d\sigma^2_i$, where $\sigma_i=s(x)^{\frac{4}{n-2}}$ on $C_i$ (cf. Lemma \ref{LemmaPsiProps}). Clearly, $g^\epsilon$ is complete. We may now apply the positive energy theorem from \cite{PMTIV} (see also \cite[Proposition 14]{EichmairPMT} for an explaination why the proof applies to $g^\epsilon$ ) to get
\begin{equation}
	\bar{E}_{ADM} + 2\epsilon B \geq 0.
\end{equation}
Since $\epsilon$ was arbitrary it follows that $\bar{E}_{ADM}\geq 0$. 
\\ \indent It remains only to establish the causality condition. As in \cite{SakovichPMTah} we use the equivariance condition stated in Subsection \ref{SubSectionIDsetsMasses}; we may change coordinates by the boosts on Minkowski spacetime and from the equivariance of the mass vector under such boosts it follows that 
\begin{equation}
	E'= \frac{E - \theta |\vec{P}|}{\sqrt{1-\theta^2}},
\end{equation}
where $\theta\in (0,1)$. Thus, the assumption that $0\leq E < |\vec{P}|$ leads to a contradiction for choices $\theta \in \big( \frac{E}{|\vec{P}|}, 1 \big)$, which would imply $E'<0$. 
	
\end{proof}
 
\subsection{Rigidity; $E=0$}
 

We finally turn to the question of rigidity. Our result does not quite appear to be optimal for the same reasons as in \cite[Remark 9.2]{SakovichPMTah}. Firstly, we need to use Wang's asymptotics because of the lack of barriers in the general case. Further, the Jang equation does not allow for the use of the full mass vector but only its zeroth component $E$. 
 
\begin{theorem}\label{TheoremRigidity}
	
Let $(M^n,g,k)$ be initial data as in Theorem \ref{TheoremPositiveMass}. If $(M^n,g,k)$ has Wang's asymptotics and $E=0$, then $(M^n,g)$ embeds isometrically into Minkowski space $\M^{n+1}$ as a spacelike graphical hypersurface with $k$ as its second fundamental form.
	
\end{theorem}

\begin{proof}
	
We suitably modify the proof of \cite[Proposition 15]{EichmairPMT}, which builds upon the ideas from Schoen and Yau in \cite{PMTIII}.
\\ \indent We let $(M^n,g^j, k^j)$ be a sequence of initial data sets with Wang's asymptotics satisfying the strict dominant energy condition that approximates $(M^n,g,k)$, taken from Theorem \ref{TheoremDensity}. The energies converge, $E^j\rightarrow E=0$, as a consequence of the continuity of the mass functional. Let $(\hat{M}_j^n, \hat{g}^j)\subset (M^n \times \rn , g^j + dt^2)$ be the outermost in $M^n$ graphical parts of the associated Jang deformations constructed in Sections \ref{SectionBarriers} through \ref{SectionJangSolution} with graphing functions $f^j$. The convergence of $(g^j, k^j)\rightarrow (g,k)$ assures a uniform supremum bound on $|k^j|^2_{g^j}$ and in turn a uniform bound on the mean curvatures of the graphs from \eqref{EquationMeanCurvIsBounded}. It follows that we have $C^{3,\alpha}_{loc}$-smooth convergence to a geometric limit $(\hat{M}^n, \hat{g})$ and smooth convergence of the boundary components $\partial U_{f^j}\rightarrow \partial U_f$. The geometric limit $(\hat{M}^n, \hat{g})$ is the graphical component over some open domain $U_f\subset M^n$ with graphing function $f$ that solves Jang's equation.
\\ \indent We need to describe the asymptotics of $f$. We denote for brevity the right hand side of the equation that $\alpha$ satisfies by $\textbf{M}$:
\begin{equation}
	\textbf{M} =  \bigg(\frac{n-2}{2} \bigg)\trace_\Omega (\textbf{m}) + \trace_\Omega(\textbf{p}),
\end{equation}
so that $\Delta^\Omega(\alpha) -(n-3)\alpha= \textbf{M}$, and similarly we define $\textbf{M}_j$ for $\alpha_j$. We note that from Lemma \ref{PropositionWangMassVector} the mean values of $\textbf{M}$ and $\textbf{M}_j$ on $\mathbb{S}^{n-1}$ are $E(n-1)\omega_{n-1}$ and $E^j(n-1)\omega_{n-1}$, respectively. Since the kernel of the linear operator $L(\alpha)= \Delta^\Omega \alpha - (n-3)\alpha$ is trivial and $L(\alpha_j- \alpha)= \textbf{M}_j-\textbf{M}$, we get from strong $L^q$-regularity \cite[Theorem 27]{Besse} that for any $1<q<\infty$
\begin{equation}
	||\alpha_j - \alpha||_{W^{2,q}(\mathbb{S}^{n-1})} \leq C ||\textbf{M}_j-\textbf{M}||_{L^q(\mathbb{S}^{n-1})}.
\end{equation}
Moreover, $\textbf{M}_k-\textbf{M}$ converge uniformly to zero on $\mathbb{S}^{n-1}$ (see \cite{DahlSakovichDensityThm}) and so it follows that we have convergence $\alpha_j \rightarrow \alpha$ in $W^{2,q}(\mathbb{S}^{n-1})$. From the Morrey embedding and Schauder estimates it follows that $\alpha_j \rightarrow \alpha$ in $C^{3,\alpha}(\mathbb{S}^{n-1})$. From this convergence and from the arguments in Section \ref{SectionBarriers} it follows that there exists some $R>0$, uniform in $j$, such that $f^j_\pm=\sqrt{1+r^2} + \alpha_j r^{-(n-3)} + \Ol(r^{-(n-2-\epsilon)})$ are defined on $\{ r>R\}\subset M^n$, where the $\Ol$-term does not depend on $j$. It follows that the barriers $f_\pm$ of the Jang graph over $(M^n, g, k)$ have the same asymptotics. Arguing as in Section \ref{SectionJangAE} we can assert that the metric $\hat{g}=g + df \otimes df$ is asymptotically Euclidean as in Corollary \ref{CorollaryJangGraphAF}. From Propositions \ref{PropositionWangMassVector} and \ref{PropositionJangGraphADMmass} we know that the ADM energies converge to zero: $\hat{E}^j_{ADM}= (n-1)E^j\rightarrow E=0$.
\\ \indent We know from Section \ref{SubSectionGraphTopology} that the components in $\partial U_{f^j}$ have all positive Yamabe type due to the strict dominant energy condition. However, the limit may have components of zero Yamabe type. 
\\ \indent We let $t_0^j \rightarrow \infty$ be a sequence such that $\pm t_0^j$ are regular values for both $f$ and $f^j$ for all $j$. We let $\tilde{g}^j$ be metrics as in Subsection \ref{SubSectionGraphTopology} such that $\tilde{g}^j=\hat{g}^j$ on $\hat{M}_j^n\cap (M^n\times (-t_0^j, t_0^j))$. Let $u^j\in C^{2,\alpha}_{loc}(\hat{M}^n_j)$ be the solutions to $-\Delta^{\tilde{g}^j}u^j+ c_nR_{\tilde{g}^j}u^j=0$ from Proposition \ref{PropositionConformalChange}. The ADM energy of $\tilde{g}^j_{u^j }$ is $\hat{E}^j_{ADM}+ 2A^j$, where $A^j\leq -2c_n\hat{E}_{ADM}^j=-2c_n(n-1)E^j\leq0$ is the coefficient in the expansion of $u^j$ in the infinity $\hat{N}^n_j$. Using Theorem \ref{TheoremConformalChangeToAS}, Lemma \ref{LemmaCappingCylinders} and the proof of Theorem \ref{TheoremPositiveMass} we find that $\hat{E}^j_{ADM}+2A^j\geq 0$. Since $\hat{E}^j_{ADM}\rightarrow 0$ it follows that $A^j\rightarrow 0$. From the equation that $u^j$ satisfies together with the Sobolev inequality it follows that $u^j\rightarrow 1$ uniformly as $r\rightarrow \infty$. Standard elliptic regularity then shows that $u^j$ converges to the constant function $u\equiv 1$ on $\hat{M}^n$. Hence, $R_{\hat{g}}=0$ and from Lemma \ref{LemmaPsiProps} it follows that $\hat{A}=k$ on $\hat{M}^n$.
\\ \indent We view the Riemannian manifold $(\hat{M}^n,\hat{g})$ as a Riemannian initial data set with vanishing energy; $\hat{E}(n-1)E=0$. Let $s\in C^{3,\alpha}_{loc}(\hat{M}^n)$ be a positive distance function that agrees with the coordinate distance $r$ for $r>2r_0$ and such that $s(p) = |t|^{-1}$ for $|t|$ large.
\\ \indent We now show that $\text{Ric}_{\hat{g}}=0$ using the variational argument of Schoen and Yau in \cite{PMTI}, following the proof of \cite[Proposition 16]{EichmairPMT}. We let $h\in C^{2,\alpha}_{c}(\text{Sym}^2(T^\ast\hat{M}^n))$ be a compactly supported symmetric $(0,2)$-tensor and for small values of $\kappa$, we consider the metric $\hat{g}_\kappa=\hat{g} + \kappa h$. Let $\sigma_0$ be small so that for all $\sigma\in (0,\sigma_0)$, both $\sigma$ and $\sigma^{-1}$ are regular values of $s$. Let $0\leq q\in C^{2,\alpha}(\hat{M}^n)$ be a function that coincides with $r^{-2n}$ on $\{r>2r_0\}$ and such that $\text{supp}(q)\cap \{s<\sigma_0\}=\emptyset$. For $\sigma\in (0,\sigma_0)$ and sufficiently small $\kappa$, we consider the mixed Dirichlet/Neumann problem
\begin{equation} \label{EquationMixed}
	\begin{cases}
		-\Delta_{\hat{g}_{\kappa}}u_{\kappa, \sigma} + c_n R_{\hat{g}_\kappa}u_{\kappa, \sigma} &= \kappa^2 q  \quad \: \qquad  \text{on} 	\qquad S_\sigma   \\
		\vec{n}(u_{\kappa, \sigma}) &= 0 \qquad \qquad   \: \text{on} \qquad \partial^- S_\sigma \\
		u_{\kappa, \sigma} &= 1 \qquad \qquad   \: \text{on} \qquad \partial^+ S_\sigma ,
	\end{cases}
\end{equation}
where $S_\sigma= \{\sigma \leq s \leq \sigma^{-1}\}$ and $\partial S^+_\sigma=\{ s=\sigma \}$ and $\partial S^-_\sigma = \{ s=\sigma^{-1}  \}$. To solve \eqref{EquationMixed}, by the Fredholm theory, it suffices to show that the homogeneous problem has only the trivial solution. Indeed, if $w_\sigma$ solves the homogeneous problem, then we can multiply the equation that $w_\sigma$ satisfies by $w_\sigma$, integrate over $S_\sigma$, use the Sobolev Inequality in Lemma \ref{LemmaSobolevInequality} on $(\hat{M}^n\cap \{s\geq \sigma_0\}), \hat{g}_\kappa )$ which yields  
\begin{equation}
1\leq C \bigg(  \int_{ \{  \sigma_0 < s(x) < \sigma^{-1} \} }  |R_{\hat{g}_\kappa}|^{\frac{n}{2}} d\mu_{\hat{g}_\kappa}    \bigg)^{\frac{n}{2}}
\end{equation}
as in the proof of Theorem \ref{TheoremConformalChangeToAS}. Since $||R_{\hat{g}_\kappa}||_{L^{\frac{n}{2}}}=\Ol(|\kappa|)$, we obtain a contradiction, which implies $w_\sigma=0$.
\\ \indent Decomposing $v_{\kappa, \sigma}= 1- u_{\kappa, \sigma}$, we may perform a similar argument to get
\begin{equation}
	\begin{split}
		\bigg(  \int_{ \{ \sigma_0 < s < \sigma^{-1} \} } &|v_{\kappa, \sigma}|^{\frac{2n}{n-2}}  d\mu_{\hat{g}_\kappa}  	\bigg)^{\frac{n-2}{n}} \\ &\leq C \int_{\{ \sigma < s < \sigma^{-1}\}}  |dv_{\kappa, \sigma}|^2_{\hat{g}_\kappa} d\mu_{\hat{g}_\kappa}    \\
		&=C \int_{\{ \sigma < s < \sigma^{-1}  \}} \bigg(-c_n R_{\hat{g}_\kappa } v^2_{\kappa, \sigma} + (\kappa^2 q - c_n 	R_{\hat{g}_\kappa})v_{\kappa, \sigma } \bigg) d\mu_{\hat{g}_\kappa} \\
		&\leq C \bigg(  \int_{\{ \sigma_0 < s < \sigma^{-1}  \}} |R_{\hat{g}_\kappa}|^\frac{n}{2}  d\mu_{\hat{g}_\kappa} \bigg)^\frac{n}{2} 	\bigg(   \int_{\{ \sigma_0 < s < \sigma^{-1} \}}  |v_{\kappa, \sigma}|^{\frac{2n}{n-2}}_{\hat{g}_\kappa} d\mu_{\hat{g}_\kappa} \bigg)^\frac{ n-2 }{n} \\
		&\qquad + C\bigg(  \int_{\{ \sigma_0 < s < \sigma^{-1} \}}  |\kappa^2 q - c_n R_{\hat{g}_\kappa}|^{\frac{2n}{n+2}} d\mu_{\hat{g}_\kappa} 	\bigg)^\frac{n+2}{2n}  \bigg(   \int_{ \{ \sigma_0 < s < \sigma^{-1}\}} |v_{\kappa, \sigma}|^{\frac{2n}{n-2}} d\mu_{\hat{g}_\kappa}     \bigg)^\frac{n-2}{2n}.
	\end{split}
\end{equation}
In turn, this implies that $||v_{\kappa, \sigma}||_{L^\frac{2n}{n-2}( \{ \sigma_0 < s < \sigma^{-1}  \}  ) }= \Ol(|\kappa|)$. Standard elliptic theory applied as in the proof of Proposition \ref{PropositionConformalChange} implies the same bound $||v_{\kappa, \sigma}||_{C^{2,\alpha}(\{ \sigma_0 < s < \sigma^{-1}  \}) } = \Ol(|\kappa|)$, where the $\Ol$-term does not depend on $\sigma$. Moreover, arguing as in the proof of Proposition \ref{PropositionConformalChange} that $v_{\kappa, \sigma}$ is $\hat{g}_\kappa$-harmonic on $\{ \sigma < s < \sigma_0\}$ and using the Maximum principle we obtain the same bound for the $L^\infty$-norm of $v_{\kappa, \sigma}$ on $\{\sigma < s<\sigma_0\}$. A standard bootstrapping argument implies $||u_{\kappa, \sigma}-1||_{C^{2,\alpha}(\hat{M}^n)}=\Ol(|\kappa|)$.
\\ \indent As in the proof of Proposition \ref{PropositionConformalChange}, we take a subsequential limit as $\sigma \rightarrow 0$ for a global solution $u_\kappa$ on $\hat{M}^n$. The previous estimate on $v_{\kappa, \sigma}$ is uniform in $\sigma$ so $||u_{\kappa}-1||_{C^{2,\alpha}(\hat{M}^n)}=\Ol(|\kappa|)$ ,$-\Delta_{\hat{g}_\kappa}u_{\kappa} + c_nR_{\hat{g}_\kappa}u_\kappa = \kappa^2q$, and $u\rightarrow 1$ as $r\rightarrow \infty$. Since each $u_{\kappa, \sigma}$ is harmonic on $\{\sigma < s < \sigma_0\}$ and satisfies the Neumann boundary condition $\vec{n}(u_\kappa)=0$ on $\partial S_\sigma^+$ it follows from the divergence theorem and the $C^{2,\alpha}$-smooth connvergence that 
\begin{equation}
	\int_{\{  s=\sigma_0 \}} \vec{n}(u_\kappa) d\mu_{\hat{g}_\kappa} = 0
\end{equation}
and an asymptotic analysis as in the proof of Proposition \ref{PropositionConformalChange} yields the fall-off 
\begin{equation}
	u_\kappa = 1 + \frac{A_\kappa}{r^{n-2} } + \Ol(r^{-(n-1-\epsilon)}), 
\end{equation}
for large $r$. Clearly, for $\kappa=0$ we have $\hat{g}_0=\hat{g}$, $R_{\hat{g}_0}=0$ and $u_0=1$ and further since $||u_\kappa - 1||_{C^{2,\alpha}(\hat{M}^n)}=\Ol(|\kappa|)$ it follows that $A_\kappa$ is differentiable at $\kappa=0$. An integration by parts as in the proof of Proposition \ref{PropositionConformalChange} yields
\begin{equation}
	4(n-1)\omega_{n-1}  A_\kappa = \int_{\hat{M}^n} \big( \kappa^2 q - c_n R_{\hat{g}_\kappa}u_\kappa   \big) u_\kappa d\mu_{\hat{g}_\kappa}.
\end{equation}
It follows that
\begin{equation}
	\begin{split}
		4(n-1)\omega_{n-1}\frac{d}{d\kappa}A_\kappa \bigg|_{\kappa=0} &= - \int_{\hat{M}^n} \frac{d}{d\kappa} \bigg|_{\kappa=0} 	R_{\hat{g}_\kappa} d\mu^{\hat{g}_\kappa} \\
		&= \int_{\hat{M}^n} \bigg( \Delta_{\hat{g}} \trace_{\hat{g}}(h)- \diver_{\hat{g}}\diver_{\hat{g}} (h) + \langle h, 	\text{Ric}_{\hat{g}}\rangle_{\hat{g}} \bigg) d\mu^{\hat{g}} \\
		&=\int_{\hat{M}^n} \langle h, \text{Ric}_{\hat{g}}\rangle_{\hat{g}} d\mu^{\hat{g}},
	\end{split}
\end{equation}
where the two first terms vanish due to the compact support of $h$ together with the divergence theorem. Since the scalar curvature of $u_\kappa^{\frac{4}{n-2}}\hat{g}_\kappa$ is non-negative everywhere and positive for $r>2r_0$, for $\kappa\not= 0$, we have from Theorem \ref{TheoremPositiveMass} that the ADM energy of $u_\kappa^\frac{4}{n-2} \hat{g}_{\kappa}$ is non-negative. From the expansion of $u_\kappa$ and the fact that $\hat{g}$ has vanishing ADM energy, it follows that the energy is $\frac{n-2}{2}A_\kappa$. Hence $\frac{d}{d\kappa}A_\kappa \big|_{\kappa =0 }=0$ with
\begin{equation}
	\int_{\hat{M}^n} \langle h, \text{Ric}_{\hat{g}}\rangle_{\hat{g}} d\mu_{\hat{g}}=0
\end{equation}
as a consequence.
\\ \indent Now, take an arbitrary coordinate chart and let $\chi \in C_c^{3,\alpha}(\hat{M}^n)$ be a non-negative function supported in the chart. Let $h_k \in C^{2,\alpha}_0(\text{Sym}^2 (T^\ast \hat{M}^n))$ be a sequence that approximates $\chi \text{Ric}_{\hat{g}}$ in $C^{0,\alpha}(\hat{M}^n)$. Passing to the limit in the integral above it follows that
\begin{equation}
	\int_{\hat{M}^n} \chi |\text{Ric}_{\hat{g}}|_{\hat{g}}^2 d\mu_{\hat{g}} = 0,
\end{equation}
and hence $\text{Ric}_{\hat{g}}\equiv 0$ identically. 
\\ \indent We now rule out cylindrical ends and establish the isometry to Euclidean space. If there are cylindrical ends we can construct geodesic lines going to this infinity (see \cite[Chapter 3]{Peterson}). Since $\text{Ric}_{\hat{g}}=0$ the Cheeger-Gromoll theorem applies and so $\hat{M}^n$ must split off a factor $\rn$ isometrically, which contradicts the metric fall-off properties. Hence $\hat{M}^n$ has no cylindrical ends so that $U_f=M^n$.
\\ \indent By the Bishop-Gromov volume comparison theorem we have that the density quotient 
\begin{equation}
r \rightarrow \frac{\mu_{\hat{g}}(B_r(p))}{r^n\omega_{n-1}}
\end{equation}
is non-increasing for any $p\in \hat{M}^n$ and tends to $1$ as $r\rightarrow 0$. Explicit computations show that the quantity converges to $1$ as $r\rightarrow \infty$ and so it is constantly equal to $1$. As a consequence, $(\hat{M}^n, \hat{g})\simeq (\rn^n, \delta)$.
\\ \indent It only remains to construct the embedding in the Minkowski space. We identify the graph of $f:M^n\rightarrow \rn$ with $M^n$ diffeomorphically via projection. This allows us to view $f$ as a function on its own graph $\hat{M}^n$. By the above $\hat{g}=\delta$ and so $g = \delta  - df\otimes df$, which is the metric induced on the graph of a function $f:\rn^n\rightarrow \rn$ in Minkowski space $M^{n+1}=(\rn^{n+1}, - dt^2 + \delta)$. With this at hand it is easy to check that both $(1+|df|^2_g)=(1-|df|^2_\delta)^{-1}$ and $\sqrt{1+|df|^2_\delta}\Hess^\delta_{ij}(f) = \Hess^g_{ij} (f)$. It follows that
\begin{equation}
	\frac{\Hess^g_{ij} (f)}{\sqrt{1+|df|^2_g}} = \frac{\Hess^\delta_{ij} (f)}{\sqrt{ 1- |df|^2_\delta}},
\end{equation}
where the left hand side is the second fundamental form $\hat{A}$ of the graph of $f:M^n\rightarrow \rn$ in $(M^n\times \rn, g + dt^2)$ and the right hand side is the second fundamental form of the graph of $f:\rn^n\rightarrow \rn$ in Minkowskispace. Since $\hat{A}=k$ this completes the proof.


\end{proof}

\appendix \label{SectionAppendix}

\section{Computations for Wang's asymptotics}

This appendix contains some elementary computations for asymptotically hyperbolic initial data $(M^n, g,k)$ with Wang's asymptotics as in Definition \ref{DefinitionWangAsymptotics}. Indices are raised with the hyperbolic metric $b$, the standard metric on the unit sphere is denoted by $\Omega$ and we recall that the chart is supressed for convenience, so that for instance we write $\Psi_\ast (g) = g$. 

\begin{lemma}\label{LemmaWangGeometry}
Let $(M^n,g,k)$ be asymptotically hyperbolic initial data of type $(\ell, \alpha, \tau=n, \tau_0)$ with Wang's asymptotics as in Definition \ref{DefinitionWangAsymptotics}. Then $\Gamma^r_{r \mu}=0$, $\Gamma^\mu_{rr}=0$ and
\begin{equation}
	\begin{split}
		\Gamma^r_{rr} &= -\frac{r}{1+r^2}, \\
		\Gamma^r_{\mu \nu} &= -\frac{1}{2}(1+r^2)\bigg( \frac{2}{r}b_{\mu\nu } - (n-2)\frac{\textbf{m}_{\mu\nu}}{r^{n-1}} + \Ol  ( r^{-n}  )   	\bigg), \\
		\Gamma^\mu_{r\nu} &= \frac{\delta^\mu_\nu}{r} - \frac{n}{2}\frac{\textbf{m}_\nu^\mu}{r^{n-1}} + \Ol  ( r^{-(n+2)} ), \\
		\Gamma^\sigma_{\mu\nu} &= \frac{1}{2}b^{\rho \sigma}\bigg(  b_{\rho \mu, \nu} + b_{\rho \nu, \mu} -b_{ \mu \nu, \rho}    \bigg) - 	\frac{1}{2}\frac{\textbf{m}^{\rho \sigma}}{r^{n-2}} \bigg(   b_{\rho \mu, \nu} + b_{\rho \nu, \mu} -b_{ \mu \nu, \rho}    \bigg) \\
		&\qquad + \frac{1}{2}  \frac{b^{\sigma \rho}}{r^{n+2}}  \bigg(   \textbf{m}_{\rho \mu, \nu} + \textbf{m}_{\rho \nu, \mu} -\textbf{m}_{ 	\mu \nu, \rho}    \bigg)+\Ol  ( r^{-(n+1)}  ).
	\end{split}
\end{equation}
Furthermore, the Ricci tensor $\text{Ric}^g$ has components
\begin{equation}
	\begin{split}
		\text{Ric}^g_{rr}&= -\frac{(n-1)}{1+r^2} - \frac{n(n+3)}{2} \frac{\trace_{\Omega}(\textbf{m})}{r^{n+2}} + \Ol^{\ell-2, \alpha} ( 	r^{-(n+2+\epsilon)} ), \\
		\text{Ric}_{r\mu}^g&= \frac{n}{2}\bigg(  \frac{\trace_{\Omega}(\textbf{m})_{,\mu}}{r^{n+1}}- \frac{\textbf{m}_{\mu, 	\nu}^\nu}{r^{n-1}}    \bigg) + \frac{n}{2}\bigg(   \frac{\Gamma^\nu_{\rho \mu}\textbf{m}^\rho_\nu - \Gamma^\nu_{\nu\rho}\textbf{m}^\rho_\mu}{r^{n-1}}     \bigg) +\Ol^{\ell-2, \alpha} ( r^{-(n+2)} ), \\
		\text{Ric}_{\mu\nu}^g&= -(n-1)b_{\mu\nu} + \Ol(1).
	\end{split}
\end{equation}
\end{lemma}
	
In Proposition \ref{PropositionWangMassVector} we compute the mass vector in Definition \ref{DefinitionMassFunctional}.

\begin{proposition}\label{PropositionWangMassVector}
Let $(M^n,g,k)$ be asymptotically hyperbolic initial data of type $(\ell, \alpha, \tau=n, \tau_0)$ with Wang's asymptotics as in Definition \ref{DefinitionWangAsymptotics}. Then the components of the mass vector $(E,\vec{P})$ are
\begin{equation}
	E=\frac{1}{(n-1)\omega_{n-1}}\int_{\mathbb{S}^{n-1}} \bigg( \trace_{\Omega}(\textbf{p}) +\bigg( \frac{n-2}{2}   \bigg)\trace_{\Omega}(\textbf{m})     \bigg)dS
\end{equation}
and
\begin{equation}
	P^i=\frac{1}{(n-1)\omega_{n-1}}\int_{\mathbb{S}^{n-1}}  \bigg( \trace_{\Omega}(\textbf{p}) +\bigg( \frac{n-2}{2}  \bigg)\trace_{\Omega}(\textbf{m})     \bigg)x^idS,
\end{equation}
where the $x^i$ are the coordinate functions from $\rn^n$ to $\mathbb{S}^{n-1}$ and $\omega_{n-1}=|\mathbb{S}^{n-1}|_\delta$.
\end{proposition}



\begin{proof}

We first calculate $E$ using Definition \ref{DefinitionMassFunctional}. Clearly the only non-zero components of $e=g-b$ are
\begin{equation}
	e_{\mu \nu}=g_{\mu \nu}-b_{\mu \nu} = \frac{\textbf{m}_{\mu \nu}}{r^{n-2}}+ R_{\mu\nu},
\end{equation}
where $R_{\mu\nu}$ is a function that falls off as $\Ol(r^{-\tau})$ and with derivatives falling off as $\partial^k_r \partial^\ell_\mu R_{\mu\nu}= \Ol(r^{-(\tau + k)})$. The unit normal with respect to the coordinate sphere of radius $R$ is $\vec{n}^r=\sqrt{1+r^2}\partial_r$ and so we need only the $r$-component of the $1$-form appearing in \eqref{EquationMassFunctional}. The radial component of $\diver^b(e)$ is
\begin{equation}
	\begin{split}
		\diver^{b}(e)_r&=   b^{ij}(\nabla_i e)_{rj} \\
		&= b^{\mu \nu}(\nabla_\mu e)_{r\nu } \\ 
		&=b^{\mu\nu}\bigg(   e_{r\nu ,\mu }-e_{m\mu }\Gamma_{r\nu }^m-e_{rm}\Gamma^m_{\mu \nu }   \bigg) \\
		&=-b^{\mu\nu}e_{\sigma\mu }\Gamma_{r\nu }^\sigma \\
		&=-b^{\mu\nu}\bigg(  \frac{\textbf{m}_{\sigma \mu} }{r^{n-2}}  +\Ol ( r^{-(n-1)} )     \bigg)\bigg( \frac{\delta^\sigma_\nu}{r} - 	\frac{n}{2}\frac{\textbf{m}_\nu^\sigma}{r^{n-1}} + \Ol ( r^{-(n+2)} )\bigg) \\ 
		&=- \frac{\trace_{\Omega}(\textbf{m})}{r^{n+1}} +\Ol ( r^{-(n+2)} ),
	\end{split}
\end{equation}
where we used Christoffel symbols from Lemma \ref{LemmaWangGeometry}. Similarly, since $(\langle b, e \rangle_b  )_{,r}= \langle b, \nabla_r b\rangle_b$, we have
\begin{equation}
	\begin{split}
		d \trace^b(e)_r&= \trace^b(e)_{,r} \\
		&=\langle b, \nabla_r b\rangle_b\\
		&= b^{ij} (\nabla_r b)_{ij}  \\
		&=b^{ij} \big( e_{ij,r} - \Gamma_{ri}^\ell e_{\ell j}  - \Gamma^\ell_{rj} e_{i\ell}     \big) \\
		&=b^{\mu \nu} \big( e_{\mu \nu,r} - \Gamma_{r\nu}^\rho e_{\rho \mu}  - \Gamma_{r\mu}^\rho e_{\rho \nu}     \big) \\
		&= b^{\mu \nu} \big( -(n-2)\frac{\textbf{m}_{\mu\nu}}{r^{n-1}} - 2 \frac{e_{\mu\nu}}{r } + \Ol(r^{-n}) \\
		&=-n\frac{\trace_{\Omega}(\textbf{m})}{r^{n+1}} +\Ol ( r^{-(n+2)} ).
		%
	\end{split}
\end{equation}
It follows that, for $V_0=\sqrt{1+r^2}$, we have
\begin{equation}
	\begin{split}
		V_0 \big(\diver^b(e)_r  -d \trace^b(e)_r\big)	&=(n-1)\frac{\trace^b(\textbf{m})}{r^{n }} + \Ol ( r^{-(n+1)} ) .
	\end{split}
\end{equation}
Furthermore, since $dV_0 = \frac{r}{\sqrt{1+r^2}}dr$, we have
\begin{equation}
		\trace^b(e)d V_0  =\bigg(  \frac{\trace_{\Omega}(\textbf{m})}{r^{n}} +\Ol ( r^{-(n +1)} )    \bigg) dr  .
\end{equation}
Moreover, since $\nabla^bV_0 = r\sqrt{1+r^2}\partial_r$, we obtain
\begin{equation}
	\begin{split}
		(e+2\eta)(\nabla^b V_0,\cdot )_r &= (e+2\eta)(\nabla^b V_0,\partial_r ) \\
		&=(e+2\eta)_{rr}r\sqrt{1+r^2}   \\
		&=2\eta_{rr}r\sqrt{1+r^2} ,  
	\end{split}
\end{equation}
where 
\begin{equation}\label{EquationAux0000}
	\begin{split}
		\eta_{rr}&=(k-g)_{rr}-\trace_g(k-g)g_{rr} \\
		&=- g^{\mu\nu}(k-g)_{\mu\nu}    g_{rr}\\
		&=-\bigg( \frac{\trace_{\Omega}(\textbf{p})-\trace_{\Omega}(\textbf{m})}{r^n}   \bigg) \frac{1}{1+r^2}+  \Ol ( r^{-(n+3)} ) .
	\end{split}
\end{equation}
In summary, we obtain
\begin{equation}
	\begin{split}
		V_0 \big(\diver^b(e)_r&  -d \trace^b(e)_r\big)  + \trace^b(e)(d V_0)_r   -(e+2\eta)(\nabla^bV_0,\cdot )_r \\
		&\qquad= 2\bigg(    \frac{\trace_{\Omega}(\textbf{p})}{r^n}  + \bigg(\frac{n-2}{2}\bigg)\frac{\trace_{\Omega}(\textbf{m}) }{r^n}   \bigg) + \Ol(r^{-(n+1)}).
	\end{split}
\end{equation}
Recalling that $\vec{n}^r = \sqrt{1+r^2}\partial_r$ we conclude that
\begin{equation}
		E =\frac{1}{(n-1)\omega_{n-1}}\int_{\mathbb{S}^{n-1}}   \bigg( \trace_{\Omega}(\textbf{p}) +\bigg( \frac{n-2}{2}   \bigg)\trace_{\Omega}(\textbf{m})     \bigg)   dS \\
\end{equation}
as asserted.
\\ \indent We briefly comment on the proof of the expression for $P^i$. We have $V_i=x^ir$ and so the first terms in the charge integral will become
\begin{equation}
	V_i\bigg( \diver^b(e)_r  -d \trace^b(e)_r \bigg) =x^i\bigg( (n-1)\frac{\trace_{\Omega}(\textbf{m})}{r^{n}} +\Ol ( r^{-(n +1)} )    \bigg).
\end{equation}
Further, we have both $d V_i =x^idr+rx^i_{,\mu}dx^{\mu}$ and $\nabla^bV_i = (1+r^2)x^i\partial_r + b^{\mu\nu}rx^i_{,\nu}\partial_\mu$.	In turn, using \eqref{EquationAux0000}, we obtain
\begin{equation}
	\begin{split}
		(e+2\eta)(\nabla^b V_i,\partial_r)&=2\eta_{rr}(1+r^2)x^i \\
		&=\bigg( -2\bigg( \frac{\trace_{\Omega}(\textbf{p})-\trace_{\Omega}(\textbf{m})}{r^n}   \bigg) + \Ol ( r^{-(n +1)} )  \bigg)x^i.
	\end{split}
\end{equation}
Conclusively, we get
\begin{equation}
	\begin{split}
		V_i \big(\diver^b(e)_r&  -d \trace^b(e)_r\big)  + \trace^b(e)(d V_i)_r   -(e+2\eta)(\nabla^b V_i,\cdot )_r \\
		&\qquad= 2\bigg(    \frac{\trace_{\Omega}(\textbf{p})}{r^n}  + \bigg(\frac{n-2}{2}\bigg)\frac{\trace_{\Omega}(\textbf{m}) }{r^n}   \bigg)x^i + \Ol(r^{-(n+1)})
	\end{split}
\end{equation}
and so the assertion follows.

\end{proof}

\section{Geometry of a smooth approximate Jang graph}\label{SectionJangGraph}

We present some useful properties of the Jang graph obtained in Proposition \ref{PropositionJangLimit} with asymptotics as in Proposition \ref{PropositionBarrierExistence} under the extra assumption that the geometry is smooth and asymptotically flat as in Corollary \ref{CorollaryJangGraphAF}. Throughout this section we have an asymptotically hyperbolic initial data $(M^n,g,k)$ of type $(\ell=\infty, \alpha, \tau=n, \tau_0)$ with Wang's asymptotics as in Definition \ref{DefinitionWangAsymptotics} and a function $f:M^n\rightarrow \rn$ that is smooth and solves Jang's equation \emph{approximately} outside of a compact set, that is we have $\J(f)=\Ol(r^{-(n+1-\epsilon)})$. We recall from Section \ref{SectionBarriers} that such a function has the asymptotics
\begin{equation}\label{EquationJangAsymptotics}
	f = \sqrt{1+r^2} + \frac{\alpha}{r^{n-3}} + q(r, \theta), 
\end{equation}
where $q$ is a smooth function such that $\partial^k_r \partial^\ell_\mu q = \Ol(r^{-(n-2+k-\epsilon)})$. From Corollary \ref{CorollaryJangGraphAF} we know that the \emph{exact} solution to Jang's equation must also have the asymptotics of \eqref{EquationJangAsymptotics}. We let $\hat{M}^n$ denote the graph of $f$ in $M^n\times \rn$, and similarly we use hatted symbols for the geometric quantities, so that for instance the induced metric on $\hat{M}^n$ is $\hat{g}$ and the Christoffel symbols are denoted by $\hat{\Gamma}$. The projection diffeomorphism $\Pi:\hat{M}^n\rightarrow M^n$ pulls back vector fields $\Pi^\ast (\partial_i)= \partial_i + f_{,i} \partial_t$. Throughout the section we mean by $\Ol_\infty(r^{-\tau})$ a smooth function that falls off as $\Ol(r^{-\tau})$. 

\begin{lemma}\label{LemmaJangGraphMetric}
The induced metric $\hat{g}$ has the coordinate expression $\hat{g}_{ij} = g_{ij} + f_{,i}f_{,j}$ with the components:
\begin{equation}
	\begin{split}
		\hat{g}_{rr}&= 1-2(n-3)\frac{\alpha}{r^{n-2}}+\Ol_\infty (  r^{-(n-1-\epsilon)} ), \\
		\hat{g}_{\mu r}&=\frac{\alpha_{,\mu}}{r^{n-3}}+\Ol_\infty(  r^{-(n-2-\epsilon)} ), \\ 
		\hat{g}_{\mu \nu}&=\delta_{\mu \nu} + \frac{\textbf{m}_{\mu\nu}}{r^{n-2}}  +\Ol_\infty (  r^{-(n-1-\epsilon)}). 
	\end{split}
\end{equation}
The components of the inverse metric $\hat{g}^{ij} = g^{ij} - \frac{f^{,i}f^{,j}}{1+|df|_g^2}$ are:
\begin{equation}
	\begin{split}
		\hat{g}^{rr}&=1+2(n-3)\frac{\alpha}{r^{n-2}}+ \Ol_\infty (  r^{-(n-1-\epsilon)} ), \\
		\hat{g}^{\mu r}&= -\delta^{\mu\nu}\frac{\alpha_{,\nu}}{r^{n-3}} + \Ol_\infty (  r^{-(n-\epsilon)} ), \\
		\hat{g}^{\mu\nu}&=  \delta^{\mu \nu} + \Ol_\infty (  r^{-(n+1-\epsilon)} ). 
	\end{split}
\end{equation}
\end{lemma}

\begin{lemma}\label{LemmaJangGraphChristoffelSymbols}
	
The Christoffel symbols of $(\hat{M}^n, \hat{g})$ are:

\begin{equation}
	\begin{split}
		\hat{\Gamma}^r_{rr}&= (n-2)(n-3)\frac{\alpha}{r^{n-1}} +  \Ol_\infty(r^{-(n-\epsilon)}), \\
		\hat{\Gamma}^\mu_{rr}&= \Ol_\infty ( r^{-(n+1-\epsilon)} ), \\
		\hat{\Gamma}^r_{r\mu } &= -(n-2)\frac{\alpha_{,\mu}}{r^{n-2}} + \Ol_\infty ( r^{-(n-1-\epsilon)} ), \\
		\hat{\Gamma}^\mu_{r\nu} &=  \frac{\delta^\mu_\nu}{r} - \bigg(\frac{n-2}{2}\bigg) \frac{\delta^{\mu\rho}\textbf{m}_{\nu\rho}}{r^{n-1}}+ \Ol_\infty ( r^{-(n+2-\epsilon)} ), \\
		\hat{\Gamma}^r_{\mu\nu} &= -\frac{\delta_{\mu\nu}}{r} +\frac{\Hess^\Omega_{\mu\nu}(\alpha)}{r^{n-3}} +  \bigg(\frac{n-2}{2}\bigg)\frac{\textbf{m}_{\mu\nu}}{r^{n-1}} \\
		&\qquad -2(n-3)\frac{\alpha}{r^{n-1}}\delta_{\mu\nu} + \Ol_\infty ( r^{-(n- \epsilon)} ), \\
		\hat{\Gamma}^\rho_{\mu \nu} &=  \frac{1}{2}\delta^{\rho\sigma}\big( \delta_{\mu \sigma, \nu} + \delta_{\nu \sigma, \mu}  - \delta_{\mu \nu, \sigma}   \big) + \Ol_\infty( r^{-(n-1-\epsilon)}).
	\end{split}
\end{equation}
	
\end{lemma}

\begin{lemma}\label{LemmaJangGraphRicci}
	
The components of the Ricci tensor of $(\hat{M}^n, \hat{g})$ are:

\begin{equation}
	\begin{split}
		\text{Ric}^{\hat{g}}_{rr} &= - 2\frac{\Delta^\Omega \alpha}{r^n} +  n(n-1)(n-3)\frac{\alpha}{r^n} +  \Ol_\infty (r^{-(n+1-\epsilon)}), \\
		\text{Ric}^{\hat{g}}_{\mu r} &= -2(n-1)\frac{\alpha_{,\mu}}{r^{n-1}} + \Ol_\infty( r^{-(n-\epsilon)}), \\
		\text{Ric}^{\hat{g}}_{\mu\nu} &=  (n-1) \frac{\Hess_{\mu\nu}^\Omega \alpha}{r^{n-2}} +   \delta_{\mu\nu} \frac{\Delta^\Omega \alpha}{r^n} -   n(n-3) \frac{\alpha}{r^n} \delta_{\mu\nu} + \Ol_\infty (r^{-(n-1-\epsilon)}), \\
	\end{split}
\end{equation} 
In particular, we have 
\begin{equation}
	R_{\hat{g}} = 2(n-2)\frac{\Delta^\Omega \alpha}{r^n}   + \Ol_\infty(r^{-(n+1-\epsilon)}).
\end{equation}
	
\end{lemma} 

\begin{proof}

We assist the reader with estimates on the scalar curvature, as the assertions in this Lemma are proven by routine computations.
\\ \indent We recall from \cite[Equations 2.23-2.25]{PMTII} that, for any function $f:U_f\rightarrow \rn$, we have 
\begin{equation}
	\begin{split}
		R_{\hat{g}}  &= 2(\mu -J(\omega))+|\hat{A}-k|_{\hat{g}}^2+2|q|_{\hat{g}}^2 -2\diver^{\hat{g}}(q) +H_{\hat{M}^n}^2 -( \trace^{\hat{g}}(k))^2 \\
		&\qquad  + 2k(\vec{n}, \vec{n}) \big( H_{\hat{M}^n} - \trace^{\hat{g}}(k) \big) + 2\vec{n} \big( H_{\hat{M}^n} - \trace^{\hat{g}}(k) \big) 
	\end{split}
\end{equation}
where $\hat{A}$ is the second fundamental form of the graph of $f$, $k$ is the symmetric tensor from the initial data $(M^n,g,k)$ extended trivially to $(M^n\times\rn, g + dt^2 )$ and 
\begin{equation}
	\omega=\frac{\nabla^g f}{\sqrt{1+|d f|_g^2}}, \qquad \text{and}\qquad q_i= \frac{f^{,j}}{\sqrt{1+|d f|_g^2}}(\hat{A}_{ij}-k_{ij}).
\end{equation}
If $f$ solves Jang's equation $\J(f)=0$, the last three terms on the right hand side vanish, and we recover the \emph{Schoen-Yau} identity.
\\ \indent The first two of the last three terms have the fall-off rate $\Ol(r^{-(n+1-\epsilon)})$ by boundedness of $k(\vec{n}, \vec{n})$ and the equation that $f$ satisfies. For the last term, the same fall-off rate is obtained after recalling the asymptotics of $\vec{n}^k$ from the proof of Lemma \ref{LemmaRhoExpansionMeanCurv}. 
\\ \indent It follows from Definition \ref{DefinitionAHinitialData} that $\mu- J(\omega)= \Ol(r^{-(n+\tau_0)})$, but we can without loss of generality assume $\epsilon \leq \tau_0$. We estimate the remaining terms; firstly, we claim 
\begin{equation}
	|\hat{A}-k|_{\hat{g}}^2=   \Ol_\infty(r^{-(n+1-\epsilon)})     
\end{equation}
Indeed, from Definition \ref{DefinitionWangAsymptotics} and Lemmas \ref{LemmaJangGraphMetric} and \ref{LemmaJangGraph2ndFF} we get
\begin{equation}
	\begin{split}
		\hat{A}_{rr}-k_{rr} &= (n-2)(n-3)\frac{\alpha }{r^n}+\Ol_\infty ( r^{-(n+1-\epsilon)} ), \\
		\hat{A}_{r \mu} -k_{\mu r} &= -(n-2)\frac{\alpha_{,\mu}}{r^{n-1}} + \Ol_\infty ( r^{-(n-\epsilon)} ), \\
		\hat{A}_{ \mu \nu}-k_{\mu \nu}	&= \frac{\Hess_{\mu \nu}^{\Omega}(\alpha)}{r^{n-2}} -\bigg(  \frac{n-2}{2} 	\bigg)\frac{\textbf{m}_{\mu\nu}}{r^{n-2}} - \frac{\textbf{p}_{\mu\nu}}{r^{n-2}} + \Ol_\infty ( r^{-(n-1-\epsilon)} ).
	\end{split}
\end{equation}
From these estimates and Lemma \ref{LemmaJangGraphMetric} it is easily seen that $|\hat{A}-k|_{\hat{g}}^2$ has the asserted decay.
\\ \indent As for the $|q|_{\hat{g}}^2$-term, we note that from Lemma \ref{LemmaJangGraphMetric} we have
\begin{equation}
	\frac{1}{\sqrt{1+|df|_g^2}} =  \frac{1}{\sqrt{1+r^2}}+(n-3)\frac{\alpha}{r^{n-1}} +\Ol_\infty ( r^{-(n-\epsilon)} )    .
\end{equation}
We first compute radial component 
\begin{equation}
	q_r =  \frac{ g^{rr}f_{,r}   }{\sqrt{1+|d f|_g^2}}(\hat{A}_{rr}-k_{rr}) +\frac{  g^{\mu \nu}f_{,\nu}  }{\sqrt{1+|d f|_g^2}}( \hat{A}_{ r\mu}-k_{ r\mu}) .\\
\end{equation}
The first term is
\begin{equation}
	\begin{split}
		\frac{ g^{rr}f_{,r}   }{\sqrt{1+|d f|_g^2}}(\hat{A}_{rr}-k_{rr}) &=(1+r^2)\bigg( \frac{1}{\sqrt{1+r^2}}+(n-3)\frac{\alpha }{r^{n-1}} 	+\Ol ( r^{-(n-\epsilon)}  )  \bigg) \\
		&\qquad \times \bigg( \frac{r}{\sqrt{1+r^2}}-(n-3)\frac{\alpha }{r^{n-2}}+ \Ol ( r^{-(n-1-\epsilon)}  ) \bigg) \\
		&\qquad \times \bigg( (n-2)(n-3)\frac{\alpha }{r^{n}}+ \Ol ( r^{-(n+1-\epsilon)}   ) \bigg) \\
		&=(n-2)(n-3)\frac{\alpha }{r^{n-1}}+ \Ol ( r^{-(n-\epsilon)}  )
	\end{split}
\end{equation}
and, similarly, the second term is
\begin{equation}
	\frac{  g^{\mu \nu}f_{,\nu}  }{\sqrt{1+|d f|_g^2}}( \hat{A}_{ r\mu}-k_{ r\mu}) = \Ol_\infty(r^{-(2n-1)}), 
\end{equation}
so that  
\begin{equation}\label{EquationAux1010}
	q_r=  (n-2)(n-3)\frac{\alpha }{r^{n-1}}  +  \Ol_\infty ( r^{-(n-\epsilon)} ).
\end{equation}
In a similar fashion, we estimate the tangential component $q_\mu$:
\begin{equation}
	\begin{split}
		q_{\mu}&= \frac{ g^{rr}f_{,r}   }{\sqrt{1+|d f|_g^2}}(\hat{A}_{\mu r}-k_{\mu r}) +\frac{  g^{\lambda \rho}f_{,\rho}  }{\sqrt{1+|d f|_g^2}}( \hat{A}_{ \mu \lambda}-k_{\mu \lambda}) ,
	\end{split}
\end{equation}
where the terms are
\begin{equation}
	\begin{split}
		\frac{ g^{rr}f_{,r}   }{\sqrt{1+|d f|_g^2}}(\hat{A}_{\mu r}-k_{\mu r})  &=  -(n-2)\frac{\alpha_{,\mu}}{r^{n-2}}   + \Ol_\infty ( r^{-(n-1-\epsilon)} ), \\
		\frac{  g^{\lambda \rho}f_{,\rho}  }{\sqrt{1+|d f|_g^2}}( \hat{A}_{ \mu \lambda}-k_{\mu \lambda})	&=\Ol_\infty ( r^{-2(n-1)} ),
	\end{split}
\end{equation}
which gives
\begin{equation}
	q_{\mu}= -(n-2)\frac{\alpha_{,\mu}}{r^{n-2}} + \Ol_\infty ( r^{-(n-1-\epsilon)} ).
\end{equation}
The asserted estimate on the norm of $q$ follows:
\begin{equation}
	\begin{split}
		|q|_{\hat{g}}^2 &= \hat{g}^{ij}q_{i}q_j \\
		&= \Ol ( r^{-(n+1-\epsilon)} ),
	\end{split}
\end{equation}
using Lemma \ref{LemmaJangGraphMetric}. 
\\ \indent Finally, we estimate $\diver^{\hat{g}}(q)$. The Christoffel symbols of the Jang graph are denoted by $\hat{\Gamma}$, and estimates thereof are found in Lemma \ref{LemmaJangGraphChristoffelSymbols}. The covariant derivative of $q$ has components 
\begin{equation}
	(\hat{\nabla}_{i}q)_j= q_{j,i}-\hat{\Gamma}_{ij}^kq_k 
\end{equation}
and from \eqref{EquationAux1010} we see that
\begin{equation}
	q_{r,r} =-(n-1)(n-2)(n-3)\frac{\alpha }{r^n}+ \Ol_\infty(r^{-(n+1-\epsilon)}) , \\
\end{equation}
$q_{r,\mu}  =\Ol_\infty(r^{-(n-1)})$, $q_{\mu,r} = \Ol_\infty(r^{-(n-\epsilon)})$ and $q_{\mu,\nu} =\Ol_\infty(r^{-(n-1-\epsilon)})$.
From Lemma \ref{LemmaJangGraphChristoffelSymbols} we find: 
\begin{equation}
	\begin{split}
		(\hat{\nabla}_r q)_r 
		&= q_{r,r}- \big( \hat{\Gamma}_{r r}^rq_r +\hat{\Gamma}_{r r}^{\mu}q_{\mu}  \big) \\
		&=-(n-1)(n-2)(n-3)\frac{\alpha }{r^n} + \Ol_\infty ( r^{-(n+1-\epsilon)} ).
	\end{split}
\end{equation}
It follows that 
\begin{equation}
	\hat{g}^{rr}(\hat{\nabla}_r q)_r = -(n-1)(n-2)(n-3)\frac{\alpha }{r^n} + \Ol_\infty ( r^{-(n+1-\epsilon)} ).
\end{equation}
It is not difficult to see that the mixed terms of $\nabla q$ are 
\begin{equation}
	\begin{split}
		(\hat{\nabla}_r q)_\mu 
		&=  (n-1)(n-2)\frac{\alpha_{,\mu}}{r^{n-1}} +  \Ol_\infty(r^{-(n-\epsilon)})
	\end{split}
\end{equation}
and similarly
\begin{equation}
	\begin{split}
		(\hat{\nabla}_\mu q)_r 
		&= (n-2)^2\frac{\alpha_{,\mu}}{r^{n-1}}+\Ol_\infty(r^{-(n-1)})
	\end{split}
\end{equation}
so that  
\begin{equation}
	\hat{g}^{r\mu}(\hat{\nabla}_{r}q)_\mu = -(n-1)(n-2)\frac{|d\alpha|^2_\Omega}{r^{-2(n-1)}}  + \Ol_\infty(r^{-(2n-1-\epsilon)})
\end{equation}
and 
\begin{equation}
	\hat{g}^{r\mu}(\hat{\nabla}_{\mu}q)_r =- (n-2)^2\frac{|d\alpha|^2_\Omega}{r^{-2(n-1)}}  + \Ol_\infty(r^{-(2n-1-\epsilon)}).
\end{equation}
Finally the tangential term is
\begin{equation}
	\begin{split}
		(\hat{\nabla}_\mu q)_\nu &= q_{\nu, \mu} - \hat{\Gamma}^r_{\mu \nu} q_r- \hat{\Gamma}^\rho_{\mu \nu} q_\rho \\
		&= \delta_{\mu\nu} (n-2)(n-3)\frac{\alpha}{r^{n }} -(n-2)\frac{\Hess^\Omega_{\mu\nu}(\alpha)}{r^{n-2}}+ \Ol_\infty ( r^{-(n-1-\epsilon)}  ) 
	\end{split}
\end{equation}
so that 
\begin{equation}
	\begin{split}
		\hat{g}^{\mu\nu}(\hat{\nabla}_\mu q)_\nu &= (n-1)(n-2)(n-3)\frac{\alpha}{r^{n }}- (n-2)\frac{\Delta^\Omega (\alpha)}{r^{n}} + \Ol_\infty ( r^{-(n +1-\epsilon)} ). 
	\end{split}
\end{equation}
The assertion on the fall-off rate of $\diver^{\hat{g}}(q)$ follows.
\\ \indent The asserted decay of $R_{\hat{g}}$ now follows.

\end{proof}

\begin{lemma}\label{LemmaJangGraph2ndFF}
	
The second fundamental form $\hat{A}$ of $(\hat{M}^n, \hat{g})$ has components:
\begin{equation}
	\begin{split}
		\hat{A}_{rr}&=\frac{1}{1+r^2}+(n-2)(n-3)\frac{\alpha}{r^{n}}+\Ol_\infty (  r^{-(n+1-\epsilon)}  ), \\
		\hat{A}_{\mu r}&= -(n-2)\frac{\alpha_{,\mu}}{r^{n-1}} +\Ol_\infty ( r^{-(n-\epsilon)} ),\\
		\hat{A}_{\mu \nu}&=\frac{\Hess_{\mu \nu}^{\Omega}(\alpha)}{r^{n-2}} + \delta_{\mu \nu} - \bigg(  \frac{n-2}{2} \bigg) 	\frac{\textbf{m}_{\mu\nu}}{r^{n-2}} + \Ol_\infty (  r^{-(n-1-\epsilon)} ).
	\end{split}
\end{equation}

\end{lemma}

\begin{lemma}\label{LemmaJangGraphMetricDerivative}
	
The covariant derivative $\nabla \hat{g}$ taken with respect to the Euclidean metric $\delta$ has components:
\begin{equation}
	\begin{split}
		(\nabla_r \hat{g} )_{rr} &=2(n-2)(n-3)\frac{\alpha}{r^{n-1}} + \Ol_\infty ( r^{-(n-\epsilon)}  ), \\
		(\nabla_r \hat{g} )_{r\mu} &=-(n-2)\frac{\alpha_{,\mu}}{r^{n-2}}       + \Ol_\infty ( r^{-(n-1-\epsilon)}  ), \\
		(\nabla_r \hat{g} )_{ \mu \nu} &= \Ol_\infty ( r^{-(n-1-\epsilon)} ) \\
		(\nabla_\mu \hat{g})_{rr}	&=- 2(n-2)\frac{\alpha_{,\mu}}{r^{n-2}}       + \Ol_\infty ( r^{-(n-1-\epsilon)} ), \\
		(\nabla_\rho \hat{g})_{r\mu} &=\frac{\alpha_{,\mu\rho}}{r^{n-3}}     -2(n-3)\delta_{\rho \mu} \frac{\alpha}{r^{n-1}} - \Gamma^\sigma_{\rho 	\mu} \frac{\alpha_{,\sigma}}{r^{n-3}} + \Ol_\infty ( r^{-(n-2-\epsilon)} ), \\
		(\nabla_\rho \hat{g})_{\mu\nu} 	&=    \frac{\delta_{\mu\rho}\alpha_{,\nu}+ \delta_{\rho\nu}\alpha_{,\mu}}{r^{n-2}}     + 	\Ol_\infty(r^{-(1-\epsilon)}). \\
	\end{split}
\end{equation}

\end{lemma}

\begin{lemma}\label{LemmaJangGraphMetricInverseDerivative}
	The covariant derivative $\nabla \hat{g}^{-1}$ taken with respect to the Euclidean metric $\delta$ has components:
	\begin{equation}
		\begin{split}
			(\nabla_r \hat{g})^{rr} &=  - 2(n-2)(n-3)\frac{\alpha}{r^{n-1} } + \Ol_\infty(r^{-(n-\epsilon)}), \\
			(\nabla_r \hat{g} )^{r\mu} &=  \Ol(r^{-(n+1-\epsilon)}), \\
			(\nabla_r \hat{g} )^{ \mu \nu} &=  \Ol(r^{-(n+2-\epsilon)}), \\
			(\nabla_\sigma \hat{g})^{rr}	&= 2(n-2)\frac{\alpha_{,\sigma}}{r^{n-2}} + \Ol(r^{-(n-1-\epsilon)}), \\
			(\nabla_\sigma \hat{g})^{r\mu} &= \Ol(r^{-(n-1)}), \\
			(\nabla_\sigma \hat{g})^{\mu\nu} 	&= -\delta^\mu_\sigma \delta^{\nu \beta} \frac{\alpha_{,\beta}}{r^{n-2}} - \delta_\sigma^{\nu} \delta^{\mu\beta} \frac{\alpha_{,\beta}}{r^{n-2}} + \Ol(r^{-(n+1-\epsilon)}).    
		\end{split}
	\end{equation}
	
\end{lemma}

\begin{lemma}\label{LemmaJangGraph2ndFFDerivative}
	
The components of the covariant derivative $\nabla \hat{A}$, taken with respect to the Euclidean metric $\delta$, are:
\begin{equation}
	\begin{split}
		(\nabla_r \hat{A})_{rr}	&=- \frac{2r}{(1+r^2)^2} - n(n-2)(n-3)\frac{\alpha}{r^{n+1}} + \Ol_\infty(  r^{-(n+2-\epsilon)} ), \\
		(\nabla_r \hat{A} )_{\mu r} 	&= \Ol_\infty (  r^{-(n +1-\epsilon)} ), \\
		(\nabla_r \hat{A} )_{\mu \nu} 	&= - \frac{2}{r}\delta_{\mu\nu} +  \Ol_\infty (  r^{-(n-1)} ), \\
		(\nabla_\mu \hat{A} )_{rr} &= \Ol_\infty(r^{-n}), \\
		(\nabla_\mu \hat{A} )_{r\nu}	&= -\frac{4}{r}\delta_{\mu\nu}    +  \Ol_\infty(r^{-(n-1)}), \\
		(\nabla_\rho \hat{A} )_{\mu\nu}	&= \Ol_\infty(r^{-(n-2)}).
	\end{split}
\end{equation}

\end{lemma}

	
	

	
	

\section{The ADM energy of the Jang graph}\label{SectionJangGraphADMmass}

We calculate the ADM energy of the Jang graph obtained in Proposition \ref{PropositionJangLimit} with asymptotics as in Proposition \ref{CorollaryJangGraphAF}. The notation used here are as in Section \ref{SectionJangGraph}.

\begin{proposition}\label{PropositionJangGraphADMmass}
The ADM-mass of the Jang graph $(\hat{M}^n, \hat{g} )$ obtained in Proposition \ref{PropositionJangLimit}, with asymptotics as in Corollary \ref{CorollaryJangGraphAF}, is 
\begin{equation}
	E_{ADM}=\frac{1}{\omega_{n-1}}\int_{\mathbb{S}^{n-1}} \bigg(   \trace^{\Omega}(\textbf{p}) +\bigg(\frac{n-2}{2} \bigg)\trace^{\Omega}(\textbf{m})  \bigg) dS.
\end{equation}
\end{proposition}

\begin{proof}

Throughout this proof the geometric quantities, such as Christoffel symbols and covariant derivatives, are associated to the Euclidean metric $\delta$. 
\\ \indent Similarly to the proof of Proposition \ref{PropositionWangMassVector}, we let the exhaustion of $\hat{M}^n$ be coordinate balls and hence we need the radial components of the 1-form $\mathbb{U}(\hat{g},\delta)$, the divergence $\diver^{\delta}(\hat{g})$ and the differential of the trace $d \trace^{\delta}(\hat{g})$. The radial component of the divergence term is $\diver^{\delta}(\hat{g})_r = (\nabla_r \hat{g} )_{rr} +  \delta^{\mu \nu} (\nabla_\mu \hat{g} )_{r\nu}$. Using Lemma \ref{LemmaJangGraphMetric} and the fact that $\Gamma^k_{rr}=0$ we compute
\begin{equation}
	(\nabla_{r}\hat{g})_{rr}  =2(n-1)(n-3)\frac{\alpha}{r^{n-1}}+ \Ol_\infty (  r^{-(n-\epsilon)} ).
\end{equation}
Similarly
\begin{equation}
	\begin{split}
		(\nabla_{\mu }\hat{g})_{r\nu }&= 
		\frac{Hess_{\mu\nu}^\Omega (\alpha)}{r^{n-3}} - 2(n-3) \frac{\alpha}{r^{n-1}} \delta_{\mu\nu} + \Ol_\infty(r^{-(n-2-\epsilon)}). 
	\end{split}
\end{equation}
It follows that
\begin{equation}
	\begin{split}
		\delta^{\mu\nu}(\nabla_{\mu }\hat{g})_{r\nu } 
		&=\frac{\Delta^\Omega(\alpha)}{r^{n-1}} -2(n-1)(n-3)\frac{\alpha}{r^{n-1}}   +\Ol_\infty (  r^{-(n-\epsilon)} )  \\
	\end{split}
\end{equation}
and so in total
\begin{equation}
	\diver^\delta(\hat{g})_r = \frac{\Delta^\Omega(\alpha )}{r^{n-1}}     +\Ol_\infty (  r^{-(n-\epsilon)} ).
\end{equation}
To find the component of the gradient of the trace we first note that 
\begin{equation}
	\trace^\delta(\hat{g} )= n    -2(n-3)\frac{\alpha}{r^{n-2}}+\Ol_\infty (  r^{-(n-1-\epsilon)} )
\end{equation}
directly from Lemma \ref{LemmaJangGraphMetric}. Hence, 
\begin{equation}
	d\trace^\delta(\hat{g})_r  = 2(n-2)(n-3)\frac{\alpha}{r^{n-1}} +\Ol_\infty (  r^{-(n- \epsilon)} )
\end{equation}
so that in turn
\begin{equation}
	\mathbb{U}(\hat{g},\delta)_r =\frac{\Delta^\Omega (\alpha )}{r^{n-1}} -2(n-1)(n-3)\frac{\alpha }{r^{n-1}}   + \Ol_\infty (  r^{-(n- \epsilon)} ).
\end{equation}
It follows, using $\vec{n}_r=\partial_r$, that the ADM mass is
\begin{equation}
	\begin{split}
		E_{ADM} &=\frac{1}{2(n-1)\omega_{n-1}}\lim_{R\rightarrow \infty} \int_{\{r=R\}} \mathbb{U}(\hat{g},\delta)(\vec{n}_r) d\mu^{\delta} \\
		&=\frac{1}{2(n-1)\omega_{n-1}}\int_{\mathbb{S}^{n-1}} \bigg(  \Delta^\Omega(\alpha ) -2(n-1)(n-3)\alpha    \bigg) dS \\
		&=\frac{1}{\omega_{n-1}}\int_{\mathbb{S}^{n-1}} \bigg(   -(n-3)\alpha   \bigg) dS \\
		&=\frac{1}{\omega_{n-1}}\int_{\mathbb{S}^{n-1}} \bigg(   \trace^\Omega(\textbf{p}) +\bigg(\frac{n-2}{2} \bigg)\trace^\Omega(\textbf{m})  \bigg) dS  ,
	\end{split}
\end{equation}
where we used \eqref{EquationAlpha} that $\alpha$ satisfies.

\end{proof}

\section{Some properties of Fermi coordinates}\label{SectionRicattiSystem}

In this section we provide the proof of Proposition \ref{PropositionFermiCoord} used in Section \ref{SectionJangAE}: 

\begin{proposition}
	There exists constants $\rho_0>0$ and $C\geq 1$ such that $|\hat{A}_\rho|_{\hat{g}_\rho}<C$ and $C^{-1}\delta\leq \hat{g}^\rho\leq C\delta$ for any $0\leq \rho \leq \rho_0$. Furthermore, all partial derivatives of $(\hat{g}_\rho)_{ij}$ and $(\hat{A}_\rho)^i_j$ up to order $3$ in the Fermi coordinates are bounded. 
\end{proposition}

\begin{proof}

We follow the proof in Appendix C of \cite{SakovichPMTah}. It is well known that the $(1,1)$-tensor with components $(A_\rho)_i^j$ satisfies the Mainardi equation:
\begin{equation}
	-(A_\rho)^i_{j,\rho} + (A_\rho)^i_k (A_\rho)_j^k = \text{Riem}^i_{\rho \rho j}.
\end{equation}
Here indices are raised by $g_\rho$. It is convenient to write this on the form 
\begin{equation}
	- A'(\rho) + A^2(\rho) = R^N(\rho),
\end{equation}
where $R^N(\rho)$ acts on vectors $V$ orthogonal to $\partial_\rho$ by $R^N(\rho)(V)=\langle R^N V, V\rangle = \text{sec}(V, \partial_\rho)$ and is nothing but the normal sectional curvature operator. We write $\Lambda(\rho)$ for the largest eigenvalue of $A(\rho)$. The eigenvalues of $A(0)$ are bounded and we want to show that for some $\rho_0>0$ the eigenvalues of $A(\rho)$ are bounded when $\rho<\rho_0$. For convenience we will throughout the proof supress the tangential indices.
\\ \indent Since $\Lambda(\rho)$ is obtained through the Rayleigh quotient it is Lipschitz continuous, and hence almost everywhere differentiable by Rademachers theorem. Let $(q, \tilde{\rho})$ be a point where $\Lambda$ is differentiable. Let $v$ be a unit eigenvector normalized with respect to the Euclidean metric and extend it parallelly so that $v(q, \rho)=v(q, \tilde{\rho})$ for all $\rho \in [0, \rho_0]$. Let $\varphi(\rho) = v^T A(\rho)v$. Then $\varphi(\tilde{\rho})= \Lambda(\tilde{\rho})$ and $\varphi(\rho)\leq \Lambda (\rho)$. Further, there is
\begin{equation}
	\begin{split}
		- \Lambda'(\tilde{\rho}) + \Lambda^2(\tilde{\rho}) &= -\varphi'(\tilde{\rho}) + \varphi^2(\tilde{\rho}) \\
		&=v^T \big(  - A'(\tilde{\rho}) + A^2(\tilde{\rho}) \big) v\\
		&=v^T \big( R^N (\tilde{\rho}) \big) v.
	\end{split}
\end{equation}
The curvature is uniformly bounded and so we conclude that we have 
\begin{equation}
	-C_1 < - \Lambda'(\tilde{\rho}) + \Lambda^2(\tilde{\rho}) < C_1, 
\end{equation}
for some $C_1>0$.
\\ \indent We consider, for $C_0>|\lambda(0)|$, the initial value problem 
\begin{equation}
	\begin{cases}
		- \mu'(\rho) + \mu^2(\rho)  &=  - C_1, \\
		\mu(0) &= - C_0, 
	\end{cases}
\end{equation}
which is solved by $\mu(\rho) = \sqrt{C_1}\tan \big( \sqrt{C_1}\rho + \arctan(\frac{C_0}{\sqrt{C_1}})    \big)$. By decreasing $\rho_0$ if necessary, we may thus assume $\mu(\rho)$ is bounded on $[0, \rho_0]$. Furthermore, we may assert
\begin{equation}\label{Equation112}
	\Lambda'(\rho) - \Lambda^2(\rho) < C_1 = \mu'(\rho) - \mu^2(\rho)
\end{equation}
or
\begin{equation}\label{Equation111}
	-\big(\Lambda'(\rho) + \mu'(\rho) \big) + \big( \Lambda^2(\rho) + \mu^2(\rho) \big) <0,
\end{equation}
for almost every $\rho\in [0,\rho]$. 
\\ \indent We now show that $\Lambda(\rho) < \mu(\rho)$ for $\rho\in [0, \rho_0]$. $\Lambda$ is again almost everywhere differentiable, and so we can write $\Lambda(\rho) - \Lambda(0)= \int_0^\rho \Lambda'(\tau) d\tau$ by the Lebesgue differentiation theorem. Combining this with \eqref{Equation111} we obtain 
\begin{equation}
	\Lambda(\rho) + \mu(\rho) = \int_0^\rho \big( \Lambda'(\tau) + \mu'(\tau) \big) d\tau + \Lambda(0) + \mu(0) >0
\end{equation}\label{Equation113}
and so $\Lambda(\rho)> - \mu(\rho)$. Conversely, from \eqref{Equation112} we find
\begin{equation}
	\Lambda(\rho) - \mu(\rho) = \int_0^\rho \big( \Lambda'(\tau) + \mu'(\tau) \big) d\tau + \Lambda(0) + \mu(0) > \int_0^\rho \big(\Lambda^2(\tau) - \mu^2(\tau) \big) d\tau.
\end{equation}
Now let $\rho^\ast = \inf\{ \rho \: | \: \Lambda(\rho) > \mu(\rho) \}$. Since $\Lambda(0)<\mu(0)$ we must have $\rho^\ast>0$. It follows that both $\mu(\rho^\ast)= \Lambda(\rho^\ast)$ and $\Lambda(\rho) < \mu(\rho)$ for $0\leq \rho < \rho^\ast$. Since also $\Lambda(\rho) > - \mu(\rho)$ on $[0, \rho_0]$ it follows that $\Lambda^2(\rho ) - \mu^2(\rho)<0$ for $\rho \in [0,\rho_0]$. But this is a contradiction in view of \eqref{Equation113}.
\\ \indent Similarly, one can show that the smallest eigenvalue $\lambda(\rho)$ of $A(\rho)$ satisfies the differential inequality 
\begin{equation}
	-C_1 < -\lambda'(\rho) + \lambda^2(\rho) < C_1,
\end{equation}
for some $C_1>0$ and repeating the arguments as above yields a lower bound for $\lambda(\rho)$. The asserted estimate on the norm follows.
\\ \indent To get the uniform equivalence of $g_\rho$ and $\delta$ we note that from the proof of Lemma \ref{LemmaJangEqHeightFun} we have $g^\rho_{ij, \rho} = -2(\hat{A}_\rho)^k_i g^\rho_{jk}$. This implies, for $\Theta(\rho)$ the largest eigenvalue of $g_\rho$, that $\Theta'(\rho)\leq C_ \Theta(\rho)$, where $C_1>0$. Hence, $(\Theta(\rho) e^{-C_1\rho}) \leq 0$ which is solved by $\Gamma(\rho) = C_0 e^{-C_1\rho}$ for some $C_0$ and we chose $C_0>\Theta(0)$. Then 
\begin{equation}
	(\Theta(\rho) - \Gamma(\rho)) e^{-C_1\rho} = \int_0^\rho \big( 	(\Theta(\tau) - \Gamma(\tau)) e^{-C_1\tau} \big) d\tau + \Theta(0) - \Gamma(0) <0
\end{equation}
and so $\Theta(\rho) < \Gamma(\rho)$ for $\rho\in [0, \rho_0]$. The lowest eigenvalue can be estimated similarly and this yields the uniform equivalence $C^{-1}\delta \leq g_\rho  \leq C \delta$ for some $C>0$.
\\ \indent We now estimate the derivatives of $(\hat{A}_\rho)^i_j$ and $g^\rho_{ij}$. The Mainardi equation gives the bound on the first order derivative $(\hat{A}_\rho)^i_{j,\rho}$ and $g^\rho_{ij, \rho} = -2(\hat{A}_\rho)^k_i g^\rho_{jk}$ implies the estimates of the first order derivative of $g^\rho$ in the $\rho$-direction. Differentiating these equations in the tangential direction and commuting derivatives yields
\begin{equation}\label{EquationRicatti2ndDerivative1}
	(\hat{A}_\rho)^i_{j,k\rho} =      (A_\rho)^i_{\ell,k} (A_\rho)_j^\ell +(A_\rho)^i_\ell (A_\rho)_{j,k}^\ell - \text{Riem}^i_{\rho \rho j,k}
\end{equation} 
and
\begin{equation}\label{EquationRicatti2ndDerivative2}
	g^\rho_{ij,k\rho} = -2 (\hat{A}_\rho)^\ell_{j,k} g^\rho_{\ell i} - 2 (\hat{A}_\rho)^\ell_j g^\rho_{\ell i,k}. 
\end{equation}
The formula for the components of the $(1,4)$-tensor $\nabla \text{Riem}$ is
\begin{equation}
	\text{Riem}^i_{\rho \rho j, k} = \nabla_k \text{Riem}^i_{\rho \rho j} - \Gamma^i_{k\ell} \text{Riem}^\ell_{j\rho\rho} + \Gamma^\ell_{k\rho} \text{Riem}^i_{\ell \rho j} + \Gamma^\ell_{k\rho}\text{Riem}^i_{j\rho \ell} + \Gamma^\ell_{kj} \text{Riem}^i_{\rho\rho\ell}.
\end{equation}
Further, we recall $\Gamma^\ell_{k\rho} = (\hat{A}_\rho)^\ell_k$ from the proof of Lemma \ref{LemmaJangEqHeightFun}. Hence the right hand sides of the above second derivatives are bounded. We may write these as a the system 
\begin{equation}
	\begin{split}
		(\partial \hat{A}_\rho)' &= K_1 ( \partial \hat{A}_\rho) + K_2 (\partial g_\rho) + K_3, \\
		(\partial g_\rho)' &= K_4 (\partial \hat{A}_\rho) + K_5 (\partial g_\rho),
	\end{split}
\end{equation}
where $K_i$ are bounded $n^3\times n^3$-matrices and $\partial \hat{A}_\rho$ and $\partial g_\rho$ are treated as vectors in $\rn^{n^3}$ with components $(\hat{A}_\rho)^i_{j,k}$ and $g^\rho_{ij,k}$, respectively. We set $x(\rho)= |\partial_\rho \hat{A}_\rho |$ and $y(\rho)= |\partial g_\rho|$. From the Cauchy-Schwarz inequality it follows that $x' \leq | (\partial \hat{A}_\rho)'|$ and $y' \leq |(\partial g_\rho)'|$. In turn, there is
\begin{equation}
	\begin{split}
		x' &\leq c_1 x + c_2 y + c_3, \\
		y' &\leq c_4 x + c_5 y + c_6,
	\end{split}
\end{equation}
for positive real constants $c_i>0$. Thus, there is $x<\tilde{x}$ and $y<\tilde{y}$ on $[0,\rho_0]$, where $(\tilde{x}, \tilde{y})$ solves 
\begin{equation}
	\begin{split}
		\tilde{x}' &\leq c_1 \tilde{x} + c_2 \tilde{y} + c_3, \\
		\tilde{y}' &\leq c_4 \tilde{x} + c_5 \tilde{y} + c_6,
	\end{split}
\end{equation}
such that $\tilde{x}(0)>x(0)$ and $\tilde{y}(0)> y(0)$. It follows that the derivatives $(\hat{A}_\rho)^i_{j,k}$ and $g^\rho_{ij,k}$ are bounded and so all first order derivatives are bounded. Inserting this into Equations \ref{EquationRicatti2ndDerivative1} and \ref{EquationRicatti2ndDerivative2} we find that the second derivatives $\hat{g}^\rho_{ij,k\rho}$ and $(\hat{A}^\rho)^i_{j,k\rho}$ are bounded. Differentiating the Mainardi equation and the equation for $\hat{g}^\rho_{ij}$ above with respect to $\rho$ we find that also both $\hat{g}_{ij,\rho\rho}$ and $(\hat{A}^\rho)^i_{j,\rho\rho}$ are bounded. 
\\ \indent Finally, by taking tangential derivatives of Equations \ref{EquationRicatti2ndDerivative1} and \ref{EquationRicatti2ndDerivative2} we obtain
\begin{equation}
	\begin{split}
		(\partial \partial \hat{A}_\rho)' &= K_1 ( \partial \partial  \hat{A}_\rho) + K_2 (\partial \partial g_\rho) + K_3, \\
		(\partial \partial g_\rho)' &= K_4 (\partial \partial \hat{A}_\rho) + K_5 (\partial \partial g_\rho) + K_6,
	\end{split}
\end{equation}
where $K_i$, $i=1,\ldots, 6$, are $n^4\times n^4$-matrices with bounded entries and $\partial \partial \hat{A}_\rho$ and $\partial \partial g_\rho$ are treated as vectors in $\rn^{n^4}$ with components $(\hat{A})^i_{j,k\ell}$ and $\hat{g}^\rho_{ij,k\ell}$. Repeating the above arguments yields the boundedness of the second order tangential derivatives.
\\ \indent The same procedure gives boundedness of the third derivatives, as explained in \cite{SakovichPMTah}.
	
\end{proof}

\section{Geometric Measure Theory}\label{SectionGMT}

In this section we review some preliminaries from Geometric Measure Theory that are used in Section \ref{SectionJangSolution}. The reader is assumed to be familiar with basic notions, such as currents, varifolds and minimizing currents (see for instance \cite{SimonGMT} and \cite{KrantzParksGMT}). Our focus will be on the important class of $\lambda$-minimizing currents $\F_\lambda$. Here, we essentially follow  \cite[Appendix A]{EichmairPlateau}.
\\ \indent The following definition, generalizing the notion of a minimizing current, was first introduced in \cite{DuzaarSteffen93}. 

\begin{definition}\label{DefinitionLambdaMin}
	Let $T=\partial [[E]]\in \D_n(\rn^{n+\ell})$ be a current that is a boundary of an $\H^{n+1}$-measurable set $E\subset \rn^{n+\ell}$ that has locally finite perimeter. Then $T$ is said to be $\lambda$\emph{-minimizing} if, for every open subset $W\subset \subset \rn^{n+\ell}$ and every integer multiplicity current $X\in \D_{n+1}(\rn^{n+\ell})$ with compact support in $W$, we have 
	\begin{equation}\label{EquationLambdaMinimizing}
		\mathbb{M}_W(T)\leq \mathbb{M}_W(T+ \partial X) + \lambda \mathbb{M}_W(X).
	\end{equation}
	If $N^{n+1}\subset \rn^{n+\ell}$ is an embedded, oriented $C^2$-manifold, $E\subset N^{n+1}$, and $T=\partial [[E]]$ satisfies \eqref{EquationLambdaMinimizing} for all $X\in \D_{n+1}(\rn^{n+\ell})$ with $\text{supp}(X)\subset N^{n+1}\cap W$ we say that $T$ is $\lambda$-\emph{minimizing in} $N^{n+1}$.
	\\ \indent The collection of $\lambda$-minimizing boundaries is denoted by $\F_\lambda$.
\end{definition}

\noindent Throughout this section it will be assumed, unless stated otherwise, that our currents $T\in \F_\lambda$ are $\lambda$-minimizing in $N^{n+1}$, which is an orientable $C^2$-manifold embedded into $\rn^{n+\ell}$.
\\ \indent We note that a current $T$ such that \eqref{EquationLambdaMinimizing} holds need not be integer multiplicity, but the requirement that $T=\partial [[E]]$ where $E$ has locally finite perimeter together with the local mass bounds $\mathbb{M}_W(T)<\infty$, for any open $W\subset \subset \rn^{n+\ell}$, implies that $T$ is integer multiplicity with multiplicity $1$ (see \cite[Remark 5.2 in Chapter 7]{SimonGMT}). That the currents in $\F_\lambda$ have locally bounded mass will be shown in the proof of Theorem \ref{TheoremGMTclosure} below.
\\ \indent It is well-known (see, for instance, \cite[Chapter 7]{SimonGMT}) that the underlying varifolds of a minimizing current ($\lambda=0$ in \eqref{EquationLambdaMinimizing}) is stationary. In the special case of the current being an oriented $C^2$-manifold, the minimizing property translates into vanishing mean curvature. In the $\lambda$-minimizing case the underlying varifold has bounded generalized mean curvature, as is shown in Lemma \ref{LemmaBoundedMeanCurv}.

\begin{lemma}\label{LemmaBoundedMeanCurv}
	Let $T\in \F_\lambda$ be $\lambda$-minimizing in $N^{n+1}$ and let $\vec{H}^T$ be the tangential generalized mean curvature vector in $N^{n+1}$. Then $|\vec{H}^T|\leq \lambda$ $\mu_V$-almost everywhere, where $V$ is the associated varifold to $T$.
\end{lemma}

\begin{proof}
	
The proof given in a more general case can be found in \cite{DuzaarSteffen93}. Here, we follow \cite[Remark A.2 in Appendix A]{EichmairPlateau}. Take any $X\in C^1_c(W, \rn^{n+\ell})$ such that $X(p)\subset T_pN^{n+1}$ for each $p$, where $W \subset \subset \rn^{n+\ell}$ is open and let $\varphi:[0,1]\times \rn^{n+\ell}\rightarrow \rn^{n+\ell}$ be the associated flow generated by $X$. We denote the current $\big[ \big[ [0,1] \big] \big]$ equipped with the standard orientation by $[[0,1]]$. From the homotopy formula \cite[(2.25) in Chapter 6]{SimonGMT} we have  
\begin{equation} 
	\varphi^t_{\#}(T) - T  = \partial \varphi_{\#} \big( [[0,t]]\times T  \big), \\
\end{equation}
since $\partial T = 0$ and $\varphi^0_{\#}(T)=T$. The $\lambda$-minimizing property in \eqref{EquationLambdaMinimizing} implies, since $\varphi^t_{\#} (T) =T$ outside of the compact support of $X$, that
\begin{equation}
	\begin{split}
		\mathbb{M}_W(T) &\leq \mathbb{M}_W\big(T+ \partial \varphi_{\#} \big( [[0,t]]\times T  \big)\big) + \lambda \mathbb{M}_W(  \varphi_{\#} \big( [[0,t]]\times T  \big)) \\
		&= \mathbb{M}_W ( \varphi^t_{\#}(T)  ) + \lambda \mathbb{M}_W (  \varphi_{\#} \big( [[0,t]]\times T  \big)),
	\end{split}
\end{equation}
for any $W\subset \subset \rn^{n+\ell}$ open. For $t$ close to zero it follows that 
\begin{equation}\label{EquationAuxMass}
	0 \leq \big(\mathbb{M}_W ( \varphi^t_{\#}(T)  ) - \mathbb{M}_W(T) \big)t^{-1} + \lambda t^{-1}\mathbb{M}_W \big(  \varphi_{\#} ( [[0,t]]\times T  ) \big).
\end{equation}
We now take the $\limsup$ as $t\rightarrow  0$. In the first term nothing but the first variation of the associated varifold $V=(M,\theta)$, where $M$ is the rectifiable set and $\theta$ the multiplicity function of the integer multiplicity current $T=(M,\theta, \xi)$. For the second term we can assert that 
\begin{equation}\label{EquationLimsup}
	\limsup_{t\downarrow 0}\lambda t^{-1}\mathbb{M}_W \big(  \varphi_{\#} ( [[0,t]]\times T  ) \big) \leq \lambda \mathbb{M}_W(T) \sup |X|, 
\end{equation}
(cf. the discussion leading up to \cite[(2.27) in Chapter 6]{SimonGMT} for details). It is standard that the first variation for a varifold under any flow is related to the generalized mean curvature vector $\vec{H}$ in $\rn^{n+\ell}$ of $V$ via
\begin{equation}\label{EquationFirstVariation}
	\frac{d}{dt}\bigg|_{t=0} \mathbb{M}_W (\varphi^t_{\#}(T) ) = \int_{M} \diver^M (X) d\mu_V = - \int_M X \cdot \vec{H} d\mu_V.
\end{equation}
We decompose $\vec{H}=\vec{H}^N + \vec{H}^T$ where $\vec{H}^N$ is the mean curvature vector of $N^{n+1}$ and $\vec{H}^T$ is tangential to $N^{n+1}$. Since $X=X^T$ by assumption we obtain conclusively 
\begin{equation}\label{EquationMeasures}
	  \int_W X \cdot \vec{H}^T d\mu_V   \leq  \lambda \mathbb{M}_W(T) \sup |X|,
\end{equation}
from \eqref{EquationAuxMass}, keeping \eqref{EquationLimsup} and \eqref{EquationFirstVariation} in mind. We assume further that $|X|\leq 1$ and view \eqref{EquationMeasures} as an inequality of measures: $\mu_X(W)\leq \lambda \mu_V(W)$, where we $\mu_X$ is signed. By changing sign of $X$ if necessary, we see that $\mu_X< < \mu_V$ and so the Radon-Nikodym theorem implies that
\begin{equation}
	\mu_X (W) = \int_W f d\mu_V,
\end{equation}
where $f$ is integrable with respect to $\mu_V$. Further, we must have $|f|\leq \lambda$ $\mu_V$-almost everywhere, as otherwise the inequality $\mu_X\leq \lambda \mu_V$ would be violated. From the uniqueness of $f$, we identify $f=X \cdot \vec{H}^T$, and since again $|X|\leq 1$ it follows that $|\vec{X}^T|\leq \lambda$ $\mu_V$-almost everywhere, as asserted.

\end{proof}

The following example shows that any graph with bounded mean curvature is a $\lambda$-minimizing current.

\begin{example}\label{ExampleBoundedMeanCurvGraph}
	We study a real-valued function $f\in C^2(\rn^{n})$ where the associated graph in $\rn^{n}\times \rn$ has bounded mean curvature $H_{\hat{g}}$ at each point $\vec{x}=(x^1,\ldots, x^n)\in \rn^{n}$. With the above notation we set
	\begin{equation}
		E=\{ (\vec{x}, t)\: |\: f(\vec{x})\leq t         \}
	\end{equation}
	and let $T=\partial [[E]]$, which is precisely the graph of $f$. Clearly, $E$ has locally finite perimeter and $T$ is integer multiplicity one. 
	\\ \indent We claim that $T$ has the $\lambda$-minimizing property. Let $W\subset \subset \rn^{n+1}$ and $X\in \D_{n+1}(\rn^{n+1})$ with $\text{supp}(X)\subset W$ be as in Definition \ref{DefinitionLambdaMin}. Let 
	\begin{equation}
		\begin{split}
			\vec{n}&=\frac{ \nabla^\delta f - \partial_t}{\sqrt{1+|\nabla^\delta f|_\delta^2}} \\ 
		\end{split}
	\end{equation}
	be the downward pointing unit normal of the graph of $f$. Extend $\vec{n}$ trivially to all of $\rn^{n+1}$. From \eqref{EquationDivergenceForm}, we know that the mean curvature of the graph is $H_f=\diver_\delta(\vec{n})$. Let $\omega = dx^1\wedge \ldots \wedge dx^n \wedge dt$ be the top-form and $\sigma = \vec{n} \lrcorner \omega$ be the area form of $T$. We introduce the notation
	\begin{equation}
		\begin{split}
			d\hat{x}^k &= dx^1 \wedge \ldots \wedge dx^{k-1}\wedge dx^{k+1}\wedge \ldots \wedge dt, \\
			\hat{\partial}_{k} &= \partial_1\wedge \ldots \wedge \partial_{k-1}\wedge \partial_{k+1}\wedge \ldots \wedge \partial_{t}.
		\end{split} 
	\end{equation}
	From the definition of the "elbow" operation $\lrcorner$ we find, for any $1\leq k\leq n$:
	\begin{equation}
		\begin{split}
			\langle \sigma , \hat{\partial}_k \rangle &= \langle \vec{n}\lrcorner \omega, \hat{\partial}_k \rangle \\
			&=\langle \omega , \vec{n}\wedge \hat{\partial}_k \rangle \\
			&=\bigg\langle \omega, \frac{(-1)^{k-1} f^{,k}}{\sqrt{1+|\nabla^\delta f|^2_\delta} }\omega \bigg\rangle \\
			&=\frac{(-1)^{k-1}f^{,k}}{\sqrt{1+|\nabla^\delta f|^2_\delta}}.
		\end{split}
	\end{equation}
	Similarly, for $\hat{\partial}_{t}$ we have
	\begin{equation}
		\begin{split}
			\langle \sigma , \hat{\partial}_{t} \rangle &=  \langle \vec{n}\lrcorner \omega , \hat{\partial}_{t}    \rangle \\       
			&=\langle \omega , \vec{n} \wedge \hat{\partial}_{t} \rangle \\
			&=\bigg\langle \omega, \frac{(-1)^n}{\sqrt{1+|\nabla^\delta f|^2_\delta} }\omega \bigg\rangle \\ 
			&=\frac{(-1)^n}{\sqrt{1+|\nabla^\delta  f|^2_\delta}}.
		\end{split}
	\end{equation}
	Thus, we have
	\begin{equation}
		\sigma =  \frac{(-1)^{k-1}f^{,k}}{\sqrt{1+|\nabla^\delta f|^2_\delta}} d\hat{x}_{k}  + 	\frac{(-1)^n}{\sqrt{1+|\nabla^\delta f|^2_\delta} } d\hat{t} ,
	\end{equation}
	where $k=1,\ldots,n$. 
	\\ \indent Similarly, we obtain $d\sigma = \diver_\delta(\vec{n})\omega$. Indeed, since $dx^k\wedge d\hat{x}^\ell=\delta_{k\ell}(-1)^{k-1} \omega$, we see that 
	\begin{equation}
		\begin{split}
			d\sigma &= \bigg( \frac{(-1)^{k-1}f^{,k}}{\sqrt{1+|\nabla^\delta f|^2_\delta}} \bigg)_{,k} dx^k \wedge d\hat{x}^{k} + 	\bigg( \frac{(-1)^n}{\sqrt{1+|\nabla^\delta f|^2_\delta}} \bigg)_{,t} dt  \wedge d\hat{t}  \\
			&= \bigg( \frac{ f^{,k}}{\sqrt{1+|\nabla^\delta f|^2_\delta}} \bigg)_{,k} \omega \\
			&= \diver_\delta (  \vec{n} ) \omega.
		\end{split}
	\end{equation}
	Thus,
	\begin{equation}\label{EquationMeanCurvatureForm}
		d\sigma (\vec{x},t) = H_f (\vec{x},t)dx^1\wedge \ldots \wedge dx^k \wedge dt.
	\end{equation}
	Let $\D_X(W)$ denote all smooth functions on $\rn^{n+1}$ with support in $W$ that are equal to one in a neighbourhood of the support of $X$. Then 
	\begin{equation}
		\begin{split}
			\mathbb{M}_W(T+\partial X) &= \sup_{\omega \in \D^n(W), |\omega|\leq 1} (T+\partial X)(\omega) \\
			&\geq \sup_{\varphi \in \D_X(W), |\omega|\leq 1} (T+\partial X)(\varphi \sigma) \\
			& \geq \sup_{\varphi\in \D_X(W), |\omega|\leq 1}  T (\varphi \sigma) - \sup_{\varphi \in \D_X(W), |\omega|\leq 1}  \partial X 	(\varphi \sigma) \\
			&=\mathbb{M}_W(T) - |X(d\sigma)| \\
			&\geq  \mathbb{M}_W(T) - \lambda \mathbb{M}_W(X),
		\end{split}
	\end{equation}
	where the first inequality is due to inclusion, the second inequality comes from the triangle inequality and the last inequality follows from \eqref{EquationMeanCurvatureForm} and $|H_{\hat{g}}|\leq \lambda$. Hence, $T$ is $\lambda$-minimizing. 
\end{example}

In the following theorem we prove the compactness of $\F_\lambda$.

\begin{theorem} \label{TheoremGMTclosure}
	
Let $N^{n+1}\subset \rn^{n+\ell}$ be an embedded, orientable $C^2$-manifold and let $\{ T_k\}\subset \F_\lambda$ be $\lambda$-minimizing in $N^{n+1}$. Then there exists a subsequence $\{T_{k'}\}$ such that $T_{k'}\rightharpoonup T \in \F_\lambda$, where $T$ is $\lambda$-minimizing in $N^{n+1}$. Furthermore, we have the convergence of the indicator functions $\chi_{E_k} \rightarrow \chi_E$ in $L^1_{loc}(\H^{n+1})$, and $\mu_{T_k}\rightarrow \mu_T$ as Radon measures. 

\end{theorem}

\begin{proof}
		
We follow the proof of \cite[Lemma A.2]{EichmairPlateau}, which is very similar to the proofs of \cite[Theorems 2.4 and 5.3 in Chapter 7]{SimonGMT}. We start by proving the local boundedness the mass for every $T_k$ (by suitable modifications of the proof of \cite[Theorem 5.3 in Chapter 7]{SimonGMT}). For a fixed $q\in N^{n+1}$ we define $r(p)=|p-q|$ to be the Euclidean distance to $q$. For $\rho>0$ we may slice
\begin{equation}
	\partial [[E_k\cap B_\rho(q)]] = T_k\llcorner B_\rho(q) + \langle [[E_k]] , r, \rho \rangle 
\end{equation}
and note that these currents are compactly supported in $\bar{B}_\rho(q)$, so that for any open set $W$ such that $B_\rho(q)\subset W \subset \subset \rn^{n+\ell}$ the $\lambda$-minimizing property implies 
\begin{equation}
	\mathbb{M}(T_k \llcorner B_\rho(q)) \leq \mathbb{M}(\langle [[E_k]], r, \rho \rangle ) + \lambda \mathbb{M}( [[E_k\cap B_\rho(q) ]]).
\end{equation}
Now define $\tilde{E}_k = N^{n+1} \setminus E_k$ and $\tilde{T}_k = \partial [[\tilde{E}_k]]$. Then $\tilde{T}_k= -T_k$ is also $\lambda$-minimizing in $N^{n+1}$ and we have 
\begin{equation}\label{EquationMassBounds}
	\begin{split}
		\mathbb{M}(T_k \llcorner B_\rho(q)) &\leq \min\bigg\{ \mathbb{M}\langle [[E_k]], r, \rho \rangle + \lambda \mathbb{M}( [[E_k\cap B_\rho(q) ]]), \\
		&\qquad \mathbb{M}\langle [[\tilde{E}_k]], r, \rho \rangle + \lambda \mathbb{M}( [[\tilde{E}_k\cap B_\rho(q) ]])\bigg\},
	\end{split}
\end{equation}
for Lebesgue almost every $\rho>0$. Since $[[E_k]]+ [[\tilde{E}_k]]= [[ N^{n+1} ]]$ we also have 
\begin{equation}
	\langle [[E_k]], r, \rho \rangle+ \langle [[\tilde{E}_k]], r, \rho \rangle = \langle  N^{n+1}, r, \rho \rangle 
\end{equation}
for Lebesgue almost every $\rho>0$, which in turn implies 
\begin{equation}
		\mathbb{M} \big( \langle [[E_k]], r, \rho \rangle \big)   + \mathbb{M} \big( \langle [[\tilde{E}_k]], r, \rho \rangle \big) = \mathbb{M} \big( \langle  N^{n+1}, r, \rho \rangle \big),
\end{equation}
since $E_k$ and $\tilde{E}_k$ are disjoint as sets. Further, it is clear that 
\begin{equation}
	\mathbb{M} \big(  \langle N^{n+1}, r, \rho \rangle \big) \leq \H^n \big(   N^{n+1} \cap \partial B_\rho (q) \big),
\end{equation}
and
\begin{equation}\label{EquationMassBoundsII}
	\mathbb{M} \big( [[ E_k \cap B_\rho (q)  ]] \big) \leq \H^{n+1} \big(N^{n+1} \cap B_\rho (q) \big).
\end{equation}
The same estiates hold for $\tilde{E}_k$. Combining \eqref{EquationMassBounds} - \eqref{EquationMassBoundsII} we obtain 
\begin{equation}\label{EquationLocalMassBound}
	\mathbb{M}\big(   T_k \llcorner B_\rho(q)  \big) \leq \frac{1}{2} \H^n \big(   N^{n+1} \cap \partial B_\rho (q) \big) + \frac{\lambda }{2} \H^{n+1} \big( N^{n+1} \cap B_\rho (q) \big)
\end{equation}
for Lebesgue almost every $\rho>0$. The local boundedness follows.
\\ \indent We now prove the statement about the indicator functions $\chi_{E_k}$. From \eqref{EquationLocalMassBound} and \cite[Remark 5.2 in Chapter 7]{SimonGMT} together with the compactness results for $BV_{loc}$-functions \cite[Theorem 2.6 in Chapter 2]{SimonGMT} the sequence $\{ \chi_{E_k}\}$ has a convergent subsequence $\{\chi_{E_{k'}}\}$ that converges in $L^1_{loc}$ to an indicator function $\chi_E \in BV_{loc}$, where $E$ is some $\H^{n+1}$-measurable set. The $L^1$-convergence implies the current convergence $[[E_{k'}]]\rightharpoonup [[E]]$ and, in turn, also $T_{k'} \rightharpoonup T$. 
\\ \indent Our next aim is to show that $T$ is $\lambda$-almost minimizing. For this, following \cite[Lemma A.2]{EichmairPlateau}, we modify the proof of \cite[Theorem 2.4 in Chapter 7]{SimonGMT}. For simplicity, we only consider the setting when $T_k$ are $\lambda$-minimizing in $\rn^{n+1}$. The argument extends to the general of $\lambda$-minimizing boundaries in a submanifold $N^{n+1}$ by the same techniques as mentioned in \cite[Remark 2.5 (2) in Chapter 7]{SimonGMT}.
\\ \indent Let $K\subset \rn^{n+1}$ be an arbitrary compact set and let $\varphi: \rn^{n+1}\rightarrow [0,1]$ a smooth function such that $\varphi\equiv 1$ in a neighbourhood of $K$, support inside an $\epsilon$-neighbourhood $U_\epsilon=\{ p\: |\: \text{dist}(p,K)<\epsilon\}$ of $K$. For $\gamma \in (0,1)$ we denote the superlevel set
\begin{equation}
	W_\gamma = \{ p\in \rn^{n+1}\: |\: \varphi(p)>\gamma    \}.
\end{equation}
We define the current $R_k=[[E]] - [[E_k]]$ and observe that $\mathbb{M}_{W_0}(R_{k'})\rightarrow 0$ as $k'\rightarrow \infty$, where $k'$ is the index of the subsequence of indicator functions $\chi_{k'}$ that converges in $L^1_{loc}$. 
\\ \indent We now slice the currents $\{ R_k\}$ with respect to $\varphi$. From slicing theory \cite[Section 4 in Chapter 6]{SimonGMT} we may choose $\alpha \in (0,1)$ and a subsequence of $\{R_k\}$, still denoted by $\{R_k\}$, so that
\begin{equation}
	P_k= \partial (R_k \llcorner W_\alpha) - (\partial R_k)\llcorner W_\alpha
\end{equation}
is integer multiplicity with support in $\partial W_\alpha$ and such that $\mathbb{M}(P_k)\rightarrow 0$. Furthermore, $\alpha$ can be chosen so that both 
\begin{equation}
	\mathbb{M}_{W_0}(T_k\llcorner \partial W_{\alpha}) = 0
\end{equation}
for all $k$ and $\mathbb{M}_{W_0}(T \llcorner \partial W_\alpha)=0$. By taking restriction to $W_\alpha$, we have
\begin{equation}
	T\llcorner W_\alpha = T_k  \llcorner W_\alpha + \partial (R_k\llcorner W_\alpha) -P_k
\end{equation}
where both $P_k$ and $\partial (R_k\llcorner W_\alpha)$ are integer multiplicity with support in $\overline{W}_\alpha$ and whose masses tends to zero in the limit.
\\ \indent Consider a compactly supported $X\in \D_{n+1}(\rn^{n+1})$ with $\text{supp}(X)\subset K$ and take $\gamma \in (0, \alpha)$. Then the $\lambda$-minimizing property implies 
\begin{equation}\label{EquationAux1}
	\begin{split}
		\mathbb{M}_{W_\gamma}(T_k\llcorner W_\alpha)&\leq \mathbb{M}_{W_\gamma}(T_k\llcorner W_\alpha - P_k) + \mathbb{M}_{W_\gamma}(P_k) \\
		&\leq \mathbb{M}_{W_\gamma}(T_k\llcorner W_\alpha - P_k + \partial(R_k\llcorner W_\alpha) + \partial X ) \\
		&\qquad +  \lambda\mathbb{M}_{W_\gamma}(X) + \lambda\mathbb{M}_{W_\gamma}(R_k\llcorner W_\alpha)   + \mathbb{M}_{W_\gamma}(P_k) \\
		&= \mathbb{M}_{W_\gamma}(T\llcorner W_\alpha + \partial X) + \lambda\mathbb{M}_{W_\gamma}(X) \\
		&\qquad + \lambda\mathbb{M}_{W_\gamma}(R_k\llcorner W_\alpha) + \mathbb{M}_{W_\gamma}(P_k),
	\end{split}
\end{equation}
since both $X$ and $R_k\llcorner W_\alpha$ are compactly supported. Taking the limit $\gamma \rightarrow 0$ we obtain
\begin{equation}\label{EquationAux2}
	\begin{split}
		\mathbb{M}_{W_\alpha}(T_k ) &\leq \mathbb{M}_{W_\alpha}(T  + \partial X) + \lambda\mathbb{M}_{W_\alpha}(X) \\
		&\qquad + \lambda\mathbb{M}_{W_\alpha}(R_k ) + \mathbb{M}(P_k).
	\end{split}
\end{equation}
If we now let $X\equiv 0$ and take the superior limit in \eqref{EquationAux2}, then recalling that the masses of $P_k$ and $R_k$ tend to zero  we obtain $\limsup_k \mathbb{M}_{W_\alpha}(T_k) \leq \mathbb{M}_{W_\alpha}(T)$. From the lower semi-continuity of the mass it then follows that
\begin{equation}
	\mathbb{M}_{W_\alpha} (T_k ) \rightarrow \mathbb{M}_{W_\alpha}(T).
\end{equation}
In other words, no mass is lost under the weak convergence. Thus, taking the limit $k\rightarrow \infty$ in \eqref{EquationAux2} and recalling that $K$ was arbitrary we conclude that $T\in \F_\lambda $ as asserted.
\\ \indent Finally, we verify the Radon measure convergence. For this, we again follow the proof of \cite[Theorem 2.4 in Chapter 7]{SimonGMT}. We let $X\equiv 0$ in \eqref{EquationAux1} and since by construction $K\subset W_{\gamma}\subset U_\epsilon$ we get
\begin{equation}
	\begin{split}
		\limsup_k \mu_{T_k}(K) &\leq \limsup_k\mathbb{M}_{W_\gamma}(T_k) \\
		&\leq \mathbb{M}_{U_\epsilon}(T).
	\end{split}
\end{equation}
In the limit $\epsilon \rightarrow 0$ we thus get
\begin{equation}
	\limsup_k \mu_{T_k}(K)\leq \mu_T(K).
\end{equation}
Summing up, we see that the Radon measures $\{ \mu_{T_k}\}$ are upper semi-continuous when restricted to compact sets and from the lower semi-continuity of the mass we know that they are lower semi-continuous on when restricted to open sets. As explanied in the end of the proof in \cite[Theorem 2.4 in Chapter 7]{SimonGMT}, using an approximation argument we can show that this implies Radon measure convergence, that is for $f\in C_c(\rn^{n+1})$ we have
\begin{equation}
	\int_{\rn^{n+1}} f d\mu_{T_k} \rightarrow \int_{\rn^{n+1}} f d\mu_T.
\end{equation}
	
\end{proof}

At this point it is convenient to state the approximate monotonicity formula for currents in $\F_\lambda$ (see \cite[Theorem 3.17 in Chapter 4]{SimonGMT}):

\begin{equation}\label{EquationApproximateMonotonicity}
	F(\rho)\frac{\mu_T(B_\rho(q))}{\omega_n\rho^n} - F(\sigma) \frac{\mu_T(B_\sigma(q))}{\omega_n\sigma^n} = G(\sigma, \rho) \int_{B_\rho(q)- B_\sigma(q)} \frac{|\nabla^\perp r|^2}{r^n} d\mu_V,
\end{equation}
where $V$ is the varifold associated to $T$, $0<\sigma < \rho$, $F(\rho)\in [e^{-\Lambda \rho}, e^{\Lambda \rho}]$ and $G\geq 0$ is continuous and bounded for small $\rho$. It follows that the function $\rho \rightarrow F(\rho)\frac{\mu_T(B_\rho(q))}{\omega_n\rho^n} $ has a limit as $\rho\rightarrow 0$ and since $\lim_{\rho\rightarrow 0}F(\rho)=1$ it also follows that the density  
\begin{equation}
	\Theta^n(\mu_T, q) = \lim_{\rho \rightarrow 0} \frac{\mu_T(B_\rho(q))}{\omega_n\rho^n} 
\end{equation}
is defined at every point $q$. This will be $\mu_T$-almost everywhere equal to the multiplicity function: $\Theta^n(\mu_T, q)=\theta(q)$.
\\ \indent The following Lemma will be useful. 

\begin{lemma}\label{LemmaDensityConvergence}
	
Let $T, T_k\in \F_\lambda$, $T_k\rightharpoonup T$ and $q_k\rightarrow q$ with $q_k\in \text{supp}(T_k)$ and $q\in \text{supp}(T)$. Then 
\begin{equation}
	\limsup_k \Theta^n(\mu_{T_k},q_k) \leq   \Theta^n(\mu_T, q) .
\end{equation}
	
\end{lemma}

\begin{proof}
	
From the approximate monotonicity formula \eqref{EquationApproximateMonotonicity} it follows that for $\rho>0$ sufficiently small so that $F(\rho)<1+\epsilon_1$ and for $\epsilon_2>0$ fixed we have
\begin{equation}
	\Theta^n( \mu_{T_k}, q_k )  \leq (1+\epsilon_1)\frac{\mu_{T_k}(B_\rho(q_k))}{\omega_n\rho^n} \leq (1+\epsilon_1) \frac{\mu_{T_k}(B_{\rho+\epsilon_2}(q))}{\omega_n\rho^n},
\end{equation}
for sufficiently large $k$. Further, from the proof of Proposition \ref{TheoremGMTclosure} we know that mass is not lost under current convergence, $\mathbb{M}_W(T_k)\rightarrow \mathbb{M}_W(T)$. Taking the superior limit of both sides we obtain
\begin{equation}
	\limsup_k \Theta^n( T_k, q_k ) \leq (1+\epsilon_1) \frac{\mu_{T }(B_{\rho+\epsilon_2}(q))}{\omega_n\rho^n}.
\end{equation}
Taking the limit $\epsilon_2 \rightarrow 0$ followed by $\rho\rightarrow 0$ and finally $\epsilon_1\rightarrow 0$ the assertion follows.
	
\end{proof}

We define the map $\eta_{q,\gamma}:\rn^{n+\ell}\rightarrow \rn^{n+\ell}$ by 
\begin{equation}\label{EquationEta}
	\eta_{q, \gamma}(p) = \frac{p-q}{\gamma}
\end{equation}
and note the following property for the pushforward with respect to $\eta_{q, \gamma}$. 

\begin{lemma}\label{LemmaRescaleInclusion}
	If $T\in \F_\lambda$, then $\eta_{q, \gamma \: \#}T \in \F_{\lambda \gamma}$.
\end{lemma} 

\begin{proof}
	
Let $\omega = \sum_{\alpha} \omega_\alpha dx^{\alpha} \in \D^n (\rn^{n+\ell})$. Then for $\eta_{q,\gamma}$ as in \eqref{EquationEta} we have 
\begin{equation}
	\eta_{q, \gamma}^{\#} \omega = \sum_\alpha (\omega_\alpha \circ \eta_{q, \gamma} )\: \frac{dx^\alpha}{\gamma^n}.
\end{equation}
Consequently, for any open $W\subset \subset \rn^{n+\ell}$ and any $T\in \D_n(\rn^{n+\ell})$ we have
\begin{equation}
	\begin{split}
		\mathbb{M}_W(\eta_{q, \gamma \: \#}T) &= \sup_{\omega \in \D^n(\rn^{n+\ell}), |\omega|\leq1, \text{supp}(\omega)\subset W} (\eta_{q, \gamma \: \#} T )(\omega) \\
		&= \sup_{\omega \in \D^n(\rn^{n+\ell}), |\omega|\leq1, \text{supp}(\omega)\subset W} T( \eta_{q, \gamma}^{\#} \omega) \\
		&= \sup_{\omega \in \D^n(\rn^{n+\ell}), |\omega|\leq1, \text{supp}(\omega)\subset W} T \bigg( \sum_\alpha (\omega_\alpha \circ \eta_{q, \gamma}) \: \frac{dx^\alpha}{\gamma^n} \bigg) \\
		&= \frac{\mathbb{M}_{\eta^{-1}_{q, \gamma}(W)} (T)   }{\gamma^n}.
	\end{split}
\end{equation}
Hence, for a $\lambda$-minimizing $T$ and $X\in \D_{n+1}(\rn^{n+\ell})$ compactly supported in $W$, and $Y\in \D_{n+1}(\rn^{n+\ell})$ such that $\eta_{q, \gamma\: \#} Y=X$ we have
\begin{equation}
	\begin{split}
		\mathbb{M}_W(\eta_{q, \gamma \: \#} T) &= \frac{\mathbb{M}_{\eta^{-1}_{q, \gamma}(W)} (T)   }{\gamma^n} \\
		&\leq \bigg( \mathbb{M}_{\eta^{-1}_{q, \gamma}(W)} (T + \partial Y) + \lambda \mathbb{M}_{\eta^{-1}_{q, \gamma}(W)} (Y)      \bigg) \gamma^{-n} \\
		&= \mathbb{M}_{W} (\eta_{q, \gamma \: \#}T + \partial X) + \frac{\lambda \gamma}{\gamma^{n+1}} \mathbb{M}_{\eta^{-1}_{q, \gamma}(W)} (Y)  \\
		&=\mathbb{M}_{W} (\eta_{q, \gamma \: \#}T + \partial X) +  \lambda \gamma  \mathbb{M}_{ W }(X)
	\end{split} 
\end{equation}
as asserted.

\end{proof} 

\noindent In particular, Lemma \ref{LemmaRescaleInclusion} implies that if $\gamma<1$ and $T\in \F_\lambda$, then $\eta_{q, \gamma \: \#} T \in \F_{  \lambda}$.
\\ \indent The following theorem shows that the tangent cones of currents in $\F_\lambda$ are minimizing.

\begin{theorem}\label{TheoremConeIsMinimizing}
 	
Suppose $T\in \F_\lambda$ is $\lambda$-minimizing in an embedded orientable submanifold $N^{n+1}$. Then, for each $p\in \emph{supp} (T) $ and each sequence of positive real numbers $\{\lambda_k\}$ tending to zero there exists a subsequence $\{\lambda_{k'}\}$ and a minimizing integer multiplicity current $\C\in \D_n (\rn^{n+\ell})$ with $0\in \text{supp}(\C) \subset T_p N^{n+1}$ such that 
\begin{equation}
	\mu_{\eta_{p,\lambda_{k'}\: \#} T} \rightarrow \mu_C
\end{equation}
as Radon measures. Further, there exists an $\H^{n+1}$-measurable set $F$ in $T_p N^{n+1}$ such that $\C=\partial [[F]]$ and
\begin{equation}
	\chi_{\text{pr}_{T_p N^{n+1}}} (\eta_{p,\lambda_{k'}}  (E)) \rightarrow \chi_{F}
\end{equation} 
in the $L^1_{loc}(\H^{n+1})$ sense, where $\text{pr}_{T_pN^{n+1}}$ is the orthogonal projection onto $T_p N^{n+1}$. Finally 
\begin{equation}
	\eta_{0, \gamma\#}\C = \C
\end{equation}
and $\eta_{0,\gamma}(F)=F$ as sets for any $\gamma>0$. 
 	
\end{theorem}
 
\begin{proof}
 	
We write for brevity $\eta_{p,\lambda_{k}\: \#} T=T_k$ and observe that $\eta_{p,\lambda_k} N^{n+1}\rightarrow T_p N^{n+1}$ smoothly. It is not difficult to modify the proof of Theorem \ref{TheoremGMTclosure} to the case where the $E_k\subset N^{n+1}_k$ and where $N^{n+1}_k$ converge to $N^{n+1}$ smoothly using nearest point projection and the homotopy formula. From Lemma \ref{LemmaRescaleInclusion} we know that $T_k \in \F_{\lambda \gamma_k}$ and so by Theorem \ref{TheoremGMTclosure} we obtain the assertions about subconvergence.
\\ \indent It only remains to prove the minimizing property. We know from Lemma \ref{LemmaRescaleInclusion} that 
\begin{equation}
	\mathbb{M}_{W} (T_k ) \leq \mathbb{M}_{W}(T_k + \partial X)+ \lambda \gamma_k \mathbb{M}_{W}(X),
\end{equation}
where $W$ and $X$ are as in Definition \ref{DefinitionLambdaMin}. Since $\F_{\lambda \gamma_k}\subset \F_\lambda$ and since we know from the proof of Theorem \ref{TheoremGMTclosure} that mass in not lost under current convergence we obtain $\mathbb{M}_W(\C) \leq \mathbb{M}_W(\C + \partial X)$ after taking $\liminf$ on both sides.


\end{proof}
 
We now discuss the regularity of currents in $\F_\lambda$. 
 
\begin{definition}\label{DefinitionRegularSingularSets}
 	For $T\in \F_\lambda$ the set 
 	\begin{equation}
 		\emph{reg}(T) = \{ p \in \text{supp}(T) \: |\:T\llcorner B_\rho(p)\: \text{is a connected} \: C^{1,\alpha}-\text{graph} \}
 	\end{equation}
 	for some $\alpha \in (0,1)$ and some $\rho>0$, is called the \emph{regular set} and the set 
 	\begin{equation}
 		\emph{sing}(T) = \emph{supp}(T) - \emph{reg}(T) 
 	\end{equation}	
 	is called the \emph{singular set}.
\end{definition}

The following theorem shows that for $n\leq 6$ any current in $\F_\lambda$ has $\text{sing}(T)=\emptyset$.
 
\begin{theorem} \label{TheoremGMTregularity}
	
Let $T\in \F_\lambda$ be $\lambda$-minimizing in $N^{n+1}$. Then $\emph{sing}(T)=\emptyset$ for $n\leq 6$, $\emph{sing}(T)$ consists of isolated points if $n=7$ and $\H^{n-7+\alpha}(\emph{sing}(T))=0$ for $\alpha>0$.  
	
\end{theorem}
 
\begin{proof}
 	
We perform the tangent cone analysis and the abstract dimension reduction argument as in \cite[Theorem A.1]{EichmairPlateau} and \cite[Theorem 5.8 in Chapter 7]{SimonGMT}. Since the argument is rather well-known we provide only a sketch. 
\\ \indent We take any $q\in \text{sing}(T)$ and recall that the density $\Theta^n(\mu_T, q)$ will exist everywhere from the approximate monotonicity formula in \eqref{EquationApproximateMonotonicity}. If $q$ is in the singular set, then some critera in Allard's theorem must be violated. The mean curvature is bounded and so there must be some $\delta_0>0$ such that $\Theta^n(\mu_T, q)\geq 1+\delta_0$.
\\ \indent We define the set of weak limits
\begin{equation}
	\T=\{ S \: |\:    \eta_{q_k, \lambda_k \: \#}T  \rightharpoonup S    \}, 
\end{equation}
for some convergent sequences $\{q_k\}$ and $\{\lambda_k\}$ with limits $q=\lim_k q_k$ and $\lambda = \lim_k \lambda_k$ where  $0<\lambda_k<1$ and $0 \leq \lambda <1$. We note that $\limsup_k \mathbb{M}_{W}(S_k)< \infty $ in view of the current convergence and the facts that $T$ is integer multiplicity and that the limits are $\lambda$-minimizing from Theorem \ref{TheoremGMTclosure}. Moreover, it is not difficult to see that $\T_{p,\tau}= \T$ whenever $0<\tau<1$. 
 	\\ \indent We now construct a function $\varphi_S$ for any $S\in \T$ that will satisfy the criteria of the reduction argument in \cite[Appendix A]{SimonGMT}. We let $\varphi_S:\rn^{n+\ell}\rightarrow \rn^{n+1}$ be defined by 
 	\begin{equation}
 	\varphi_S^0(p)=\theta_S(p), \qquad \varphi_S^k(p)=\theta_S(p)\xi_S^k(p), \qquad k = 1, \ldots, n+1,
 	\end{equation}
 	where the $\xi_S^k$ is the $k$:th component of the orientation vector $\vec{S}(p)$ of $S$, and let $\F=\{ \varphi_S\: : \: S\in \T  \}$. It follows by the theory in \cite{SimonGMT} that either $\text{sing}(S)= \emptyset$ or 
 	\begin{equation}
 		\text{dim} \big( B_1(0)\cap \text{sing}(S)\big) \leq d,
 	\end{equation}
 	where $d\in \{ 0, \ldots, n-1\}$, for all $S\in \T$. Furthermore it also follows that there is some $S\in \T$ and some $d$-dimensional linear subspace $L$ of $\rn^{n+\ell}$ such that $\text{sing}(S)=L$ and 
 	\begin{equation}
 		\eta_{q,\lambda\:\#} S = S
 	\end{equation}
 	for all $q\in L$ and all $\lambda>0$. Without loss of generality we may assume that $L=\rn^d \times \{0\}$ so that $S=[[R^d]]\times S_0$, where $\partial S_0=0$ and $\text{sing}(S_0)=\{0\}$ and that $S_0$ is minimizing in $\rn^{n+\ell-d}$. Further, the assumption that $T\subset N^{n+1}$ implies that the rescaling gives that $\text{supp}(S)\subset \rn^{n+1}$ (after some orthogonal transformation, if necessary) and so we may assume that $S_0$ is an $(n-d)$-dimensional minimizing cone in $\rn^{n-d+1}$. The singular of this minimizing cone is the origin. The assertion follows from the non-existence of stable minimal hypercones by \cite{SimonNonExistence} in dimension $\leq 6$, so that $\text{sing}(T)=\emptyset$ in case $n\leq 6$. If $n=7$ we obtain that $\text{sing}(T)$ consists of isolated points from the theory in \cite[Theorem 5.8 in Chapter 7]{SimonGMT}.

\end{proof}
 
We state an important result concerning convergence in the case of smooth hypersurfaces. 
 
\begin{lemma}\label{LemmaSmoothConvergence}
	
Let $\{T_k\}\subset \F_\lambda$ be a sequence of integer multiplicity currents that are $\lambda$-minimizing in $N^{n+1}$ and suppose $T_k\rightharpoonup T\in  \F_\lambda$, where $T$ is also $\lambda$-minimizing in $N^{n+1}$. If $T$ and $T_k$ have empty singular sets then there exists a subsequence $\{T_{k'}\}$ that converges to $T$ in $C^{1,\alpha}_{loc}$.
	
\end{lemma}
 
\begin{proof}
 	
Let $p\in \text{supp}(T)$ and let $p_k\in \text{supp}(T_k)$ converge to $p$. Such a sequence $\{p_k\}$ must exist by the current convergence. By assumption, the currents $T$ and $T_k$ are locally the graphs of $C^{1,\alpha}$-functions $f$ and $f_k$, defined on the tangent planes at $p$ and $p_k$, respectively. We can without loss of generality assume that $p =0$ and $\nabla^{\rn^n} f =\vec{0}$. Since the generalized mean curvatures $H_k^T$ of $T_k$ are locally bounded it follows from Allard's Regularity Theorem \cite[Theorem 5.2 in Chapter 5]{SimonGMT} that there exist $0<\rho<1$, $\gamma(n, \ell, 2n) \in(0,1)$ and $\delta\in (0,1/2)$ such that the weighted H\"older norms are uniformly bounded: 
\begin{equation}
	\begin{split}
		& \rho^{-1} \sup_{q\in B_{\gamma\rho}^n(0)}  |f_k(q)| +  \sup_{q\in B_{\gamma\rho}^n(0)} |\nabla^\delta f_k(q)| \\
		& \qquad + \rho^{1-n/p} \sup_{p,q\in B_{\gamma \rho}^n(0), p\not=q}\frac{|\nabla^\delta f_k(p)- \nabla^\delta f_k(q)|}{|p-q|^{\frac{1}{2}}} 	\leq C(n, \ell, 2n)\delta^{\frac{1}{2(n+1)}}.
	\end{split}
\end{equation}
Here $B^{n+\ell}_R(0)$ is the ball in $R^{n+\ell}$ and $B^n_R(0)$ the ball in $\rn^n \times \{p^{n+1}=0, \ldots, p^{n+\ell}=0\}$. Since the tangent spaces converge, $T_{p_k}\text{graph}(f_k)\rightarrow  \rn^n \times \{p^{n+1}=0, \ldots, p^{n+\ell}=0\}$, it is clear that we may write $\text{supp} T_k \cap B^N_{\rho \gamma}(0) = \text{graph}(f_k)$ for large enough $k$, where now the domain of $f_k$ is $\rn^n \times \{p^{n+1}=0, \ldots, p^{n+\ell}=0\}$. The assertion now follows from the Arzela-Ascoli Theorem.

\end{proof}
 
 We end this section by recalling why the regularity results above do not in general hold when $n\geq7$:
 
 \begin{example}\label{ExampleSimonsCone}
 	The following is an example given by \cite{SimonNonExistence}. Consider the set 
 	\begin{equation}
 	C = \bigg\{ (x,y)\in \rn^4\times \rn^4\: \bigg| \: ||x||=||y||   \bigg\},
 	\end{equation}
 	with the induced Euclidean metric $\delta$. Then $C$ is a stable minimal hypercone.
 \end{example}

\bibliographystyle{amsalpha}


\bibliography{PMTah4567}

\providecommand{\bysame}{\leavevmode\hbox to3em{\hrulefill}\thinspace}
\providecommand{\MR}{\relax\ifhmode\unskip\space\fi MR }
\providecommand{\MRhref}[2]{%
  \href{http://www.ams.org/mathscinet-getitem?mr=#1}{#2}
}
\providecommand{\href}[2]{#2}
\begin{thebibliography}{CGNP18}

\bibitem[ACG08]{ACG07}
Lars Andersson, Mingliang Cai, and Gregory Galloway, \emph{Rigidity and
  positivity of mass for asymptotically hyperbolic manifolds}, Ann. Henri
  Poincar\'{e} \textbf{9} (2008), no.~1, 1--33.

\bibitem[ADM59]{ADM}
Richard Arnowitt, Stanley Deser, and Charles Misner, \emph{Dynamical structure
  and definition of energy in general relativity}, Phys. Rev. (2) \textbf{116}
  (1959), 1322--1330. \MR{113667}

\bibitem[AEM11]{AEM11}
Lars Andersson, Michael Eichmair, and Jan Metzger, \emph{Jang's equation and
  its applications to marginally trapped surfaces}, Complex analysis and
  dynamical systems {IV}. {P}art 2, Contemp. Math., vol. 554, Amer. Math. Soc.,
  Providence, RI, 2011, pp.~13--45.

\bibitem[AMO21]{agostiniani2023greens}
Virginia Agostiniani, Lorenzo Mazzieri, and Francesca Oronzio, \emph{A green's
  function proof of the positive mass theorem}, 2021,
  \url{http://arxiv.org/abs/1010.4256}.

\bibitem[Bes08]{Besse}
Arthur Besse, \emph{Einstein manifolds}, Classics in Mathematics,
  Springer-Verlag, Berlin, 2008, Reprint of the 1987 edition.

\bibitem[BK11]{BK11}
Hubert Bray and Marcus Khuri, \emph{P.{D}.{E}.'s which imply the {P}enrose
  conjecture}, Asian J. Math. \textbf{15} (2011), no.~4, 557--610.

\bibitem[BKKS22]{BKKS19}
Hubert Bray, Demetre Kazaras, Marcus Khuri, and Daniel Stern, \emph{Harmonic
  functions and the mass of 3-dimensional asymptotically flat {R}iemannian
  manifolds}, J. Geom. Anal. \textbf{32} (2022), no.~6, Paper No. 184, 29.

\bibitem[BKS21]{BKS}
Edward Bryden, Marcus Khuri, and Christina Sormani, \emph{Stability of the
  spacetime positive mass theorem in spherical symmetry}, J. Geom. Anal.
  \textbf{31} (2021), no.~4, 4191--4239.

\bibitem[CCS16]{CCS16}
Carla Cederbaum, Julien Cortier, and Anna Sakovich, \emph{On the center of mass
  of asymptotically hyperbolic initial data sets}, Ann. Henri Poincar\'{e}
  \textbf{17} (2016), no.~6, 1505--1528.

\bibitem[CGNP18]{Chrusciel2018OnTM}
Piotr Chrusciel, Gregory Galloway, Luc Nguyen, and Tim-Torben Paetz, \emph{On
  the mass aspect function and positive energy theorems for asymptotically
  hyperbolic manifolds}, Classical and Quantum Gravity \textbf{35} (2018).

\bibitem[CH03]{ChruscielHerzlich}
Piotr Chrusciel and Marc Herzlich, \emph{The mass of asymptotically hyperbolic
  {R}iemannian manifolds}, Pacific J. Math. \textbf{212} (2003), no.~2,
  231--264.

\bibitem[CJL04]{CJL04}
Piotr Chrusciel, Jacek Jezierski, and Szymon Leski, \emph{The trautman-bondi
  mass of hyperboloidal initial data sets}, Adv. Theor. Math. Phys. \textbf{8}
  (2004), no.~1, 83--139.

\bibitem[CM06]{CM06}
Piotr Chrusciel and Daniel Maerten, \emph{Killing vectors in asymptotically
  flat space-times. {II}. {A}symptotically translational {K}illing vectors and
  the rigid positive energy theorem in higher dimensions}, J. Math. Phys.
  \textbf{47} (2006), no.~2, 022502, 10.

\bibitem[CMT06]{CMT06}
Piotr Chrusciel, Daniel Maerten, and Paul Tod, \emph{Rigid upper bounds for the
  angular momentum and centre of mass on non-singular asymptotically anti-de
  {S}itter space-times}, J. High Energy Phys. (2006), no.~11, 084, 42.

\bibitem[CN01]{ChruscielNagy}
Piotr Chrusciel and Gabriel Nagy, \emph{The mass of spacelike hypersurfaces in
  asymptotically anti-de {S}itter space-times}, Adv. Theor. Math. Phys.
  \textbf{5} (2001), no.~4, 697--754.

\bibitem[DS93]{DuzaarSteffen93}
Frank Duzaar and Klaus Steffen, \emph{{$\lambda$} minimizing currents},
  Manuscripta Math. \textbf{80} (1993), no.~4, 403--447.

\bibitem[DS21]{DahlSakovichDensityThm}
Mattias Dahl and Anna Sakovich, \emph{A density theorem for asymptotically
  hyperbolic initial data satisfying the dominant energy condition}, Pure Appl.
  Math. Q. \textbf{17} (2021), no.~5, 1669--1710.

\bibitem[EHLS16]{EichmairHuangLeeSchoen}
Michael Eichmair, Lan-Hsuan Huang, Dan Lee, and Richard Schoen, \emph{The
  spacetime positive mass theorem in dimensions less than eight}, J. Eur. Math.
  Soc. (JEMS) \textbf{18} (2016), no.~1, 83--121.

\bibitem[Eic09]{EichmairPlateau}
Michael Eichmair, \emph{The {P}lateau problem for marginally outer trapped
  surfaces}, J. Differential Geom. \textbf{83} (2009), no.~3, 551--583.

\bibitem[Eic13]{EichmairPMT}
\bysame, \emph{The {J}ang equation reduction of the spacetime positive energy
  theorem in dimensions less than eight}, Comm. Math. Phys. \textbf{319}
  (2013), no.~3, 575--593.

\bibitem[Eld13]{Eldering}
Jaap Eldering, \emph{Normally hyperbolic invariant manifolds}, Atlantis Studies
  in Dynamical Systems, vol.~2, Atlantis Press, Paris, 2013, The noncompact
  case. \MR{3098498}

\bibitem[EM16]{EichmairMetzgerJenkinsSerrinType}
Michael Eichmair and Jan Metzger, \emph{Jenkins-{S}errin-type results for the
  {J}ang equation}, J. Differential Geom. \textbf{102} (2016), no.~2, 207--242.

\bibitem[GT01]{GilbargTrudinger}
David Gilbarg and Neil Trudinger, \emph{Elliptic partial differential equations
  of second order}, Classics in Mathematics, Springer-Verlag, Berlin, 2001,
  Reprint of the 1998 edition.

\bibitem[Har02]{HartmanODE}
Philip Hartman, \emph{Ordinary differential equations}, Classics in Applied
  Mathematics, vol.~38, Society for Industrial and Applied Mathematics (SIAM),
  Philadelphia, PA, 2002, Corrected reprint of the second (1982) edition
  [Birkh\"{a}user, Boston, MA; MR0658490 (83e:34002)], With a foreword by Peter
  Bates.

\bibitem[HKK22]{HKK20}
Sven Hirsch, Demetre Kazaras, and Marcus Khuri, \emph{Spacetime harmonic
  functions and the mass of 3-dimensional asymptotically flat initial data for
  the {E}instein equations}, J. Differential Geom. \textbf{122} (2022), no.~2,
  223--258.

\bibitem[HL20]{HuangLeeRigidityI}
Lan-Hsuan Huang and Dan Lee, \emph{Equality in the spacetime positive mass
  theorem}, Comm. Math. Phys. \textbf{376} (2020), no.~3, 2379--2407.
  \MR{4104553}

\bibitem[HL23]{HuangLeeRigidityII}
Lan-Hsuan Huang and Dan~A. Lee, \emph{Equality in the spacetime positive mass
  theorem ii}, 2023.

\bibitem[Jan78]{JangPaper}
Pong~Soo Jang, \emph{On the positivity of energy in general relativity}, J.
  Math. Phys. \textbf{19} (1978), no.~5, 1152--1155.

\bibitem[KP08]{KrantzParksGMT}
Steven Krantz and Harold Parks, \emph{Geometric integration theory},
  Cornerstones, Birkh\"{a}user Boston, Inc., Boston, MA, 2008.

\bibitem[LUY21]{LUY21}
Martin Lesourd, Ryan Unger, and Shing-Tung Yau, \emph{The positive mass theorem
  with arbitrary ends}, 2021.

\bibitem[Mae06]{M06}
Daniel Maerten, \emph{Positive energy-momentum theorem for
  {A}d{S}-asymptotically hyperbolic manifolds}, Ann. Henri Poincar\'{e}
  \textbf{7} (2006), no.~5, 975--1011.

\bibitem[Mey63]{MeyersPDEexpansion}
Norman Meyers, \emph{An expansion about infinity for solutions of linear
  elliptic equations}, J. Math. Mech. \textbf{12} (1963), 247--264.

\bibitem[Mic11]{MichelMassFormalism}
Benoit Michel, \emph{Geometric invariance of mass-like asymptotic invariants},
  J. Math. Phys. \textbf{52} (2011), no.~5, 052504, 14.

\bibitem[MOM04]{MalecMurchada}
Edward Malec and Niall \'{O}~Murchadha, \emph{The {J}ang equation, apparent
  horizons and the {P}enrose inequality}, Classical Quantum Gravity \textbf{21}
  (2004), no.~24, 5777--5787.

\bibitem[Pet16]{Peterson}
Peter Petersen, \emph{Riemannian geometry}, third ed., Graduate Texts in
  Mathematics, vol. 171, Springer, Cham, 2016.

\bibitem[Sak21]{SakovichPMTah}
Anna Sakovich, \emph{The {J}ang equation and the positive mass theorem in the
  asymptotically hyperbolic setting}, Comm. Math. Phys. \textbf{386} (2021),
  no.~2, 903--973.

\bibitem[Sch89]{SchoenNotes}
Richard Schoen, \emph{Variational theory for the total scalar curvature
  functional for {R}iemannian metrics and related topics}, Topics in calculus
  of variations ({M}ontecatini {T}erme, 1987), Lecture Notes in Math., vol.
  1365, Springer, Berlin, 1989, pp.~120--154.

\bibitem[Sim68]{SimonNonExistence}
James Simons, \emph{Minimal varieties in riemannian manifolds}, Ann. of Math.
  (2) \textbf{88} (1968), 62--105.

\bibitem[Sim14]{SimonGMT}
Leon Simon, \emph{Introduction to geometric measure theory}, 2014,
  \url{https://web.stanford.edu/class/math285/ts-gmt.pdf}.

\bibitem[Spr07]{Spruck}
Joel Spruck, \emph{Interior gradient estimates and existence theorems for
  constant mean curvature graphs in {$M^n\times\bold R$}}, Pure Appl. Math. Q.
  \textbf{3} (2007), no.~3, Special Issue: In honor of Leon Simon. Part 2,
  785--800.

\bibitem[SY79a]{PMTV}
Richard Schoen and Shing-Tung Yau, \emph{Complete manifolds with nonnegative
  scalar curvature and the positive action conjecture in general relativity},
  Proc. Nat. Acad. Sci. U.S.A. \textbf{76} (1979), no.~3, 1024--1025.

\bibitem[SY79b]{PMTI}
\bysame, \emph{On the proof of the positive mass conjecture in general
  relativity}, Comm. Math. Phys. \textbf{65} (1979), no.~1, 45--76.

\bibitem[SY81a]{PMTIII}
\bysame, \emph{The energy and the linear momentum of space-times in general
  relativity}, Comm. Math. Phys. \textbf{79} (1981), no.~1, 47--51.

\bibitem[SY81b]{PMTII}
\bysame, \emph{Proof of the positive mass theorem. {II}}, Comm. Math. Phys.
  \textbf{79} (1981), no.~2, 231--260.

\bibitem[SY22]{PMTIV}
\bysame, \emph{Positive scalar curvature and minimal hypersurface
  singularities}, Surveys in differential geometry 2019. {D}ifferential
  geometry, {C}alabi-{Y}au theory, and general relativity. {P}art 2, Surv.
  Differ. Geom., vol.~24, Int. Press, Boston, MA, [2022] \copyright 2022,
  pp.~441--480.

\bibitem[Wan01]{WangPMT}
Xiaodong Wang, \emph{The mass of asymptotically hyperbolic manifolds}, J.
  Differential Geom. \textbf{57} (2001), no.~2, 273--299.

\bibitem[Wit81]{Witten81}
Edward Witten, \emph{A new proof of the positive energy theorem}, Comm. Math.
  Phys. \textbf{80} (1981), no.~3, 381--402.

\bibitem[WX15]{WangXuEM}
Yaohua Wang and Xu~Xu, \emph{Hyperbolic positive energy theorem with
  electromagnetic fields}, Classical Quantum Gravity \textbf{32} (2015), no.~2,
  025007, 20.

\bibitem[XZ08]{XieZhang08}
Naqing Xie and Xiao Zhang, \emph{Positive mass theorems for asymptotically
  {A}d{S} spacetimes with arbitrary cosmological constant}, Internat. J. Math.
  \textbf{19} (2008), no.~3, 285--302.

\bibitem[Zha99]{Zhang99}
Xiao Zhang, \emph{Angular momentum and positive mass theorem}, Comm. Math.
  Phys. \textbf{206} (1999), no.~1, 137--155.

\bibitem[Zha04]{Zhang04}
\bysame, \emph{A definition of total energy-momenta and the positive mass
  theorem on asymptotically hyperbolic 3-manifolds. {I}}, Comm. Math. Phys.
  \textbf{249} (2004), no.~3, 529--548.

\end{thebibliography}



\end{document}